\documentclass[11pt,reqno]{article}
\usepackage{psfrag}
\usepackage{float, graphicx}
\usepackage[]{epsfig}
\usepackage{amsmath, amsthm, amssymb}
\usepackage{mathscinet}
\usepackage{epsfig}
\usepackage{verbatim}
\usepackage{multicol}
\usepackage{url}
\usepackage{latexsym}
\usepackage{mathrsfs}
\usepackage{graphicx}
\usepackage{amsmath}
\usepackage{enumerate}
\usepackage[normalem]{ulem}
\usepackage{bm}
\usepackage{dsfont}
\usepackage{stmaryrd}
\usepackage{enumitem}
\usepackage{cite}
\usepackage{color}
\usepackage{xcolor}
\usepackage[colorlinks=true, bookmarks=true]{hyperref}
 \hypersetup{linkcolor=black, citecolor=green}

\numberwithin{equation}{section}

\usepackage[letterpaper, hmargin=1in, top=1in, bottom=1.2in, footskip=0.6in]{geometry}
\usepackage{titlesec}
\titleformat{\subsection}[runin]{\normalfont\bfseries}{\thesubsection.}{.5em}{}[.]\titlespacing{\subsection}{0pt}{2ex plus .1ex minus .2ex}{.8em}
\titleformat{\subsubsection}[runin]{\normalfont\itshape}{\thesubsubsection.}{.3em}{}[.]\titlespacing{\subsubsection}{0pt}{1ex plus .1ex minus .2ex}{.5em}
\titleformat{\paragraph}[runin]{\normalfont\itshape}{\theparagraph.}{.3em}{}[.]\titlespacing{\paragraph}{0pt}{1ex plus .1ex minus .2ex}{.5em}

\setlist[enumerate]{leftmargin=2em}
\setlist[itemize]{leftmargin=1.5em,label={\tiny$\blacksquare$}}
\newlist{itemizetight}{itemize}{10}
\setlist[itemizetight]{leftmargin=1.5em,itemsep=0pt,label={\tiny$\blacksquare$}}

\usepackage[titles]{tocloft}
\usepackage[labelfont=sc,font=small,labelsep=period]{caption}
\makeatletter\def\SetFigFont#1#2#3#4#5{\small}

\newcommand{\pbb}[1]{\biggl({#1}\biggr)}

\def\cA{{\mathcal A}}
\def\cB{{\mathcal B}}

\def\cH{{\mathcal H}}
\def\cA{{\mathcal A}}

\def\sB{{\sf B}}

\newcommand{\fa}{{\frak a}}
\newcommand{\fb}{{\frak b}}
\newcommand{\fc}{{\frak c}}
\newcommand{\fd}{{\frak d}}
\newcommand{\fe}{{\frak e}}

\newcommand{\cR}{{\mathcal R}}

\newcommand{\cT}{{\mathcal T}}

\newcommand{\be}{\begin{equation}}
\newcommand{\ee}{\end{equation}}

\newcommand{\td}{\tilde}

\newcommand{\cE}{{\cal E}}
\newcommand{\cG}{{\cal G}}

\newcommand{\cC}{\mathbb S}
\newcommand{\bI}{\mathbb I}

\newcommand{\cY}{{\cal Y}}

\newcommand{\bC}{{\mathbb C}}

\newcommand{\T}{\mathbb T}

\newcommand{\bU}{ {\mathbb  U}}

\newcommand{\bT}{\T}

\newcommand{\bB}{{\mathbb B}}
\newcommand{\bX}{{\mathbb X}}
\newcommand{\bG}{{\mathbb G}}
\newcommand{\tbG}{{\tilde\bG}}
\newcommand{\bH}{{\mathbb H}}
\newcommand{\bV}{{\mathbb V}}
\newcommand{\bK}{\mathbb{K}}
\newcommand{\bY}{\mathbb{Y}}

\newcommand{\ii}{\mathrm{i}} 

\newcommand{\deq}{\mathrel{\mathop:}=}

\newcommand{\id}{\mspace{2mu}\mathrm{i}\mspace{-0.6mu}\mathrm{d}} 
\renewcommand{\epsilon}{\varepsilon}
\renewcommand{\leq}{\leqslant}
\renewcommand{\geq}{\geqslant}
\newcommand{\floor}[1] {\lfloor {#1} \rfloor}

\newcommand{\ind}[1]{{\bf 1} (#1)}

\newcommand{\inda}[1]{{\bf 1} \pa{#1}}

\renewcommand{\P}{\mathbb{P}}
\newcommand{\E}{\mathbb{E}}
\newcommand{\R}{\mathbb{R}}
\newcommand{\C}{\mathbb{C}}
\newcommand{\N}{\mathbb{N}}
\newcommand{\Z}{\mathbb{Z}}

\def\bC{{\mathbb C}}

\newcommand{\cX}{\cG_0}

\newcommand{\hcX}{\hcG_0}

\newcommand{\pB}[1]{\Bigl({#1}\Bigr)}
\newcommand{\pa}[1]{\left({#1}\right)}

\renewcommand{\qq}[1]{[\![{#1}]\!]}

\newcommand{\h}[1]{\{{#1}\}}

\newcommand{\ha}[1]{\left\{{#1}\right\}}

\newcommand{\absb}[1]{\bigl\lvert #1 \bigr\rvert}

\newcommand{\absa}[1]{\left\lvert #1 \right\rvert}

\newcommand{\avg}[1]{\langle #1 \rangle}

\DeclareMathOperator{\re}{Re}
\DeclareMathOperator{\im}{Im}

\newcommand{\del}{\partial}

\newcommand{\dist}{\mathrm{dist}}
\newcommand{\I}{Q}
\newcommand{\vE}{\vec{E}}

\newcommand{\bS}{\mathbb{S}}
\newcommand{\cGT}{\cal G^{(\T)}}

\newcommand{\tcG}{\tilde{\cal G}}
\newcommand{\hcG}{\hat{\cal G}}
\newcommand{\tcGT}{\tilde{\cal G}^{(\T)}}
\newcommand{\hcGT}{\hat{\cal G}^{(\T)}}
\newcommand{\GT}{G^{(\T)}}
\newcommand{\tG}{\tilde{G}}
\newcommand{\tGT}{\tilde{G}^{(\T)}}
\newcommand{\hG}{\hat{G}}
\newcommand{\hGT}{\hat{G}^{(\T)}}
\newcommand{\tP}{\tilde{P}}
\newcommand{\ta}{\tilde{a}}

\newcommand{\sG}{{\sf G}}
\newcommand{\sS}{{\sf S}}
\newcommand{\GNd}{\sG_{N,d}}
\newcommand{\GNdp}{\tilde \sG_{N,d}}
\newcommand{\Pp}{\tilde \P}

\newcommand{\n}{\omega}
\newcommand{\msc}{m_{sc}}
\newcommand{\md}{m_d}
\newcommand{\Cw}{C}

\renewcommand{\cal}{\mathcal}

\newcommand{\rn}[1]{%
       \lowercase\expandafter{\romannumeral#1}%
}

\renewcommand{\mod}{\text{ mod }}

\newcommand{\RN}[1]{%
       \uppercase\expandafter{\romannumeral#1}%
}

\renewcommand{\Im}{{\mathrm{Im}}}

\newcommand{\sumdots}{[\,\cdots]}

\theoremstyle{plain} 
\newtheorem{theorem}{Theorem}[section]
\newtheorem*{theorem*}{Theorem}
\newtheorem{lemma}[theorem]{Lemma}
\newtheorem*{lemma*}{Lemma}
\newtheorem{corollary}[theorem]{Corollary}
\newtheorem*{corollary*}{Corollary}
\newtheorem{proposition}[theorem]{Proposition}
\newtheorem*{proposition*}{Proposition}

\newtheorem*{assumption*}{Assumption}
\newtheorem{claim}[theorem]{Claim}

\newtheorem{definition}[theorem]{Definition}
\newtheorem*{definition*}{Definition}

\newtheorem*{example*}{Example}
\newtheorem{remark}[theorem]{Remark}

\newtheorem*{remark*}{Remark}
\newtheorem*{remarks*}{Remarks}

\theoremstyle{definition}

\newcommand{\TE}{\mathrm{TE}}

\begin{document}
\date{}

\title
{Local Kesten--McKay law for random regular graphs 
}
\author{Roland Bauerschmidt\footnote{University of Cambridge, Statistical Laboratory, DPMMS. E-mail: {\tt rb812@cam.ac.uk}.} \and
Jiaoyang Huang\footnote{Harvard University, Department of Mathematics. E-mail: {\tt jiaoyang@math.harvard.edu}.} \and
Horng-Tzer Yau\footnote{Harvard University, Department of Mathematics. E-mail: {\tt htyau@math.harvard.edu}, partially supported by NSF grant DMS-1307444, DMS-1606305 and a Simons Investigator award.}}

\maketitle

\begin{abstract}
We study the adjacency matrices of random $d$-regular graphs with large but fixed degree $d$.
In the bulk of the spectrum $[-2\sqrt{d-1}+\varepsilon, 2\sqrt{d-1}-\varepsilon]$ down to the optimal spectral scale,
we prove that  the 
Green's functions can be  approximated by those  of certain  infinite tree-like (few cycles) graphs  that
depend only on the local structure of the original graphs.
This  result implies that
the Kesten--McKay law holds for the spectral density  down to the smallest  scale
and  the complete delocalization of bulk eigenvectors.
Our method is based on estimating the Green's function of the adjacency matrices and 
a resampling of the boundary edges of large balls in the graphs.  
\end{abstract}

\tableofcontents

\newpage
\section{Main results I: Spectral density and eigenvectors}
\label{sec:intro}

\subsection{Introduction}
\label{sec:intro-intro}

Random regular graphs with fixed degree $d$ are fundamental 
models of sparse random graphs 
and they arise naturally in many different  contexts. %
The spectral properties of their adjacency matrices  are  of particular interest
in computer science,  combinatorics, and statistical physics. The relevant topics include 
the theory of expanders (see e.g.\ \cite{MR2072849}), quantum chaos (see e.g.\ \cite{MR3204183}),
and  graph $\zeta$-functions (see e.g.\ \cite{MR2768284}).
There has been significant progress in the understanding of
the spectra of (random and deterministic) regular graphs.
For \emph{fixed} degree these results generally concern properties of 
eigenvalues and eigenvectors near the macroscopic scale,
and their proofs use the local tree-like structure of these graphs
as an important input.
On the other hand, dense regular graphs belong
to the random matrix universality class and their spectral properties 
are known to resemble those of Wigner matrices.
In this paper, we introduce an approach that allows 
the Green's function method of random matrix theory to make use of 
the local tree-like structure of the random regular graph,
while it also captures  key random matrix behavior.

\smallskip

Throughout the paper, $A=A(\cG)$ is the adjacency matrix of a (uniform) random $d$-regular graph $\cG$ on $N$ vertices.
Thus $A$ is uniformly chosen among all symmetric $N\times N$ matrices with entries in $\{0,1\}$ with
$\sum_j A_{ij}=d$ and $A_{ii}=0$ for all $i$.
Note that $A$ has the \emph{trivial} constant eigenvector with eigenvalue $d$.
We also use the rescaled adjacency matrix $H=A/\sqrt{d-1}$,
and we denote the set of (simple) $d$-regular graphs on $N$ vertices by $\GNd$.

Below we first discuss some known consequences of the tree-like and of the random matrix-like structure.

\paragraph{Tree-like structure}
It is well known that most regular graphs of a fixed degree $d\geq 3$ are \emph{locally tree-like}
in the sense that: (i) for any fixed radius $R$ (and actually for $R=c\log_{d-1} N$),
the radius-$R$ 
neighborhoods of almost all vertices are the same as those in the infinite $d$-regular tree;
(ii) the $R$-neighborhoods of all vertices have bounded excess,
which is the smallest number of edges that must be removed to yield a tree;
see e.g.\ Proposition~\ref{prop:structure} below.
The tree-like structure is important for the following results,
valid in general for deterministic graphs and in some cases requiring randomness as well.
\begin{enumerate}
\item
For regular graphs with locally tree-like structure,
the macroscopic spectral density of $A$ converges to the Kesten--McKay law \cite{MR0109367,MR629617}, 
characterized by the density 
$\frac{d}{d^2-x^2}\frac{1}{2\pi}\sqrt{[4(d-1)-x^2]_+}$.
For random regular graphs, the Kesten--McKay law was established  on spectral scales $(\log N)^{-c}$ \cite{MR3025715,MR3433288,MR3322309} by 
using the fact that the locally tree-like structure holds with high probability in neighborhoods of  radius
$ \Omega( \log N)$.

\item
For regular graphs with locally tree-like structure,
the eigenvectors $v$ of $A$ are weakly delocalized:
their entries are uniformly bounded by $(\log N)^{-c}\|v\|_2$ \cite{MR3025715,MR3433288,MR3038543} and their $\ell^2$-mass cannot concentrate on a small set \cite{MR3038543}.
If, in addition, the graphs are expanders,  
the eigenvectors of $A$ also satisfy the quantum ergodicity property \cite{MR3322309,1505.03887,1512.06624}.

\item
For random regular graphs
using the locally tree-like structure  as important input,
for any fixed $\varepsilon>0$,
the nontrivial eigenvalues of $A$ are contained in $[-2\sqrt{d-1}-\varepsilon, 2\sqrt{d-1}+\varepsilon]$.
This was conjectured in \cite{MR875835} and proved in \cite{MR2437174}; see also \cite{MR3385636,Bord15} for recent alternative arguments.
It was also shown that the scale $\varepsilon$ can actually be taken to be  $(\log N)^{-c}$ in \cite{Bord15}.
\end{enumerate}

\paragraph{Random matrix-like structure} For random matrices of Wigner type, 
precise estimates on the spectral properties of these matrices were proved (see e.g.,  \cite{MR2810797,MR2784665,1504.03605,ErdYauBook}):
\begin{enumerate}
\item
The spectral density in the bulk is given by the semicircle law on all scales larger than $N^{-1}$.
\item
The eigenvectors are uniformly bounded in $\ell^\infty$-norm by $N^{-1/2}$ (up to logarithmic correction).
\item
The extremal eigenvalues are concentrated on scale $N^{-2/3}$.
\item Both bulk and edge universality holds; in particular, the  distributions 
of the extremal eigenvalues are the same as  those of Gaussian matrix ensembles
(Tracy--Widom distributions).
\end{enumerate}
The first three properties usually can be proved via estimates on the Green's function;
the proofs of   universality involve Dyson Brownian Motion or  other comparison methods
(see \cite{ErdYauBook} for a review).

For random $d$-regular graphs with $d\in [N^\varepsilon, N^{2/3-\varepsilon}]$, properties (i), (ii),
and also bulk universality
were proved in \cite{1503.08702-cpam,1505.06700-aop}
(the lower bound on $d$ can be relaxed to $(\log N)^4$ for properties (i) and  (ii)).
Simulations indicate that (i)-(iv)  hold for random regular graphs of fixed degree \cite{MR2433888,MR1691538,MR2647344,MR2218022}.

In this paper, we consider random regular graphs of large but fixed degree~$d$.
One of our key ideas to  prove the  properties (i) and (ii) is to use switchings to 
resample the boundaries of large balls (see Section~\ref{sec:switch}).
This operation preserves 
the local tree-like structure and it also captures sufficient global structure in random regular graphs.
This  resampling generalizes  and adds a geometric component to
the  local resampling method we introduced with A.\ Knowles in \cite{1503.08702-cpam}
for random regular graphs with $d \gg \log N$.
The idea of using some form of switchings in studying random  regular graphs 
goes back at least  to \cite{MR790916}, where it was used  in the enumeration of such graphs;
see also \cite{MR1725006} for further applications in enumeration. 
Finally, to analyze the propagation of the boundary  effect  to the interior of the ball in
the Green's function, we explicitly compute the Green's function of  the tree-like graphs.

\paragraph{Notation} \label{sec:asymp-notation}

For two quantities $X$ and $Y$ depending on $N$, we use the notations
$X=O_{\leq }(Y)$ if $Y$ is positive and $|X|\leq Y$;
$X=O(Y)$ if $X,Y$ are positive and there exists some universal constant $C$ such that $X\leq CY$; $X=o(Y)$,
$X\ll Y$ or $Y\gg X$ if $Y$ is positive and $\lim_{N\rightarrow \infty}X/Y=0$;
$X=\Omega(Y)$ if $X,Y$ are positive and $\liminf_{N\rightarrow \infty}X/Y>0$.
We write $\qq{a,b} = [a,b]\cap\Z$ and $\qq{N} = \qq{1,N}$.

\subsection{Spectral density and eigenvector delocalization}

Our main result, Theorem~\ref{thm:mr},
is a precise estimate  on the local profile of the Green's function down to the smallest possible spectral scales, with high probability. 
Its statement requires several definitions, and we therefore only state it in Section~\ref{sec:main}.
In the remainder of this section, we state some direct consequences of Theorem~\ref{thm:mr}, which can be stated in elementary terms.
The proofs of these corollaries are given in Section~\ref{sec:main-pfcor}.

\subsubsection{Spectral density}

With high probability, the spectral measure of
the rescaled adjacency matrix  $H=A/\sqrt{d-1}$ converges weakly to the rescaled Kesten--McKay law
with  density  given by
\begin{align}\label{kmlaw}
\rho_{d}(x)=\left(1+\frac{1}{d-1}-\frac{x^2}{d}\right)^{-1}\frac{\sqrt{[4-x^2]_+}}{2\pi}.
\end{align}
This convergence can be expressed as $m(z) = m_d(z) + o(1)$ for any $z \in \C_+$ independent of $N$,
where $m_d(z)$ is the Stieltjes transform of $\rho_d$,
and $m(z)$ is the Stieltjes transform of the empirical spectral measure of $H$,
\begin{equation}
  m_d(z) = \int \frac{1}{\lambda-z} \, \rho_d(\lambda) \, d\lambda, \qquad
  m(z) = \frac{1}{N} \sum_{j=1}^N \frac{1}{\lambda_j-z},
\end{equation}
and $\C_+ = \{z \in \C : \im [z] > 0\}$ is the upper half-plane.
The imaginary part of the spectral parameter $z \in \C_+$ determines the scale of the convergence.
In particular, the convergence $m(z) \to m_d(z)$ for all fixed $z$ corresponds to the convergence on the macroscopic scale,
i.e., for intervals containing order $N$ eigenvalues.
The following theorem gives the convergence on the optimal mesoscopic scale $\im[z] \gg 1/N$,
away from the spectral edges at $\pm 2$.

\begin{theorem}[Local Kesten--McKay Law] \label{thm:localmckay}
Fix $\alpha>4$, $\n\geq 8$ and $\sqrt{d-1}\geq (\n+1)2^{2\omega+45}$.
Then with probability $1-o(N^{-\n+8})$  with respect to the uniform measure on $\GNd$,
\begin{equation} \label{e:localmckay}
  \left|m(z) - \md(z)\right| = O(\log N)^{-\alpha} 
\end{equation}
uniformly for
\begin{equation}
\label{spectralD}
z \in \cal D \deq \ha{z\in \C_+:
\Im [z]\geq  \frac{(\log N)^{48\alpha+1}}{N},\quad |z\pm 2|\geq (\log N)^{-\alpha/2+1}}.
\end{equation}
\end{theorem}

While Theorem~\ref{thm:localmckay} shows that the spectral density (or its Stieltjes transform, which is the trace of the Green's function) does concentrate, 
the individual entries of the Green's function of the random regular graph with bounded degree do \emph{not} concentrate;
see also Remark~\ref{rk:Giinonconc} below.
This is different from the typical examples in random matrix theory,
and it is one of the reasons that  the fixed degree graphs require a more delicate analysis.
For example, the random regular graph contains a triangle with probability uniformly bounded from below.
For graphs with bounded degree, triangles and other short cycles have a strong local influence on the
elements of the Green's function, and thus the spectrum.

The spectral density of random regular graphs at scales much larger than the typical eigenvalue spacing
has been studied in \cite{MR2999215,MR3025715,MR3433288,MR3322309}. 
Results for  spectral density near the typical eigenvalue spacing only appeared very recently \cite{1503.08702-cpam}, 
where  the semicircle law down to the optimal mesoscopic scale was established 
for degree $d\in [\xi^4, N^{2/3}\xi^{-2/3}]$ with $\xi=(\log N)^2$.
The methods of the current paper could be extended from fixed $d$ to $d$ growing slowly with $N$,
for example to the range $d_0 \leq d \leq (\log N)^4$ beyond which the results of \cite{1503.08702-cpam} apply.
Thus the results of this paper complement those of \cite{1503.08702-cpam}.
For simplicity, we restrict this paper to the most interesting case of fixed degree $d$.

\subsubsection{Eigenvectors}

Theorem~\ref{thm:mr} implies delocalization estimates of the eigenvectors in the bulk of the spectrum.

\begin{theorem}[Eigenvector delocalization]\label{thm:delocalizationev}
Fix $\alpha>4$, $\n\geq 8$ and $\sqrt{d-1}\geq (\n+1)2^{2\omega+45}$.
Then, with probability $1-o(N^{-\n+8})$ with respect to the uniform measure on $\GNd$,
the eigenvectors $v$ of $H$ whose eigenvalue $\lambda$ obeys $|\lambda \pm 2| \geq (\log N)^{1-\alpha/2}$
are simultaneously delocalized:
\begin{equation}\label{evest}
  \|v\|_\infty \leq 
\frac{\sqrt{2}(\log N)^{24\alpha+1/2}}{\sqrt{N}} \|v\|_2.
\end{equation}
\end{theorem}

Theorem~\ref{thm:delocalizationev} shows that with probability $1-o(N^{-\omega+8})$,
the eigenvectors are completely delocalized.
On the other hand, it is easy to see that, with probability $\Omega(N^{-d+2})$,
the random $d$-regular graph has a localized eigenvector (see Figure~\ref{fig:fork}).
In particular, \eqref{evest} cannot hold with probability higher than polynomial in $1/N$.
Moreover, the Erd\H{o}s--R\'enyi graph with finite average degree $d$ has localized
eigenvectors with probability $\Omega(1)$.
Thus \eqref{evest} with probability tending to $0$ is false
for the Erd\H{o}s--R\'enyi graph with finite average degree $d$.

\begin{figure}
\centering
\input{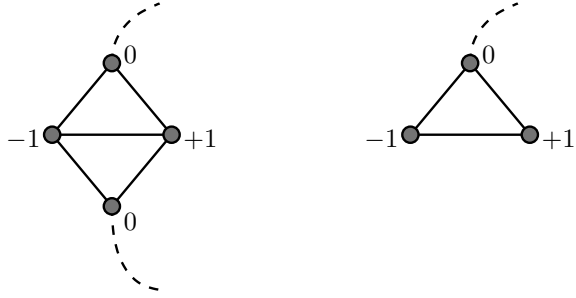}
\caption{
Theorem~\ref{thm:delocalizationev} shows
that a random $d$-regular graph has only completely delocalized eigenvectors with probability $1-o(N^{-\n+8})$.
On the other hand, it is not difficult to show that a random $d$-regular graph has localized eigenvectors with probability $\Omega(N^{-d+2})$. For example,
a random $3$-regular graph contains the subgraph shown on the left with probability $\Omega(N^{-1})$.
For comparison, also notice that an Erd\H{o}s--R\'enyi graph with finite average degree contains localized eigenvectors
with probability $\Omega(1)$; see the right figure.
\label{fig:fork}}
\end{figure}

The delocalization of eigenvectors of (random and deterministic) regular graphs has been studied in
\cite{MR2999215,MR3025715,MR3433288,MR3322309,MR3245884,MR3038543,1505.03887,1512.06624}
(see also \cite{1601.03678} for a survey of results on eigenvector delocalization in random matrices).
Our result implies the optimal bound of order $1/\sqrt{N}$
(up to logarithmic corrections) on the $\ell^\infty$-norms of the (bulk) eigenvectors
of random regular graphs.

For (deterministic) locally tree-like regular graphs,
it was previously proved that the eigenvectors $v$ are weakly delocalized in the sense that $\|v\|_\infty \leq (\log N)^{-c}\|v\|_2$ \cite{MR3025715,MR3433288,MR3038543},
and that eigenvectors cannot concentrate on a small set, in the sense that any vertex set $\bV \subset \qq{N}$ with $\sum_{i\in \bV} |v_i|^2 \geq \varepsilon \|v\|_2$
must have at least $N^{c(\varepsilon)}$ elements \cite{MR3038543}.
Moreover, for deterministic locally tree-like  regular expander graphs,
it was proved that the eigenvectors $v$ satisfy a quantum ergodicity property:
for all $a \in \R^N$ with $\|a\|_\infty\leq 1$ and $\sum_i a_i=0$,
averages of $|\sum_i a_i v_i^2|^2$ over many eigenvectors $v$
are close to $0$ \cite{MR3322309,1505.03887,1512.06624}. 

Theorem~\ref{thm:delocalizationev} and the exchangeability of the random
regular graph also imply the following
\emph{isotropic} version of Theorem~\ref{thm:delocalizationev},
implying that the eigenvectors are delocalized not only in the standard basis,
but in any deterministic orthonormal basis. In addition,  a
\emph{probabilistic} version of the quantum unique ergodicity property (QUE)
holds for these graphs. 
Note that estimates \eqref{e:QUE}, \eqref{e:flat} are not uniform
over \emph{all} $q$ or $\bX$.
Therefore $q$ and $\bX$ cannot be chosen depending on the random graph.

\begin{corollary}
\label{c:que}
Under the assumptions of Theorem~\ref{thm:delocalizationev}, 
the following estimates hold with probability $1-o(N^{-\n+8})$
with respect to the uniform measure on $\GNd$.
For any deterministic $q^1, \dots, q^M \in \R^N$ with $\|q^m\|_2 = 1$ and $\sum_i q^m_i = 0$
($M$ can depend on $N$),
and for all normalized eigenvectors $v^k$ whose eigenvalue $\lambda_k$
obeys $|\lambda_k \pm 2| \geq (\log N)^{1-\alpha/2}$,
we have:
\begin{enumerate}
\item
(Isotropic delocalization)
The eigenvectors are delocalized in directions $q^m$:
\begin{align}
  \max_{k,m}
  |\avg{q^m,v^k}| \leq 
  \frac{(\log N)^{24\alpha+1/2}(\log N+\log M)^2}{\sqrt{N}}.
\end{align}
\item
(Probabilistic QUE)
The eigenvector densities are flat with respect to the test vectors $q^m$:
\begin{equation} \label{e:QUE}
  \max_{k,m} \absa{\sum_{i = 1}^N q^m_i (v_i^k)^2}
  \leq \frac{(\log N)^{48\alpha+1}(\log N+\log M)^2}{N}
\end{equation}
In particular, with probability $1-o(N^{-\n+8})$,
simultaneously for any deterministic index sets $\bX^1, \dots, \bX^M \subset\qq{N}$,
and all eigenvectors $v^k$ with $|\lambda_k \pm 2| \geq (\log N)^{1-\alpha/2}$,
\begin{equation}\label{e:flat}
  \sum_{i\in \bX^m} (v_i^k)^2
  = \frac{|\bX^m|}{N}
  + O\left(\frac{(\log N)^{48\alpha+1} (\log N+\log M)^2\sqrt{|\bX^m|}}{N}\right).
\end{equation}
\end{enumerate}
\end{corollary}

The proof of Corollary~\ref{c:que}
makes strong use of the exchangeability of the random regular graph.
On the other hand, the proof of Theorem~\ref{thm:mr},
and its consequences Theorem~\ref{thm:localmckay} and Theorem~\ref{thm:delocalizationev},
do not exploit exchangeability
in a significant way, and we believe that the method could be extended, for example, to graphs with more general degree sequences.

\subsection{Related results}
\label{sec:intro-related}

Macroscopic eigenvalue statistics for random regular graphs of fixed degree have been studied
using the techniques of Poisson approximation of short cycles \cite{MR3078290,MR3315475} and (non-rigorously) using the replica method \cite{PhysRevE.90.052109}.
These results show that the macroscopic eigenvalue statistics for random regular graphs of fixed degree
are different from those of a Gaussian matrix.
However, this is not predicted to be the case for the \emph{local} eigenvalue statistics.
Spectral properties of regular directed graphs have also been studied recently \cite{1508.00208,Cook2015}.

The second largest eigenvalue $\lambda_2$ of regular graphs is of particular interest.
For the case of fixed degree, see in particular \cite{MR2437174,MR3385636,Bord15,FriedmanKahnSzemeredi,1510.06013}.
The conjecture that the distribution of the second largest eigenvalue on scale $N^{-2/3}$ is the
same as that of the largest eigenvalue of the Gaussian Orthogonal Ensemble \cite{MR2072849}
would imply that slightly more than half of all regular graphs are Ramanujan graphs,
namely $d$-regular graphs with $\lambda_2 \leq 2\sqrt{d-1}$
(for explicit and probabilistic constructions of sequences of Ramanujan graphs, see \cite{MR963118,MR939574,MR3374962}).
The spectrum of random regular graphs has also received interest from the study of $\zeta$-functions,
as it can be related by an exact relationship to the poles of the Ihara $\zeta$-function of regular graphs
\cite{MR0223463,MR1194071}; see also \cite{MR2768284,MR1967891}.

Another interesting direction related to the spectral properties of
random regular graphs concerns  the phase diagram of the Anderson model.
The model was originally defined on the square lattice $\Z^d$,
but only limited progress was made for the delocalization problem in this setting.
A simplified model
on the infinite regular tree (Bethe lattice) is  well-understood 
\cite{MR1492789,MR2215610,MR2257129,MR2259096,MR3405613,MR3055759,MR2905787,MR2885163,MR2329431,MR2241809};
see also \cite{MR3364516} for a review.
At large disorder, it is known that the Anderson model on the random regular graph exhibits Poisson statistics \cite{MR3369316}.
The eigenstates of the Anderson model on the random regular graph have also been studied in connection with
many-body localization \cite{PhysRevLett.113.046806,1401.0019}.

In random matrix theory, the local spectral statistics of the generalized Wigner matrices are  well
understood; see in particular
 \cite{MR1810949,MR2810797,MR2639734,MR2662426,MR2981427,MR2784665,MR3372074,MR3541852,MR2964770, ErdYauBook}.
Many results on local eigenvalue statistics also exist for Erd\H{o}s-R\'enyi random graphs, in particular \cite{MR3098073,MR2964770,1504.05170,1510.06390};
the latter results apply down to logarithmically small average degrees.
Similar results have also been proved for more general degree distributions \cite{1509.03368,AEK16}.
However, these types of results are false for the Erd\H{o}s--R\'enyi graph with \emph{bounded} average degree.
For a review of other results for discrete random matrices, see also \cite{MR2432537}.
For the eigenvectors of random regular graphs with $d \in [N^\varepsilon, N^{2/3-\varepsilon}]$,
the asymptotic normality was proved in \cite{1609.09022};
see also the prior results for generalized Wigner matrices \cite{MR3034787,MR2930379,BY2016}.
For random regular graphs of fixed degree,
a Gaussian wave correlation structure for the eigenvectors was predicted in \cite{0907.5065}
and partially confirmed in \cite{1607.04785}.

\section{Main results II: Local approximation of the Green's function}
\label{sec:main}

\subsection{Graphs}

The main result of this paper, Theorem~\ref{thm:mr} below, is a precise local approximation result of the Green's function,
which in particular implies the results stated in Section~\ref{sec:intro}.
To state the main result, we require several definitions, which we give below.

\paragraph{Graphs, adjacency matrices, Green's functions}

Throughout this paper, graphs $\cG$ are always simple (i.e., have no self-loops or multiple edges) and have vertex degrees at most $d$
(non-regular graphs are also used).
The geodesic distance (length of the shortest path between two vertices) in the graph $\cG$ is denoted by $\dist_{\cG}(\cdot, \cdot)$.
For any graph $\cG$, the \emph{adjacency matrix} is the (possibly infinite) symmetric matrix $A$ indexed by the vertices of the graph,
with $A_{ij}= A_{ji} = 1$ if there is an edge between $i$ and $j$, and $A_{ij}=0$  otherwise.
Throughout the paper, we denote the \emph{normalized adjacency matrix} by $H=A/\sqrt{d-1}$,
where the normalization by $1/\sqrt{d-1}$
is chosen independently of the actual degrees of the graph.
Moreover, we denote the (unnormalized) adjacency matrix of a directed edge $(i,j)$ by $e_{ij}$, i.e.\ $(e_{ij})_{kl} = \delta_{ik}\delta_{jl}$.
The \emph{Green's function} of a graph $\cG$ is the unique matrix $G = G(z)$ defined by $G(H-z)=I$ for $z \in \C_+$, where $\C_+$ is the upper half plane.

In Appendix~\ref{app:Green}, several well-known properties of Green's function are summarized;
they will be used throughout the paper.
The Green's function $G(z)$ encodes all spectral information of $H$ (and thus of $A$).
In particular, the \emph{spectral resolution} is given by $\eta = \im [z]$:
the \emph{macroscopic} behavior corresponds to $\eta$ of order $1$,
the \emph{mesoscopic} behavior to $1/N \ll \eta \ll 1$,
and the \emph{microscopic} behavior of individual eigenvalues corresponds
to $\eta$ below $1/N$.

\paragraph{Subsets and Subgraphs}

Let $\cG$ be a graph, and denote the set of its edges by the same symbol $\cG$ and its vertices by $\bG$.
More generally, throughout the paper, we use blackboard bold letters for set or subsets of vertices, and calligraphic letters for graphs or subgraphs.
For any subset $\bX\subset \bG$, we define the graph $\cG^{(\bX)}$ by removing the vertices $\bX$ and edges adjacent to $\bX$ from $\cG$,
i.e., the adjacency matrix of $\cG^{(\bX)}$ is the restriction of that of $\cG$ to $\bG \setminus \bX$.
We write $G^{(\bX)}$ for the Green's function of $\cG^{(\bX)}$.
For any subgraph $\cal X \subset \cG$, we denote by $\del \cal X=\{v\in \bG: \dist_{\cG}(v,\cal X)=1\}$
the {\it vertex boundary} of $\cal X$ in $\cG$,
and by $\partial_E \cal X = \{e\in \cG: \text{$e$ is adjacent to $\cal X$ but $e\not\in \cal X$}\}$ the {\it edge boundary} of $\cal X$ in $\cG$.
Moreover, for any subset $\bX \subset \bG$, we denote by $\del\bX$ and $\del_E \bX$ the vertex and edge boundaries
of the subgraph induced by $\cG$ on $\bX$.

\paragraph{Neighborhoods}

Given a subset $\bX$ of the vertex set of a graph $\cG$ and an integer $r>0$,
we denote the $r$-neighborhood of $\bX$ in $\cG$ by $\cB_r(\bX,\cG)$,
i.e., it is the subgraph induced by $\cG$ on the set $\bB_r(\bX,\cG) = \h{j \in \cal G: \dist_{\cal G}(\bX,j) \leq r}$.
In particular $\cB_r(i,\cG)$ is the radius-$r$ neighborhood of the vertex $i$.

Moreover, given vertices $i,j$ in $\cal G$ and $r >0$, we denote by $\cE_r(i,j,\cG)$ the smallest subgraph of $\cG$ 
that contains all paths of length at most $r$ between $i$ and $j$. Namely,
\begin{equation} \label{e:Erdef}
\cE_r(i,j,\cG):=\{e\in \cG: \text{there exists a path from $i$ to $j$ of length at most $r$ containing $e$} \}.
\end{equation}
Notice that $\cE_{2r}(i,j,\cal G)\subset \cB_{r}(i,\cal G)\cup \cB_{r}(j,\cal G)$.

\paragraph{Trees}

The infinite $d$-regular tree is the unique (up to isometry) infinite connected $d$-regular graph without cycles, and is denoted by $\cY$.
The rooted $d$-regular tree with root degree $c$ is the unique (up to isometry) infinite connected graph that is $d$-regular
at every vertex except for a distinguished root vertex $o$, which has degree $c$.

\subsection{Tree extension}
\label{sec:intro-TE}

The local approximation of the Green's function of a graph will be defined
in terms of the \emph{tree extension}, defined next.

\begin{definition}[deficit function]
Given a graph $\cG$ with vertex set $\bG$ and degrees bounded by $d$,
a \emph{deficit function} for $\cG$ is a function
$g: \bG \to \qq{0,d}$ satisfying $\deg_{\cG} (v) \leq d-g(v)$ for all vertices $v\in \bG$.
We call a vertex $v\in \bG$ \emph{extensible} if $\deg_\cG(v)<d-g(v)$.
\end{definition}

\begin{figure}[t]
\centering
\input{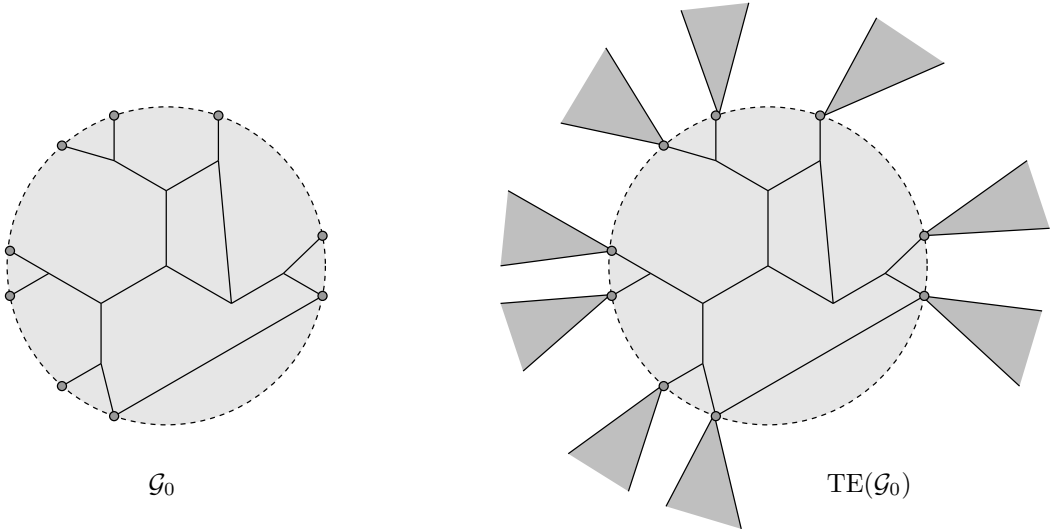}
\caption{The left figure illustrates a finite graph $\cG_0$; its extensible vertices are shown as grey circles.
The right figure shows the tree extension $\TE(\cG_0)$, in which a rooted tree (darkly shaded) is attached to every extensible vertex.
\label{fig:TE}}
\end{figure}

\begin{definition}[tree extension] \label{def:treeextension}
Let $\cG_0$ be a finite graph with deficit function $g$.
\begin{enumerate}
\item
The \emph{tree extension} (abbreviated $\TE$) of $\cG_0$ is the (possibly infinite) graph $\TE(\cG_0)$ defined
by attaching to any extensible vertex $v$ in $\cG_0$
a rooted $d$-regular tree with root degree $d-g(v)-\deg_{\cG_0}(v)$.
\item
The Green's function of $\cG_0$ \emph{with tree extension}, denoted $P(\cG_0)$,
is the Green's function of the (possibly infinite) graph $\TE(\cG_0)$.
\end{enumerate}
\end{definition}

See Figure~\ref{fig:TE} for an illustration of the tree extension.
In our main result, stated in Section~\ref{sec:intro-mr}, we approximate the Green's function of a regular graph at vertices $i,j$
by that of the tree extension of a neighbourhood of $i,j$. This requires specification
of a deficit function, which we will usually do using the following conventions for deficit functions,
assumed throughout the paper.

\paragraph{Conventions for deficit functions}

Throughout this paper, all graphs $\cG$ are equipped with a deficit function $g$.
The interpretation of the deficit function $g(v)$ is that it measures the difference to the desired degree of the vertex $v$.
We use the following conventions for deficit functions.
\begin{itemize}
\item
If the deficit function of $\cG$ is not specified explicitly, it is given by $g(v) = d-\deg_{\cG}(v)$.

Thus no vertex is extensible and the tree extension of $\cG$ is trivial: $\cG = \TE(\cG)$.
\item
If $\bX$ is a subset of the vertices of $\cG$, and $g$ is the deficit function of $\cG$,
then the deficit function $g'$ of $\cG^{(\bX)}$ is given by $g'(v) = g+\deg_{\cG}(v)-\deg_{\cG^{(\bX)}}(v)$,
unless specified explicitly.

Thus when removing the edges incident to $\bX$ from $\cG$,
these are also absent in the tree extension.
\item
If $\cH \subset \cG$ is a subgraph (which was not obtained as $\cG^{(\bX)}$),
then the deficit function of $\cH$ is given by the restriction of the deficit function of $\cG$ on $\cH$,
unless specified explicitly.

Thus any vertex $v$ in $\cH \subset \cG$ has the same degree in the tree extension $\TE(\cH)$ as in $\TE(\cG)$.
\end{itemize}
The above conventions are illustrated in Figure~\ref{fig:TE23}.
In particular, in the case that $\cG$ is a $d$-regular graph, the deficit function is always $g\equiv 0$,
so that $\TE(\cG) = \cG$.
Moreover, by our conventions, the tree extension of a subgraph $\cH \subset \cG$ is again a $d$-regular graph.

\begin{figure}[h]
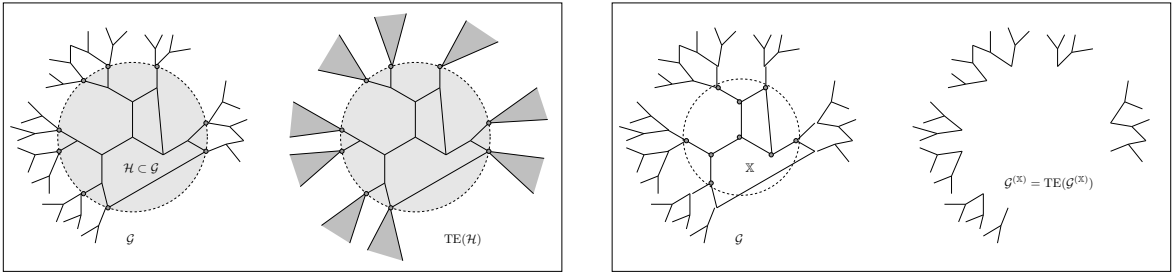

\centering
\scalebox{0.47}{\input{TE2.pspdftex}}
\hspace{1em}
\scalebox{0.47}{\input{TE3.pspdftex}}

\caption{
Given a graph $\cG$ (with the standard deficit function $g=d-\deg_{\cG}$),
the left figure illustrates a subgraph $\cH\subset \cG$, which by our conventions inherits its deficit function from $\cG$ by restriction.
Thus all vertices in $\cH$ have the same degrees in the tree extension $\TE(\cH)$ as in $\cG = \TE(\cG)$.
The right figure illustrates the graph $\cG^{(\bX)}$ obtained by removing a vertex set $\bX$.
By our convention on the deficit function, the tree extension of $\cG^{(\bX)}$ is then trivial.
\label{fig:TE23}}
\end{figure}

\begin{definition}
Given an integer $r>0$,
we call $P_{ij}(\cE_r(i,j,\cG))$ the \emph{localized Green's function} of $\cG$
at vertices $i,j$.
\end{definition}

Thus the localized Green's functions at $i,j$ is the Green's function of a graph that itself
depends on a small neighborhood of $i,j$. However, the dependence of the graph on $i,j$ is weak, in the sense that,
up to a small error, 
the graph $\cE_r(i,j,\cG)$ could be replaced by any neighborhood of $i,j$ that is not too small and not too large;
see Proposition~\ref{boundPij} and Remark~\ref{rk:princG0}.

In our main result, stated in Section~\ref{sec:intro-mr} below,
we will show that the Green's function $G_{ij}(\cG)$ can be approximated by the localized Green's function
$P_{ij}(\cE_r(i,j,\cG))$.
To interpret this result, we note the following elementary properties 
of the localized Green's function.
\begin{itemize}
\item
If $\dist(i,j) > r$, then $\cE_r(i,j,\cG)$ is the empty graph, and therefore $P_{ij}(\cE_r(i,j,\cG),z) = 0$.
\item
If $\cE_r(i,j,\cG)$ has no cycles (thus it is a tree), then $\TE(\cE_r(i,j,\cG))$ is an infinite tree.
In particular, if $\cG$ is $d$-regular, then $\TE(\cE_r(i,j,\cG))$ is the infinite $d$-regular tree $\cY$,
and therefore $P_{ij}(\cE_r(i,j,\cG),z) = G_{ij}(\cY,z)$. By a straightforward calculation (see Section~\ref{sec:TE}),
it then follows that 
\begin{equation} \label{e:Gtreemkm-bis}
  P_{ij}(\cE_r(i,j,\cG),z)
  = 
  G_{ij}(\cY, z)
  =
  \md(z)\left(-\frac{\msc(z)}{\sqrt{d-1}}\right)^{\dist_{\cY}(i,j)},
\end{equation}
where $\md$ and $\msc$ are the Stieltjes transforms of the Kesten--McKay and semicircle laws; see \eqref{e:mdmsc} below.
\item
If $\cE_r(i,j,\cG)$ has bounded excess, then upper bounds similar to the right-hand side of \eqref{e:Gtreemkm-bis} hold.
In particular, $P_{ij}(\cE(i,j,r),z)$ is uniformly bounded in $z \in \C_+$ and decays exponentially in the distance
with rate $\log ( |m_{sc}(z)|/\sqrt{d-1})$ (see Section~\ref{sec:TE}).
\end{itemize}

\paragraph{Kesten--McKay and semicircle law} 
Throughout this paper, the Stieltjes transforms of the Kesten--McKay law and that of the
closely related semicircle law play an important role.
Let $\rho_{d}(x)$ be the density of the (normalized) Kesten--McKay law \eqref{kmlaw}
and $\rho_{sc}(x):= \frac{1}{2\pi}\sqrt{[4-x^2]_+}$ that of Wigner's semicircle law.
We denote their Stieltjes transforms by 
\begin{equation} \label{e:mdmsc}
\md(z)=\int \frac{\rho_{d}(x)}{x-z} \, dx,
\quad
\msc(z):=\int \frac{\rho_{sc}(x)}{x-z} \, dx,
\quad z\in \C_+.
\end{equation}
Then $\md(z)$ is explicitly related to $\msc(z)$ by the equation (see also Proposition~\ref{greentree})
\begin{equation} \label{e:mdstieltjes}
\md(z)=\frac{1}{-z-d(d-1)^{-1}\msc(z)}=\frac{\msc(z)}{1-({d-1})^{-1}\msc^2(z)}.
\end{equation}
Moreover, it is well known that $\msc(z)$ is a holomorphic bijection from the upper half plane $\C_+$ to the upper half unit disk
${\mathbb D}_+:= \{z\in \C_+: |z|< 1\}$, and that it satisfies the algebraic equation
\begin{equation}
z=-\left(\msc(z)+\frac{1}{\msc(z)}\right),\quad z\in \C_+,
\end{equation} 
and in particular that $|\msc(z)| \leq 1$.

\subsection{Main result} \label{sec:intro-mr}

Recall that $\GNd$ denotes the set of simple $d$-regular graphs on the vertex set $\qq{N}$.
Throughout the paper, we control error estimates in terms of (large powers of) the parameter
\begin{equation}  \label{e:qdef}
  q(z) := \frac{|\msc(z)|}{\sqrt{d-1}} \leq \frac{1}{\sqrt{d-1}},
\end{equation}
where $z \in \C_+$.
We will often omit the parameter $z$ from the notation if it is clear from the context.

Our main result is the following theorem.

\begin{theorem}\label{thm:mr}
Fix $\alpha>4$, $\n\geq 8$ and $\sqrt{d-1}\geq (\n+1)2^{2\omega+45}$,
and set $\ell_*=\floor{\alpha \log_{d-1}\log N}$ and $r_*=2\ell_*+1$.
Then, for $\cG$  chosen uniformly from $\GNd$, 
the Green's function satisfies
\begin{align}\label{e:locallaw}
  \left|G_{ij}(\cG,z)-P_{ij}(\cE_{r_*}(i,j,\cG),z)\right|\leq |\msc(z)|q(z)^{r_*}, 
\end{align}
with probability $1-o(N^{-\n+8})$,
uniformly in $i,j\in\qq{N}$, and uniformly in $z\in \cal D$, where $\cal D$ is as in \eqref{spectralD}.
Here we assume that $N\geq N_0(\alpha,\omega,d)$ is large enough and that $Nd$ is even.
\end{theorem}

We emphasize that, for fixed $d$,
the right-hand side of \eqref{e:locallaw} converges to $0$, as $N\to\infty$,
uniformly in the spectral domain $\cal D$.
The constants in the statement of the theorem can be improved at the expense of
a longer proof and a more complicated statement. We do not pursue this.

\subsection{Interpretation of Theorem~\ref{thm:mr}; proofs of Theorems~\ref{thm:localmckay}, \ref{thm:delocalizationev}, and Corollary~\ref{c:que}}
\label{sec:main-pfcor}

Theorem~\ref{thm:mr}  states that, in $\cal D$, the Green's function $G_{ij}(\cG)$ is well approximated by
$P_{ij}(\cE_{r_*}(i,j,\cG))$, which is \emph{random}, but only depends on the \emph{local graph structure} of $\cG$ near the vertices $i$ and $j$.
Since the local structure of a random regular graph is well understood,
the theorem has a number of consequences.
Specifically, under the assumptions of the theorem, it is well known that 
there are $\kappa>0$ and $\delta>0$ such that, with $R= \floor{\kappa  \log_{d-1} N}$,
one can assume that the radius-$R$ neighborhoods
of all but $N^\delta$ many vertices of $\cG$ coincide with those of the infinite $d$-regular tree,
and that the $R$-neighborhoods of all other vertices have excess at most $\n$
(see e.g.\ Proposition~\ref{prop:structure}).
Moreover, for the vertices $i$ that have radius-$R$ tree neighborhoods, we have (see e.g.\ Proposition~\ref{greentree})
\begin{equation} \label{e:intro-Pii-tree}
  P_{ii}(\cE_{r_*}(i,i,\cG)) = \md.
\end{equation}
The vertices whose $R$-neighbourhood has bounded excess still satisfy (see e.g.\ Proposition~\ref{boundPij})
\begin{equation} \label{e:intro-Pii-notree}
  |P_{ii}(\cE_{r_*}(i,i,\cG))| \leq 3 |\msc|/2\leq 3/2.
\end{equation}
Together with this information on the local graph structure,
the result of Theorem~\ref{thm:mr} 
implies the results stated in Section~\ref{sec:intro}.

\begin{remark} \label{rk:Giinonconc}
The equation \eqref{e:locallaw} implies that the 
\emph{individual} entries of the Green's function do \emph{not} concentrate. For example,
\begin{align*}
  G_{ii}(z) = P_{ii}(\cE_{r_*}(i,i,\cG),z) + O(\log N)^{-\alpha}
\end{align*}
and the first term on the right-hand side can be easily seen to depend strongly on the local
graph structure. Its fluctuation is of order $1$.
\end{remark}

\begin{proof}[Proof of Theorem~\ref{thm:localmckay}]
\eqref{e:locallaw} and \eqref{e:intro-Pii-tree} imply that $G_{ii}(z) = \md(z) + O(|\msc(z)|q(z)^{r_*})$ for all $z \in \cal D$
and at least $N-N^{-\delta}$ vertices $i \in \qq{N}$.
For the remaining vertices, by \eqref{e:intro-Pii-notree}, we still have $|G_{ii}(z)| \leq 2$.
Thus
\begin{equation*}
  m(z)
  = \frac{1}{N} \sum_{i=1}^N G_{ii}(z)
  = \md(z) + O(|\msc(z)|q(z)^{r_*}) + O(N^{-1+\delta})
  = \md(z) + O(\log N)^{-\alpha},
\end{equation*}
as claimed.
\end{proof}

\begin{proof}[Proof of Theorem~\ref{thm:delocalizationev}]
\eqref{e:locallaw} and \eqref{e:intro-Pii-notree} imply that $|G_{ii}(z)| \leq 2$
for all $i \in \qq{N}$ and all $z\in \cal D$. 
Taking $z=\lambda+\ii(\log N)^{48\alpha+1}/N$, it follows that
\begin{align*}
2\geq \Im[G_{ii}(z)]\geq \frac{\Im[z] |v_i|^2}{|\lambda-z|^2}=(\log N)^{-1-48\alpha}N|v_i|^2,
\end{align*}
which implies the claim \eqref{evest}.
\end{proof}

\begin{proof}[Proof of Corollary~\ref{c:que}]
In \cite[Section 8]{1503.08702-cpam}, it is proved that any exchangeable random vector
$(Y_i)_{i = 1}^N$ satisfies, for any  (deterministic) $q \in \R^N$ with $\sum_i q_i =0$ and $\sum_i q_i^2=1$,
and for any $p \geq 1$,
\begin{equation} \label{e:aY}
\E\absa{\sum_{i=1}^N q_i Y_i}^p = \pa{O\pbb{\frac{p^2}{\log p}}}^p\E|Y_1|^p.
\end{equation}
Let $\phi$ be the indicator function of the event that
for all eigenvectors $v$ with $H$-eigenvalue $|\lambda\pm 2| \geq (\log N)^{1-\alpha/2}$
the estimate $\|v\|_\infty \leq \xi \|v\|_2$ holds,
where $\xi = \sqrt{2} (\log N)^{24\alpha+1/2}/\sqrt{N}$.
Let $v^k$ be the normalized eigenvector corresponding to the $k$-th largest $H$-eigenvalue $\lambda_k$,
and set $Y_i = v_i^k \phi \ind{|\lambda_k \pm 2| \geq (\log N)^{1-\alpha/2}}$.
The $(Y_i)_{i=1}^N$ are exchangeable, by the exchangeability of the random regular graph.
By \eqref{e:aY} with $p=\zeta (\log \zeta)^{1/4}$ and Markov's inequality, for $\zeta$ large enough,
\begin{equation*}
  \P\pa{\phi \inda{|\lambda_k \pm 2|\geq (\log N)^{1-\alpha/2}} \sum_{i=1}^N q_i^m v_i^k
    \leq \zeta^2 \xi} \geq 1-e^{-\zeta(\log \zeta)^{1/4}}.
\end{equation*}
By a union bound over $k \in \qq{N}$ and over $m \in \qq{M}$, it follows
\begin{equation*}
  \P\pa{\phi \max_{k,m}\sum_{i=1}^N q_i^m v_i^k \leq \zeta^2\xi} \geq 1- NM e^{-\zeta (\log \zeta)^{1/4}},
\end{equation*}
where the maximum over $k$ is over all $k$ with $|\lambda_k\pm 2| \geq (\log N)^{1-\alpha/2}$.
Since $\P(\phi=0) = o(N^{-\omega+8})$, by Theorem~\ref{thm:delocalizationev},
and choosing $\zeta = 2^{-1/4}(\log M+\log N)$, we have
\begin{equation*}
  \P\pa{\max_{k,m}\sum_{i=1}^N q_i^m v_i^k  \leq \frac{(\log N)^{24\alpha+1/2}(\log N+\log M)^2}{\sqrt{N}}} \geq 1- o(N^{-\omega+8}),
\end{equation*}
which implies the claim. The proof of \eqref{e:QUE} is analogous, using
$Y_i = (v_i^k)^2 \phi \ind{|\lambda \pm 2| \geq (\log N)^{1-\alpha/2}}$.
\end{proof}

\section{Proof outline and main ideas}
\label{sec:outline}

In this section, we give a high-level outline of the proof of Theorem~\ref{thm:mr}, whose details occupy the
remainder of the paper. The proof is based on the general principle that, for small distances,
a random regular graph behaves almost deterministically, while on the other hand, for large distances, it
behaves much like a random matrix.

\subsection{Parameters}
\label{sec:outline-parameters}

Throughout the paper,
we fix constants $\alpha>4$, $\omega\geq 8$, $0<\delta< 1/\n$, $0<\kappa <\delta/(2\n+2)$, $\sqrt{d-1}\geq (\n+1)2^{2\omega+45}$,
and set $\ell_*=\lceil \alpha \log_{d-1}\log N\rceil$ and $r_*=2\ell_*+1$.
We also set $R = \floor{\kappa \log_{d-1} N}$, and write $r=2\ell+1$, where $\ell$ is a parameter chosen such that
\begin{gather}  
\ell\in\qq{\ell_*,2\ell_*}.
\label{e:constchoice}
\end{gather}
We always assume that $Nd$ is even and sufficiently large (depending on the previous parameters).

\subsection{Structure of the proof}
\label{sec:outline-structure}

The proof consists of several sections, which we briefly describe in this section.
Here, we also define several subsets of $\GNd$,
namely the sets
\begin{equation*}
\Omega^-(z,\ell) \subset \Omega(z,\ell) \subset  \Omega_1^+(z,\ell) \subset \bar\Omega \subset \GNd,\quad \Omega_1'(z,\ell) \subset \bar\Omega \subset \GNd.
\end{equation*}
These sets depend on parameters $z \in \C_+$ and $\ell \in \N$
(and also on the previously fixed parameters). 

\paragraph{Small distance structure; the set $\bar\Omega$}

The small distance behavior is captured in terms of cycles in neighborhoods of radius $R$.
For any graph, we define the \emph{excess} to be the smallest number of edges that must be removed to yield a graph with no cycles
(a forest).
Then, with $R,\n,\delta$ as fixed above,
we define the set $\bar\Omega\subset \GNd$ to consist of graphs such that
\begin{itemize}
\item
the radius-$R$ neighborhood of any vertex has excess at most $\n$; 
\item
the number of vertices that have an $R$-neighborhood that contains a cycle is at most $N^\delta$.
\end{itemize}
The set $\bar\Omega$ provides rough a priori stability at small distances.
All regular graphs appearing throughout the paper will be members of $\bar\Omega$.
It is well-known that $\P(\bar\Omega) \geq 1-o(N^{-\n+\delta})$; see Proposition~\ref{prop:structure}.

\paragraph{Green's function approximation; the sets $\Omega(z,\ell)$ and $\Omega^-(z,\ell)$}

For $z\in\C_+$, we define the 
set $\Omega(z,\ell) \subset \bar\Omega$ be the set of graphs $\cG$
such that for any two vertices $i,j$ in $\qq{N}$, it holds that
\begin{align}\label{rigid0}
  \left|G_{ij}(z)-P_{ij}(\cE_r(i,j,\cG),z)\right|
  \leq |\msc|q^r.
\end{align}
Our main goal is to prove that $\Omega(z,\ell)$ has high probability uniformly in the spectral domain $z\in \cal D$.
That $\Omega(z,\ell)$ has high probability is not difficult to show if $|z|$ is large enough; see Section~\ref{sec:initial}.
To extend this estimate to smaller $z$, we define the set $\Omega^-(z,\ell) \subset \Omega(z,\ell)$ by the same conditions as $\Omega(z,\ell)$,
except that the right-hand side in \eqref{rigid0} is smaller by a factor $1/2$:
 \begin{align}\label{defOmega-}
  \left|G_{ij}(z)-P_{ij}(\cE_r(i,j,\cal G),z)\right|
  \leq \frac{1}{2}|\msc|q^r.
 \end{align}
Our main goal is to show that, for any $z\in \cal D\cap \Lambda_\ell$ (where the spectral domain is defined in \eqref{spectralD} and $\Lambda_\ell$ is defined in \eqref{lambdadef}), if $\Omega(z,\ell)$ has high probability, then the event $\Omega(z,\ell)\setminus \Omega^-(z,\ell)$ has very small probability,
so that $\Omega^-(z,\ell)$ still has high probability. 
Then, by the Lipschitz-continuity of the Green's function, it follows that $\Omega^-(z,\ell) \subset \Omega(z',\ell)$ for small $|z-z'|$,
and thus that $\Omega(z',\ell)$ also has high probability.
This can then be repeated to show that $\Omega(z,\ell)$ holds for all $z \in \cal D\cap \Lambda_\ell$ with high probability. Since these sets $\Lambda_\ell$ all together cover $\cal D$, it follows that $\Omega(z, \ell_*)$ holds for all $z\in \cal D$ with high probability.

\paragraph{Local resampling}

To show that $\Omega(z,\ell) \setminus \Omega^-(z,\ell)$ has small probability, we use the random matrix-like
structure of random regular graphs at large distances. To this end, we fix a vertex, 
without loss of generality chosen to be $1$, and abbreviate the $\ell$-neighborhood of $1$
(as a set of vertices in $\qq{N}$ and as a graph, respectively; see Section~\ref{sec:main} for our notational conventions) by
\begin{equation}
\T=\bB_\ell(1, \cG),\quad \cT=\cB_{\ell}(1,\cG).
\end{equation}

In Section~\ref{sec:switch}, we resample the boundary of the neighborhood $\cT$
by switching the boundary edges with uniformly chosen edges from the remainder of the graph.
The switched graph is often denoted by $\tcG$.
On the vertex set $\T$, it coincides with the unswitched graph $\cG$, but the boundary of $\cT$ in the switched graph $\tcG$ is now essentially
random compared to the original graph $\cG$.

Given $\cG$, the switching is specified by the resampling data $\bf S$, which consists of $\mu$
independently chosen oriented edges from $\cGT$.
The local resampling is implemented by switching a boundary edge of $\cT$ with one of the independently chosen edges encoded by $\bf S$.
In fact, in this operation, not all pairs of edges can be switched (are \emph{switchable}) while keeping the graph simple.
Therefore, given $\bf S$, we denote by $W_{\bf S} \subset \qq{1,\mu}$ the index set for switchable edges (see Section~\ref{sec:switch} for the definition),
whose switching leaves the uniform measure on $\GNd$ invariant.
For notational convenience, without loss of generality, we assume that $W_{\bf S}=\{1,2,3,\dots,\nu\}$ where $\nu\leq \mu$
throughout the paper (except in the definition in Section~\ref{sec:switch}).

\paragraph{Switching from $\cG$ to $\tcG$}

Throughout Sections~\ref{sec:dist}--\ref{sec:improved}, we
condition on a graph $\cG$ that satisfies certain estimates,
and only use the randomness of the switching that specifies how to modify $\cG$ to $\tcG$.
By our choice of $\ell$ and using $\cT$ has bounded excess (which we can and do assume), the number of edges
in the boundary of $\cT$ is about $(\log N)^{O(1)}$.
The randomness of these edges ultimately provides access to concentration estimates, which exhibit the
random matrix-like structure of the random regular graph at large distances.

Note that, if we remove the vertex set $\T$ from $\cG$,
our switchings have a simpler effect than in $\cG$: they only consist of removing the edges
$\{b_i,c_i\}$ and adding instead $\{a_i,b_i\}$, for $i \in W_{\bf S}$.
Therefore, instead of studying the change from $\cG$ to $\tcG$ at once, it will be convenient to 
analyze the effect of the switching in several steps.
For this, we define the following graphs (which need not be regular).
\begin{itemizetight}
\item $\cG$ is the original unswitched graph;
\item $\cGT$ is the unswitched graph with vertices $\T$ removed;
\item $\hcGT$ is the intermediate graph obtained from $\cGT$ by removing the edges $\{b_i,c_i\}$ with $i\in W_{\bf S}$;
\item $\tcGT$ is the switched graph obtained from $\hcGT$ by adding the edges $\{a_i,b_i\}$ with $i \in W_{\bf S}$; and
\item $\tcG$ is the switched graph $T_{\bf S}(\cG)$ (including vertices $\T$).
\end{itemizetight}
Following the conventions of Section~\ref{sec:intro-TE},
the deficit functions of these graphs are given by $d-\deg$,
where $\deg$ the degree function of the graph considered,
and we abbreviate their Green's functions by $G$, $\GT$, $\hGT$, $\tGT$, and $\tG$ respectively.

\paragraph{Distance estimates}

To use the local resampling, we require some estimates on the local distance structure of
graphs and some a priori estimates on their Green's functions. These are collected in Sections~\ref{sec:dist}--\ref{sec:distGF}.
In fact, we use both the usual graph distance (of the unswitched and switched graphs)
and a notion of ``distance'' that is defined in terms of the size of the  Green's function of the graph from which the set $\T$ is removed
(again for the unswitched and switched graph).

The need for the Green's function distance arises as follows.
While estimates that involve sums over the diagonal of the Green's function can be controlled
quite well using only the graph distance, estimates of sums of off-diagonal terms are more delicate because
the number terms is \emph{squared} compared to the diagonal terms.
By direct combinatorial arguments, it would be difficult to control large distances sufficiently precisely.
However, to understand spectral properties, it is the size of the Green's function rather than distances themselves
that is relevant;
and while the size of the Green's function between two vertices is directly related to the distance between them
if there are only few cycles, on a global scale (where many cycles could be present) cancellations can make the Green's function much smaller.
These cancellations are captured in terms of a Ward identity, which states that the Green's function of any symmetric matrix obeys
(see also Appendix~\ref{app:Green})
\begin{equation}
  \frac{1}{N} \sum_{i=1}^N |G_{ij}(z)|^2 = \frac{\im G_{ii}(z)}{\im z}.
\end{equation}

\paragraph{Removing the neighborhood $\T$
and stability under resampling; the sets $\Omega^+_1(z,\ell)$}

Our goal is to show that estimates on the Green's function of $\cG$ improve near the vertex $1$
under the above mentioned local resampling.
For this, we work with the Green's function
of the graph $\cGT$ obtained from $\cG$ by removing the vertex set $\T$ (on which the graph
does not change under switching).

As a preliminary step to showing that the estimates for the Green's function improve, we
show that they are stable under the operation of removing $\T$ and resampling, i.e.,
roughly that the estimates analogous to those assumed continue to hold.
More precisely, in Section~\ref{sec:stabilityT}, we show that if $\cG \in \Omega(z,\ell)$,
then the (non-regular) graph $\cGT$ obeys the analogous estimate
\begin{equation} \label{e:GTstab}
  \left|G_{ij}(\cG^{(\T)},z)-P_{ij}(\cE_r(i,j,\cG^{(\T)}),z)\right|
  \leq 2|\msc|q^r.
\end{equation}
We define the set $\Omega^+_1(z,\ell) \subset \bar\Omega$ similarly as the set $\Omega(z,\ell)$,
except that $\cG$ is replaced by the graph $\cGT$ (and with different constant),
i.e., $\Omega_1^+(z,\ell)$ is the set of $\cG \in \bar\Omega$ such that
\begin{align}\label{e:defOmega+}
  \left|G_{ij}(\cal G^{(\T)},z)-P_{ij}(\cE_r(i,j,\cal G^{(\T)}),z)\right|
  \leq 2^{10}|\msc|q^r.
\end{align}
Clearly, by \eqref{e:GTstab}, we have $\Omega(z,\ell)\subset \Omega_1^{+}(z,\ell)$.
In Section~\ref{sec:stability},
we show that if $\cGT$ obeys the (stronger) estimate \eqref{e:GTstab}, then with high probability
the resampled graph obeys $\cGT \in \Omega^+(z,\ell)$.

\paragraph{Locally improved Green's function approximation; the sets $\Omega_1'(z,\ell)$}

The set $\Omega'_1(z,\ell) \subset \bar\Omega$ is defined by the
improved estimates \eqref{G11bound}--\eqref{G1itreebound} near the vertex $1$, with constant $K=2^{10}$.
In Sections~\ref{sec:boundarydecay}--\ref{sec:improved}, it is proved that if we start with a graph $\cG\in \Omega_1^+(z,\ell)$,
with high probability with respect to the local resampling around vertex $1$, the switched graph $\tcG$ belongs to $\Omega_1'(z,\ell)$.

\paragraph{Involution}

To sum up, the argument outlined above shows that, for any graph $\cG$ in $\Omega(z,\ell)$,
with high probability with respect to the randomness of the local resampling,
the switched graph $\tcG$ is in the set $\Omega_1'(z,\ell)$.
However, our goal was to show that a uniform $d$-regular graph $\cG$
is in $\Omega_1'(z,\ell)$, except for an event
of small probability.
This follows from the statement we proved for $\tcG$ using that our switching
 acts as an involution on
the larger product probability space
(see Proposition~\ref{prop:resample}).

\paragraph{Self-consistent equation}

The sets $\Omega_1^+(z,\ell)$ and $\Omega'_1(z,\ell)$ depend on the choice of vertex $1$.
However, for any $i\in\qq{N}$, we can define $\Omega_i'(z,\ell)$ in the same way,
by replacing the vertex $1$ in the above definitions by vertex $i$ (or using symmetry).
By a union bound, then also the union of the events $\Omega_i'(z,\ell)$
over $i \in \qq{N}$ holds with high probability.
On the latter event, we derive
(in Section~\ref{sec:pfmr})
a \emph{self-consistent equation} for
the quantity
\begin{equation*}
  \I(\cG) = \frac{1}{Nd} \sum_{(i,j) \in \vec{E}} G^{(i)}_{jj}(\cG),
\end{equation*}
where the sum ranges over the set of oriented edges in $\cG$, and $G^{(i)}(\cG)$ 
is the Green's function of the graph $\cG$ with vertex $i$ removed.
On the infinite $d$-regular tree, it is straightforward computation 
to show that $G^{(i)}_{jj}(z) = \msc(z)$ 
holds for any directed edge $(i,j)$  (see Proposition~\ref{greentree}).
For the random regular graph, we will show that $Q(\cG)$ obeys (see \eqref{IGequation})
\begin{equation*}
 \I(\cal G)-\msc= \frac{d-2}{d-1}\md\msc^{2\ell+1}(\I(\cal G)-\msc)+ \text{error}.
\end{equation*}
The main result of Section~\ref{sec:pfmr}, proved using this self-consistent equation,
is that, for any $z\in \cal D\cap \Lambda_\ell$, 
\begin{equation*}
  \bigcap_{1\leq i\leq N} \Omega_i'(z,\ell)\subset \Omega^-(z,\ell),
\end{equation*}
where $\Lambda_\ell \subset \C_+$ is a domain on which the self-consistent equation is not singular
(see Section~\ref{sec:pfmr} for details).
In the final step, we will use different choices of $\ell$ to  cover the entire spectral domain $\cal D$.

\paragraph{Conclusion}

In summary, in Sections~\ref{sec:stabilityT}--\ref{sec:pfmr},
we show that the probability of $\Omega(z,\ell)\setminus \Omega^-(z,\ell)$ is negligible.
By the Lipschitz property of the Green's function,  $\Omega^-(z,\ell)\subset \Omega(z',\ell)$ given that $|z-z'|$ is small enough.
It follows that if $\Omega^-(z,\ell)$ holds with high probability, then $\Omega^-(z,\ell)\cap \Omega^-(z', \ell)$ holds with high probability. This can then be repeated to show that $\Omega(z,\ell)$ holds for all $z \in \cal D\cap \Lambda_\ell$ with high probability. 
The proof of Theorem~\ref{thm:mr} is then completed by showing that $\cal D\subset \cup_{\ell\in\qq{\ell_*, 2\ell_*}}\Lambda_\ell$ and thus $\Omega(z,\ell_*)$ holds for all $z\in \cal D$ with high probability.

\subsection{Random walk picture}
\label{sec:Gexpan}

We conclude this section with the following random walk heuristic for the Green's function.
The Green's function $G_{ij}(\cG,z)$ has a  \emph{formal} expansion in terms of walks  in $\cG$ from $i$ to $j$
with complex $z$-dependent weights: 
\begin{equation} \label{e:Gexpan}
  G_{ij}(\cG,z) = -\sum_{w: i\to j} (d-1)^{-|w|/2} z^{-|w|-1},
\end{equation}
where the sum ranges over all walks $w$ from $i$ to $j$ of length $|w|$. In several parts of the paper,
it might be useful to think about the Green's function in this picture, though we never use it directly.
However, the expansion  \eqref{e:Gexpan} is only absolutely convergent for $|z| > \sqrt{d-1}$,
where the formal sum is dominated by the shortest walks.
In our case of primary interest, $\im z \ll 1$ (with $\re z$ inside the spectrum of the adjacency matrix),
the expansion becomes highly
oscillatory and is not absolutely convergent.
Long walks become dominant and the Green's function can only remain bounded due to significant cancellations.

On the tree, it is easy to compute the Green's function exactly.
In particular, one finds that the Green's function is bounded for all $z\in \C_+$, and that, roughly speaking, each step of a walk contributes a factor $-\msc(z)/\sqrt{d-1}$.
A popular and very efficient method to exhibit the required cancellations that result from the tree structure is via \emph{nonbacktracking walks}.

Our main effort is not in exhibiting the cancellations resulting from the tree structure,
but it is rather in exhibiting the cancellations of \emph{very long walks},
where the tree structure ceases to be effective. To obtain these cancellations,
we exploit the randomness of the random regular graph in combination with a Ward identity.
Using a multiscale approach (implemented via a continuity argument), we successively prove that the Green's function remains bounded even for small $\im z$,
and moreover that it has good decay.
Such bounds, together with linear algebra (Schur complement formula, resolvent identity), allow to obtain
well-defined partially resummed versions of random walk identities, in which the crucial cancellations
are accounted for nonperturbatively.

\section{Structure of random and deterministic regular graphs}
\label{sec:rrg}

In this section, we collect some properties of random and deterministic regular graphs,
which we use in the remainder of the paper.

\paragraph{Excess of random regular graphs}

For any graph $\cG$, we define its \emph{excess} to be the smallest number of edges that must be removed to yield a graph with no cycles
(a forest). It is given by
\begin{align}\label{def:excess}
\text{excess}(\cG)\deq \#\text{edges}(\cG) -\#\text{vertices}(\cG) +\#\text{connected components}(\cG).
\end{align} 
There are different conventions for the normalization of the excess.
Our normalization is such that the excess of a tree or forest is $0$.
Note that if $\cH \subset \cG$ is a subgraph, then $\text{excess}(\cH) \leq \text{excess}(\cG)$.
We will use the following well-known estimates for the excess in random regular graphs.

\begin{proposition} \label{prop:structure}
Let $\delta > 0$ and $\n \geq 1$ be an integer. There is $\kappa>0$ such that, if $R=\floor{\kappa \log_{d-1} N}$,
then the following holds for a uniformly chosen random $d$-regular graph $\cG$ on $\qq{N}$,
with probability at least $1-o(N^{-\n+\delta})$
for $N \geq N_0(d,\n,\delta)$ large enough.
\begin{itemize}
\item All $R$-neighborhoods have excess at most $\omega$:
\begin{equation}
\label{e:structure1}
\text{for all $i \in \qq{N}$, the subgraph $\cB_R(i,\cG)$ has excess at most $\n$.}
\end{equation}
\item Most $R$-neighborhoods are trees:
\begin{equation}
\label{e:structure2}
|\{i \in \qq{N}: \text{the subgraph $\cB_R(i,\cG)$ contains a cycle}\}| \leq N^{\delta}.
\end{equation}
\end{itemize}
In fact, one can take $\kappa < \delta/(2\omega+2)$.
\end{proposition}

\begin{proof}
The statements are well known;
for completeness, we sketch proofs in Appendix~\ref{app:structure}.
\end{proof}

\paragraph{Excess and the number of non-backtracking walks}

The next proposition bounds the number of non-back\-track\-ing walks (NBW) between two vertices in a graph in terms of
the excess of the graph.
Here a non-backtracking walk of length $n$ is a sequence of vertices $(i_0, \dots, i_n)$
such that the edge $\{i_{k-1},i_{k}\}$ is adjacent to $\{i_k,i_{k+1}\}$ and such that the walks makes no
steps backwards, i.e., $i_{k-1} \neq i_{k+1}$.

\begin{proposition}\label{prop:numberpath}
Let $\cG$ be a graph with excess at most $\n$. Then the following hold.
\begin{itemize}
\item
For any vertices $i,j\in \bG$, and any $k\geq 1$, we have
\begin{align}\label{pathnum}
|\{\text{NBW from $i$ to $j$ of length $\dist_{\cG}(i,j)+k-1$}\}| \leq 2^{\n k}.
\end{align}
\item
For any subgraph $\cH \subset \cG$ and two vertices $i,j$ in $\cH$
such that  $\cE_\ell(i,j,\cG) \subset \cH$, we have
\begin{align}\label{pathN}
|\{\text{NBW from $i$ to $j$ of length
$\ell+k$ which are not completely in $\cal H$}\}|
\leq 2^{\n(k+1)+1}.
\end{align}
\end{itemize}
The graph $\cG$ does not need to be regular or finite, and self-loops and multi-edges are allowed.
\end{proposition}

\begin{proof}
The statements are presumably also well known; lacking a reference, we include their proofs in
Appendix~\ref{app:numberpath}.
\end{proof}

\paragraph{Boundary of a neighborhood}

In the remainder of the paper, given a graph $\cG$ on $\qq{N}$,
we will often fix a vertex, chosen without loss of generality to be $1$,
and denote its $\ell$-neighborhood by $\cT=\cB_{\ell}(1,\cG)$,
with corresponding vertex set $\T=\bB_\ell(1, \cG)$.
We further enumerate $\partial_E \T$ as $\{e_1, e_2,\dots, e_\mu\}$, i.e., $e_i$ are the edges with one vertex in $\T$
and one in $\qq{N} \setminus \T$,
and correspondingly $\partial \T$ as $\{a_1, \dots, a_\mu\}$, where $a_i$ is the endpoint of $e_i$ not in $\T$.
We also write $\bT_i=\{v\in \cG: \dist_{\cG}(1,v)=i\}$ for $i=0,1,2,\dots, \ell$.

\begin{proposition}\label{GTstructure}
Let $\cal G$ be a $d$-regular graph on $\qq{N}$, assume
that $\cB_R(1,\cG)$ has excess at most $\n$, and that $\ell\ll R$.
Then the following hold.
\begin{itemize}
\item
After removing $\T$, most boundary vertices of $\cT$ are isolated from the other boundary vertices:
\begin{align}\label{fewclose}
|\{p\in \qq{1,\mu}: \exists q\in \qq{1,\mu} \setminus \{p\}, \dist_{\cGT}(a_p, a_q)\leq R/2\}|\leq 2\n.
\end{align}
\item
After removing $\T$, any vertex $x\in \qq{N} \setminus \T$ can only be close to few boundary vertices of $\bT$:
\begin{align}
\label{distfinitedega}
|\{p\in\qq{1,\mu}: \dist_{\cGT}(x,a_p)\leq R/2\}|\leq \n+1,
\\
\label{distfinitedegl}
 |\{v\in \T_\ell: \dist_{\cG \setminus \cT}(x,v)\leq R/2\}|\leq \n+1.
\end{align}
\end{itemize}
\end{proposition}

Notice that the graph $\cG \setminus \cT$ is slightly larger than $\cG^{(\T)}$
because the edges between the vertices $\T_\ell$ and $\qq{N}\setminus \T$
are not removed.

\paragraph{Bound on deficit functions}

Finally, we have the following deterministic bound on the deficit functions for the
connected components of the subgraph obtained from $\cB_R(1,\cG)$ by removing a set of vertices $\bU$.

\begin{proposition} \label{p:deficitbound}
Let $\cal G$ be a $d$-regular graph on $\qq{N}$,  and assume
that $\cB:=\cB_R(1,\cG)$ has excess at most $\n$. Then the following hold.
\begin{itemize}
\item
Let $\cA$ be the annulus obtained by removing $\T$ from $\cB$.
Then the {sum of the} deficit function over any connected component of $\cA$ satisfies $\sum g(v)\leq \n+1$.
\item
Given $\bU\subset \bB_\ell(1,\cG)$, let $\cB^{(\bU)}$ be the subgraph given by removing the vertices $\bU$ from $\cB$.
Then the sum of the deficit function over any connected component of $\cB^{(\bU)}$ satisfies $\sum g(v)\leq \n+|\bU|$. 
\end{itemize}
For the above statements,
recall that we view $\cal A$ and $\cB^{(\bU)}$ as subgraphs of $\cB$ (which has zero deficit function)
and that their deficit functions are given by our conventions in Section~\ref{sec:intro-TE}.
\end{proposition}

In the remainder of this section, we prove Propositions~\ref{GTstructure} and \ref{p:deficitbound}.

\subsection{Proof of Proposition~\ref{GTstructure}}
\label{sec:GTstructurepf}

\begin{figure}[t]
\centering
\input{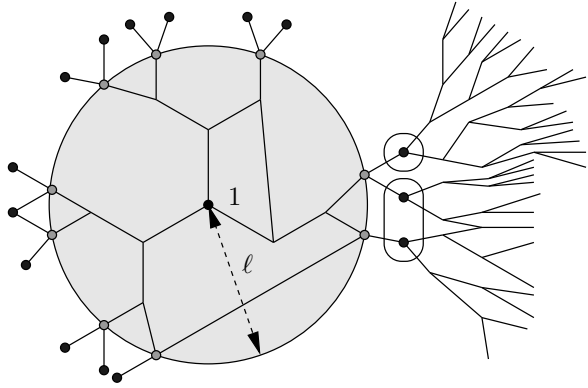}
\caption{The two vertices $a_i$ which are encircled together are close in the sense that they are
in the same connected component of the annulus $\cA$.
Proposition~\ref{GTstructure} shows that,
since $\cB$ has excess at most $\n$, this happens for at most $2\n$ of the $a_i$.
In addition, it shows that any vertex $x$ outside $\T$ can only be close to at most $\n+1$ of the $a_i$.
\label{fig:resample2}}
\end{figure}

Abbreviate $\cB = \cB_R(1,\cG)$.
By assumption the ball $\cB$ has excess at most $\n$. 
Let $\cA$ be the annulus obtained by removing $\T$ from $\cB$.
We partition $\qq{1,\mu}$ into sets  $\{A_1, A_2, A_3, \dots\}$,
such that $i$ and $j$ are in the same set $A_k$ if and only if $a_i$ and $a_j$ are in the same connected component of $\cA$.
We label the sets $A_k$ such that $|A_1|\geq |A_2|\geq \dots\geq |A_{\alpha}|>1=|A_{\alpha+1}|=\cdots$
and let $i_j$ be a labeling such that $A_1\cup \cdots\cup A_{\alpha}=\{{i_1}, {i_2},\dots, {i_\beta}\}$.

\begin{lemma} \label{lem:beta2n}
\begin{equation} \label{e:beta2n}
  \alpha \leq \n, \quad \beta \leq 2\n, \quad |A_j| \leq \omega+1 \quad \text{for all $j$}.
\end{equation}
\end{lemma}

\begin{proof}
For any finite graph $\cG$, we set
\begin{align}\label{c-c}
  \chi(\cG) \deq \#\text{connected components}(\cG)-\text{excess}(\cG)
  =
  \#\text{vertices}(\cG)-\#\text{edges}(\cG),
\end{align}
where the second equality follows from the definition \eqref{def:excess} of $\text{excess}(\cG)$.
In particular, for any $e \in \cG$, we have $\chi(\cG \setminus e) = \chi(\cG) + 1$.

As a ball, $\cB$ has by definition exactly one connected component, and by assumption it has excess at most $\n$.
Thus $\chi(\cB) \geq 1-\n$.
We recall that $e_i$ is the edge on the boundary of $\cT$ containing $a_i$.
Thus the graph $\cB \setminus \{e_{i_1}, \dots, e_{i_\beta}\}$ has at most $\alpha+1$ connected components:
the component containing the vertex~$1$  and the components containing the vertices $a_i$ with $i\in A_j$ for some $j\in\qq{1,\alpha}$.
(Notice for $i\in A_j$ with $j>\alpha$, we did not remove the edge $e_{i}$. Therefore $a_i$ is still connected to $1$.)
Thus $\chi(\cB \setminus \{e_{i_1}, \dots, e_{i_\beta}\}) \leq \alpha+1$. 
It follows that
\begin{equation*}
1+\alpha\geq \chi(\cB \setminus \{e_{i_1}, \dots, e_{i_\beta}\}) =\chi(\cB)+\beta\geq 1-\n+\beta,
\end{equation*}
and thus $\beta \leq \alpha+\n$.
Since, by definition, we have $\beta=\sum_{i=1}^{\alpha}|A_i|\geq 2\alpha$,
the first two inequalities in \eqref{e:beta2n} follow.
The third inequality is trivial for $i> \alpha$, and for $i \leq \alpha$, we have
\begin{equation*}
\n+\alpha\geq \beta=\sum_{j=1}^{\alpha}|A_j|\geq |A_i|+2(\alpha-1),
\end{equation*}
which implies that $|A_i|\leq \n - \alpha + 2 \leq \n+1$ as claimed. 
\end{proof}

\begin{proof}[Proof of \eqref{fewclose}]
By definition,
any $i,j$ such that $\dist_{\cGT}(a_i, a_j)\leq R/2$ belong to the same connected component of $\cA$.
(Indeed, $a_i$ is at distance $\ell+1$ from the vertex $1$ and $R\gg\ell$, and thus $\cB_{R/2}(a_i, \cG)\subset \cB$ for any $i\in \qq{1,\mu}$.)
In particular, if the set $A_i$ containing $i$ has size $1$, then for any $j\in\qq{1,\mu}\setminus\{i\}$,
we have $\dist_{\cGT}(a_i, a_j)> R/2$.
Recalling that $\beta\leq 2\n$ is the number of $i$ for which the set $A_i$ containing it has size greater than $1$,
the claim \eqref{fewclose} follows from \eqref{e:beta2n}.
\end{proof}

\begin{proof}[Proof of \eqref{distfinitedega}]
The claim is trivial if $x\not\in \cB$, since we then have $\dist_{\cGT}(x,\{a_1,a_2,\dots, a_\mu\})\geq R - \ell > R/2$ by definition.
Thus assume that $x \in \cB$. Let $A_j$ be such that $x$ and the vertices $a_i$ with $i\in A_j$ are in the same connected component of $\cA$.
We first show that those vertices $a_p$ with $p \in A_k$ where $k\neq j$ do not contribute to \eqref{distfinitedega}.
Indeed, then $x$ and $a_p$ are in the different connected components of $\cA$.
But since $\cB_{R/2}(a_p, \cG)\subset \cB$, it then follows that $\dist_{\cGT}(x,a_p) > R/2$.
Therefore $|\{p\in \qq{1,\mu}: \dist_{\cGT}(x,a_p)\leq R/2\}|\leq |A_j|$,
and the claim \eqref{distfinitedega} follows from the third inequality in \eqref{e:beta2n}.
\end{proof}

\begin{proof}[Proof of \eqref{distfinitedegl}]
By the same proof, \eqref{distfinitedega} also holds with $\ell$ replaced by $\ell-1$,
i.e., with $\T=\bB_\ell(1,\cG)$ replaced by $\bB_{\ell-1}(1,\cG)$,
including in Lemma~\ref{lem:beta2n}. This gives
\begin{align*}
 |\{v\in \T_\ell: \dist_{\cG ^{(\bB_{\ell-1}(1,\cG))}}(x,v)\leq R/2\}|\leq \n+1.
\end{align*}
Then claim then follows since $\cG\setminus \cT\subset\cG ^{ (\bB_{\ell-1}(1,\cG))}$.
\end{proof}

\subsection{Proof of Proposition~\ref{p:deficitbound}}

\begin{proof}[Proof of Proposition~\ref{p:deficitbound}]
For the first statement,
viewing the annulus $\cA$ as a subgraph of $\cGT$,
the bound \eqref{e:beta2n} immediately implies that the sum of deficit function over
any connected component of $\cA$ satisfies $\sum g(v)\leq \max_j |A_j|\leq \n+1$.

For the second statement, let $k=|\bU|$ and write $\bU=\{u_1,u_2,\dots, u_k\}$.
Let $\bX\subset \qq{N}$ be the set of vertices of any given connected component of $\cA$.
Define
\begin{align*}
B_i:=\bX \cap \del u_i=\left\{v^i_1, v^i_2,\dots, v^i_{|B_i|}\right\},\quad i=1,2,\dots, k,
\end{align*}
where $\del u$ is the set of neighbors of the vertex $u$ in $\cG$.
Notice that $g(v)=0$ unless $v \in B_1 \cup \cdots \cup B_k$. Thus
\begin{equation}
\sum_{v \in \bX} g(v)
\leq \sum_{i=1}^k \sum_{v \in B_k} g(v)
\leq \sum_{i=1}^k |B_i|
\leq |\bU| + \sum_{i=1}^k (|B_i|-1),
\end{equation}
so that the claim follows from
\begin{align}\label{e:sumdeficit}
\sum_{i=1}^k (|B_i|-1) \leq \n.
\end{align}
To prove \eqref{e:sumdeficit}, we consider the graph
\begin{align*}
  {\cH} = \cB\setminus \bigcup_{i=1}^k\{\{u_i, v_1^i\}, \{u_i, v_2^i\}, \cdots, \{u_i, v_{|B_i|-1}^i \}\}.
\end{align*}
Note that $\cH$ is obtained from $\cB$ by removing exactly $\sum_{i=1}^k (|B_i|-1)$ edges and that $\cH$ is connected. 
Since by assumption $\cB$ has excess at most $\n$, after removing any $\n+1$ edges, it cannot be connected.
This implies \eqref{e:sumdeficit}.
\end{proof}

\section{Trees and tree extension}
\label{sec:TE}

For the infinite regular tree and for the rooted infinite regular tree with given root degree,
it is elementary to compute the Green's function explicitly,
as done in the following proposition.

\begin{proposition}\label{greentree}
Let $\cY$ be the infinite $d$-regular tree.
For all $z \in \C_+$, its Green's function is
\begin{equation} \label{e:Gtreemkm}
  G_{xy}(z)=m_{d}(z)\left(-\frac{\msc(z)}{\sqrt{d-1}}\right)^{\dist_{\cY}(x,y)}.
\end{equation}
Let $\cY_0$ be the rooted infinite $d$-regular tree with root degree $d-1$.
Its Green's function is
\begin{equation} \label{e:Gtreemsc}
  G_{xy}(z)=m_{d}(z)\left(1-\left(-\frac{\msc(z)}{\sqrt{d-1}}\right)^{2\ell(x,y)+2}\right)\left(-\frac{\msc(z)}{\sqrt{d-1}}\right)^{\dist_{\cY_0}(x,y)},
\end{equation}
where $\ell(x,y)$ is the depth of the common ancestor of the vertices $x$ and $y$ in $\cY_0$.
In particular, if $x$ is the root of $\cY_0$, then $G_{xx}(z)=\msc(z)$.
\end{proposition}

The proof is given below.
More general results for Green's functions on regular trees are discussed e.g.\
in \cite[Section~3]{MR3055759} and references given there.

The main results of this section are the following estimates for $P_{ij}(\cG_0,z)$,
the Green's function of the tree extension $\TE(\cG_0)$ of a graph $\cG_0$,
as defined in Definition~\ref{def:treeextension}.

\begin{proposition}\label{boundPij}
Let $\n \geq 6$ and $\sqrt{d-1} \geq 2^{\n+2}$.
Let $\cG_0$ be a finite graph with vertex set $\bG_0$ and deficit function $g$.
Assume that
(\rn{1}) any connected component of $\cal G_0$ has excess at most $\omega$,
and that
(\rn{2}) the sum of deficit function over any connected component of $\cal G_0$ satisfies $\sum g(v)\leq 8\n$.
Then the following holds for all $z \in \C_+$ and all $i,j\in \bG_0$.
\begin{enumerate}
\item
The Green's function $P_{ij}(\cG_0)$ of $\TE(\cG_0)$ satisfies
\begin{equation} \label{e:boundPij}
\left|P_{ij}(\cG_0,z)\right|
\leq 2^{\n+2}|\msc(z)|q^{\dist_{\cG_0}(i,j)},
\end{equation}
and the diagonal terms satisfy the better estimate
\begin{equation} \label{e:boundPii}
\left| P_{ii}(\cG_0,z)-\md(z) \right| \leq \frac{|\msc(z)|}{4}.
\end{equation}
\item
Let $\cH_0 \subset \cG_0$ be a subgraph with vertex set $\bH_0$.
Then for any two vertices $i,j$ in $\cH_0$ such that 
$\cE_\ell(i,j) \subset \cH_0$,
the $ij$-th entries of the Green's functions of the tree extensions of $\cal G_0$ and $\cal H_0$ satisfy
\begin{align} \label{e:compatibility}
\left|P_{ij}(\cG_0, z)-P_{ij}(\cal H_0,z)\right|
\leq 2^{2\n+3}|\msc(z)|q^{\ell+1}.
\end{align}
\end{enumerate}
\end{proposition}

Item (i) states that $P_{ij}(\cG_0)$ is bounded and has (up to constants) the same decay as the Green's function of the infinite $d$-regular tree $\cY$.
In particular, \eqref{e:boundPij} and \eqref{e:boundPii} together with \eqref{e:mdstieltjes}  imply that
\begin{equation} \label{e:boundPiimsc}
  |P_{ij}(\cG_0,z)| \leq (1+\delta_{ij}/2)|m_{sc}(z)|.
\end{equation}
 
Item (ii) states that $P_{ij}(\cG_0)$ depends only weakly on $\cG_0$.
Especially, it implies the the following principle,
which is used repeatedly throughout Sections~\ref{sec:stabilityT}--\ref{sec:improved}.

\begin{remark}[Localization principle] \label{rk:princG0}
Let $\bX$ be a (small) set of vertices in a graph $\cG$.
For vertices $i,j\in \bX$, it is often convenient to replace $P_{ij}(\cE_r(i,j,\cG))$,
namely the $ij$-th entry of the Green's function of the graph $\TE(\cE_r(i,j,\cG))$ which itself depends on $i,j$,
by $P_{ij}(\cG_0)$ of a graph $\cG_0$ which is independent of $i,j$ and contains $\cE_r(i,j,\cG)$ for $i,j \in \bX$.
In this situation, we abbreviate $P=P(\cX)$.
The estimate \eqref{e:compatibility} then implies that $P_{ij}$ and $P_{ij}(\cE_r(i,j,\cG))$ are close
in the sense
\begin{align} \label{replaceEr0}
\left|P_{ij}(\cE_{r}(i,j,\cG))-P_{ij}\right|
\leq 2^{2\n+3}|\msc|q^{r+1}
\end{align}
provided the assumptions of \eqref{e:compatibility} are obeyed. 
\end{remark}

\subsection{Proof of Proposition~\ref{greentree}}

The proof of Proposition~\ref{greentree} is a straightforward consequence of the Schur complement formula
\eqref{e:Schur1}.

\begin{proof}
  Let $\dist_{\cY}(x,y)=1$.
  The Schur complement formula implies
  \begin{equation}
    G^{(x)}_{yy} = \frac{-1}{z+(d-1)^{-1} \sum_{k,l\in \partial y \setminus \{x\}} G^{(xy)}_{kl}}
    = \frac{-1}{z+(d-1)^{-1} \sum_{k\in \partial y \setminus \{x\}} G^{(y)}_{kk}},
  \end{equation}
  where $\del y$ is the set of adjacent vertices of $y$ in $\cY$.
  By homogeneity, $G^{(x)}_{yy}$ is independent of $x$ and $y$ if $\dist_{\cY}(x,y)=1$ and therefore equal to the
  unique solution to the equation $m = -1/(z+m)$ with $\im m >0$, which is $\msc$.
  Applying the Schur complement formula again, it follows that
  \begin{equation}
    G_{xx} = \frac{-1}{z+d(d-1)^{-1}\msc} = \md.
  \end{equation}
  This proves \eqref{e:Gtreemkm} and \eqref{e:Gtreemsc} for $x=y$. The case $\dist_{\cY}(x,y)=1$ then follows, e.g., from
  \begin{equation}
    1 = \sum_{y} G_{xy} (H_{yx}-z \delta_{yx}) = \frac{d}{\sqrt{d-1}} G_{xy} - z \md,
  \end{equation}
  which, using $1+z\md + \frac{d}{d-1} \md\msc=0$, implies
  \begin{equation}
    G_{xy} = \frac{\sqrt{d-1}}{d} (1+z\md) = - \frac{\md\msc}{\sqrt{d-1}},
  \end{equation}
  as claimed. The general case is similar by induction.
\end{proof}

\subsection{Proof of Proposition~\ref{boundPij} for $g\equiv 0$}

For the proof of Proposition~\ref{greentree}, we require the notion of covering of a graph.
Given a graph $\cG$, a graph $\tcG$ together with a surjective map $\pi: \tbG \to \bG$ is a \emph{covering} of $\cG$
if for each $x\in \tbG$, the restriction of $\pi$ to the neighborhood of $x$ is a bijection onto the neighborhood of $\pi(x)$ on $\bG$.
Every $d$-regular graph is covered by the infinite $d$-regular tree $\cY$ which is its universal covering. 

The Green's functions of a graph $\cG$ and a cover $\tcG$ with covering map $\pi: \tbG \to \bG$ obey the following identity.
For each $x\in \tbG$ and $\pi(x)=i\in \bG$, we have
\begin{align}\label{defg}
G_{ij}(z)=\sum_{y:\pi(y)=j}\tilde{G}_{xy}(z),
\end{align}
if the right-hand side is summable (see Appendix~\ref{app:Green} for the elementary proof of \eqref{defg}).
In particular, if $\cal G$ is an infinite simple $d$-regular graph
and $\pi: \bY\rightarrow \bG$ its universal covering map, where $\bY$ are the vertices of $\cY$,
then by \eqref{e:Gtreemkm} and \eqref{defg},
for any vertex $x\in \bY$ such that $\pi(x)=i$, the resolvent entries of the graph $\cG$ are given by
\begin{align}
\notag G_{ij}
&=m_{d}\sum_{y:\pi(y)=j}\left(-\frac{\msc}{\sqrt{d-1}}\right)^{\dist_{\cY}(x,y)}\\
\label{tree-likeG}
&=m_{d}\sum_{k\geq \dist_{\cal G}(i,j)}|\{\text{non-backtracking paths from $i$ to $j$ of length $k$}\}|\left(-\frac{\msc}{\sqrt{d-1}}\right)^{k}.
\end{align}
For the number of non-backtracking paths, recall the estimates of Proposition~\ref{prop:numberpath}.
Using these, the proofs of \eqref{e:boundPij} and \eqref{e:compatibility} are straightforward
from \eqref{tree-likeG} if $g\equiv 0$.

\begin{proof}[Proof of \eqref{e:boundPij} for $g\equiv 0$]
For vertices $i,j$ in different connected components of $\cG_0$,
we have $P_{ij}(\cG_0)=0$ and there is nothing to prove.
Therefore, we can assume that $i$ and $j$ are in the same connected component.

Since we assume $g\equiv 0$,
the tree extension  $\cG_1=\TE(\cG_0)$ is $d$-regular, and 
\eqref{tree-likeG} implies
\begin{align}\label{eqn:formulaPij}
  P_{ij}(\cG_0,z)=\md\sum_{k\geq \dist(i,j)}|\{\text{NBW in $\cG_1$ from $i$ to $j$ of length $k$}\}|\left(-\frac{\msc}{\sqrt{d-1}}\right)^{k}
  .
\end{align}
Since $\cG_0$ has excess at most $\n$, the same is true for $\cG_1$.
By the estimates for the number of non-backtracking paths from Proposition~\ref{prop:numberpath},
the right-hand side of \eqref{tree-likeG} is summable, provided that $\sqrt{d-1}\geq 2^{\n+2}$, and
\begin{align*}
|G_{ij}|
  \leq |\md|\sum_{k\geq 1}2^{\n k}q^{\dist_{\cG_0}(i,j)+k-1}
  &=|\md|2^\n q^{\dist_{\cG_0}(i,j)}\sum_{k\geq 1}(2^{\n} q)^{k-1}\\
  &\leq |\md|2^\n q^{\dist_{\cG_0}(i,j)}\sum_{k\geq 0}4^{-k}\leq 
 2^{\n+1}|\msc|q^{\dist_{\cG_0}(i,j)}.
\end{align*}
This completes the proof if $g \equiv 0$.
\end{proof}

\begin{proof}[Proof of \eqref{e:compatibility} for $g\equiv 0$]
As in the proof of \eqref{e:boundPij}, we can assume that $i$ and $j$ are in the same connected component of $\cG_0$.
By \eqref{tree-likeG}, since all the non-backtracking paths from $i$ to $j$ of length $\leq \ell$
are contained in $\cal H_0$, we have
\begin{align*}
&P_{ij}(\cal G_0, z)-P_{ij}(\cal H_0,z)\\
&= \md\sum_{k\geq 1}|\{\text{NBW from $i$ to $j$ of length $\ell+k$, not completely in $\cal H_0$}\}|\left(-\frac{\msc}{\sqrt{d-1}}\right)^{\ell+k}
  .
\end{align*}
By \eqref{pathN}, we therefore have
\begin{align*}
\left|P_{ij}(\cal G_0, z)-P_{ij}(\cal H_0,z)\right|
&\leq |\md|\sum_{k=1}^{\infty}2^{\n(k+1)+1}q^{\ell+k}
\\
&= |\md|2^{2\n+1}q^{\ell+1}\sum_{k=1}^{\infty}\left(2^{\n}q\right)^{k-1}
\leq 2^{2\n+2}|\msc|q^{\ell+1},
\end{align*}
again provided that $\sqrt{d-1}\geq 2^{\n+2}$.
This completes the proof if $g \equiv 0$.
\end{proof}

\subsection{Proof of Proposition~\ref{boundPij} for $g \not\equiv 0$}

To extend the bounds \eqref{e:boundPij} and \eqref{e:compatibility} to $g \not\equiv 0$,
we use an alternative representation of $P_{ij}(\cG_0)$ given as follows.
In Definition~\ref{def:treeextension}, $P_{ij}(\cG_0,z)$ is defined as the Green's function
of the infinite graph obtained by attaching a $d$-regular tree at every extensible vertex of $\cG_0$.
The next lemma shows that it is equivalently given by attaching to every extensible
vertex a self-loop with $z$-dependent complex weight.
The proof of the lemma follows by application of the Schur complement formula.

\begin{lemma} \label{lem:H2}
Let $z\in \C_+$. Then for vertices $i,j\in\bG_0$,
\begin{align*}
P_{ij}(\cal G_0,z)=(H_2-z)^{-1}
\end{align*}
where $H_2$ is the normalized $z$-dependent adjacency matrix obtained by attaching to
any extensible vertex $v$ in $\cG_0$ a self-loop with complex weight $-\msc(z)(d-g(v)-\deg_{\cal G_0}(v))/\sqrt{d-1}$.
\end{lemma}

\begin{proof}
Let $\cG_1=\TE(\cG_0)$,
and denote the normalized adjacency matrix of $\cG_0$ and $\cG_1$ by $H_0$ and $H_1$ respectively. Then
$H_1$ has the block form
\begin{align*}
 H_1=\left[\begin{array}{cc}
            H_0 & B'\\
            B & D
           \end{array}
\right]
\end{align*}
where $D$ is the normalized adjacency matrix of several copies of $\cY_0$,
i.e.\ infinite $d$-regular tree with root degree $d-1$,
and $B_{xy}$ is $1/\sqrt{d-1}$ if $y$ is an extensible vertex of $\cG_0$ and $x$ the root of
one of the former copies of the tree $\cY_0$, and $B_{xy}=0$ otherwise.
By the Schur complement formula \eqref{e:Schur}, it follows that, for any $i,j\in \bG_0$,
\begin{align*}
 G_{ij}(\cal G_1,z)=(H_1-z)^{-1}_{ij}=(H_0-z-B'(D-z)^{-1}B)^{-1}_{ij}.
\end{align*}
Since $B'(D-z)^{-1}B$ is a diagonal matrix, indexed by the extensible vertices in $\cG_0$
(which are disjoint),
and since $B$ is normalized by $1/\sqrt{d-1}$, it follows from \eqref{e:Gtreemsc} with $x=y$ that
\begin{align*}
 (B'(D-z)^{-1}B)_{vv}=\msc{\bf 1}_{v \text{ is extensible}}\frac{d-g(v)-\deg_{\cal G_0}(v)}{d-1}
  .
\end{align*}
Thus $H_2=H_0-B'(D-z)^{-1}B$ and the claim of the lemma follows.
\end{proof}

As previously, we abbreviate $\cG_1 =\TE(\cG_0)$, and denote by $\cG_2$ the finite $z$-dependent graph with complex weight obtained by
attaching  at each extensible vertex $v$ of $\cG_0$ a self-loop with weight $-\msc(z)(d-g(v)-\deg_{\cG_0}(v))/\sqrt{d-1}$.
Moreover, to extend \eqref{e:boundPij} and \eqref{e:compatibility} from $g \equiv 0$ to $g \not\equiv 0$,
we denote by $\cG_0'$ the same graph as $\cG_0$ but with deficit function $g \equiv 0$,
by $\cG_1'=\TE(\cG_0')$ its tree extension,
and by $\cG_2'$ the finite $z$-dependent graph with complex weight obtained by
attaching  at each extensible vertex $v$ of $\cal G_0'$ a self-loop with weight $-\msc(z)(d-\deg_{\cG_0}(v))/\sqrt{d-1}$. We denote the normalized adjacency matrices of $\cG_2$ and $\cG_2'$ by $H_2$ and $H_2'$ respectively. 

\begin{proof}[Proof of \eqref{e:boundPij} for $g \not\equiv 0$]
By Lemma~\ref{lem:H2} and the case $g\equiv 0$, we have
\begin{align*}
  \Gamma' := \max_{i,j\in \bG_0}|G_{ij}(\cal G_2',z)|\left(|\msc|q^{\dist_{\cG_0}(i,j)}\right)^{-1}\leq 2^{\n+1}.
\end{align*}
Our goal is to estimate
\begin{align*}
  \Gamma := \max_{i,j\in \bG_0}|G_{ij}(\cal G_2,z)|\left(|\msc|q^{\dist_{\cG_0}(i,j)}\right)^{-1}.
\end{align*}
Notice that $H_2-H_2'$ is a diagonal matrix with entries
\begin{equation} \label{e:H2diag}
  (H_2-H_2')_{vv}=\frac{\msc g(v)}{d-1},\quad v\in \bG_0,
\end{equation}
and the resolvent formula \eqref{e:resolv} implies
\begin{align}{\label{reso}}
  G(\cG_2',z)_{ij}-G(\cG_2,z)_{ij}=\sum_{v\in \bG_0}G_{iv}(\cG_2',z)(H_2-H_2')_{vv}G(\cG_2,z)_{vj}.
\end{align}
By multiplying both sides of \eqref{reso} by $(|\msc|q^{\dist_{\cG_0}(i,j)})^{-1}$, we obtain
\begin{align*}
  |G_{ij}(\cG_2,z)|\left(|\msc|q^{\dist_{\cG_0}(i,j)}\right)^{-1}
  &\leq \Gamma'+\sum_{v\in \bG_0} \Gamma'\Gamma |(H_2-H_2')_{vv}||\msc|q^{\dist_{\cG_0}(i,v)+\dist_{\cG_0}(v,j)-\dist_{\cG_0}(i,j)}\\
  &\leq \Gamma'+\frac{1}{d-1}\Gamma'\Gamma\sum_{v\in \bG_0}g(v)
    \leq 2^{\n+1}+\frac{8\n 2^{\n+1} }{d-1}\Gamma\leq 2^{\n+1}+\Gamma/2,
\end{align*}
where the first inequality uses the triangle inequality $\dist_{\cG_0}(i,v)+\dist_{\cG_0}(v,j)-\dist_{\cG_0}(i,j)\geq 0$ and $q\leq 1$,
and the second and third inequalities follow from the assumptions $\sum g(v) \leq 8\n$, $\sqrt{d-1} \geq 2^{\n+2}$, and $\n\geq 6$.
By taking the maximum on the left-hand side of the above inequality and rearranging it, we get $\Gamma\leq 2^{\n+2}$.
\end{proof}

\begin{proof}[Proof of \eqref{e:compatibility} for $g\not\equiv 0$]
The extension to the case $g\not\equiv 0$ again follows by comparing to the case $g\equiv 0$.
We define $\cH_2$ and $\cH_2'$ analogously to $\cG_2'$ and $\cG_2'$.
Our goal now is to bound
\begin{align*}
\Gamma:=\max_{i,j\in \bH_0}|G_{ij}(\cal G_2)-G_{ij}(\cal H_2)|\left(|\msc|q^{\ell(i,j)+1}\right)^{-1}.
\end{align*}
The resolvent identity \eqref{e:resolv} and \eqref{e:H2diag} imply
\begin{align}
\label{Gquan}G_{ij}(\cal G_2')-G_{ij}(\cal G_2)&=\sum_{v\in \bG_0}G_{iv}(\cal G_2')\frac{\msc g(v)}{d-1}G_{vj}(\cal G_2),\\
\label{Hquan}G_{ij}(\cal H_2')-G_{ij}(\cal H_2)&=\sum_{v\in \bH_0}G_{iv}(\cal H_2')\frac{\msc g(v)}{d-1}G_{vj}(\cal H_2).
\end{align}
For vertices $i,j\in \bH_0$, set
\begin{align*}
\ell(i,j):= \max\{\ell: \text{all paths in $\cG_0$ from $i$ to $j$ of length $\leq \ell$ are contained in $\cal H_0$}\},
\end{align*}
and given any $v\in \bG_0$, we abbreviate $\ell=\ell(i,j)$, $\ell_1=\dist_{\cG_0}(i,v)$, $\ell_2=\dist_{\cG_0}(v,j)$.
To bound $\Gamma$, we distinguish two cases:
\begin{enumerate}
\item
$\ell_1+\ell_2\geq \ell+1$.
Then already \eqref{e:boundPij} implies
\begin{align*}
\left|G_{iv}(\cal G_2')\frac{\msc g(v)}{d-1}G_{vj}(\cal G_2)\right|,\left|G_{iv}(\cal H_2')\frac{\msc g(v)}{d-1}G_{vj}(\cal H_2)\right|\leq \frac{2^{2\n+4}g(v)}{d-1}|\msc|q^{\ell+1}.
\end{align*}
\item
$\ell_1+\ell_2\leq \ell$.
Then by assumption we must have $v\in \bH_0$, and $\ell(i,v)\geq \ell-\ell_2$ and $\ell(v,j)\geq \ell-\ell_1$.
Therefore, using the case $g \equiv 0$ for $|G_{iv}(\cG_2')-G_{iv}(\cH_2')|$ and \eqref{e:boundPij} for $|G_{vj}(\cG_2)|$ and $|G_{iv}(\cH_2')|$,
\begin{align*}
&\left|G_{iv}(\cal G_2')\frac{\msc g(v)}{d-1}G_{vj}(\cal G_2)-G_{iv}(\cal H_2')\frac{\msc g(v)}{d-1}G_{vj}(\cal H_2)\right|
\\
&\leq \frac{|m_{sc}|g(v)}{d-1}
\pB{|G_{iv}(\cal G_2')-G_{iv}(\cal H_2')||G_{vj}(\cal G_2)|+
|G_{iv}(\cal H_2')||G_{vj}(\cal G_2)-G_{vj}(\cal H_2)|}
\\
&\leq \frac{2^{\n+2}(\Gamma+2^{2\n+2})g(v)}{d-1}|\msc|q^{\ell+1}.
\end{align*}
\end{enumerate}
Taking the difference of \eqref{Gquan} and \eqref{Hquan},
dividing both sides by $|\msc|q^{\ell(i,j)+1}$, and then taking the maximum over $i,j\in \bH_0$,
this leads to
\begin{align*}
\Gamma\leq 2^{2\n+2}+\frac{2^{\n+2}(\Gamma+2^{2\n+2})}{d-1}\sum_{v\in \bG_0}g(v).
\end{align*}
Since by assumptions $\sum g(v) \leq 8\n$, $\sqrt{d-1} \geq 2^{\n+2}$ and $\n\geq 6$,
again rearranging the above expression, we get $\Gamma\leq 2^{2\n+3}$. This finishes the proof.
\end{proof}

\section{Initial estimates}\label{sec:initial}

As the first step of the proof of Theorem~\ref{thm:mr},
we now show that \eqref{e:locallaw} holds whenever $|z| \geq 2d-1$.
Indeed, the following proposition states that \eqref{e:locallaw} holds
deterministically for $|z| \geq 2d-1$ under the assumption that the graph
has locally bounded excess, which is guaranteed to hold with high probability by \eqref{e:structure1}.
(Related results appear in \cite{MR3025715}.)

\begin{proposition} \label{prop:betacomp}
Let $\n\geq 6$, $\sqrt{d-1}\geq 2^{\n+2}$ and $N\geq N_0(\n, d)$ large enough.
Let $\cG$ be a $d$-regular graph on $N$ vertices, with excess at most $\n$ in any radius-$R$ neighborhood.
Then for any $z \in \C_+$ with $|z|\geq 2d-1$, and any $i,j \in \bG$, the Green's function of $\cG$ satisfies
 \begin{align}\label{largedgreen}
  \left|G_{ij}(z)-P_{ij}(\cE_r(i,j,\cal G),z)\right|\leq \frac{1}{2}|\msc|q^{r}.
 \end{align}
\end{proposition}

\subsection{Proof of Proposition~\ref{prop:betacomp}}

To prove Proposition~\ref{prop:betacomp}, we need an upper bound on the entries of the Green's function.
It can be obtained, for example, by the Combes--Thomas method \cite{MR0391792}.

\begin{lemma}\label{greendecay}
For any finite simple graph $\cal G$ with degree bounded by $d$, 
and any $z$ with $|z|\geq 2d-1$,
\begin{equation}\label{betabound}
 |G_{ij}(z)| \leq \frac{1}{d}, \qquad
 |G_{ij}(z)|\leq \frac{1}{(d-1)^{ \dist_\cal G(i,j)/2}}.
\end{equation}
\end{lemma}

\begin{proof}
We denote the normalized adjacency matrix of $\cG$ by $H$ (where we recall that the normalization of the entries is always by $1/\sqrt{d-1}$).
The first bound in \eqref{betabound} is immediate since
the spectrum of $H$ is contained
in $[-d/\sqrt{d-1},d/\sqrt{d-1}]$, which implies that
\begin{equation}
  |G_{ij}(z)| \leq \frac{1}{|z|-d/\sqrt{d-1}} \leq \frac{1}{d}.
\end{equation}
To show the second bound, set $\tau = \frac12 \log (d-1)$.
Fix a vertex $i$, and define the diagonal matrix $M$ by
\begin{align*}
  M_{jj}=\exp\{\tau \dist_\cal G(i,j)\}.
\end{align*}
Then we have
\begin{align*}
  G_{ij}e^{\tau \dist_\cal G(i,j)}
  =\langle \delta_j, MGM^{-1}\delta_i\rangle
  =\langle \delta_j,(MHM^{-1}-z)^{-1}\delta_i\rangle
  .
\end{align*}
The entries of the matrix $MHM^{-1}$ are given by
\begin{align*}
  (MHM^{-1})_{xy}=e^{\tau(\dist_\cal G(i,x)-\dist_\cal G(i,y))}H_{xy}.
\end{align*}
If $H_{xy}\neq 0$, then $|\dist_\cal G(i,x)-\dist_\cal G(i,y)|\leq 1$, and 
\begin{align*}
  \max_x \sum_y |(MHM^{-1})_{xy}|\leq de^\tau/\sqrt{d-1}\leq d,\\
  \max_y \sum_x |(MHM^{-1})_{xy}|\leq de^\tau/\sqrt{d-1}\leq d.
\end{align*}
Therefore $\|MHM^{-1}\|_{\infty\to \infty}$ and $\|MHM^{-1}\|_{1\to 1}$ are bounded by $d$,
and by interpolation 
\begin{align*}
 \|MHM^{-1}\|_{2\to 2} \leq \sqrt{\|MHM^{-1}\|_{1\to 1}\|MHM^{-1}\|_{\infty\to\infty}}\leq d.
\end{align*}
Therefore, the spectrum of $MHM^{-1}$ is contained in the set $\left\{z \in \C: |z|\leq d\right\}$.
In particular, for $z$ such that $|z|\geq 2d-1$, its distance to the spectrum of $MHM^{-1}$ is at least $1$,
and thus
 \begin{align*}
  |G_{ij}e^{\tau\dist_\cal G(i,j)}|=|\langle\delta_j,(MHM^{-1}-z)^{-1}\delta_i\rangle|\leq 1,
 \end{align*}
which implies \eqref{betabound}.
This completes the proof.
\end{proof}

\begin{proof}[Proof of Proposition~\ref{prop:betacomp}]
Let $r_0:=\lceil r+1-2(r+2)\log_{d-1}|\msc|\rceil=O(r)$.
Then, for vertices $i,j$ such that $\dist_\cal G(i,j)\geq r_0$, Lemma~\ref{greendecay} implies
\begin{align*}
|G_{ij}(z)|\leq \frac{1}{(d-1)^{r_0/2}}\leq |\msc|q^{r+1},
\end{align*}
and in particular \eqref{largedgreen} follows since $q \leq 1/\sqrt{d-1} \leq 1/2$
and  $P_{ij}(\cE_r(i,j,\cG))=0$.

Thus we can assume $\dist_\cG(i,j)< r_0$.
Let $\cG_0:=\cal B_{r_0+r}(i, \cG)$, 
let $\cG_1 = \TE(\cG_0)$ be the tree extension of $\cG_0$,
and let $P$ be the Green's function of $\cG_1$.
Then, by \eqref{e:compatibility}, we have
\begin{equation}  \label{e:initial-Papprox}
|P_{ij}-P_{ij}(\cE_r(i,j,\cG))|\leq 2^{2\n+3}|m_{sc}|q^{r+1}.
\end{equation}
Therefore it suffices to prove the claim with $P_{ij}(\cE_r(i,j,\cG))$ replaced by $P_{ij}$, and an additional factor $1/2$
on the right-hand side.
Let $\T_0:= \bB_{r_0}(i,\cG)$ 
and $\partial\T_0 = \{v\in \cal G: \dist_{\cal G}(v,\T_0)=1\}$.  
By the Schur complement formula \eqref{e:Schur1},
\begin{align*}
G|_{\T_0}=(H-z-B'G^{(\T_0)} B)^{-1},\\
P|_{\T_0}=(H-z-B'P^{(\T_0)}B)^{-1},
\end{align*}
where $H$ is the normalized adjacency matrix on $\T_0$ induced by $\cG$
and $B$ is the part of the adjacency matrix of the edges from $\del \T_0$ to $\T_0$.
Taking the difference of the last two equations, for any $i,j\in \T_0$,
\begin{align*}
|(G-P)_{ij}|
\leq& \sum_{x,y\in \del \T_0}|(PB')_{ix}|\pa{|G^{(\T_0)}_{xy}|+|P^{(\T_0)}_{xy}|}|(BG)_{yj}|.
\end{align*}
Since the radius-$R$ neighborhood of $i$ has excess at most $\n$, each row of $B$ contains at most $\n+1$ nonzero entries.
Therefore, by \eqref{e:boundPij}, Lemma~\ref{greendecay}, and noticing that $\dist_{\cG}(i,x)\geq r_0+1$ and $\dist_{\cG}(y,j)\geq r_0+1-\dist_{\cG}(i,j)$, we have
\begin{align*}
|(PB')_{ix}|\leq 2^{\n+2}(\n+1)q^{r_0+1}, \quad |(BG)_{yj}|\leq \frac{\n+1}{(d-1)^{(r_0+1-\dist_\cG(i,j))/2}},
\end{align*}
where we recall the definition $q=|\msc|/\sqrt{d-1}$.
Moreover, 
it follows from \eqref{fewclose}
that 
\begin{align*}
|\{x\in \del \T_0: \dist_{\cG^{(\T_0)}}(x, \del \T_0\setminus \{x\})\leq R/2\}|\leq 2\n,
\end{align*}
using that $R>2r_0$.
Therefore, by the second bound of \eqref{betabound}, $|G^{(\T_0)}_{xy}|\leq (d-1)^{-R/4}$
for all $x,y\in \del \T_0$ except for the diagonal entries and at most $4\n^2$ off-diagonal entries.
By the first bound of \eqref{betabound}, for these remaining entries we have
$|G^{(\T_0)}_{xy}| \leq 1/d$. The same bounds hold for $P^{(\T_0)}$ instead of $G^{(\T_0)}$.
As a result, we obtain
\begin{align*}
|(G-P)_{ij}|
&\leq \frac{2^{\n+2}(\n+1)^2q^{r_0+1}}{(d-1)^{(r_0-\dist_\cal G(i,j)+1)/2}} \sum_{x,y\in \del\T_0}\pa{|(G^{(\T_0)}_{xy}| + |P^{(\T_0)})_{xy}|}\\
&\leq \frac{2^{\n+3}(\n+1)^2|\msc|^{r_0+1}}{(d-1)^{r_0+1-\dist_\cal G(i,j)/2}} \left(\frac{|\del \T_0|+4\n^2}{d}+\frac{|\del\T_0|^2}{(d-1)^{R/4}}\right).
\end{align*}
Using that $|\del \T_0| \leq d(d-1)^{r_0}$, 
that $|\msc|\leq 1/d$ for $|z|\geq 2d-1$,
that $d-1\geq 2(\n+1)$,
as well as that $R>4r_0$,
the right-hand side is bounded by
\begin{align*}
|(G-P)_{ij}|
&\leq \frac{2^{\n+3}(\n+1)^2|\msc|^{r_0+1}}{(d-1)^{r_0+1-\dist_\cal G(i,j)/2}} \left(\frac{d(d-1)^{r_0}+4\n^2}{d}+\frac{d^2(d-1)^{2r_0}}{(d-1)^{R/4}}\right)\\
&\leq \frac{2^{\n+4}(\n+1)^2}{(d-1)^{r_0+2-\dist_\cal G(i,j)/2}}\leq \frac{2^{\n+2}}{(d-1)^{r_0/2}}\leq 2^{\n+2}|\msc|q^{r+1},
\end{align*}
where we used that $\dist_{\cG}(i,j)<r_0$.
Together with \eqref{e:initial-Papprox}, we conclude that
\begin{align}
\left|G_{ij}(z)-P_{ij}(\cE_r(i,j,\cal G),z)\right|
\leq
(2^{2\n+3}+2^{\n+2})|\msc|q^{r+1}
\leq
\frac{1}{2}|\msc|q^{r},
\end{align}
where the last inequality follows from $q=|\msc|/\sqrt{d-1}\leq 2^{-3(\n+2)}$,
using that $\sqrt{d-1}\geq 2^{\n+2}$.
\end{proof}

\section{Local resampling by switching}
\label{sec:switch}

In this section, we define a \emph{local resampling} of a random regular graph by using switchings.
We effectively resample the edges on the boundary of balls of radius $\ell$,
by switching them with random edges from the remainder of the graph.
This resampling generalizes the local resampling introduced in \cite{1503.08702-cpam},
where switchings were used to resample the neighbors of a vertex (corresponding to $\ell=0$).
The local resampling provides an effective access to the randomness of the random regular graph,
which is fundamental for the remainder of the paper.

\subsection{Definitions}

To introduce the local resampling, we require some definitions.

\paragraph{Graphs and edges}

We consider simple $d$-regular graphs on vertex set $\qq{N}$
and identify such graphs with their sets of edges throughout this section.
(Deficit functions do not play a role in this section.)
For any graph $\cG$, we denote the set of unoriented edges by $E$,
and the set of oriented edges by $\vec{E}:=\{(u,v),(v,u):\{u,v\}\in E\}$.
For a subset $\vec{S}\subset \vec{E}$, we denote by $S$ the set of corresponding non-oriented edges.
For a subset $S\subset E$ of edges we denote by $[S] \subset \qq{N}$ the set of vertices incident to any edge in $S$.
Moreover, for a subset $\bV\subset\qq{N}$ of vertices, we define $E|_{\bV}$ to be the subgraph of $\cG$ induced on $\bV$.

\begin{figure}[h]
\centering
\input{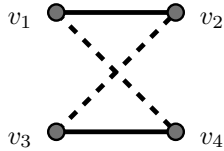}
\caption{
The switching encoded by the two directed edges $\vec S=\{(v_1, v_2), (v_3, v_4)\}$
replaces the unoriented edges $\{v_1,v_2\}, \{v_3,v_4\}$ by $\{v_1,v_4\},\{v_2,v_3\}$.
\label{fig:switching}}
\end{figure}

\paragraph{Switchings}

A (simple) switching is encoded by a pair of oriented edges $\vec S=\{(v_1, v_2), (v_3, v_4)\} \subset \vec{E}$.
We assume that the two edges are disjoint, i.e.\ that $|\{v_1,v_2,v_3,v_4\}|=4$.
Then the switching consists of
replacing the edges $\{v_1,v_2\}, \{v_3,v_4\}$ by the edges $\{v_1,v_4\},\{v_2,v_3\}$,
as illustrated in Figure~\ref{fig:switching}.
We denote the graph after the switching $\vec S$ by $T_{\vec S}(\cG)$,
and the new edges $\vec S' = \{(v_1,v_4), (v_2,v_3)\}$ by
\begin{equation}
  T(\vec S) = \vec S'.
\end{equation}
(Double switchings, which we used in \cite{1503.08702-cpam}, are not needed in this paper;
henceforth we will therefore refer to simple switchings as switchings.)

\paragraph{Resampling data}

\begin{figure}[t]
\centering
\input{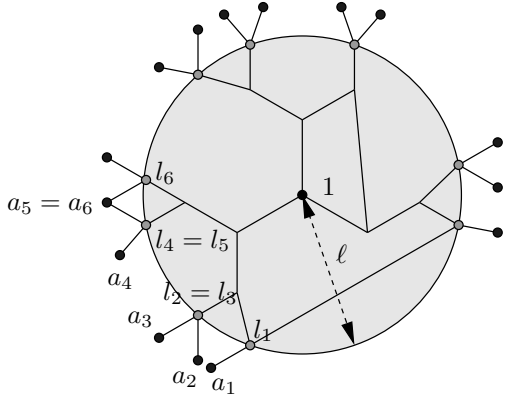}
\caption{
The figure illustrates the neighborhood $\cT = \cB_\ell(1,\cG)$ (within the shaded area)
and its edge boundary $\del_E \cT$, consisting of the edges $e_i=\{l_i,a_i\}$, $1\leq i\leq \mu$.
Our local resampling switches the switchable edges $e_i$ (corresponding to $i\in W_{\bf S}$)
with randomly chosen edges from the remainder of the graph (not shown).
Several exceptional cases can occur.
In particular, the vertices $a_i$ are not necessarily distinct (e.g., $a_5=a_6$ in the figure),
and the boundary vertices $l_i$ may have different degrees in the graph obtained by removing the set $\cT$
(e.g., $l_1$ has only one outgoing edge in the figure, while most of the other $l_i$ has two outgoing edges).
\label{fig:resample1}}
\end{figure}

Our local resampling involves a center vertex, which by symmetry we now assume to be $1$,
and a radius $\ell$.
Given a $d$-regular graph $\cG$, we abbreviate $\bT=\bB_\ell(1, \cG)$ and $\cT=\cB_{\ell}(1,\cG)$.
The edge boundary $\del_E \cT$ of $\cT$ consists of the edges in $\cG$ with one vertex in $\T$ and the other vertex in $\qq{N}\setminus\T$,
as illustrated in Figure~\ref{fig:resample1}.
Our local resampling switches the edge boundary of $\cT$ with randomly chosen edges in $\cGT$
if the switching is admissible (see below), and leaves them in place otherwise.

To be precise, given a graph $\cG$, 
we enumerate $\del_E \cT$ as $ \del_E \cT = \{e_1,e_2,\dots, e_\mu\}$,
and orient the edges $e_i$ by defining $\vec{e}_i$ to have the same vertices as $e_i$
and to be directed from a vertex $l_i \in \T$ to a vertex $a_i \in \qq{N} \setminus \T$.
The directed edges $\vec{e}_i = (l_i,a_i)$ are illustrated in Figure~\ref{fig:resample1}.
Note that $\mu$ and the edges $e_1, \dots, e_\mu$ depend on $\cG$.

Then we choose $(b_1,c_1), \dots, (b_\mu,c_\mu)$ to be independent, uniformly chosen oriented edges from the graph $\cGT$, i.e.,
the edges of $\cG$ that are not incident to $\T$,
and define 
\begin{equation}
  \vec{S}_i = \{\vec{e}_i, (b_i,c_i)\},
  \qquad
  {\bf S}=(\vec S_1, \vec S_2,\dots, \vec S_\mu).
\end{equation}
The sets $\bf S$ will be called the \emph{resampling data} for $\cG$.
By definition, the edges $e_i$ are distinct, but
the vertices $a_i$ are not necessarily distinct and neither are the vertices $l_i$.

\paragraph{Admissible switchings}

For $i\in\qq{1,\mu}$,
we define the indicator functions
\begin{align*}
I_i &\equiv I_i(\cG,{\bf S}):={\bf 1}(|[S_{i}]|=4, E|_{[S_{i}]}=S_{i}),\\
J_i &\equiv J_i(\cG,{\bf S}):={\bf 1}(|[S_i]\cap [S_j]|\leq 1 \text{ for all } j\neq i),
\end{align*}
and the set of \emph{admissible switchings}
\begin{align}\label{Wdef}
W_{\bf S}\equiv W(\cG,{\bf S}):=\{i\in \qq{1,\mu}: I_i(\cG,{\bf S}) J_i(\cG,{\bf S})=1\}.
\end{align}
The interpretation of $I_i=1$ is that the graph $E|_{[S_i]}$ is $1$-regular.
The interpretation of $J_i=1$ is that the edges of $S_i$ do not interfere with the edges of any other $S_j$.
Indeed, the condition $|[S_i] \cap [S_j]| \leq 1$ guarantees that 
the switchings encoded by $\vec S_i$ and $\vec S_j$ do not influence each other,
meaning that $T_{\vec S_i}$ and $T_{\vec S_j}$ commute. 
We say that the index $i \in \qq{1,\mu}$ is \emph{admissible} or \emph{switchable} if $i\in W_{\bf S}$.

Let $\nu:=|W_{\bf S}|$ be the number of admissible switchings and $i_1,i_2,\dots, i_{\nu}$
be an arbitrary enumeration of $W_{\bf S}$.
Then we define the switched graph by
\begin{equation} \label{e:Tdef1}
T_{\bf S}(\cG) := \left(T_{\vec S_{i_1}}\circ \cdots \circ T_{\vec S_{i_\nu}}\right)(\cG)
\end{equation}
and the switching data by
\begin{equation} \label{e:Tdef2}
  T({\bf S}) := (T_i(\vec{S_i}), \dots, T_\mu(\vec{S_\mu})), 
  \quad
  T_i(\vec{S}_i) =
  \begin{cases}
    T(\vec{S}_i) & (i \in W_{\bf S})\\
    \vec{S}_i & (i \not\in W_{\bf S}).
  \end{cases}
\end{equation}

\subsection{Reversibility}\label{sec:rev}

To make the structure more clear, we introduce an enlarged probability space.
Equivalently to the definition above, the sets $\vec{S}_i$ are uniformly distributed over 
\begin{align*}
{\sf S}_{i}(\cG)=\{\vec S\subset \vec{E}: \vec S=\{\vec e_i, \vec e\}, \text{$\vec{e}$ is not incident to $\T$}\},
\end{align*}
i.e., the set of pairs of oriented edges in $\vec{E}$ containing $\vec{e}_i$ and another oriented edge in $\cGT$.
Therefore ${\bf S}=(\vec S_1,\vec S_2,\dots, \vec S_\mu)$ is uniformly distributed over the set
${\sf S}(\cG)=\sf S_1(\cG)\times \cdots \times \sf S_\mu(\cG)$.

\begin{definition}
For any graph $\cG\in \GNd$,  denote by $\iota(\cG) = \{\cG\} \times \sf S(\cG)$ the \emph{fibre}
of local resamplings of $\cG$ (with respect to vertex $1$),
and define the enlarged probability space
\begin{align*}
\GNdp = \iota(\GNd) = \bigsqcup_{\cG\in \GNd}\iota(\cG)
\end{align*}
with the probability measure $\Pp(\cG, {\bf S}):= \P(\cG)\P_{\cG}({\bf S}) = (1/|\GNd|)(1/|\sS(\cG)|)$
for any $(\cG, {\bf S})\in  \GNdp$.
Here $\P(\cG)=1/|\GNd|$ is the uniform probability measure on $\GNd$,
and for $\cG\in\GNd$, we denote the uniform probability measure on $\sS(\cG)$ by $\P_{\cG}$.
\end{definition}

Let $\pi: \GNdp \to \GNd$, $(\cG,{\bf S}) \mapsto \cG$ be the canonical projection onto the first component.
\begin{proposition}
$\pi$ is measure preserving: $\P = \Pp \circ \pi^{-1}$.
\end{proposition}

\begin{proof}
Note that $\pi^{-1}(\cG) = \iota(\cG)$. Therefore
\begin{equation}
  \Pp(\pi^{-1}(\cG)) = \Pp(\iota(\cG)) = \sum_{{\bf S} \in \sS(\cG)} \Pp(\cG,{\bf S}) = \P(\cG) \sum_{{\bf S} \in \sS(\cG)} \frac{1}{|\sS(\cG)|} = \P(\cG),
\end{equation}
as claimed.
\end{proof}

On the enlarged probability space, we define the maps
\begin{alignat}{2}
\label{e:Ttildedef}
\tilde T &: \GNdp \to \GNdp, &\quad
\tilde T(\cal G, {\bf S}) &:= (T_{\bf S}(\cal G), T({\bf S})),
\\
\label{e:Tdef}
T &: \GNdp \to \GNd, &\quad
T(\cal G, {\bf S}) &:= \pi(\tilde T(\cG,{\bf S})) = T_{\bf S}(\cal G).
\end{alignat}
For the statement of the next proposition,
recall that $\GNd$ denotes the set of simple $d$-regular graph on $\qq{N}$.
For any finite graph $\cT$ on a subset of $\qq{N}$,
we define $\GNd(\cT):=\{\cG\in \GNd: \cB_\ell(1,\cG)=\cT\}$
to be the set of $d$-regular graphs whose radius-$\ell$ neighborhood of the vertex $i$ in $\cG$ is $\cT$.

\begin{proposition} \label{prop:reverse}
For any graph $\cT$, we have
\begin{equation}\label{e:tildeTT}
  \tilde T(\iota(\GNd(\cT))) \subset \iota(\GNd(\cT)),
\end{equation}
and $\tilde T$ is an involution: $\tilde T \circ \tilde T = \id$.
\end{proposition}

\begin{proof}
The first claim is obvious by construction.
To verify that $\tilde T$ is an involution,
let $(\cG, {\bf S}) \in \GNdp$ and abbreviate $(\tcG, \td {\bf S}) = \tilde T(\cG, {\bf S})$.
Then, due to \eqref{e:tildeTT}, the edge boundaries of the $\ell$-neighborhoods of $1$ have the same number of edges $\mu$
in $\tcG$ and $\cG$.
Moreover, we can choose the (arbitrary) enumeration of the boundary of the $\ell$-ball in $\tcG$ such that,
for any $i \in \qq{1,\mu}$, we have $T_i(\vec S_i) \in \sS_i(\tcG)$.
Define
\begin{align*}
\td W_{\tilde {\bf S}}\equiv W(\tcG,\td{\bf S}):=\{i\in \qq{1,\mu}: I_i(\tcG,\td{\bf S}) J_i(\tcG,\td{\bf S})=1\}.
\end{align*}
We claim that $\td W_{\tilde {\bf S}} = W_{\bf S}$.
First, by definition of switchings, we have $[T_i(S_i)]=[S_i]$ for any $i\in \qq{1,\mu}$.
Thus $J_i(\tcG, \td {\bf S})=J_i(\cG,{\bf S})$,
and it suffices to verify that $I_i(\tcG, \td{\bf S}) = I_i(\cG,{\bf S})$ also holds for all $i \in \qq{1,\mu}$.
For $i\not\in W_{\bf S}$, the switching of $S_i$ does not take place, i.e., $\tcG|_{[S_i]}=\cG|_{[S_i]}$
and therefore $I_i(\tcG, \td{\bf S})=I_i(\cG, {\bf S})$.
On the other hand, for $i\in W_{\bf S}$, the subgraph $\cG|_{[S_i]}$ is $1$-regular, i.e., $I_i(\cG, {\bf S}) =1$,
and the other $S_j$ with $j \in W_{\bf S}$ intersect $S_i$ at most at one vertex. Therefore, $\tcG|_{[S_i]}=T_{\vec S_i}\cG|_{[S_i]}$ and 
the graph $\tcG|_{[S_i]}$ is again $1$-regular, i.e., $I_i(\tcG, \td{\bf S})=1$ as needed.

In summary, we have verified the claim $\td W_{\tilde {\bf S}}=W_{\bf S}$. By definition of our switchings,
it follows that $T(\td {\bf S})=\bf S$ and $T_{\td{\bf S}}(\tcG)=\cG$.
Therefore $\tilde T$ is an involution.
\end{proof}

\begin{proposition}
$\tilde T$ and $T$ are measure preserving: $\Pp \circ \tilde T^{-1} = \Pp$ and $\Pp \circ T^{-1} = \P$.
\end{proposition}

In other words, that $T$ is measure preserving means that
if $\cG$ is uniform over $\GNd$, and given $\cG$, we choose $\bf S$ uniform over $\sS(\cG)$,
then $T_{\bf S}(\cal G)$ is uniform over $\GNd$.

\begin{proof}
We decompose the enlarged probability space according to the $\ell$-neighborhood of $1$ as
\begin{equation} \label{e:GNdpdecomp}
  \GNdp = \bigcup_{\cT} \GNdp(\cT), \quad \text{where } \GNdp(\cT) = \iota(\GNd(\cT)).
\end{equation}
Notice that, given any $\cT$, the size of the set $\sS(\cG)$ is (by construction) independent of the graph $\cG \in \GNd(\cT)$. 
Therefore, given any $\cT$, the restricted measure $\Pp|{\GNdp(\cT)}$ is uniform, i.e.,
proportional to the counting measure on the finite set $\GNdp(\cT)$.
Since, by Proposition~\ref{prop:reverse},
the map $\tilde T$ is an involution on $\GNdp(\cT)$, it is in particular a bijection
and as such preserves the uniform measure $\Pp|{\GNdp(\cT)}$.
Since $\tilde T$ acts diagonally in the decomposition \eqref{e:GNdpdecomp},
this implies that the map $\tilde T$ preserves the measure $\Pp$.
Since $\P = \Pp \circ \pi^{-1}$ and $T = \pi \circ \tilde T$,
it immediately follows that also $T$ is measure preserving:
\begin{equation*}
\Pp \circ T^{-1} = \Pp \circ \tilde T^{-1} \circ \pi^{-1} = \Pp \circ \pi^{-1} = \P,
\end{equation*}
as claimed.
\end{proof}

\begin{figure}[h]
\centering\input{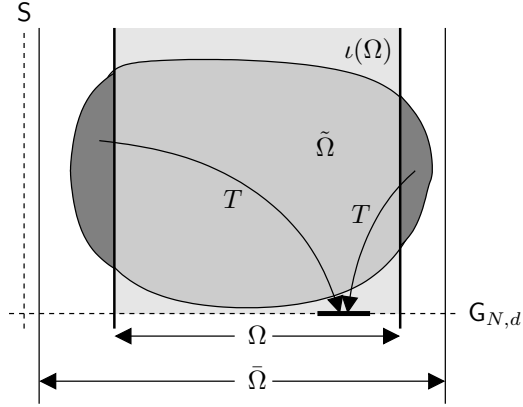}
\caption{
The figure illustrates the idea of 
Proposition~\ref{prop:resample}.
The horizontal axis represents the set of graphs $\GNd$,
and the vertical direction the fibres of possible switchings.
In particular, the sets $\Omega$, $\Omega'$, $\Omega^+$, $\bar\Omega$ are represented on the horizontal axis.
The area in medium and dark grey represents $\tilde \Omega = T^{-1}(\Omega)$.
The sets $\Omega'$ and $\Omega^+$ and their preimages can be illustrated analogously,
but for simplicity we assume for the figure that $\Omega = \Omega^+$.
The lightly shaded area bounded by the vertical bars is $\iota(\Omega)$.
In \eqref{e:Omegasplit}, we devide $\tilde\Omega \setminus \tilde\Omega'$ into the part contained in $\iota(\Omega^+)$
(the second term) and the part outside of $\iota(\Omega^+)$ (the first term).
The part inside $\iota(\Omega^+)$ is small because of assumption (iii).
To bound the part outside $\iota(\Omega^+)$, we use that $\tilde T$ is an involution.
This implies that the image under $\tilde T$ of the area in dark gray
is contained in $\iota(\Omega)$ (thus its projection to the horizontal axis lies in $\Omega$ as shown in the figure),
and not intersecting $\tilde\Omega^+$.
Its contribution is small by assumption (ii), which implies that $\iota(\Omega)$ contains most of $\tilde\Omega^+$.
\label{fig:Omega}}
\end{figure}

The following general proposition, which makes use of the involution property of $\tilde T$,
is central to our approach.
The idea of its proof is illustrated in Figure~\ref{fig:Omega}.

\begin{proposition} \label{prop:resample}
Given events $\Omega \subset \Omega^+\subset\bar\Omega \subset \GNd$ and $\Omega'\subset \bar\Omega$,
assume
\begin{enumerate}
\item
$\P(\GNd \setminus \bar\Omega)\leq q_0$,
\item
$\P_{\cG}(T_{\bf S}(\cal G) \in \bar\Omega\setminus \Omega^+) \leq q_1$ for all $\cal G \in \Omega$,
and
\item
$\P_{\cG}(T_{\bf S}(\cal G) \in \bar\Omega\setminus \Omega') \leq q_2$ for all $\cal G \in \Omega^+$.
\end{enumerate}
Then
$
  \P(\Omega \setminus (\Omega \cap \Omega')) \leq q_0+q_1+q_2
$.
\end{proposition}

Roughly, the proposition shows that if, for most graphs $\cG \in \GNd$, an event holds  for the
switched graph $T_{\bf S}(\cG)$ with high probability under the randomness of the switching $\bf S$,
then it also holds with high probability on $\GNd$.
This enables us to condition on a (good) graph $\cG$ for much of the paper, and then only
use with the randomness of ${\bf S}$ which has a simple probabilistic structure.

More specifically,
in our application of the proposition, the set $\bar\Omega$ is a large set of regular graphs obeying 
rough a priori estimates (there are only few cycles), the set $\Omega$ is a set
of graphs for which the Green's function obeys good estimates, and the set $\Omega'$ is
a sets of graphs on which the Green's function obeys better estimates (near a given vertex).
The proposition states that if with respect to the resampling most graphs obey the better
estimates, then these estimates also hold on the original probability space with high probability.
The sets to which the proposition will be applied are further discussed in Section~\ref{sec:outline}.
The proposition will be applied in Section~\ref{sec:pfmr}.

\begin{proof}
We define $\tilde\Omega = T^{-1}(\Omega)$, $\tilde\Omega' = T^{-1}(\Omega')$ and $\tilde \Omega^+ = T^{-1}(\Omega^+)$,
and abbreviate $A \setminus B = A \setminus (A \cap B)$ for any sets $A,B$. Since
\begin{equation}
  T^{-1}(\Omega \setminus \Omega')
  = T^{-1}(\Omega) \setminus T^{-1}(\Omega') = \tilde\Omega \setminus \tilde \Omega',
\end{equation}
and since $T$ is measure preserving,
and since $\tilde T$ is a measure preserving involution, we have
\begin{align} \label{e:Omegasplit}
  \P(\Omega \setminus \Omega')
  =
  \tilde \P(\tilde\Omega \setminus \tilde \Omega')
  &=
  \tilde \P(\tilde\Omega \setminus (\tilde\Omega' \cup \iota(\Omega^+)))
  + 
  \tilde \P((\iota(\Omega^+) \cap \tilde\Omega) \setminus \tilde \Omega')
  \nonumber\\
  &\leq
  \tilde \P(\tilde\Omega \setminus (\tilde\Omega' \cup \iota(\Omega^+)))
  + 
  \tilde \P((\iota(\Omega^+) \cap T^{-1}(\bar\Omega)) \setminus \tilde \Omega')
  \nonumber\\
  &=
  \tilde \P(\hat\Omega)
  + 
  \tilde \P((\iota(\Omega^+) \cap T^{-1}(\bar\Omega)) \setminus \tilde \Omega')
  ,
\end{align}
where $\hat\Omega
= \tilde T(\tilde\Omega \setminus (\tilde\Omega' \cup \iota(\Omega^+))$.
To bound the probability of $\hat\Omega$, we make the following observations.
First,
$\hat\Omega \subset \tilde T(\tilde\Omega) \subset \iota(\Omega)$.
Second,
any element $(\cal {\hat G}, \hat {\bf S}) \in \hat \Omega$ can be written as
$\cal {\hat G} = T(\cG,{\bf S})$ for some $\cG \not \in \Omega^+$ and ${\bf S} \in \sS(\cG)$.
Since $\tilde T$ is an involution, this $(\cG, {\bf S})$ must in fact be given by
$(\cG, {\bf S})=\tilde T(\cal {\hat G}, \hat {\bf S})$.
Together this implies that $(\cal {\hat G}, \hat {\bf S})\not\in \tilde\Omega^+$,
and thus that $\hat\Omega$ has no intersection with $\tilde\Omega^+$.
As a result,
\begin{align*}
  \P(\Omega \setminus \Omega')
  &\leq
  \Pp(\iota(\Omega) \setminus \tilde \Omega^+)
  +
  \Pp((\iota(\Omega^+) \cap T^{-1}(\bar\Omega)) \setminus \tilde \Omega')\\
  &= \Pp(\iota(\Omega)\setminus  T^{-1}(\bar\Omega))
    +\Pp((\iota(\Omega) \cap T^{-1}(\bar\Omega) )\setminus \tilde \Omega^+)
    + \Pp((\iota(\Omega^+)\cap T^{-1}(\bar\Omega)) \setminus \tilde \Omega')\\
  &\leq \Pp(T^{-1}(\GNd)\setminus  T^{-1}(\bar\Omega))
    +\Pp((\iota(\Omega) \cap T^{-1}(\bar\Omega) )\setminus \tilde \Omega^+)
    + \Pp((\iota(\Omega^+)\cap T^{-1}(\bar\Omega)) \setminus \tilde \Omega')\\
  &\leq
 q_0+ q_1+q_2,
\end{align*}
where
the second inequality follows since $\GNd \supset \Omega$ and
the last inequality follows from the assumptions (\rn 1)--(\rn 3).
\end{proof}

\subsection{Estimates for local resampling}

In the following, we give some basic estimates for the local resampling.
In particular, we show that, with high probability, most edges are switchable.

\begin{proposition} \label{lem:switchable}
Let $\delta>0$.
\begin{enumerate}
\item For any $x \in \qq{N}\setminus \T$, 
\begin{equation} \label{e:approxunifbd}
  \P_{\cG}(b_i = x)=  \P(c_i = x)
  \leq \frac{2}{N}.
\end{equation}
\item
For any positive integer $\n$, we have
\begin{equation} \label{e:Wbd}
  \P_{\cG}(|W_{\bf S}|> \mu-3\n)= 1-o(N^{-\n+\delta}).
\end{equation}
\end{enumerate}
\end{proposition}

\begin{proof}
To prove (\rn 1), we recall that, for any $i$,
the oriented edge $(b_i,c_i)$ is uniformly chosen from the oriented edges of $\cGT$.
By definition of $\T$, there are at least $Nd/2-(d+d(d-1)+\cdots d(d-1)^{\ell})$ edges in $\cGT$,
and since for any vertex $x\in \cGT$, the degree obeys $\deg_{\cGT}(x)\leq d$,
\begin{align*}
  \P_{\cG}(b_i = x)= \P_{\cG}(c_i = x)\leq \frac{d}{Nd-2(d+d(d-1)+\cdots d(d-1)^{\ell})}
  \leq \frac{2}{N}.
\end{align*}

To prove (\rn 2), we need to analyze the events $I_iJ_i=0$ more carefully. We define the disjoint sets
\begin{align*}
A_0&=\{i\in \qq{1,\mu}: I_i=0\},\\
A_1&=\{i\in \qq{1,\mu}\setminus A_0: |\{b_i,c_i\}\cap (\cup_{j \neq i} [e_{j}])|\geq 1\},\\
A_2&=\{i\in \qq{1,\mu}\setminus A_0\cup A_1: |\{b_i,c_i\}\cap (\cup_{j \neq i} \{b_j,c_j\})|\geq 1\},\\
A_3&=\{i\in \qq{1,\mu}\setminus A_0\cup A_1\cup A_2: \text{there exists $j$ such that } l_i = l_j \text{ and } |[e_i] \cap \{b_j,c_j\}|\geq 1\},
\end{align*}
and claim that
\begin{equation} \label{e:IJ0decomp}
  \qq{1,\mu} \setminus W_{\bf S} = \{ i \in \qq{1,\mu} : I_iJ_i = 0 \} \subset A_0 \cup A_1 \cup A_2 \cup A_3.
\end{equation}
Indeed, if $i \in \qq{1,\mu} \setminus W_{\bf S}$, then $I_i=0$ or $J_i=0$.
Clearly, if $I_i=0$ then $i \in A_0 \subset A_0 \cup A_1 \cup A_2 \cup A_3$.
On the other hand, if $J_i=0$, there exists some index $j\in \qq{1,\mu} \setminus\{i\}$ such that $|[S_i]\cap [S_j]|> 1$,
and there are two possibilities (recall that $e_i=\{l_i,a_i\}$ and $e_j=\{l_j,a_j\}$):
\begin{enumerate}
\item $l_i\neq l_j$. Then either $|\{b_i,c_i\}\cap \{b_j,c_j\}|\geq 1$; or $|\{b_i,c_i\}\cap [e_j]|\geq 1$ and $|[e_i]\cap \{b_j,c_j\}|\geq 1$.
\item $l_i= l_j$. Then either $|\{b_i,c_i\}\cap \{b_j,c_j\}|\geq 1$; or $|\{b_i,c_i\}\cap [e_j]|\geq 1$; or $|[e_i]\cap \{b_j,c_j\}|\geq 1$.
\end{enumerate}
Either way, $J_i=0$ implies $i \in A_1 \cup A_2 \cup A_3$, and \eqref{e:IJ0decomp} holds.
To bound the number of elements on the right-hand side of \eqref{e:IJ0decomp}, we first note that $|A_3|\leq 2|A_0\cup A_1|$.
In fact if $i\in A_3$, then there exists some $j$ such that $|[e_i] \cap \{b_j,c_j\}|\geq 1$, and thus $j\in A_0\cup A_1$.
Since any $\{b_j,c_j\}$ can intersect at most two edges $e_i$ with $l_i=l_j$,
\begin{align*}
|A_3|\leq \sum_{j\in A_0\cup A_1}|\{i\in \qq{1,\mu}: l_i=l_j \text{ and } |\{b_j,c_j\}\cap [e_i]|\geq 1 \}| \leq 2|A_0\cup A_1|.
\end{align*}
Therefore, it follows that
\begin{equation}
  |\qq{1,\mu} \setminus W_{\bf S}| = |\{i \in \qq{1,\mu}: I_iJ_i = 0\}| \leq 3|A_0 \cup A_1|+|A_2|\leq 3|A_0|+3| A_1|+|A_2|.
\end{equation}
We will show that
\begin{equation} \label{e:A1A0A2claim}
  \P(|A_0|+|A_1|+\tfrac12 | A_2| \geq \n) = o(N^{-\n+\delta}),
\end{equation}
which implies the claim since
\begin{align*}
  \P(|\qq{1,\mu} \setminus W_{\bf S}| \geq 3\n) \leq \P(3|A_0| + 3|A_1|+ |A_2| \geq 3\n)
  \leq \P(|A_0|+|A_1|+\tfrac12 |A_2|\geq \n) = o(N^{-\n+\delta}).
\end{align*}

To prove \eqref{e:A1A0A2claim}, first notice that there is a subset $A_2' \subset A_2$ with $|A_2'| \geq |A_2|/2$ such that
$i \in A_2'$ implies $|\{b_i,c_i\}\cap (\cup_{j \not\in A_2'} \{b_j,c_j\})|\geq 1$.
Hence, if $|A_0| +|A_1| + |A_2|/2 \geq \n$,
then there exist disjoint index sets
$\tilde A_0\subset A_0, \tilde A_1\subset A_1, \tilde A_2' =A_2' $ such that
$|\tilde A_0| + |\tilde A_1|+ |\tilde A_2'|=\n$ and
\begin{align}\label{e:conditionforB}
\begin{cases}
&\forall i\in \tilde A_0,\qquad I_i=0, \\
&\forall i\in \tilde A_1,\qquad |\{b_i,c_i\}\cap (\cup_{j \neq i} [e_{j}])|\geq 1,\\
&\forall i\in \tilde A_2',\qquad |\{b_i,c_i\}\cap (\cup_{j \not\in A_2'} \{b_j,c_j\})|\geq 1.
\end{cases}
\end{align}
The condition $I_i=0$ is equivalent to $\dist_{\cG}(\{a_i,l_i\}, \{b_i,c_i\})\leq 1$.
Therefore, by \eqref{e:approxunifbd},
\begin{align}\label{e:I_iest}
\P_{\cG}(I_i=0)&\leq \P_{\cG}(\dist_{\cG}(\{a_i,l_i\}, b_i)\leq 1)+\P_{\cG}(\dist_{\cG}(\{a_i,l_i\}, c_i)\leq 1)\\
&\leq \frac{4}{N}\#\{x\in \cGT: \dist_{\cG}(\{a_i,l_i\},x)\leq 1\}\leq \frac{8d}{N}.
\end{align}
Similarly, since $|\cup_{j} [e_j]| \leq 2\mu$,
\begin{equation}
\P_{\cG}(|\{b_i,c_i\}\cap (\cup_{j \neq i} [e_{j}])|\geq1) \leq \frac{8\mu}{N},
\end{equation}
and, for any $i \in \tilde A_2'$, we have
\begin{equation}
\P_{\cG}\pa{|\{b_i,c_i\}\cap (\cup_{j \not\in \tilde A_1'} \{b_j,c_j\}|\geq 1\Big|\vec S_j, j\not\in \tilde A_2'}
\leq \frac{8\mu}{N}.
\end{equation}
Finally, there are at most $ (3\mu)^{\n}$ 
disjoint sets $\tilde A_0, \tilde A_1, \tilde A_2' \subset \qq{1,\mu}$ such that $|\tilde A_0| + |\tilde A_1|+|\tilde A_2'|=\n$,
and therefore
\begin{align*}
\P_{\cG}\pB{|A_0|+|A_1|+\tfrac12 |A_2| \geq \n}
&\leq (3\mu)^{\n}\max_{\tilde A_0,\tilde A_1,\tilde A_2'}\P_{\cG}\pB{\text{the sets $\tilde A_0, \tilde A_1, \tilde A_2'$ satisfy \eqref{e:conditionforB}}}\\
&=(3\mu)^{\n}\max_{\tilde A_0,\tilde A_1,\tilde A_2'}
\prod_{i\in\tilde A_0}\P_{\cG}(I_i=0)
\prod_{i\in\tilde A_1}\P_{\cG}(|\{b_i,c_i\}\cap (\cup_{j \neq i} [e_{j}])|\geq1)\\
&\qquad \qquad\qquad\quad\;
 \prod_{i\in\tilde A_2'}\P_{\cG}\pa{|\{b_i,c_i\}\cap (\cup_{j \not\in \tilde A_2'} \{b_j,c_j\}|\geq 1\Big|\vec S_j, j\not\in \tilde A_2'}\\
&\leq (3\mu)^{\n}\max_{\tilde A_0,\tilde A_1,\tilde A_2'}
\prod_{i\in\tilde A_0}\frac{8d}{N}
\prod_{i\in\tilde A_1}\frac{8\mu}{N}
\prod_{i\in\tilde A_2'}\frac{8\mu}{N}
=o( N^{-\n+\delta}),
\end{align*}
where the maxima are over 
all disjoint sets $\tilde A_0, \tilde A_1, \tilde A_2'\subset \qq{1,\mu}$ such that $|\tilde A_0| + |\tilde A_1|+|\tilde A_2'|=\n$,
and where we used 
that the probability factorizes since the sets $\tilde A_0,\tilde A_1,\tilde A_2'$ are disjoint.
\end{proof}

\begin{remark}\label{defrmk}
Throughout Sections~\ref{sec:dist}--\ref{sec:improved},
we fix a $d$-regular graph $\cG\in \GNd$ on the vertex set $\qq{N}$,
and abbreviate its $\ell$-neighborhood of $1$ by 
\begin{equation}
\T=\bB_\ell(1, \cG), \quad  \cT=\cB_{\ell}(1,\cG). 
\end{equation}
We also write
\begin{equation}\label{def:Ti}
\T_i=\{v\in \cG: \dist_{\cG}(1,v)=i\},
\end{equation}
for the set of vertices at distance $i$ from $1$.

Further, we enumerate the boundary edges $\del_E \cT$ as $\{l_i, a_i\}$ for $i\in \qq{1,\mu}$,
where $l_i \in \T$ and $a_i \in \qq{N} \setminus \T$.
We denote the resampling data by ${\bf S}=(\vec S_1, \vec S_2,\dots, \vec S_\mu)$,
where $\vec{S}_i = \{(l_i, a_i), (b_i,c_i)\}$ for $i\in \qq{1,\mu}$, and $(b_1,c_1), \dots, (b_\mu,c_\mu)$
are chosen independently and uniformly among oriented edges from the graph $\cGT$.
We denote by ${\sf S}(\cG)$ the set of all possible switching data,
so that ${\bf S}$ is uniformly distributed on ${\sf S}$.
Given switching data $\bf S$, we denote the set of admissible switchings by $W_{\bf S}$.
Without loss of generality, we will assume for notational convenience that $W_{\bf S}=\{1,2,3,\dots,\nu\}$ where $|W_{\bf S}|=\nu\leq \mu$. 

To study the change of graphs before and after local resampling, we define the following graphs (which need not be regular).
\begin{itemizetight}
\item $\cG$ is the original unswitched graph;
\item $\cGT$ is the unswitched graph with vertices $\T$ removed;
\item $\hcGT$ is the intermediate graph obtained from $\cGT$ by removing the edges $\{b_i,c_i\}$ with $i\in W_{\bf S}$;
\item $\tcGT$ is the switched graph obtained from $\hcGT$ by adding the edges $\{a_i,b_i\}$ with $i \in W_{\bf S}$; and
\item $\tcG$ is the switched graph $T_{\bf S}(\cG)$ (including vertices $\T$).
\end{itemizetight}
Following the conventions of Section~\ref{sec:intro-TE},
the deficit functions of these graphs are given by $d-\deg$,
where $\deg$ is the degree function of the graph considered,
and we abbreviate their Green's functions by $G$, $\GT$, $\hGT$, $\tGT$, and $\tG$ respectively.
\end{remark}

\section{Graph distance between switched vertices}
\label{sec:dist}

\begin{figure}
\centering
\input{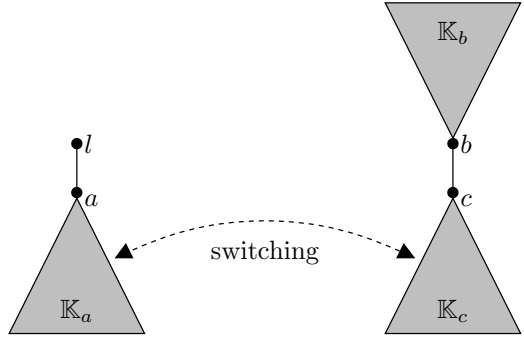}
\caption{For the vertices $x \in \{a_i,b_i,c_i\}$ participating in the switching,
we denote by $\bK_{x}$ their radius-$R/4$ neighborhoods in the unswitched graph,
with $b_i$ and $c_i$ disconnected and the set $\T$ removed.
The typical $\bK_x$ are disjoint from the other $\bK_x$,
and in the typical case, the sets $\bK_a$ and $\bK_c$ exchange their roles under switching.
\label{fig:Nx}}
\end{figure}

This section provides estimates on the distances between the vertices participating in the switching,
in the graph with vertices $\T$ removed (before and after switching).
It can be helpful to think about these estimates in terms of the sets $\bK_x \subset \qq{N} \setminus \T$
defined by
\begin{equation} \label{e:Nxdef}
\bK_{a_i} = \bB_{R/4}(a_i, \cGT),
\quad \bK_{x_i} = \bB_{R/4}(x_i, \cGT \setminus \{\{b_i,c_i\}\}),
\;\; \text{where $x_i \in \{b_i,c_i\}$},
\end{equation}
and illustrated in Figure~\ref{fig:Nx}.
In \eqref{fewclose}, \eqref{distfinitedega}, it is shown that
\begin{itemizetight}
\item{\eqref{fewclose}}
except for at most $2\n$ many, the $\bK_a$ does not intersect the other $\bK_a$.
\item{\eqref{distfinitedega}}
any $x \in \qq{N} \setminus \T$ is in at most $\n+1$ many of the sets $\bK_{a}$.
\end{itemizetight}
Roughly speaking, in this section it is shown that, for any graph $\cG \in \bar\Omega$,
the following estimates hold with high probability under $\P_{\cG}$:
\begin{itemizetight}
\item{\eqref{distfinitedegb}}
any $x \in \qq{N} \setminus \T$ is in at most $\n$ of the sets $\bK_b$;
\item{\eqref{distdega}}
any $\bK_{a}$ intersects at most $\n$ of the $\bK_{b}$;
\item{\eqref{distdegb}}
any $\bK_{b}$ intersects at most $2\n$ of the other $\bK_{a},\bK_{b}$;
\item{\eqref{treeneighbor}}
except for at most $\n$ many, the $\bK_{b}$ are trees.
\end{itemizetight}
By symmetry, the same statements hold with $b$ replaced by $c$.
More precisely, in the remainder of this section, we show that the estimates stated in the following propositions hold.

\newcommand{\sBa}{\sB_a}
\newcommand{\sBb}{\sB_b}
\newcommand{\sBc}{\sB_c}

\begin{proposition} \label{prop:distdeg}
For any graph $\cG \in \bar\Omega$ (as in Section~\ref{sec:outline-structure}),
the following holds with $\P_{\cG}$-probability at least $1-o(N^{-\omega+\delta})$:
\begin{itemize}
\item
Any vertex $x\in \qq{N}\setminus\T$ is far away from most vertices in $\{b_1, b_2,\dots, b_\mu\}$:
\begin{align}
{\label{distfinitedegb}}|\{i\in\qq{1,\mu}:\dist_{\cGT}(x,b_i)\leq R/2\}|<\n.
\end{align}
\item
Most indices $i \in \qq{1,\mu}$ are good:
\begin{align}
\label{distdega}
|\sBa| &< 3\n,&
\text{with } \sBa &=\{i\in \qq{1,\mu}: {\dist_{\cGT}(a_i, \{a_j, b_k: j\in\qq{1,\mu}\setminus \{i\}, k\in \qq{1,\mu}\})\leq R/2}\},
\\
\label{distdegb}
|\sBb| &<2\n, &
\text{with } \sBb &=\{i\in \qq{1,\mu}: {\dist_{\cGT}(b_i, \{a_j, b_k: j\in\qq{1,\mu}, k\in \qq{1,\mu}\setminus \{i\}\})\leq R/2}\},
\\
\label{treeneighbor}
|\sBc| &< \n,&
\text{with } \sBc &=\{i\in \qq{1, \mu}: \text{$\cB_R(c_i, \cGT)$ is not a tree }\}.
\end{align}
\end{itemize}
\end{proposition}

Note that $\sBa$ is the set of indices $i$ such that $\bK_{a_i}$is not disjoint from all sets other $\bK_{a}$ and $\bK_{b}$,
and that $\sBb$ is the set of indices $i$ such that $\bK_{b_i}$ is not disjoint from all other sets $\bK_{b}$ and $\bK_{a}$.

We will show that the estimates  \eqref{distdega} and \eqref{distdegb} also imply the following
estimates for the switched graph $\tcGT$.

\begin{proposition} \label{structuretGT}
Assume \eqref{distdega} and \eqref{distdegb}.
\begin{itemize}
\item
For any index $i\in \qq{1,\mu}\setminus(\sBa \cup \sBb)$,
\begin{align}\label{tcGTdist}
\dist_{\tcGT}(\{a_i,b_i\}, \{a_j,b_j: j\in \qq{1,\mu}\setminus \{i\}\})>R/2. 
\end{align}
\item
For any vertex $x\in \qq{N}\setminus \T$, 
\begin{align}\label{e:lessshortdist}
|\{i\in\qq{1,\mu}: \dist_{\tcGT}(x,\{a_i,b_i\})\leq R/4\}|\leq 5\n.
\end{align}
\end{itemize}
\end{proposition}

The remainder of this section is devoted to the proofs of Propositions~\ref{prop:distdeg}--\ref{structuretGT}.

\subsection{Proof of Proposition~\ref{prop:distdeg}}

Recall that the oriented edges $(b_i,c_i)$ are independent and distributed approximately uniformly,
so that \eqref{e:approxunifbd} holds. The claims essentially follow from this.

\begin{proof}[Proof of \eqref{distfinitedegb}]
In any graph with degree bounded by $d$,
the number of vertices at distance at most $R/2$ from vertex $x$ is bounded by $1+d+d(d-1)+\cdots +d(d-1)^{R/2-1}\leq 2(d-1)^{R/2}$.
By \eqref{e:approxunifbd} therefore
\begin{equation} \label{e:distxbi}
\P_{\cG}(\dist_{\cGT}(x,b_i)\leq R/2)\leq \frac{4(d-1)^{R/2}}{N}.
\end{equation}
Since the $b_1, \dots, b_\mu$ are independent, it therefore follows that
\begin{align*}
\P\left(|\{i\in \qq{1,\mu}:\dist_{\cGT}(x,b_i)\leq R/2\}|\geq \n\right)\leq \binom{\mu}{\omega} \pbb{\frac{4(d-1)^{R/2}}{N}}^\n\ll N^{-\n+\delta},
\end{align*}
where, in the last inequality, we used that $(4(d-1)^{R/2}\mu)^\omega \leq 2^{3\n}(d-1)^{(R/2+\ell+1)\omega} \ll N^\delta$ by the choice
of parameters in Section~\ref{sec:outline}.
\end{proof}

\begin{proof}[Proof of \eqref{distdega}]
Recall the annulus $\cA$, and sets $A_1, A_2, \dots$ from Lemma~\ref{lem:beta2n}.
By \eqref{e:beta2n}, $|A_1 \cup \cdots \cup A_\alpha|\leq 2\n$,
and for any $i \in A_{\alpha+1} \cup A_{\alpha+2}\cup  \cdots$, $a_i$ is at least distance $R$ in $\cGT$ from other vertices $a_j$.
It follows that
\begin{align*}
&\P_{\cG}\left(|\{i\in \qq{1,\mu}: {\dist_{\cGT}(a_i,\{a_j, b_k: j\in\qq{1,\mu}\setminus \{i\}, k\in \qq{1,\mu}\})\}|\leq R/2}\}\geq3\n\right)\\
&\leq \P_{\cG}\left(|\{i\in A_{\alpha+1} \cup A_{\alpha+2}\cup \cdots: {\dist_{\cGT}(a_i,\{b_1,b_2,\dots, b_\mu\})\}|\leq R/2} \}\geq \n\right).
\end{align*}
By a union bound, the right-hand side is bounded by
\begin{align*}
\sum_{A',B'}\P_{\cG}\left(\dist_{\cGT}(a_{i_1}, b_{j_1})\leq R/2, \dots, \dist_{\cGT}(a_{i_{\n}}, b_{j_{\n}})\leq R/2\right),
\end{align*}
where $A'=\{{i_1},\dots, {i_{\n}}\}$, $B'=\{{j_1},\dots, {j_{\n}}\}$, and the sum over $A'$ runs through the subsets of $A_{\alpha+1}\cup A_{\alpha+2}\cup\cdots$ with $|A'|=\n$,
the sum over $B'$ runs through subsets of $\qq{1,\mu}$ with $|B'| = \n$.
Notice that if $a_k$ and $a_m$ are in different connected components of $\cA$, then $\dist_{\cGT}(a_k,b_i)\leq R/2$ and $\dist_{\cGT}(a_m,b_j)\leq R/2$
imply $b_i$ and $b_j$ are in  different connected components of $\cA$ (those of $a_k$ and $a_m$, respectively),
and in particular then $b_i\neq b_j$.
As a consequence, the indices $j_1, \dots j_\n$ must be distinct,
and in particular the random variables $b_{j_1}, \dots, b_{j_\n}$ are independent.
Thus the previous expression is bounded by 
\begin{align*}
\sum_{A',B'}\P_{\cG}\left(\dist_{\cGT}(a_{i_1}, b_{j_1})\leq R/2) \cdots \P_{\cG}(\dist_{\cGT}(a_{i_{\n}}, b_{j_{\n}})\leq R/2\right)
\leq \binom{\mu}{\n}^{2}\pbb{\frac{4(d-1)^{R/2}}{N}}^{\n}\ll N^{-\n+\delta},
\end{align*}
where we used that there are $\binom{\mu}{\n}$ choices for $A'$ and $B'$ respectively, and the estimate \eqref{e:distxbi} with $x = a_{i_1}, \dots, a_{i_{\n}}$.
\end{proof}

\begin{proof}[Proof of \eqref{distdegb}]
Similarly, to prove \eqref{distdegb}, by the union bound we have
\begin{align*}
&\P_{\cG}\left(|\{i\in \qq{1,\mu}: {\dist_{\cGT}(b_i, \{a_j, b_k: j\in\qq{1,\mu}, k\in \qq{1,\mu}\setminus\{i\}\})\leq R/2}\}|\geq 2\n\right)\\
&\leq \sum_{B'}\P_{\cG}\left(\forall i\in B', \dist_{\cGT}(b_i, \{a_j, b_k: j\in\qq{1,\mu}, k\in \qq{1,\mu}\setminus\{i\}\})\leq R/2\right),
 \end{align*}
where $B'$ runs through all subsets of $\qq{1,\mu}$ with $|B'|=2\n$.
Next, we notice that, if for all $i\in B'$, we have $\dist_{\cGT}(b_i, \{a_j, b_k: j\in\qq{1,\mu}, k\in \qq{1,\mu}\setminus\{i\}\})\leq R/2$,
then there must be subset $B''\subset B'$ with $|B''|=\n$ such that for all
$i\in B''$, we have $\dist_{\cG^{(\T)}}(b_i, \{a_j, b_k: j\in\qq{1,\mu}, k\in \qq{1,\mu}\setminus B''\})\leq R/2$.
By relabeling, without loss of generality, we assume that $B''=\{{\mu-\n+1}, {\mu-\n+2},\dots, {\mu}\}$.
Conditioned on $\vec S_1, \vec S_2,\dots, \vec S_{\mu-\n}$, we have
\begin{align*}
\P_{\cG}\left(\dist(b_i, \{a_1,a_2,\dots, a_\mu, b_1, b_2,\dots, b_{\mu-\n}\})
\leq R/2 \;\Big|\; \vec S_1, \vec S_2,\dots, \vec S_{\mu-\n}\right)\leq 8\mu (d-1)^{R/2}/N,
\end{align*}
for any $i\in \qq{\mu-\n+1,\mu}$. Therefore,
\begin{align*}
&\sum_{B'}\P_{\cG}\left(\forall i\in B', \dist_{\cGT}(b_i, \{a_j, b_k: j\in\qq{1,\mu}, k\in \qq{1,\mu}\setminus\{i\}\})\leq R/2\right)\\
&\leq \sum_{B''}\P_{\cG}\left(\forall i\in B'', \dist_{\cGT}(b_i, \{a_j, b_k: j\in\qq{1,\mu}, k\in \qq{1,\mu}\setminus B''\})\leq R/2\right)\\
&\leq {\mu\choose\n}\left(\frac{8\mu(d-1)^{R/2}}{N}\right)^{\n}\ll N^{-\n+\delta}.
\end{align*}
since there are ${\mu \choose \n}$ choices for $B''$. This completes the proof.
\end{proof}

\begin{proof}[Proof of \eqref{treeneighbor}]
By the assumption $\cG\in \bar\Omega$, all except at most $N^\delta$ many vertices have radius-$R$ tree neighborhoods.
In particular, the same holds for $\cGT$.
By \eqref{e:approxunifbd}, it follows that
\begin{align*}
\P_{\cG}\left(\text{the radius-$R$ neighborhood of $c_i$ contains cycles}\right)\leq 2N^{-1+\delta}.
\end{align*}
By the union bound, and using that the number of ways to choose $\omega+1$ elements from $\mu$ elements is bounded from above by $\mu^{\omega+1}$,
\begin{align*}
&\P_{\cG}\left(|\{i\in \qq{1,\mu}: \text{ radius-$R$ neighborhood of $c_i$ contains cycles}\}|\geq \n+1\right)\\
&\leq \mu^{\n+1}(2N^{-1+\delta})^{\n+1}\leq (2\mu)^{\n+1}N^{-\n-1+(\n+1)\delta}\ll N^{-\n+\delta},
\end{align*}
given that $\delta<1/\n$ and using that $\mu \leq 2(d-1)^{\ell+1}=(\log N)^{O(1)}$ by
the choice of parameters in Section~\ref{sec:outline}. 
\end{proof}

\subsection{Proof of Proposition~\ref{structuretGT}}

\begin{proof}[Proof of \eqref{tcGTdist}]
By the definition of the sets $\sBa$ and $\sBb$, for any $i\in \qq{1,\mu}\setminus(\sBa \cup \sBb)$, we have
\begin{align}\label{cGTdist}
\dist_{\cGT}(\{a_i,b_i\}, \{a_j,b_j: j\in \qq{1,\mu}\setminus\{i\})> R/2.
\end{align}
Since $\tcGT$ is obtained from $\cGT$ by removing the edges $\{b_i,c_i\}_{i\leq \nu}$ and adding the edges $\{a_i,b_i\}_{i\leq \nu}$,
the claim \eqref{tcGTdist} directly follows from \eqref{cGTdist}.
\end{proof}

\begin{proof}[Proof of \eqref{e:lessshortdist}]
We consider three cases.
If $\dist_{\tcGT}(x,\{a_i,b_i\})> R/4$ for all $i\in \qq{1,\mu}$, then the claim is trivial.
If $\dist_{\tcGT}(x,\{a_i,b_i\})\leq R/4$ for some $i\in \qq{1,\mu}\setminus(\sBa \cup \sBb)$, then \eqref{tcGTdist} implies that
\begin{align*}
\dist_{\tcGT}(x,\{a_j,b_j\})\geq \dist_{\tcGT}(\{a_i,a_j\},\{a_j,b_j\})-\dist_{\tcGT}(x,\{a_i,b_i\})>R/2-R/4=R/4, 
\end{align*}
for any $j\in \qq{1,\mu}\setminus\{i\}$. Thus $|\{i\in\qq{1,\mu}: \dist_{\tcGT}(x,\{a_i,b_i\})\leq R/4\}|=1 \leq 5\n$ as claimed.
In the remaining case, $\dist_{\tcGT}(x,\{a_i,b_i\})\leq R/4$ is only possible for $i\in \sBa \cup \sBb$.
Therefore $|\{i\in\qq{1,\mu}: \dist_{\tcGT}(x,\{a_i,b_i\})\leq R/4\}|\leq |\sBa\cup \sBb|\leq 5\n$ as claimed.
\end{proof}

\section{The Green's function distance and switching cells}
\label{sec:distGF}

The bounds provided in the Section~\ref{sec:dist} provide accurate control for distances at most $R/2$.
However, \emph{random} vertices are typically much further from each other,
and as mentioned in Section~\ref{sec:outline},
we require stronger upper bounds on the Green's function for such large distances.
These bounds are in fact a general consequence of the \emph{Ward identity},
\begin{equation} \label{e:Ward-bis}
  \sum_{j} |G_{ij}(z)|^2 = \frac{\im [G_{ii}(z)]}{\im [z]},
\end{equation}
which holds for the Green's function of any symmetric matrix (see \eqref{e:Ward}).
To make use of it, we introduce a much coarser measure of distance in terms of the size
of the Green's function as follows.

\subsection{Definition}\label{sec:defcells}

Given a parameter $M>0$ (ultimately chosen in \eqref{e:Mnprime} below),
we define a relation $\sim$ on $\qq{N} \setminus \T$ by
setting $x \sim y$ if and only if
\begin{equation} \label{e:simdef}
\max_{u: \dist(x,u)\leq 4r,\atop v:\dist(y,v)\leq 4r}|\GT_{uv}(z)|\geq
\frac{M}{\sqrt{N\eta}},
\end{equation}
where the distance in the maximum is with respect to the graph $\cGT$, and $\eta=\Im[z]$.
The relation $\sim$ induces a graph $\cal R$ on the vertices $\{a_1, \dots, a_\mu, b_1, \dots, b_\nu\}$.
We partition $\{a_1, \dots, a_\mu, b_1, \dots, b_\nu\}$ into its $\sim$-clusters.
More precisely, we define $\bI_1$ to be the vertex set consisting of the union of the connected components of $\cal R$
 containing any element of $\{a_1,a_2,\dots, a_\mu\}$,
 and we define $\bI_2,\dots, \bI_\kappa$ be the vertex sets of the remaining connected components of $\cal R$.

\begin{figure}[t]
\centering
\input{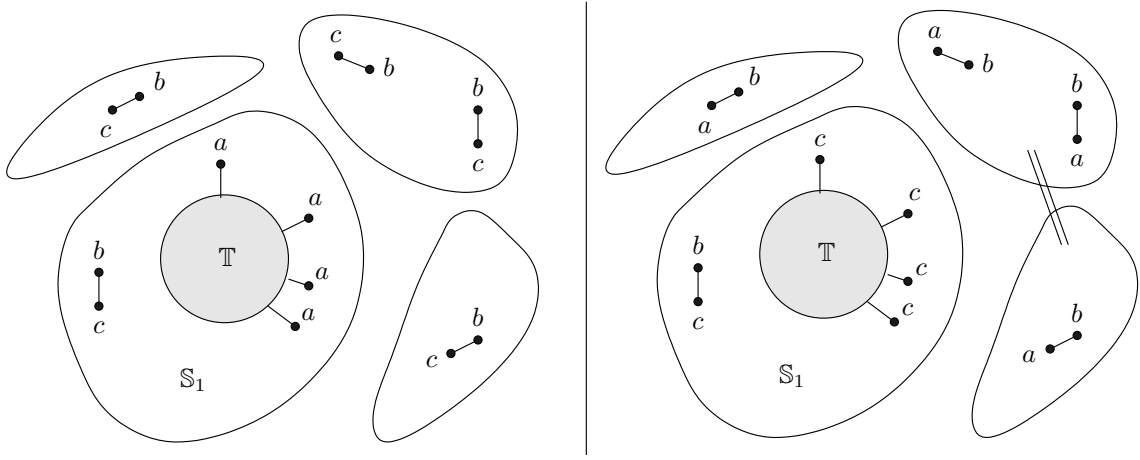}
\caption{The $\rm S$-cells are clusters of vertices that are close to one of the edges $\{b_i,c_i\}$
in the sense of the Green's function distance $\sim$.
The $\rm S$-cell $\cC_1$ contains all $a_i$ (the vertex boundary of $\cT$ in the original graph)
as well as those $b_i$ which are close to any of the $a_i$ in the sense of the Green's function distance.
Since the switching may decrease distances between vertices, the $\rm S'$-cells are defined by joining
the $\rm S$-cells which have vertices that are close to each other.
\label{fig:cells}}
\end{figure}

\begin{definition}[Cells]\label{def:cells}
$ $ 
\begin{itemize}
\item
Define sets $\cC_1, \cC_2,\dots, \cC_\kappa \subset \qq{1,N}$ called \emph{$\rm S$-cells} by
\begin{align}\label{e:defC_i}
\cC_i=\bB_{2r}(\bI_i, \cGT).
\end{align}
%
For any vertex $x\in \qq{N}\setminus\T$,
we write $x\sim \cC_i$ if there is $y\in \bI_i$ such that $x\sim y$.
\item
Define
$\cC_1', \dots, \cC_{\kappa'}'\subset \qq{1,N}$ called \emph{$\rm S'$-cells} by combining the $\rm S$-cells which are close to each other after switching:
we set $\cC_1'=\cC_1$ and join $\rm S$-cells $\cC_i$ and $\cC_j$ with $i,j>1$ if $\dist_{\tcGT}(\cC_i,\cC_j) \leq 2r$.
\end{itemize}
\end{definition}

The $\rm S$- and $\rm S'$-cells are illustrated in Figure~\ref{fig:cells}.
The $\rm S$-cells are defined in terms of the unswitched graph.
In the switching process, distances between $\rm S$-cells may decrease.
This is accounted for by the coarser $\rm S'$-cells.
For later use, we note the following elementary properties of $\rm S$-cells:
\begin{itemizetight}
\item
For any $x\in \cC_i$ and $y\in \cC_j$ such that $i \neq j$ we have $|\GT_{xy}|< M/\sqrt{N\eta}$.
\item
For any vertex $x\in \qq{N}\setminus\T$, if $x\not\sim \cC_i$, then for any $y\in \cC_i$, $|\GT_{xy}|<M/\sqrt{N\eta}$.
\item
If $b_k \in \cC_i$, then also $c_k \in \cC_i$; and, consequently, if $b_k \in \cC_i'$ then $c_k \in \cC_i'$.
\end{itemizetight}

\subsection{Estimates}

From now on, we fix the parameters $M$ and $\n'$ by
\begin{equation} \label{e:Mnprime}
M =d^{9\ell} (\log N)^{\delta}, \quad \n' = \lfloor\log N\rfloor
,
\end{equation}
where $\delta>0$ was fixed in Section~\ref{sec:outline}.
The next proposition shows that the cells do not cluster. 

\begin{proposition} \label{prop:Rdist2}
For any graph $\cG \in \Omega_1^+(z,\ell)$ (as in Section \ref{sec:outline-structure}), 
with probabililty at least $1-o(N^{-\n+\delta})$ under $\bf S$,
the following estimates hold:
\begin{itemize}
\item
Any $x \in \qq{N} \setminus \T$ is $\sim$-connected to fewer than $\n'$ of $\{b_1,b_2,\dots, b_\mu\}$,
\begin{align}\label{finitedeg}
|\{i\in \qq{1,\mu}: x\sim b_i\}|< \n'.
\end{align}
In particular, $x$ is $\sim$-connected to at most $\n'$ of the $\rm S$-cells.
\item
Except for at most $\n'$ many indices $i$,  the vertex $b_i$ is a singleton in the graph $\cR$,
and thus the $\rm S$-cell containing $b_i$ is disjoint from $\{a_j, b_k: j\in\qq{1,\mu}, k\in\qq{1,\mu}\setminus\{i\}\}$:
\begin{align}\label{deg}
|\{i\in \qq{1,\mu}: b_i\sim \{a_j, b_k: j\in\qq{1,\mu}, k\in\qq{1,\mu}\setminus\{i\}\} \}|<\n'.
\end{align}
In particular, each $\rm S$-cell contains at most $\n'$ of $\{b_1,b_2,\dots, b_\mu\}$.
\item 
Most $\rm S'$-cells are far from the other vertices participating in the switching:
\begin{align}\begin{split}\label{e:Spcellbd}
|\{i\in \qq{1,\nu}: 
& b_i\in \bS_j', \text{such that $j=1$ or}\\
&\dist(\bS_j', \{a_k, b_m, c_m: k\in\qq{1,\mu}\setminus \{i\}, m\in\qq{1,\nu}\setminus\{i\}\})\leq R/4\}|<\n'+5\n.
\end{split}\end{align}
In particular, each $\rm S'$-cell contains at most $\n'+5\n$ of $\{b_1,b_2,\dots, b_\nu\}$.  
\end{itemize}
\end{proposition}

In the remainder of this section, we prove the above proposition.
It is essentially a straightforward
consequence of the definitions, combined with union bounds.

\subsection{Proof of Proposition~\ref{prop:Rdist2}}

The following two lemmas collect some elementary properties of the Green's function graph $\cal R$
on $\{a_1, \dots, a_\mu, b_1, \dots, b_\mu\}$ that we require.

\begin{lemma}\label{l:distdiffcC}
Let $\cG\in \Omega_1^+(z,\ell)$ (as in Section \ref{sec:outline-structure}) and $x,y\in \qq{N}\setminus\T$. Then we have $x\not\sim y$ implies $\dist_{\cGT}(x,y)>8r$.
\end{lemma}

\begin{proof}
We show that $\dist_{\cGT}(x,y)\leq 8r$ implies $x\sim y$.
Assume that $\dist_{\cGT}(x,y) \leq 8r$.
Then there must be a vertex $u$ such that $\dist_{\cGT}(x,u) \leq 4r$ and $\dist_{\cGT}(y,u)\leq 4r$.
Moreover, by the definition \eqref{e:defOmega+} of $\Omega_1^+(z,\ell)$ and estimate \eqref{e:boundPii},
also $|\GT_{uu}(z)| \geq |\msc(z)|/2\geq M/\sqrt{N\eta}$, and thus $x \sim y$.
\end{proof}

\begin{lemma}\label{lem:Rdist1}
Let $\cG\in \Omega_1^+(z,\ell)$ (as in Section \ref{sec:outline-structure}) and $x \in \qq{N} \setminus \T$. Then
\begin{equation} \label{e:bsimx}
\P_{\cG}(b_i\sim x)\leq 16 (d-1)^{8r}/M^2.
\end{equation}
\end{lemma}

\begin{proof}
$\cG\in \Omega_1^+(z,\ell)$ and \eqref{e:boundPii} imply that $\im [\GT_{xx}] \leq |\GT_{xx}|\leq 2$.
Thus the Ward identity \eqref{e:Ward} implies
\begin{align}\label{e:Wardpf}
\sum_i |\GT_{xi}|^2=\Im [\GT_{xx}]/\eta\leq 2/\eta.
\end{align}
For any vertex $x\in \qq{N} \setminus \T$, set
\begin{align*}
\tilde \bV_x &:=\left\{i\in \qq{N} \setminus \T: |\GT_{xi}|\geq M/\sqrt{N\eta}\right\},
\\
\bV_x &:=\left\{i\in \qq{N} \setminus \T: \dist_{\cGT}\left(i, \bigcup_{j\in \bB_{4r}(x, \cGT)}\tilde \bV_j\right)\leq 4r \right\}.
\end{align*}
The inequality \eqref{e:Wardpf} implies $|\tilde \bV_x|\leq 2N/M^2$,
and since any vertex has at most $2(d-1)^{4r}$ vertices in its radius-$4r$ neighborhood, we also have $|\bV_x|\leq 8(d-1)^{8r}N/M^2$.
Moreover, $i\not\in \bV_x$ implies that $i\not\sim x$. Thus
\begin{equation*}
\P_{\cG}(b_i \sim x) \leq \P_{\cG}(b_i \in \bV_x) \leq \frac{2}{N} |\bV_x| \leq 16 (d-1)^{8r}/ M^2,
\end{equation*}
where the second inequality holds because $b_i$ is approximately uniform \eqref{e:approxunifbd}.
\end{proof}

\begin{proof}[Proof of \eqref{finitedeg}]
The proof is similar to that of \eqref{distfinitedegb}.
By the union bound and \eqref{e:bsimx}, we have
\begin{align*}
\P( |\{i\in \qq{1,\mu }: b_i\sim x\}|\geq \n')
\leq {\mu \choose \n'} \left(\frac{16(d-1)^{8r}}{M^2}\right)^{\n'}\leq (\log N)^{-\delta\log N}\ll N^{-\n},
\end{align*}
where, in the second inequality, we used $\binom{\mu}{\n'} \leq \mu^{\n'}$
and that
\begin{equation*}
16\mu d^{8r}/M^2 
\leq
16 d^{\ell+1}d^{17\ell} d^{-18\ell} (\log N)^{-2\delta}
\end{equation*}
since $\mu \leq d^{\ell+1}$ and by the definition of $M$ \eqref{e:Mnprime}.
\end{proof}

\begin{proof}[Proof of \eqref{deg}]
The proof is similar to that of \eqref{distdegb}.
Indeed, by the union bound and \eqref{e:bsimx},
\begin{align*}
&\P_{\cG}\left(|\{i\in \qq{1,\mu}: b_i\sim \{a_j,b_k:j\in\qq{1,\mu}, k\in \qq{1,\mu}\setminus\{i\}\} \}|\geq \n'\right)\\
&\leq {\mu \choose \n'/2}\left(\frac{16(d-1)^{8r}\mu}{M^2}\right)^{\n'/2}\leq (\log N)^{-\delta\log N/2}\ll N^{-\n},
\end{align*}
as needed.
\end{proof}

\begin{proof}[Proof of \eqref{e:Spcellbd}]
Recall the index sets $\sBa, \sBb \subset \qq{1,\mu}$ from \eqref{distdega}, \eqref{distdegb},
and let $i\not\in \sBa \cup \sBb$ be such that $b_i\not\sim \{a_k, b_m: k\in\qq{1,\mu}, m\in\qq{1,\nu}\setminus\{i\}\}$.
Denote the  $\rm S$-cell containing $b_i$ by $\cC$; then $\cC$ is not $\cC_1$ and it is disjoint from $\{b_k: k\in\qq{1,\mu}\setminus\{i\}\}$. 

By the definition of $\sBa$, $\sBb$ and since $b_j$ and $c_j$ are adjacent in $\cGT$ we have
\begin{align}\label{e:distforakbkck}
\dist_{\cGT}(\{a_i,b_i,c_i\}, \{a_k, b_m, c_m: k\in\qq{1,\mu}\setminus \{i\}, m\in\qq{1,\nu}\setminus\{i\}\})\geq R/2-2.
\end{align}
Since the graph $\tcGT$ is obtained from $\cGT$ by removing edges $\{b_j, c_j\}_{j\leq \nu}$ and adding edges $\{a_j,b_j\}_{j\leq \nu}$,
we also have
\begin{align}\label{somedist}
\dist_{\tcGT}(\{a_i,b_i,c_i\}, \{a_k, b_m, c_m: k\in\qq{1,\mu}\setminus \{i\}, m\in\qq{1,\nu}\setminus\{i\}\})\geq R/2-2.
\end{align}
Moreover, for any other $\rm S$-cell $\cC_j\neq \cC, \cC_1$, we have
\begin{align*}
\dist_{\tcGT}(\cC, \cC_j)\geq -2r+\dist_{\tcGT}(b_i, \bI_j)-2r\geq R/2-2-4r>2r,
\end{align*}
where we used \eqref{somedist}, $r\ll R$ and the definition \eqref{e:defC_i} of $\rm S$-cells, i.e. $\cC = \bB_{2r}(b_i, \cGT)$ and $\cC_j=\bB_{2r}(\bI_j, \cGT)$. Thus $\cC$ is a $\rm S'$-cell itself, and
\begin{align*}
&\dist_{\tcGT}(\cC, \{a_k, b_m, c_m: k\in\qq{1,\mu}\setminus \{i\}, m\in\qq{1,\nu}\setminus\{i\}\})\\
& \geq
\dist_{\tcGT}(b_i, \{a_k, b_m, c_m: k\in\qq{1,\mu}\setminus \{i\}, m\in\qq{1,\nu}\setminus\{i\}\})-2r
\geq R/2-2r-2> R/4,
\end{align*}
where we used $r\ll R$. Therefore, only $i\in \sBa \cup \sBb$ or $b_i\sim \{a_k, b_m: k\in\qq{1,\mu}, m\in\qq{1,\nu}\setminus\{i\}\}$ contribute to the statement \eqref{e:Spcellbd}. 
Thus, combining \eqref{deg} with the estimate $|\sBa \cup \sBb|\leq 5\n$ from \eqref{distdega}, \eqref{distdegb}, and with \eqref{deg},
the estimate \eqref{e:Spcellbd} follows.
\end{proof}

\section{Stability under removal of a neighborhood}
\label{sec:stabilityT}

The following deterministic estimate shows that
removing the neighborhood $\T$ from the graph $\cG$
has a small effect on the Green's function in the complement of $\T$.

\begin{proposition} \label{prop:stabilityGT}
Let $z \in \C_+$ and $\sqrt{d-1}\geq (\n+1)^2 2^{2\n+10}$,
and let $\cG \in \bar\Omega$ (as in Section \ref{sec:outline-structure}) be a graph such that,
for all $i,j \in \qq{N}$,
\begin{align}\label{e:asumpGT}
\left|G_{ij}-P_{ij}(\cE_r(i,j,\cG))\right|\leq |\msc|q^r.
\end{align}
Then, for all vertices $i,j\in \qq{N} \setminus \T$, we have
\begin{equation}\label{e:stabilityGT}
|\GT_{ij}-P_{ij}(\cE_r(i,j, \cGT))|
\leq 2 |\msc|q^{r}.
\end{equation}
\end{proposition}

As discussed in Section~\ref{sec:outline},
the removal of $\T$ is useful because our switchings
have a smaller effect in $\cGT$ than they do in $\cG$. Indeed, in the original graph $\cG$,
our switchings have the effect of removing two edges and adding two edges, while in $\cGT$ our switchings only remove
the edges $\{b_i,c_i\}_{i\leq \nu}$ and add the edges $\{a_i,b_i\}_{i\leq \nu}$. In the next few sections,
we therefore work with $\cGT$ and its switched version $\tcGT$,
and only return to the full graph in Section~\ref{sec:weakstab}.

The remainder of this section is devoted to the proof of the proposition.
The main ingredients are that (\rn 1) given any $i,j$,
there can only be a few vertices in $\T$ that are close to $i$ or $j$, by the
deterministic assumption on the excess of $R$-neigborhoods, 
and (\rn 2) that for all other
vertices in $\T$, the decay of the Green's function implied by \eqref{e:asumpGT}
shows that the removal of them has a small effect. 

\subsection{Step 1: Removal of vertices close to $i$ or $j$}
\label{sec:stabilityT1}
From \eqref{def:Ti},
recall that $\T_\ell=\{v\in \cG: \dist_\cG(1,v)=\ell\}$ is the set of inner vertex boundary of $\cT$.
The first step of the proof of Proposition~\ref{prop:stabilityGT}
consists of removing the vertices in $\T_\ell$ that are close to $i$ or $j$. The set of such vertices is
\begin{equation} \label{e:bSdef}
 \bU=\{v \in \T_\ell: \dist_{\cG \setminus \cT}(i,v)\leq r\}\cup \{v \in \T_\ell: \dist_{\cG \setminus \cT}(j,v)\}\leq r\}
 ,
\end{equation}
where $\cG\setminus \cT$ is obtained from $\cG$ by removing the subgraph $\cT$ induced by $\cG$ on $\T$ (but not removing $\T_\ell$).
Then $|\bU| \leq 2\n+2$ by \eqref{distfinitedegl}.
The following proposition shows that the Green's function remains to be locally approximated after removing $\bU$.

\begin{proposition}\label{stabilityGS}
Under the assumptions of Proposition~\ref{prop:stabilityGT},
for any vertex set $\bU \subset \T$ with $|\bU|\leq 2\n+2$, 
\begin{equation}\label{e:stabilityGS}
 |G^{(\bU)}_{ij}-P_{ij}(\cE_{r}(i,j, \cG^{(\bU)}))|
 \leq  3|\msc|q^{r}/2.
\end{equation}
\end{proposition}

\begin{figure}[t]
\centering
\input{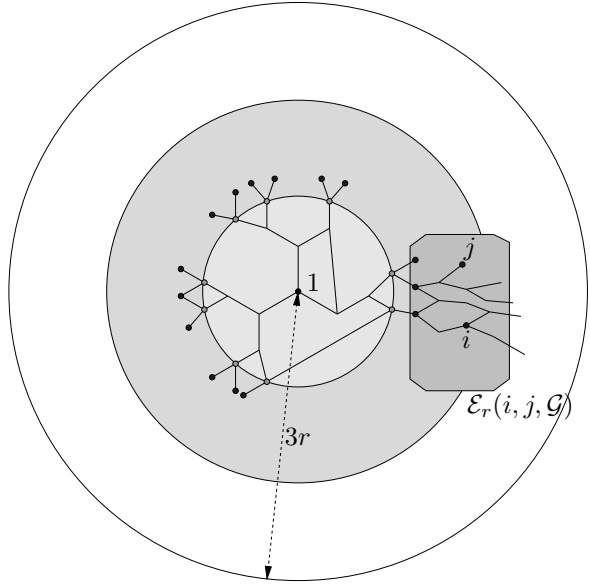}
\caption{The innermost disk shows $\T$, the second largest disk the set $\bX$,
and the outermost disk $\cG_0$. For any $i,j\in \bX$, the graph $\cE_r(i,j,\cG)$ is contained in $\cG_0$.
\label{fig:XG0}}
\end{figure}

The proof of Proposition~\ref{stabilityGS} follows a general structure that occurs repeatedly in similar
estimates throughout the paper.
\begin{enumerate}
\item
The first ingredient in this  structure, which we refer to as \emph{localization},
replaces the Green's function $P_{ij}(\cE_r(i,j,\cG))$ of the vertex-dependent graph $\cE_r(i,j,\cG)$
by the Green's function $P_{ij}=P_{ij}(\cG_0)$ of a graph $\cG_0$ that does not depend on $i,j$,
by an application of Remark~\ref{rk:princG0}. For this, among other things, we need to verify the assumptions of
Proposition~\ref{stabilityGS}.
\item
The second ingredient, which we refer to as the \emph{starting point} for the argument,
is an algebraic relation that expresses the quantity to be estimated in a convenient form. The starting point typically follows from
the Schur complement formula or the resolvent formula.
\item
The third ingredient is a collection of previously established estimates required to
estimate the expressions given by the starting point. It typically includes estimates on elements of Green's functions and graph distances.
\end{enumerate}
The actual proofs then usually follow by combination of the above ingredients. In principle, this step is straightforward,
but often several different cases need to be distinguished, which makes some of the arguments appear somewhat lengthy.

\medskip
Below we provide  the first instance of the strategy described above to prove Proposition~\ref{stabilityGS}.

\paragraph{Localization}

We approximate $P_{ij}(\cE_r(i,j,\cG))$ by a vertex independent Green's function $P_{ij}$ according to Remark~\ref{rk:princG0},
applied with $\cX=\cB_{3r}(1,\cG)$ and $\bX=\bB_{2r}(\bU, \cG)$.
We abbreviate
\begin{equation*}
\cG_1=\TE(\cG_0),
\quad P=G(\cG_1),\qquad
\cG_1^{(\bU)}=\TE(\cG_0^{(\bU)}),\quad P^{(\bU)}=G(\cG_1^{(\bU)}).
\end{equation*}

\paragraph{Verification of assumptions in Proposition \ref{boundPij}}
As subgraphs of $\cG \in \bar\Omega$, the radius-$R$ neighborhoods of $\cG_0$ and $\cG_0^{(\bU)}$ have excess at most $\n$.
By convention, the deficit function of $\cG_0$ vanishes,
on each connected component of $\cG^{(\bU)}_0$, the deficit function of $\cG^{(\bU)}$ obeys $\sum g(v) \leq \n+(2\n+2) \leq 8\n$,
by Proposition~\ref{p:deficitbound}.
Thus the assumptions for \eqref{e:compatibility} are verified for both graphs, and for any $i,j\in \bX$,
\begin{align}\label{replaceEr}
\absb{P_{ij}(\cE_{r}(i,j,\cG))-P_{ij}}
\leq 2^{2\n+3}|\msc|q^{r+1},\qquad
\absb{P_{ij}(\cE_{r}(i,j,\cG^{(\bU)}))-P_{ij}^{(\bU)}}
\leq 2^{2\n+3}|\msc|q^{r+1}
\end{align}
provided that $\sqrt{d-1}\geq 2^{\n+2}$.

\paragraph{Starting point}

To remove $\bU$, we apply the Schur complement formula \eqref{e:Schur1}:
for any $i,j\in \cG^{(\bU)}$,
\begin{align}\label{GSschur}\begin{split}
 G_{ij}-G_{ij}^{(\bU)}=&\sum_{x,y\in \bU}G_{ix}(G|_{\bU})^{-1}_{xy}G_{yj},\\
 P_{ij}-P_{ij}^{(\bU)}=&\sum_{x,y\in \bU}P_{ix}(P|_{\bU})^{-1}_{xy}P_{yj}.
 \end{split}
\end{align}
Our goal is to show that the difference $G_{ij}^{(\bU)}-P_{ij}^{(\bU)}$ is small,
by using that the difference of $G$ and $P$ is small.
As evident from the right-hand sides of \eqref{GSschur}, for this we require upper bounds on the entries
of $G$ and $(G|_{\bU})^{-1}$ (and analogously for $P$ and $(P|_{\bU})^{-1}$).

\paragraph{Green's function estimates}
By assumption \eqref{e:asumpGT} and \eqref{e:boundPij}--\eqref{e:boundPii}, we have
\begin{align} \label{Gxwbound}
\begin{cases}
|G_{xx}|\geq |\md|-|\msc|/4-|\msc|q^r \geq 3|\msc|/5,\\
|G_{xw}|\leq 2^{\n+2}|\msc|q+|\msc|q^r & (x\neq w),\\
|G_{xw}|\leq |\msc|q^r & (\dist_{\cG}(x,w)>r).
\end{cases}
\end{align}

These bounds imply the upper bounds for the entries of $(G|_{\bU})^{-1}$ stated in the following claim.
The claim essentially follows from the fact that the off-diagonal entries of $G|_{\bU}$ are much smaller than
the diagonal entries which have size roughly $\msc$.

\begin{claim}\label{l:sizeinvG}
Under the assumptions of Proposition~\ref{prop:stabilityGT},
for any $\bU \subset \T$ with $|\bU|\leq 2\n+2$,
and any $x,y \in \bU$,
\begin{equation}\label{invbound}
  |(G|_{\bU})^{-1}_{xy}| \leq 2/|\msc|,
  \quad
  |(P|_{\bU})^{-1}_{xy}| \leq 2/|\msc|.
\end{equation}
\end{claim}

\begin{proof}
By the identity $G|_{\bU}(G|_{\bU})^{-1}=I_{\bU\times\bU}$, we have
\begin{align}\label{IDGS}
\delta_{xy}=G_{xx}(G|_{\bU})^{-1}_{xy}+\sum_{w \in \bU \setminus \{x\}}G_{xw}(G|_{\bU})^{-1}_{wy}.
\end{align}
Let $\Gamma:=\max_{x,y\in \bU}|(G|_{\bU})^{-1}_{xy}|$.
Then \eqref{Gxwbound} and \eqref{IDGS} imply
\begin{align*}
|G_{xx}||(G|_{\bU})^{-1}_{xy}|\leq \delta_{xy}+\sum_{w\in \bU\setminus \{x\}}|G_{xw}|\Gamma\leq 1+(2^{\n+2}q+q^{r})|\bU||\msc|\Gamma.
\end{align*}
Taking the maximum over $x,y\in \bU$ in the equation above and using \eqref{Gxwbound} gives
\begin{align*}
\Gamma 
\leq \frac{5}{3|\msc|} + \frac{5}{3} (2^{\n+2}q+q^{r})|\bU| \Gamma
\leq \frac{5}{3|\msc|} + \frac{\Gamma}{6},
\end{align*}
provided that $\sqrt{d-1}\geq (\n+1)2^{\n+6}$. $\Gamma\leq 2/|m_{sc}|$ follows by rearranging. The same argument applies to $P|_{\bU}$,
and we obtain \eqref{invbound}.
\end{proof}

\begin{proof}[Proof of Proposition~\ref{stabilityGS}]
First consider the case that at least one of $i$ and $j$ is not in $\bX$ (i.e.\ far from $\bU$).
Then $\cE_{r}(i,j, \cG)=\cE_{r}(i,j, \cG^{(\bU)})$, and \eqref{e:stabilityGS} follows directly from
\begin{align*}
  |G_{ij}-G_{ij}^{(\bU)}|=\left|\sum_{x,y\in \bU}G_{ix}(G|_{\bU})^{-1}_{xy}G_{yj}\right|
  \leq
  (2^{\n+2}q+q^{r})q^r|\bU|^2 |\msc|^2 (2/|m_{sc}|)
  \leq
  |\msc| q^{r}/2
  ,
\end{align*}
where we used \eqref{Gxwbound}, \eqref{invbound} and that $\sqrt{d-1}\geq (\n+1)^22^{\n+7}$.

Next consider the main case $i,j \in \bX$.
By \eqref{replaceEr}, it suffices to bound the right-hand side of
\begin{align}\label{difGPS}
  |G_{ij}^{(\bU)}-P_{ij}^{(\bU)}|
  \leq  |G_{ij}-P_{ij}| + \sum_{x,y\in \bU}\left|G_{ix}(G|_{\bU})^{-1}_{xy}G_{yj}-P_{ix}(P|_{\bU})^{-1}_{xy}P_{yj}\right|,
\end{align}
which follows from taking difference of expressions in \eqref{GSschur}. 
By \eqref{e:asumpGT} and \eqref{replaceEr},
since for all vertices $i,j\in \bX$ and $x,y \in \bU \subset \bT$,
we have $\cE_{r}(i,j,\cG), \cE_{r}(i,x,\cG), \cE_{r}(y,j,\cG),\cE_{r}(x,y,\cG)\subset \cal G_0$,
\begin{align}\label{bounddiff}
  |G_{ij}-P_{ij}|, |G_{ix}-P_{ix}|, |G_{yj}-P_{yj}|, |G_{xy}-P_{xy}|\leq |\msc|q^{r}+2^{2\n+3} |\msc|q^{r+1}.
\end{align}
Together with \eqref{invbound} and the resolvent formula \eqref{e:resolv}, it follows that
\begin{equation}\label{invdiffbound}
  |(G|_{\bU})^{-1}_{xy}- (P|_{\bU})^{-1}_{xy}|
  =
  |[(G|_{\bU})^{-1}(G|_{\bU}-P|_{\bU})(P|_{\bU})^{-1}]_{xy}| \leq 4|\bU|^2 (1+2^{2\n+3}q)q^r/|\msc|
  .
\end{equation}
Using \eqref{Gxwbound},  \eqref{invbound}, \eqref{bounddiff}, and \eqref{invdiffbound},
the sum on the right-hand side of \eqref{difGPS} is bounded by
\begin{align}
\notag  &\sum_{x,y\in \bU}\Big(\left|G_{ix}-P_{ix}\right| | (G|_{\bU})^{-1}_{xy}G_{yj}|
   +|P_{ix}(P|_{\bU})^{-1}_{xy}||G_{yj}-P_{yj}|
  +\left|P_{ix}\right||(G|_{\bU})^{-1}_{xy}-(P|_\bU)^{-1}_{xy}||G_{yj}|\Big)\\
\label{differentestimate}  &
  \leq
  4|\msc|q^{r}(1+2^{2\n+3}q)\left(|\bU|^2(2^{\n+2}q+q^{r})+|\bU|^4(2^{\n+2}q+q^{r})^2\right)\leq |m_{sc}|q^r/4
\end{align}
where we used that $\sqrt{d-1}\geq (\n+1)^2 2^{2\n+10}$.
The claim follows by combining this bound for \eqref{difGPS} with \eqref{replaceEr}.
\end{proof}

\subsection{Step 2: Estimate of $\GT_{ij}$ using $G^{(\bU)}_{ij}$}
\label{sec:stabilityT2}

Next we pass from $G^{(\bU)}_{ij}$ to $G^{(\bT)}_{ij}$.
By definition of $\bU$, there are no vertices in $\bT \setminus \bU$ that are close to $i$ or $j$ in the graph $\cGT$.
Thus the step mostly follows from the decay of the Green's function together with the assumption
that there are few cycles.

\paragraph{Starting point}

Define $\cG_0=\cB_{3r}(1,\cG)$ and $\cG_1=\TE(\cG_0)$ as in Section~\ref{sec:stabilityT1}.
The normalized adjacency matrices of $\cG^{(\bU)}$ and $\cG_1^{(\bU)} = \TE(\cG_0^{(\bU)})$ have the block  matrix form
\begin{align*}
 \left[\begin{array}{cc}
        H^{(\bU)} & B'\\
        B & D
       \end{array}
\right],\quad 
\left[\begin{array}{cc}
        H^{(\bU)} & B_1'\\
        B_1 & D_1
       \end{array}
\right],
\end{align*}
where $H^{(\bU)}$ is the normalized adjacency matrix of $\cT^{(\bU)}$.
The nonzero entries of $B$ and $B_1$ occur for the indices
$(i,j) \in \{a_1,\dots, a_\mu\}\times \T_\ell\setminus \bU$ and take values $1/\sqrt{d-1}$.
Notice that $\tilde B_{ij}=(\tilde B_1)_{ij}$. 
We denote the normalized Green's functions of  $\cG^{(\bU)}$ and $\cG_1^{(\bU)}$ by $G^{(\bU)}$ and $P^{ (\bU)}$ respectively.
By the Schur complement formula \eqref{e:Schur1}, for any $i,j\in \qq{N}\setminus\T$, 
\begin{align}\label{GTGS}
 \GT_{ij}=(D-z)^{-1}_{ij}
 =G^{(\bU)}_{ij}-\sum_{x,y\in \T\setminus \bU}G^{(\bU)}_{ix}(G^{(\bU)}|_{\T\setminus\bU})^{-1}_{xy}G^{(\bU)}_{yj},
\end{align}
and also
\begin{align}
\label{invP}(P^{(\bU)}|_{\T\setminus \bU})^{-1}=&H^{(\bU)}-z-B_1'(D_1-z)^{-1}B_1,\\
\label{invG}(G^{(\bU)}|_{\T\setminus {\bU}})^{-1}=&H^{(\bU)}-z-B'(D-z)^{-1}B.
\end{align}

\begin{claim}
For any $x \in \T_\ell\setminus \bU$,
\begin{equation} \label{e:sumPTSbd}
\sum_{y\in \T\setminus \bU} |(P^{(\bU)}|_{\T\setminus \bU})^{-1}_{xy}|
\leq 2(|z|+1).
\end{equation}
\end{claim}

\begin{proof}
For any $x\in \T_\ell\setminus \bU$, by \eqref{invP} we have
\begin{align}
 \notag\sum_{y\in \T\setminus \bU} |(P^{(\bU)}|_{\T\setminus \bU})^{-1}_{xy}|
  &=\sum_{y\in \T\setminus \bU}|(H^{(\bU)}-z-B_1'(D_1-z)^{-1}B_1)_{xy}|\\
  \notag &\leq \sum_{y\in \T\setminus \bU} H^{(\bU)}_{xy} +|z|  +\sum_{y\in \T\setminus \bU} |(B_1'(D_1-z)^{-1}B_1)_{xy}|\\
 \label{express01} &\leq \frac{\n+1}{\sqrt{d-1}} +|z|+\sum_{y\in \T\setminus \bU} |(B_1'(D_1-z)^{-1}B_1)_{xy}|.
\end{align} 
In the last inequality, we used that the excess of $\cT^{(\bU)}$ is at most $\n$
so that for any $x\in \T_\ell\setminus \bU$, we have $\deg_{\cT^{(\bU)}}(x)\leq \n+1$ and
thus $\sum_{y\in \T\setminus \bU} H^{(\bU)}_{xy}\leq (\n+1)/\sqrt{d-1}$.
The terms in the last sum in \eqref{express01} vanish unless $y\in \T_\ell\setminus\bU$.
Therefore the sum is bounded by
\begin{align}
\notag &\sum_{i\in \qq{1,\mu}\atop l_i=x} (B'_1)_{l_ia_i}|(D_1-z)^{-1}_{a_ia_i}|(B_1)_{a_il_i}
+ \sum_{i\neq j\in \qq{1,\mu}\atop l_i,l_j\in \T_\ell\setminus\bU} (B'_1)_{l_ia_i}|(D_1-z)^{-1}_{a_ia_j}|(B_1)_{a_jl_j}\\
\label{express15}
&\leq \frac{1}{d-1}\sum_{i\in\qq{1,\mu}} 1_{x=l_i}|(D_1-z)^{-1}_{a_ia_i}|+\frac{1}{d-1}\sum_{i\neq j\in\qq{1,\mu}} |(D_1-z)^{-1}_{a_ia_j}|
.
\end{align}
For the first sum in \eqref{express15},
the number of vertices $a_i$ adjacent to $x$ is at most $d-1$.
For the second sum, by \eqref{fewclose}, for all pairs $i\neq j$ with at most $(2\n)^2$ exceptions, $\dist_{\cGT}(a_i,a_j)>R/2$.
For these pairs, $a_i$ and $a_j$ are in different connected components of the graph $\cG_1^{(\T)}$
which means that $|(D_1-z)^{-1}_{a_ia_j}|=0$. Therefore there are at most $(2\n)^2$ non-vanishing terms
in the second sum. We use also that
\begin{equation*}
|(D_1-z)^{-1}_{a_ia_j}|=P_{a_ia_j}({\rm TE}(\cG_0^{(\T)})) \leq 3|\msc|/2,
\end{equation*}
which follows from \eqref{e:boundPiimsc}, provided that $\sqrt{d-1}\geq 2^{2\n+3}$. Therefore,
\begin{equation} \label{express02}
\eqref{express15} \leq \left(1+\frac{(2\n)^2}{d-1}\right)\max_{i,j\in\qq{1,\mu}}|(D_1-z)^{-1}_{a_ia_j}|
 \leq \frac{3}{2}|m_{sc}|\left(1+\frac{(2\n)^2}{d-1}\right),
\end{equation}
By combining \eqref{express01} and \eqref{express02}, we have shown that
\begin{align*}
\sum_{y\in \T\setminus \bU} |(P^{(\bU)}|_{\T\setminus \bU})^{-1}_{xy}|\leq  \frac{\n+1}{\sqrt{d-1}} +|z|+\frac{3}{2}|m_{sc}|\left(1+\frac{(2\n)^2}{d-1}\right)\leq 2(|z|+1),
\end{align*}
provided that $\sqrt{d-1}\geq 8(\n+1)$. This completes the proof.
\end{proof}

\begin{claim}
For any $x,y\in \T\setminus \bU$,
\begin{align}
 \label{GTdiffbd} &|G^{(\bU)}_{xy}-P^{(\bU)}_{xy}| \leq 2|\msc|q^r,\\
 \label{GTinvbd} &|G^{(\bU)}|_{\T\setminus {\bU}})_{xy}^{-1}-(P^{(\bU)}|_{\T\setminus \bU})_{xy}^{-1}| \leq 48(|z|+1)q^r,
\end{align}
\end{claim}

\begin{proof}
Define matrices $\cal W$ and $\cal E$ by
\begin{align}
 \label{GTdiff} &G^{(\bU)}|_{\T\setminus \bU}=P^{(\bU)}|_{\T\setminus \bU}+\cal W,\\
 \label{GTinv}&(G^{(\bU)}|_{\T\setminus {\bU}})^{-1}=(P^{(\bU)}|_{\T\setminus \bU})^{-1}+\cal E.
\end{align}
From \eqref{e:stabilityGS}, \eqref{replaceEr}, for any $x,y\in \T\setminus \bU$, we have $|\cal W_{xy}|\leq 2|\msc|q^{r}$. 
We claim the same estimate holds for the entries of the matrix $\cal E$. 
Notice from \eqref{invP}, \eqref{invG} that $\cal E_{xy}\neq 0$ only for $x,y\in \T_\ell\setminus {\bU}$.
Let $\Gamma:=\max_{x,y\in \T\setminus \bU}|\cal E_{xy}|=\max_{x,y\in \T_\ell\setminus \bU}|\cal E_{xy}|$.
By taking the product of \eqref{GTdiff} and \eqref{GTinv},
\begin{align}\label{selfcontrol1}
 \cal E+(P^{(\bU)}|_{\T\setminus \bU})^{-1}\cal W\cal E+(P^{(\bU)}|_{\T\setminus \bU})^{-1}\cal W (P^{(\bU)}|_{\T\setminus \bU})^{-1}=0.
\end{align}
For any $x,y\in \T\setminus\bU$, therefore
\begin{align*}
|\cal E_{xy}|
&\leq \sum_{i,j\in \T\setminus \bU} |(P^{(\bU)}|_{\T\setminus \bU})^{-1}_{xi}||\cal W_{ij}||\cal E_{jy}|+\sum_{i,j\in \T\setminus \bU} |(P^{(\bU)}|_{\T\setminus \bU})^{-1}_{xi}||\cal W_{ij}| |(P^{(\bU)}|_{\T\setminus \bU})^{-1}_{jy}|\\
&\leq |\T\setminus \bU|(2|\msc|q^r)\Gamma \sum_{i\in \T\setminus {\bU}} |(P^{(\bU)}|_{\T\setminus \bU})^{-1}_{xi}|
+(2|\msc|q^r)\sum_{i\in \T\setminus \bU} |(P^{(\bU)}|_{\T\setminus \bU})^{-1}_{xi}|\sum_{j\in \T\setminus \bU}|(P^{(\bU)}|_{\T\setminus \bU})^{-1}_{jy}|\\
&\leq  4( |z|+1)  (d-1)^{\ell}(2|\msc|q^{r})\Gamma +4( |z|+1)^2 (2|\msc|q^{r})\\
&\leq \Gamma/2+4( |z|+1)^2 (2|\msc|q^{r}).
\end{align*}
For the second inequality, we used $|\cal W_{xy}|\leq 2|\msc|q^{r}$;
for the third inequality we used $|\T\setminus \bU|\leq |\T|\leq 1+d+d(d-1)+\cdots + d(d-1)^{\ell-1}\leq 2(d-1)^{\ell}$, and \eqref{e:sumPTSbd};
for the last inequality, we used $r= 2\ell +1$, so that $(d-1)^\ell q^r\leq (d-1)^{-1/2}$ and  $|z\msc|\leq 2$.
Taking the maximum on the right-hand side of the above inequality, and rearranging, we get
\begin{align*}
 \Gamma \leq 16(|z|+1)^2|\msc|q^r \leq 48(|z|+1) q^r,
\end{align*}
as claimed.
\end{proof}

\begin{proof}[Proof of Proposition \ref{prop:stabilityGT}]
To prove the proposition, we define $\bU$ by \eqref{e:bSdef}, and show that
\begin{equation}\label{e:stabilityofGTclaim}
 |G_{ij}^{(\T)}-G_{ij}^{(\bU)}|\leq |\msc|q^{r}/4. 
\end{equation}
This implies the claim. Indeed, the definition of $\bU$ implies
that $\cE_r(i,j,\cG^{(\bU)})=\cE_r(i,j,\cG^{(\T)})$,
and therefore \eqref{e:stabilityGT} follows from \eqref{replaceEr},
\eqref{e:stabilityofGTclaim} and Proposition~\ref{stabilityGS}:
\begin{align*}
 |\GT_{ij}-P_{ij}(\cE_{r}(i,j, \cGT))|
 &\leq |\GT_{ij}-G^{(\bU)}_{ij}| +|G^{(\bU)}_{ij}-P_{ij}^{(\bU)}|+|P_{ij}(\cE_{r}(i,j, \cG^{(\bU)}))-P^{(\bU)}_{ij}|\\
 &\leq |\msc|q^{r}/4+3|\msc|q^{r}/2+2^{2\n+3}|m_{sc}|q^{r+1}\leq 2|\msc|q^r.
\end{align*}

Thus it remains to prove \eqref{e:stabilityofGTclaim}. By definition of $\bU$, we have
$\dist_{\cG^{(\bU)}}(\{i,j\}, \T\setminus\bU)>r$, and therefore Proposition~\ref{stabilityGS} implies
\begin{align}\label{GSbound}
 \max_{x\in \T\setminus\bU}\ha{|G^{(\bU)}_{ix}|,|G^{(\bU)}_{jx}|}
  \leq 3|\msc|q^r/2\leq 2|\msc|q^r.
\end{align}
Furthermore, by \eqref{GTGS},
\begin{align}\begin{split}\label{GT-GS}
|G^{(\T)}_{ij}-G^{(\bU)}_{ij}|
&\leq \sum_{x,y\in \T\setminus \bU}|G^{(\bU)}_{ix}(H^{(\bU)}-z-B'G^{(\T)}B)_{xy}G^{(\bU)}_{yj}|\\
&\leq 4|\msc|^2q^{2r}\sum_{x,y\in \T\setminus \bU} |(H^{(\bU)}-z-B_1'(D_1-z)^{-1}B_1+\cal E)_{xy}|,
\end{split}
\end{align}
with $\cal E$ as in \eqref{GTinv}. For the sum, we have
\begin{align}\begin{split}
&\sum_{x,y\in \T\setminus \bU} |(H^{(\bU)}-z-B_1'(D_1-z)^{-1}B_1+\cal E)_{xy}|\\
&\leq \sum_{x,y\in \T\setminus \bU} H^{(\bU)}_{xy}+\sum_{i,j\in\qq{1,\mu}\atop l_i,l_j\in \T_{\ell}\setminus\bU}(B_1')_{l_ia_i}|(D_1-z)^{-1}_{a_ia_j}|(B_1)_{a_jl_j} +|z||\T\setminus \bU| +|\T\setminus \bU|^2\max_{x,y\in \T\setminus\bU}|\cal E_{xy}|\\
\label{express03}&\leq \sum_{x,y\in \T\setminus \bU} H^{(\bU)}_{xy} +\frac{1}{d-1}\sum_{i,j\in\qq{1,\mu}}|(D_1-z)^{-1}_{a_ia_j}| +2(d-1)^{\ell}|z|+4(d-1)^{2\ell}\left(48(|z|+1)q^r\right),
\end{split}\end{align}
where we used $|\T\setminus \bU|\leq 2(d-1)^\ell$ and \eqref{GTinvbd}. By our assumption $\cG\in \bar\Omega$, the subgraph $\cT$ has excess at most $\n$.
Therefore the total number of edges of $\cT$ is bounded by $|\T|+\n\leq 1+d+d(d-1)+\cdots+d(d-1)^{\ell-1}+\n\leq 2(d-1)^{\ell}$, and 
\begin{align}\label{express04}
\sum_{x,y\in \T\setminus \bU} H^{(\bU)}_{xy} \leq \frac{2(d-1)^{\ell}}{\sqrt{d-1}}.
\end{align}
By the same argument as for \eqref{express02}, we get
\begin{align}\begin{split}\label{express05}
\frac{1}{d-1}\sum_{i,j\in\qq{1,\mu}}|(D_1-z)^{-1}_{a_ia_j}| 
&=\frac{1}{d-1}\sum_{i\in\qq{1,\mu}}|(D_1-z)^{-1}_{a_ia_i}| +\frac{1}{d-1}\sum_{i\neq j\in\qq{1,\mu}}|(D_1-z)^{-1}_{a_ia_j}|\\
&\leq \frac{3|\msc|}{2}\left(\frac{\mu}{d-1}+\frac{(2\n)^2}{d-1}\right) \leq \frac{3|\msc|}{2}\left(2(d-1)^{\ell}+\frac{(2\n)^2}{d-1}\right)
\end{split}\end{align}
where we used $\mu\leq 2(d-1)^{\ell+1}$. By combining \eqref{express03}--\eqref{express05}, we have
\begin{align*}
&\sum_{x,y\in \T\setminus \bU} |(H-z-B_1'(D_1-z)^{-1}B_1+\cal E)_{xy}|
\leq 2(d-1)^\ell |z|+4(d-1)^{2\ell}\left(48(|z|+1)q^r\right)+\\&+\frac{2(d-1)^{\ell}}{\sqrt{d-1}}+\frac{3|\msc|}{2}\left(2(d-1)^{\ell}+\frac{(2\n)^2}{d-1}\right) \leq 5(|z|+1)(d-1)^\ell.
\end{align*}
Combining the above estimate with \eqref{GT-GS}, and using $|z\msc|\leq 2$,
\begin{align*}
 |G^{(\T)}_{ij}-G^{(\bU)}_{ij}|\leq  4|\msc|^2q^{2r}5(|z|+1)(d-1)^\ell\leq \frac{20(|z|+1)|\msc|}{\sqrt{d-1}}|\msc|q^r\leq |\msc|q^{r}/4.
\end{align*}
This finishes the proof.
\end{proof}

\section{Stability under switching}
\label{sec:stability}

We recall the $\rm S$-cells and $\rm S'$-cells from Definition \ref{def:cells}, and the set of switching data $\sS(\cG)$ from Section \ref{sec:rev}. The results of this section are the following stability estimates.

\begin{proposition} \label{prop:stabilitytGT}
Let $z\in\C_+$, $\cG \in \bar\Omega$ (as in Section \ref{sec:outline-structure}) be a $d$-regular graph,
and $K\geq 2$ be a constant
such that, for all $i,j \in \qq{N} \setminus \T$, 
\begin{align}\label{asumpGT}
\left|\GT_{ij}-P_{ij}(\cE_r(i,j,\cGT),z)\right|\leq K|\msc|q^r.
\end{align}
Then there exists an event $F(\cG) \subset \sS(\cG)$ with $\P_{\cG}(F(\cG))= 1-o(N^{-\n+\delta})$,
explicitly defined in Section~\ref{cdFG} below, 
such that for any ${\bf S}\in F(\cG)$ such that $\tcG  = T_{\bf S}(\cG) \in \bar\Omega$ the following hold:
\begin{itemize}
\item
For $i,j \in \qq{N}\setminus\T$,
\begin{align}\label{boundhGT}
 |\hGT_{ij}-P_{ij}(\cE_{r}(i,j,\hcGT),z)|\leq 2K |\msc| q^r.
\end{align}

\item 
For (\rn 1) $i, j \in \qq{N} \setminus \T$ in different $\rm S$-cells, or (\rn 2)  $i,j \in \qq{N} \setminus\T$ such that $j\in \cC_t$ and $i\not\sim \cC_t$ for some $t$,
\begin{equation} \label{e:diffcellest}
|\hGT_{ij}|\leq \frac{2M}{\sqrt{N\eta}}.
\end{equation}

\item
For $i,j \in \qq{N}\setminus\T$,
\begin{equation} \label{e:stabilitytGT}
|\tGT_{ij}-P_{ij}(\cE_{r}(i,j,\tcGT),z)|
\leq 2^{7}K^3|\msc|q^r.
\end{equation}

\end{itemize}
For all estimates, we assume $\sqrt{d-1}\geq \max\{(\n+1)^2 2^{2\n+10}, 2^8(\n+1)K\}$, $\n'q^r\ll1$ and that $\sqrt{N\eta}\geq M(d-1)^{\ell+1}$
(where $M$ is as in \eqref{e:Mnprime}).
\end{proposition}

In particular, for any $\cG\in \Omega(z,\ell)$, Proposition~\ref{prop:stabilityGT} implies that
the assumptions of Proposition~\ref{prop:stabilitytGT} are satisfied with $K=2$.
Thus Propositions~\ref{prop:stabilityGT} and \ref{prop:stabilitytGT} together show that, for any graph $\cG\in \Omega(z,\ell)$,
with high probability under $\bf S$, the switched graph $\tcG$ belongs to $\Omega_1^+(z,\ell)$
(as in Section~\ref{sec:outline}).

\subsection{Definition of the event $F(\cG)$}
\label{cdFG}

We fix $M$ and $\n'$ by \eqref{e:Mnprime}.
We will prove Proposition~\ref{prop:stabilitytGT} with the set $F(\cG) \subset \sS(\cG)$
defined by the following conditions on the switching data $\bf S$:
\begin{enumerate}
\item At least $\mu-3\n$ edges are switchable, i.e.\ the event in the probability in \eqref{e:Wbd} holds:
\begin{align} \label{switchnum} 
|W_{\bf S}|\geq \mu-3\n.
\end{align}
\item
All except for $\n$ of the vertices $\{c_1,c_2,\dots, c_\mu\}$ have radius-$R$ tree neighborhoods in $\cGT$,
i.e.\ \eqref{treeneighbor} holds.
\item  
The vertices $\{a_1,\dots, a_\mu,b_1,\dots, b_\mu\}$ do not cluster in the sense of distance,
i.e.\ \eqref{fewclose}--\eqref{distfinitedegl} and \eqref{distfinitedegb}--\eqref{distdegb} hold.
\item  
The vertices $\{b_1,b_2,\dots, b_\mu\}$ do not cluster in the sense of the Green's function,
i.e.\ \eqref{finitedeg}--\eqref{e:Spcellbd} hold.
\end{enumerate}
Then, for any $\cG \in \Omega_1^+(z,\ell)$, we have 
\begin{align}\label{FGmeasure}
\P_{\cG}(F(\cG))= 1-o(N^{-\n+\delta}).
\end{align}
Indeed,
(i) follows from Proposition~\ref{lem:switchable},
(ii) follows from \eqref{treeneighbor},
(iii) follows from Propositions~\ref{GTstructure} and \ref{prop:distdeg},
and
(iv) follows from Proposition~\ref{prop:Rdist2}.

\subsection{Proof of \eqref{boundhGT}}\label{proofofboundhGT}

The proof of \eqref{boundhGT} follows the structure described below \eqref{e:stabilityGS}.
Moreover, similarly to the proof of Proposition~\ref{prop:stabilityGT},
we distinguish between vertices $i,j$ that are close to the edges that get removed in going from $\cGT$ to $\hcGT$ and vertices
that are far from these edges. We first focus on $i,j$ that are close to those edges that get removed.

\paragraph{Localization}

First, we replace $P_{ij}(\cE_r(i,j,\cG))$ by the vertex-independent Green's function $P_{ij}$,
using Remark~\ref{rk:princG0} with
\begin{equation}
\cX = \cB_{3r}(\{b_1,b_2,\dots, b_\nu\}, \cGT),
\quad
\bX = \bB_{2r}(\{b_1,b_2,\dots, b_\nu\}, \cGT).
\end{equation}
Moreover, we define
$\hcX$ to be the graph obtained by removing the edges $\{b_i,c_i\}_{i\leq \nu}$ from $\cX$.
The deficit function of $\hcX$ is defined to be the restriction of that of $\hcGT$. We abbreviate
\begin{align*}
\cG_1 = \TE(\cX),\quad P=G(\cG_1),\quad \hcG_1=\TE(\hcX),\quad  \hat{P}=G(\hcG_1).
\end{align*}
Notice that $\hcG_1$ is equivalently obtained by removing the edges $\{b_i,c_i\}_{i\leq \nu}$ from $\cG_1$.
The following properties of $\cG_0$ follow from \eqref{distdega} and \eqref{distdegb}.

\begin{claim}\label{c:componentestimate0}
Assume \eqref{distdega} and \eqref{distdegb}.
Then each connected component of either $\cG_0$ or $\hcG_0$ contains at most $5\n$ elements from $\{a_1,\dots,a_\mu, b_1,\dots, b_\nu\}$.
More precisely,
\begin{align}\label{e:componentest0}
|\{i\in \qq{1,\mu}: a_i\in \bK\}|\leq 3\n, \quad |\{i\in \qq{1,\nu}: b_i\in \bK\}|\leq 2\n,
\end{align}
where $\bK$ is the vertex set of any connected component of $\cG_0$ or $\hcG_0$.
\end{claim}

\begin{proof}
The claim follows directly from \eqref{distdega} and \eqref{distdegb} and the definitions of $\cG_0$ and $\hcG_0$.
\end{proof}

\paragraph{Verification of assumptions in Proposition \ref{boundPij}}
Since both $\cG_0$ and $\hcG_0$ are subgraphs of $\cG \in \bar\Omega$, their radius-$R$ neighborhoods have excess at most $\omega$.
Let $\bK$ be the vertex set of any connected component of $\cG_0$ or $\hcG_0$.
Since the deficit function of $\cG_0$ (respectively $\hcG_0$) is the restriction of that of $\cGT$ (respectively $\hcGT$),
any of the vertices $a_i, b_i, c_i\in \bK$ contributes $1$ to the sum of the deficit function over $\bK$.
By Claim~\ref{c:componentestimate0}, the sums of the deficit functions over any of the connected components of $\cG_0$ and $\hcG_0$
are therefore bounded by $3\n+2\times 2\n\leq 8\n$.
Thus the assumptions of \eqref{e:compatibility} are verified for both $\cG_0$ and $\hcG_0$, and for any $i,j\in \bX$,
\begin{equation}\label{e:PP}
|P_{ij}-P_{ij}(\cE_r((i,j,\cGT))|\leq 2^{2\n+3}|\msc|q^{r+1},
\end{equation}
provided that $\sqrt{d-1}\geq 2^{\n+2}$, and an analogous estimate holds for $\hat P$.
Up to a small error, we can therefore use $P$ instead of  $P(\cE_r((i,j,\cGT))$ and $\hat{P}$ instead of $P(\cE_r((i,j,\hcGT))$.

\paragraph{Starting point}

By the resolvent identity \eqref{e:resolv}, we have:
\begin{align}
 \label{hGT}\hGT-\GT&=\GT\Delta \hGT,\\
  \label{hPT}\hat{P}-P&=P\Delta \hat{P},
\end{align}
where $\Delta=\sum_{k=1}^\nu (e_{b_kc_k}+e_{c_kb_k})/\sqrt{d-1}$.
Taking the difference of \eqref{hGT} and \eqref{hPT}, we obtain
\begin{equation}\label{perturbh}
\hGT_{ij}-\hat{P}_{ij}=(\GT_{ij}-P_{ij})
+\frac{1}{\sqrt{d-1}}\sum_{(x,y)\in \vE}(\GT_{ix}-P_{ix}) \hat{P}_{yj}
+ \frac{1}{\sqrt{d-1}}\sum_{(x,y)\in \vE}\GT_{ix}(\hGT_{yj}-\hat{P}_{yj}),
\end{equation}
where the summation is over the oriented edges
\begin{align}\label{e:defB}
(x,y)\in \vE=\{(b_1,c_1),\dots, (b_\nu, c_\nu), (c_1,b_1),\dots, (c_\nu, b_\nu)\}.
\end{align}
We regard \eqref{perturbh} as an equation for $\hGT-\hat P$,
and will show that $\hGT_{ij}-\hat{P}_{ij}$ is small as a consequence of
the smallness of $\GT-P$.

\paragraph{Green's function estimates}

We first collect some estimates on Green's functions, used repeatedly:
\begin{align}
\begin{cases}
\label{GPbound} |\GT_{ij}|, |P_{ij}|, |\hat P_{ij}|\leq 2|\msc|,  & (\text{all $i,j$}),\\
|\GT_{ij}|\leq K|m_{sc}|q^r, &(\dist_{\cGT}(i,j)\geq 2r),\\
 |\GT_{ib_k}|, |\GT_{ic_k}|\leq M/\sqrt{N\eta}, & (\text{$i,b_k$ are in different $\rm S$-cells, or $i\not\sim b_k$}).
\end{cases}
\end{align}
The first estimate follows from \eqref{e:boundPiimsc}, and \eqref{asumpGT};
the second estimate follows from assumption \eqref{asumpGT}, and $P_{ij}(\cE_r(i,j,\cGT))=0$;
the last estimate holds by the definition of $\sim$ in Section~\ref{sec:defcells}.

\begin{proof}[Proof of \eqref{boundhGT} for $i,j\in \bX$]
By assumption and \eqref{e:PP}, the first term in \eqref{perturbh}
is bounded by $K |\msc| q^r+2^{2\n+3}|m_{sc}|q^{r+1}$.
For the second term on the right-hand side of \eqref{perturbh},
similarly $|\GT_{ix}-P_{ix}| \leq K |\msc| q^r+2^{2\n+3}|m_{sc}|q^{r+1}$.
Moreover, $\hat{P}_{yj}=0$ if $y$ and $j$ are in different connected components of $\hcG_0$
Thus by Claim~\ref{c:componentestimate0}, we have $\hat{P}_{yj}\neq 0$ for at most $4\n$ vertices
$y\in \{b_i,c_i: i\in \qq{1,\nu}\}$, for which 
we use $|\hat{P}_{yj}|\leq 2|\msc|$ by \eqref{GPbound}. Combining these bounds,
the second term in \eqref{perturbh} is bounded by
\begin{align}\label{e:sebound}
\frac{1}{\sqrt{d-1}}\sum_{(x,y)\in \vE}|\GT_{ix}-P_{ix}| |\hat{P}_{yj}| 
 \leq 8\n (K+2^{2\n+3}q)|m_{sc}|q^{r+1}.
\end{align}
To estimate the last term in \eqref{perturbh}, we denote
\begin{align*}
 \Gamma:=\max_{i,j\in \bX}|\hGT_{ij}-\hat{P}_{ij}|.
\end{align*}
Noticing that $\bX\subset \cup_{i=1}^{\kappa} \cC_i$,
we decompose the last sum over $\vE$ in \eqref{perturbh} according to the cases in \eqref{GPbound} as
\begin{equation*}
\sum_{\vE} \sumdots
=
\sum_{\vE_1} \sumdots
+
\sum_{\vE_2} \sumdots
+
\sum_{\vE_3} \sumdots
\end{equation*}
where here and below $\sumdots$ abbreviates the terms in the last sum in \eqref{perturbh} and
\begin{align*}
&\vE_1=\{(b_k,c_k), (c_k,b_k): \text{$i, b_k$ are in different $\rm S$-cells}\},\\
&\vE_2=\{(b_k,c_k), (c_k,b_k): \text{$i, b_k$ are in the same $\rm S$-cells, and $\dist_{\cGT}(i,b_k)> 2r$}\},\\
&\vE_3=\{(b_k,c_k), (c_k,b_k): \text{$i, b_k$ are in the same $\rm S$-cells, and $\dist_{\cGT}(i,b_k)\leq 2r$}\}.
\end{align*}
Notice that, for any $(x,y)\in \vE$, we have $|(\hGT-\hat{P})_{yj}|\leq \Gamma$ by the definition of $\Gamma$.
For $(x,y)\in \vE_1$, $|\GT_{ix}|\leq M/\sqrt{N\eta}$ by \eqref{GPbound}, and
$|\vE_1|\leq 2\nu\leq 4(d-1)^{\ell+1}$,
\begin{align*}
\sum_{\vE_1}\sumdots
\leq \frac{\Gamma}{\sqrt{d-1}}\sum_{(x,y)\in \vE_1} |\GT_{ix}| \leq \frac{4(d-1)^{\ell+1/2}M\Gamma}{\sqrt{N\eta}}.
\end{align*}
For $(x,y)\in \vE_2$, $|\GT_{ix}|\leq K|m_{sc}|q^r$ by \eqref{GPbound}, and $
|\vE_2|\leq 2\n'$ by \eqref{deg},
\begin{align*}
\sum_{\vE_2}\sumdots
\leq \frac{\Gamma}{\sqrt{d-1}}\sum_{(x,y)\in \vE_2} |\GT_{ix}| \leq \frac{2K\n'|m_{sc}|q^r\Gamma}{\sqrt{d-1}}.
\end{align*}
For $(x,y)\in \vE_3$, $|\GT_{ix}|\leq 2|m_{sc}|$ by \eqref{GPbound}, and there are at most such $2\n$ terms,
i.e.\ $|\vE_{3}|\leq 2\n$,
by \eqref{distfinitedegb},
\begin{align*}
\sum_{\vE_3}\sumdots
\leq \frac{\Gamma}{\sqrt{d-1}}\sum_{(x,y)\in \vE_3} |\GT_{ix}| \leq \frac{4\n|m_{sc}|\Gamma}{\sqrt{d-1}}.
\end{align*}
Combining the sums over $\vE_1, \vE_2, \vE_3$, we get
\begin{align*}
\frac{1}{\sqrt{d-1}}\sum_{(x,y)\in \vE}|\GT_{ix}||\hGT_{yj}-\hat{P}_{yj}|\leq \frac{\Gamma}{4},
\end{align*}
provided that $\sqrt{d-1}\geq 20\n$, $\n'q^r\ll1$ and $\sqrt{N\eta}\geq (d-1)^{\ell+1} M$. Thus \eqref{perturbh} leads to 
\begin{align*}
|\hGT_{ij}-\hat{P}_{ij}|
\leq& (1+8\n q)\left(K+2^{2\n+3}q\right)|\msc|q^r+\Gamma/4.
\end{align*}
Taking the supremum over $i,j\in \bX$, we obtain
\begin{align*}
 \Gamma\leq \frac{4}{3}(1+8\n q)\left(K+2^{2\n+3}q\right)|\msc|q^r.
\end{align*}
From this estimate, and from \eqref{e:PP} to estimate $\hat{P}_{ij} - P_{ij}(\cE_r(i,j,\hcGT))$, we find
\begin{align}\label{e:boundXX}
|\hGT_{ij}-P_{ij}(\cE_{r}(i,j,\hcGT))| 
&\leq  |\hGT_{ij}-\hat P_{ij}| +|\hat P_{ij}-P_{ij}(\cE_r((i,j,\hcGT))| \nonumber\\
&\leq \frac{4}{3}(1+8\n q)\left(K+2^{2\n+3}q\right)|\msc|q^r+2^{2\n+3}|\msc|q^{r+1}
\leq 2K |\msc| q^r,
\end{align}
provided that  $\sqrt{d-1}\geq 2^{2\n+5}$.
This concludes the proof of \eqref{boundhGT} for $i,j\in \bX$.
\end{proof}

\begin{proof}[Proof of \eqref{boundhGT} in the remaining case]
In the remaining case at least one of $i,j$ is not contained in $\bX$;
and by symmetry we can assume that $i\not\in \bX$. Then
$\cE_{r}(i,j,\cGT)=\cE_{r}(i,j,\hcGT)$ and the graphs on both sides of the equality also have the same deficit function.
It therefore suffices to show that $|\hGT_{ij}-\GT_{ij}|$ is small.
By the resolvent identity \eqref{hGT}, we have
\begin{align}\label{e:hGT2}
|\hGT_{ij}-\GT_{ij}|\leq \frac{1}{\sqrt{d-1}}\sum_{(x,y)\in \vE}|\GT_{ix}||\hGT_{yj}|
.
\end{align}
Since $i\notin \bX$, we have $\dist_{\cGT}(i,\{b_k,c_k\})\geq 2r$
and therefore, by \eqref{GPbound},
\begin{equation} \label{e:GTixbd}
  |\GT_{ix}|\leq K|\msc|q^r \quad \text{for any $x\in \{b_i,c_i: i\in \qq{1,\nu}\}$}.
\end{equation}

For the case that exactly one of $i,j$ is in $\bX$, i.e.\ $i\not\in \bX$ and $j\in \bX$,
we now decompose the set $\vE$ defined in \eqref{e:defB} as $\vE=\vE_1'\cup \vE_2'\cup \vE_3'$, where 
\begin{align}\begin{split}\label{e:decomposevecE}
\vE_1'&=\{(b_k,c_k), (c_k,b_k): \dist_{\cGT}(b_k,j)\leq 2r\},\\
\vE_2'&=\{(b_k,c_k), (c_k,b_k): i\sim b_k, \dist_{\cGT}(b_k,j)> 2r\},\\
\vE_3'&=\{(b_k,c_k), (c_k,b_k): i\not\sim b_k,\dist_{\cGT}(b_k,j)> 2r\}.
\end{split}\end{align}
For $(x,y)\in \vE_1'$, since $y,j\in \bX$, $|\hGT_{yj}|\leq |P_{yj}(\cE_r((y,j,\hcGT))|+2K|m_{sc}|q^r\leq 2|m_{sc}|$ by \eqref{e:boundXX} and \eqref{e:boundPiimsc}, and there are at most $2\n$ terms, i.e. $
|\vE_1'|\leq 2\n$ by \eqref{distfinitedegb},
\begin{align*}
\sum_{\vE_1'}\sumdots
\leq \frac{1}{\sqrt{d-1}}\sum_{(x,y)\in \vE_1'} (K|m_{sc}|q^r)|\hGT_{yj}| \leq \frac{4K\n |m_{sc}|q^r}{\sqrt{d-1}}.
\end{align*}
where now $\sumdots$ refers to the terms in the sum in \eqref{e:hGT2}.
For $(x,y)\in \vE_2'$, since $y,j\in \bX$, $|\hGT_{yj}|\leq 2K|m_{sc}|q^r$ by \eqref{e:boundXX},
and $|\vE_2'|\leq 2\n'$ by \eqref{finitedeg},
\begin{align*}\sum_{\vE_2'}\sumdots
\leq  \frac{1}{\sqrt{d-1}}\sum_{(x,y)\in \vE_2'} (K|m_{sc}|q^r)|\hGT_{yj}| \leq 
\frac{2\n'(K |m_{sc}|q^r)(2K |m_{sc}|q^r)}{\sqrt{d-1}}.
\end{align*}
For $(x,y)\in \vE_3'$, $|\GT_{ix}|\leq M/\sqrt{N\eta}$ by the definition of $\sim$, $|\hGT_{yj}|\leq 2K|m_{sc}|q^r$ by \eqref{e:boundXX}, and $|\vE_3'|\leq 2\nu\leq 4(d-1)^{\ell+1}$,  
\begin{align*}
\sum_{\vE_3'}\sumdots\leq \frac{1}{\sqrt{d-1}}\sum_{(x,y)\in \vE_3'} \frac{M}{\sqrt{N\eta}}(2K|m_{sc}|q^r) \leq 8K(d-1)^{\ell+1}q^{r+1}\frac{M}{\sqrt{N\eta}}.
\end{align*}
Combining the sums over $\vE_1', \vE_2', \vE_3'$, from \eqref{e:hGT2} we obtain
\begin{align*}
|\hGT_{ij}-\GT_{ij}|\leq K|\msc|q^{r},
\end{align*}
provided that $\sqrt{d-1}\geq 20\n$, $\n'q^r\ll1$ and $\sqrt{N\eta}\geq (d-1)^{\ell+1} M$. This concludes the proof of \eqref{boundhGT} for $i\not\in\bX$ and $j\in \bX$,
\begin{align}\label{e:boundXX2}
\begin{split}
|\hGT_{ij}-P_{ij}(\cE_{r}(i,j,\hcGT))| 
&=|\hGT_{ij}-P_{ij}(\cE_{r}(i,j,\cGT))| \\
&\leq  |\hGT_{ij}-\GT_{ij}| +|\GT_{ij}-P_{ij}(\cE_r((i,j,\cGT))|
\leq 2K |\msc| q^r.
\end{split}
\end{align}

For the case that $i,j\not \in \bX$, noticing that $\dist_{\cGT}(b_k,j)>2r$,  we decompose the set $\vE$ as $\vE= \vE_2'\cup \vE_3'$, where $\vE_2'$ and $\vE_3'$ are defined in \eqref{e:decomposevecE}.
By \eqref{e:boundXX2}, for any $(x,y)\in \vE$, $P_{yj}(\cE_r(y,j,\hcGT))=0$ and thus $|\hGT|_{yj}\leq 2K|\msc|q^r$. Then the same argument as above implies
\begin{align*}
|\hGT_{ij}-\GT_{ij}|\leq K|\msc|q^{r}. 
\end{align*}
This finishes the proof of the stability of $\hGT$. 
\end{proof}

\subsection{Proof of \eqref{e:diffcellest}}

We again follow the structure described below \eqref{e:stabilityGS},
except that no localization step is required to prove \eqref{e:diffcellest}.

\paragraph{Starting point}

Under both conditions given for \eqref{e:diffcellest},
we have $|\GT_{ij}| \leq M/\sqrt{N\eta}$ by the definition of $\sim$ as in Section \ref{sec:defcells}.
By the resolvent identity \eqref{hGT}, we therefore have
\begin{equation}
\label{hGTdiffcell}
|\hGT_{ij}|\leq \frac{M}{\sqrt{N\eta}}+\frac{1}{\sqrt{d-1}}\sum_{(x,y)\in \vE}|\GT_{ix}||\hGT_{yj}|,
\end{equation}
where $\vE$ is as in \eqref{e:defB}.
Notice that if $i,j$ are in different $\rm S$-cells, then for any $(x,y)\in \vE$, either $i,x$ are in different $\rm S$-cells,
or $y,j$ are in different $\rm S$-cells.
Similarly if $j\in \cC_t$ and $i\not\sim \cC_t$ for some $t$, then for any $(x,y)\in \vE$, either $|\GT_{ix}|\leq M/\sqrt{N\eta}$,
or the vertices $y,j$ are in different $\rm S$-cells.
The claim \eqref{boundhGT} follows by analyzing \eqref{hGTdiffcell} as an inequality for these $\hGT_{ij}$.

\paragraph{Green's function estimates}
We first collect some estimates on Green's functions of $\GT$ and $\hGT$,
which are repeatedly used in the proof: for $(x,y)\in \vE$ as in \eqref{e:defB},
\begin{align}\label{GTbound2} 
&|\GT_{ix}|\leq \begin{cases}
2|\msc|,  & (\text{all $x$}),\\
K|m_{sc}|q^r, &(\dist_{\cGT}(i,x)\geq 2r),\\
M/\sqrt{N\eta}, & (\text {$i,x$ are in different $\rm S$-cells; or $i\not\sim$ the $\rm S$-cell containing $x$}).
\end{cases}\\
\label{hGTbound2}
&|\hGT_{yj}|\leq \begin{cases}
2|\msc|,  & (\text{all $y$}),\\
2K|m_{sc}|q^r, &(\dist_{\hcGT}(y,j)\geq 2r),
\end{cases}
\end{align}
These estimates follow from \eqref{asumpGT}--\eqref{boundhGT}, together with \eqref{e:boundPiimsc} for the bound for all $x,y$, and with $P_{ix}(\cE_r(i,x,\cGT))=0$ for $\dist_{\cGT}(i,x) \geq 2r$; and  $P_{yj}(\cE_r(y,j,\hcGT))=0$ for $\dist_{\hcGT}(y,j) \geq 2r$. The last bound in \eqref{GTbound2} holds by the definition of $\sim$.

\begin{proof}[Proof of \eqref{e:diffcellest}, case (\rn 1)]
We verify \eqref{e:diffcellest} in the case that $i,j$ are in different $\rm S$-cells.
Denote
\begin{align*}
\Gamma:=\max_{t_1\neq t_2}\max_{i\in \cC_{t_1}, j\in \cC_{t_2}}|\hGT_{ij}|,
\end{align*}
and now abbreviate by $\sumdots$ the terms in the sum in \eqref{hGTdiffcell}
including the $1/\sqrt{d-1}$ prefactor.
We divide the set $\vE$ according to their relations to the cells
$\cC_{t_1}$ and $\cC_{t_2}$ as $\vE=\vE_1\cup \vE_2\cup \vE_3\cup \vE_4\cup \vE_5$, where
\begin{align*}
&\vE_1=\{(x,y)\in \cC_{t_1}: \dist_{\cGT}(i,x)\leq2r\},\\
&\vE_2=\{(x,y)\in \cC_{t_1}: \dist_{\cGT}(i,x)> 2r\},\\
&\vE_3=\{(x,y)\in \cC_{t_2}: \dist_{\cGT}(y,j)\leq 2r\},\\
&\vE_4=\{(x,y)\in \cC_{t_2}: \dist_{\cGT}(y,j)> 2r\},\\
&\vE_5=\{(x,y)\not\in \cC_{t_1}\cup \cC_{t_2}\}.
\end{align*}
For $(x,y)\in \vE_1$, $|\vE_1|\leq 2\n$ from \eqref{distfinitedegb},
i.e.\ $|\{k\in \qq{1,\nu}: \dist(i,b_k)<2r\}|\leq \n$, and $|\hGT_{yj}|\leq \Gamma$,
by the definition of $\Gamma$. Thus, by \eqref{GTbound2},
\begin{align*}
\sum_{\vE_1}\sumdots\leq \frac{4\n|\msc|\Gamma}{\sqrt{d-1}}.
\end{align*}
For $(x,y)\in \vE_2$, $|\vE_2|\leq 2\n'$ from \eqref{deg}, i.e. $|\cC_{t_1}\cap \{b_1,\dots,b_\nu\}|\leq \n'$. Thus, by \eqref{GTbound2},
\begin{align*}
\sum_{\vE_2}\sumdots\leq \frac{2K\n'|\msc|q^r\Gamma}{\sqrt{d-1}}.
\end{align*}
For $(x,y)\in \vE_3$, $|\vE_3|\leq 2\n$ from \eqref{distfinitedegb}, and by \eqref{GTbound2}--\eqref{hGTbound2}, 
\begin{align*}
\sum_{\vE_3}\sumdots\leq  \frac{M}{\sqrt{N\eta}}\frac{4\n|\msc|}{\sqrt{d-1}}
\end{align*}
For $(x,y)\in \vE_4$, $|\vE_4|\leq 2\n'$ from \eqref{deg}, and $\dist_{\hcGT}(y,j)\geq \dist_{\cGT}(y,j)>2r$.
Thus, by  \eqref{GTbound2}--\eqref{hGTbound2},
\begin{align*}
\sum_{\vE_4}\sumdots\leq  \frac{M}{\sqrt{N\eta}}\frac{4K\n'|\msc|q^r}{\sqrt{d-1}}.
\end{align*}
Finally, for $(x,y)\in \vE_5$, we use $|\vE_5|\leq 2\nu\leq 4(d-1)^{\ell+1}$ and $|\hGT_{yj}|\leq \Gamma$ which holds by the definition of $\Gamma$. Thus, by \eqref{GTbound2},
\begin{align*}
\sum_{\vE_5}\sumdots\leq   \frac{4(d-1)^{\ell+1}M\Gamma}{\sqrt{d-1}\sqrt{N\eta}}.
\end{align*}
Combining the sums over $\vE_1,\dots, \vE_5$ in \eqref{hGTdiffcell} leads to
\begin{align*}
|\hGT_{ij}|\leq \frac{M}{\sqrt{N\eta}}+\frac{\left(4\n|\msc|+2K\n'|\msc|q^r\right)\Gamma}{\sqrt{d-1}}+\frac{\left(4\n|\msc|+4K\n'|\msc|q^r\right)}{\sqrt{d-1}}\frac{M}{\sqrt{N\eta}}+\frac{4M(d-1)^{\ell+1/2}\Gamma}{\sqrt{N\eta}}.
\end{align*}
By taking the maximum over $i,j$ as in the assumption and rearranging the inequality, we get
\begin{align}\label{diffcellest}
\Gamma\leq 2M/\sqrt{N\eta},
\end{align}
provided that $\sqrt{d-1}\geq 20\n$, $\n'q^r\ll1$ and $\sqrt{N\eta}\geq M(d-1)^{\ell+1}$.
\end{proof}

\begin{proof}[Proof of \eqref{e:diffcellest}, case (\rn 2)]
For $j\in \cC_t$ and $i\not\sim \cC_t$, we now decompose the set $\vE$
according to their relations to vertex $i$ and the cell $\cC_t$
as $\vE=\vE_1'\cup \vE_2'\cup \vE_3'\cup \vE_4'\cup \vE_5'$, with
\begin{align*}
&\vE_1'=\{(b_k,c_k), (c_k,b_k): i\not\sim b_k, b_k\not\in\cC_t\},\\
&\vE_2'=\{(b_k,c_k), (c_k,b_k):i\sim b_k, \dist_{\cGT}(i,b_k)\leq 2r, b_k\not\in \cC_t\},\\
&\vE_3'=\{(b_k,c_k), (c_k,b_k):i\sim b_k, \dist_{\cGT}(i,b_k)> 2r,  b_k\not\in \cC_t\},\\
&\vE_4'=\{(b_k,c_k), (c_k,b_k): b_k\in \cC_t,\dist_{\cGT}(b_k,j)\leq 2r\},\\
&\vE_5'=\{(b_k,c_k), (c_k,b_k): b_k\in \cC_t,\dist_{\cGT}(b_k,j)> 2r\}.
\end{align*}
For $(x,y)\in \vE_1'$, $|\vE_1'|\leq 2\nu\leq 4(d-1)^{\ell+1}$, and by \eqref{diffcellest}, $|\hGT_{yj}|\leq 2M/\sqrt{N\eta}$, since $y,j$ are in different $\rm S$-cells. Thus, combining with \eqref{GTbound2},
\begin{align*}
\sum_{\vE_1'}\sumdots\leq 4(d-1)^{\ell+1}\frac{2M^2}{N\eta}\frac{1}{\sqrt{d-1}}.
\end{align*}
For $(x,y)\in \vE_2'$, $y,j$ are in different $\rm S$-cells. We have: $|\vE_2'|\leq 2\n$ from \eqref{distfinitedegb}, i.e. $|\{k\in \qq{1,\nu}: \dist(i,b_k)\leq 2r\}|\leq \n$, and  $|\hGT_{yj}|\leq 2M/\sqrt{N\eta}$ by \eqref{diffcellest}. Thus, by \eqref{GTbound2},
\begin{align*}
\sum_{\vE_2'}\sumdots\leq \frac{8\n|\msc|M}{\sqrt{d-1}\sqrt{N\eta}}.
\end{align*}
For $(x,y)\in \vE_3'$, $y,j$ are in different $\rm S$-cells. We have: $|\vE_3'|\leq 2\n'$ from \eqref{finitedeg}, and  $|\hGT_{yj}|\leq 2M/\sqrt{N\eta}$ by \eqref{diffcellest}. Thus, by \eqref{GTbound2},
\begin{align*}
\sum_{\vE_3'}\sumdots\leq \frac{4K\n'q^{r+1}M}{\sqrt{N\eta}}.
\end{align*}
For $(x,y)\in \vE_4'$, $|\vE_4'|\leq 2\n$ from \eqref{distfinitedegb}, and  \eqref{GTbound2}--\eqref{hGTbound2},
\begin{align*}
\sum_{\vE_4'}\sumdots\leq \frac{4\n|\msc|M}{\sqrt{d-1}\sqrt{N\eta}}.
\end{align*}
For $(x,y)\in \vE_5'$, $|\vE_5'|\leq 2\n'$ from \eqref{deg}, and $\dist_{\hcGT}(y,j)\geq \dist_{\cGT}(y,j)\geq 2r$, since in the graph $\cGT$, $b_k$ and $c_k$ are adjacent. Thus, combining with \eqref{GTbound2}--\eqref{hGTbound2},
\begin{align*}
\sum_{\vE_5'}\sumdots\leq \frac{4K\n'q^{r+1}M}{\sqrt{N\eta}}.
\end{align*}
Therefore, \eqref{hGTdiffcell} can be bounded by
\begin{align*}
|\hGT_{ij}|\leq \frac{M}{\sqrt{N\eta}}+\left(\frac{12\n|\msc|M}{\sqrt{d-1}\sqrt{N\eta}}+\frac{8K\n'q^{r+1}M}{\sqrt{N\eta}}+\frac{8M^2(d-1)^{\ell+1/2}}{N\eta}\right)\leq  \frac{2M}{\sqrt{N\eta}},
\end{align*}
given that $\sqrt{d-1}\geq 20\n$, $\n'q^r\ll 1$ and $\sqrt{N\eta}\geq M(d-1)^{\ell+1}$.
\end{proof}

\subsection{Proof of \eqref{e:stabilitytGT}}

As in previous arguments, we follow the structure described below \eqref{e:stabilityGS}.

\paragraph{Localization}

The switching vertices that are not on the boundary of $\T$ \emph{after} switching are
given by $\{a_1,\dots, a_{\mu},b_1,\dots, b_\nu\}$.
From Section~\ref{sec:defcells}, we recall the partition $\{\bI_1, \bI_2,\dots, \bI_\kappa\}$ of this set.
(Thus the $\bI_j$ are the connected components of the Green's function graph $\cal R$, with all connected components containing any of the vertices $a_i$ joined to $\bI_1$.)
The close vertices $\bX_1 \cup \bX_2$ and the larger subgraph $\cG_0$ are defined by
\begin{equation}\label{e:XYdef}
\cG_0
:=\cB_{3r}(\{a_1,\dots, a_{\mu},b_1,\dots, b_\nu\}, \cGT),
\quad
\bX_1:= \bB_{2r}(\bI_1, \cGT),
\quad
\bX_2 := \bB_{2r}(\bI_2 \cup \cdots \cup \bI_\kappa, \cGT).
\end{equation}
By our construction of $\rm S$-cells and $\rm S'$-cells, it follows that $\bX_1=\cC_1=\cC_1'$ and $\bX_2=\cup_{i=2}^{\kappa}\cC_i=\cup_{i=2}^{\kappa'}\cC_i'$.
By our conventions, the deficit function of the graph $\cG_0$ is the restriction of that on $\cGT$.
We define the graph $\hcG_0$ by removing edges $\{b_i,c_i\}_{i\leq \nu}$
from $\cG_0$ and $\tcG_0$ by adding edges $\{a_i,b_i\}_{i\leq \nu}$ to $\hcG_0$.
The graphs $\hcG_0$ and $\tcG_0$ are given the restricted deficit functions from $\hcGT$ and $\tcGT$ respectively. We abbreviate
\begin{align*}
\cG_1=\TE(\cG_0),\quad \hcG_1=\TE(\hcG_0),\quad \hat{P}=G(\hcG_1),\quad
\tcG_1=\TE(\tcG_0), \quad \tilde{P}=G(\tcG_1).
\end{align*}
Notice that the graph $\hcG_1$ is obtained by removing the edges $\{b_i,c_i\}_{i\leq \nu}$ from $\cG_1$,
and that the graph $\tcG_1$ is obtained by adding the edges $\{b_i,a_i\}_{i\leq \nu}$ to $\hcG_1$.
We use the following fact throughout this section.

\begin{claim}\label{c:componentestimate}
If \eqref{distdega} and \eqref{distdegb} hold,
then each connected component of $\tcG_0$ contains at most $10\n$ elements in $\{a_1,\dots,a_\mu, b_1,\dots, b_\nu\}$, i.e.
\begin{align}\label{e:componentest}
|\{i\in\qq{1,\mu}: a_i\in \bK\}|+|\{i\in \qq{1,\nu}: b_i\in \bK\}|\leq 10\n,
\end{align}
where $\bK$ is the vertex set of any connected component of $\tcG_0$.
\end{claim}
\begin{proof}
\eqref{e:componentest} is a consequence of Propositions~\ref{prop:distdeg} and \ref{structuretGT}.
More precisely, if $a_i\in \bK$ or $b_i\in \bK$ for some $ i\in \qq{1,\mu}\setminus(\sBa \cup \sBb)$, then $\bK$ is disjoint from $\{a_1,a_2,\dots,a_\mu, b_1,b_2,\dots, b_\nu\}\setminus\{a_i,b_i\}$. Therefore $|\{i\in\qq{1,\mu}: a_i\in \bK\}|+|\{i\in \qq{1,\nu}: b_i\in \bK\}|\leq 2|\sBa\cup \sBb|\leq 10\n$.
\end{proof}

\paragraph{Verification of assumptions in Proposition \ref{boundPij}}

By assumption $\tcG=T_{\bf S}(\cG)\in \bar\Omega$. Since $\hcG_0$ and $\tcG_0$ can be viewed as subgraphs of $\cG$ and $\tcG$ respectively, the radius-$R$ neighborhoods of them have excess at most $\omega$. Moreover, the same argument as in Section \ref{proofofboundhGT} implies that the sum of the deficit functions on each connected component of $\hcG_0$ and that of $\tcG_0$ are bounded by $8\n$.  Therefore the assumptions for \eqref{e:compatibility} are verified for both graphs $\hcG_0$ and $\tcG_0$. Thus \eqref{e:boundPij}--\eqref{e:compatibility} hold for $\hat P$ and $\tilde P$, and as in \eqref{e:PP}, we can use $\hat P$ instead of $P(\cE_r(i,j,\hcGT))$ and $\tilde P$ instead of $P(\cE_r(i,j,\tcGT))$.

\paragraph{Starting point}
The proof is similar to that of \eqref{boundhGT}. By the resolvent identity \eqref{e:resolv}, we have
\begin{align}
\label{tGT}\tGT-\hGT&=\hGT\Delta\tGT,\\
\label{tPT}\tilde{P}-\hat P&=\hat P\Delta\tilde P.
\end{align}
where $\Delta=\sum_{k=1}^\nu (e_{b_ka_k}+e_{a_kb_k})/\sqrt{d-1}$. 
Taking difference of \eqref{tGT} and \eqref{tPT}, we have
\begin{equation}\label{perturb}
\tGT_{ij}-\tilde{P}_{ij}=(\hGT_{ij}-\hat P_{ij})
+\frac{1}{\sqrt{d-1}}\sum_{(x,y)\in \vE}(\hGT_{ix}-\hat P_{ix}) \tilde{P}_{yj}
+ \frac{1}{\sqrt{d-1}}\sum_{(x,y)\in \vE}\hGT_{ix}(\tGT_{yj}-\tilde{P}_{yj}),
\end{equation}
where the sums are over the ordered pairs
\begin{align}\label{e:defB2}
(x,y)\in \vE=\{(a_1,b_1),\dots,(a_\nu,b_\nu), (b_1,a_1),\dots, (b_\nu,a_\nu)\}.
\end{align}
We regard \eqref{perturb} as an equation for $\tGT-\tilde P$, and will show that
$\tGT-\tilde{P}$ is small, using that $\hGT-\hat P$ is small by \eqref{boundhGT}.

\paragraph{Green's function estimates}

We collect some estimates on Green's function, which are repeatedly used in the proof:
\begin{align}\label{tGPbound}
\begin{cases}
|\hGT_{ij}|, |\hat P_{ij}|, |\tilde P_{ij}|\leq 2|\msc|,&(\text{all $i,j$}),\\
|\hGT_{ij}|\leq 2K|m_{sc}|q^r, &(\dist_{\cGT}(i,j)\geq 2r),\\
|\hGT_{ib_k}|, |\hGT_{ic_k}|\leq 2M/\sqrt{N\eta}, & (\text{$i,b_k$ are in different $\rm S$-cells; or $i\not\sim b_k$}) 
\end{cases}
\end{align}
The first estimate follows from \eqref{e:boundPiimsc} and \eqref{boundhGT}; the second estimate follows from $P_{ij}(\cE_r(i,j,\hcGT))=0$ and \eqref{boundhGT}; the last estimate is from \eqref{e:diffcellest}.

\begin{proof}[Proof of \eqref{e:stabilitytGT} for $i,j\in \bX_1\cup \bX_2$]
For the second term on the right-hand side of \eqref{perturb}, it follows from \eqref{boundhGT} that $|\hGT_{ix}-\hat P_{ix}| \leq 2K |\msc| q^r+2^{2\n+3}|m_{sc}|q^{r+1}$.
Moreover, again $\tilde P_{yj} = 0$ if $y$ and $j$ are in different connected components of $\tcG_0$.
Thus, by Claim~\ref{c:componentestimate}, we have $\tilde P_{yj}\neq 0$ for at most $10\n$ vertices $y\in \{a_i: i\in \qq{1,\mu}\}\cup \{b_i: i\in \qq{1,\nu}\}$,
and for these we again have $|\tilde{P}_{yj}|\leq 2|\msc|$ by \eqref{tGPbound}.
Altogether, the second term on the right-hand of \eqref{perturb} is bounded by
\begin{align*}
 \frac{1}{\sqrt{d-1}}\sum_{(x,y)\in \vE}|\hGT_{ix}-\hat P_{ix}| |\tilde{P}_{yj}|
 \leq 20\n(2K+2^{2\n+3}q)|m_{sc}| q^{r+1}.
\end{align*}

To estimate the last term in \eqref{perturb}, we denote
\begin{align*}
 \Gamma_1:=\max_{i,j\in \bX_1}|\tilde{P}_{ij}-\tGT_{ij}|,\qquad
 \Gamma_2:=\max_{i\in \bX_2,j\in \bX_1\cup \bX_2}|\tilde{P}_{ij}-\tGT_{ij}|.
\end{align*}
Our goal is to prove that 
\begin{align}\label{Gammabound}
\Gamma_1,\Gamma_2\leq 24K^2|m_{sc}|q^{r}.
\end{align}
In the following, we first derive an estimate for $\Gamma_2$. We assume that 
$i\in \cC_t$ for some $t\neq 1$, and $j\in \bX_1\cup \bX_2$.
We decompose the set $\vE$ (as in \eqref{e:defB2}) according to their relations to the $\rm S$-cell $\cC_t$: $\vE=\vE_1\cup \vE_2\cup \vE_3$, where
\begin{align*}
\vE_1&=\{(x,y): x\not\in \cC_t\},\\
\vE_2&=\{(x,y): x\in \cC_t, \dist_{\cGT}(i,x)<2r\},\\
\vE_3&=\{(x,y): x\in \cC_t, \dist_{\cGT}(i,x)\geq 2r\}.
\end{align*}
Notice that for any $(x,y)\in \vE$, by the definition of $\Gamma_1, \Gamma_2$, we always have $|\tilde P_{yj}-\tGT_{yj}|\leq \max\{\Gamma_1,\Gamma_2\}$.
For $(x,y)\in \vE_1$, we have $|\vE_1|\leq 2\nu\leq 4(d-1)^{\ell+1}$.
Since $i,x$ are in different $\rm S$-cells, by \eqref{tGPbound}, 
\begin{align*}
\sum_{\vE_1}\sumdots\leq \frac{1}{\sqrt{d-1}}\sum_{(x,y)\in \vE_1}|\hGT_{ix}|\max\{\Gamma_1,\Gamma_2\}
\leq  \frac{4(d-1)^{\ell+1}}{\sqrt{d-1}}\frac{2M}{\sqrt{N\eta}}\max\{\Gamma_1,\Gamma_2\}.
\end{align*}
For $(x,y)\in \vE_2$, we have $|\vE_2|\leq 2\n$ by \eqref{distfinitedega} and  \eqref{distfinitedegb}.
Thus, by \eqref{tGPbound},
\begin{align*}
\sum_{\vE_2}\sumdots\leq \frac{1}{\sqrt{d-1}}\sum_{(x,y)\in \vE_2}|\hGT_{ix}|\max\{\Gamma_1,\Gamma_2\}
\leq  \frac{4\n|\msc|}{\sqrt{d-1}}\max\{\Gamma_1,\Gamma_2\}.
\end{align*}
For $(x,y)\in \vE_3$, we have $|\vE_3|\leq \n'$ from \eqref{deg}, and combined with \eqref{tGPbound},
\begin{align*}
\sum_{\vE_3}\sumdots\leq \frac{1}{\sqrt{d-1}}\sum_{(x,y)\in \vE_3}|\hGT_{ix}|\max\{\Gamma_1,\Gamma_2\}
\leq  \frac{2K\n'|m_{sc}|q^r}{\sqrt{d-1}}\max\{\Gamma_1,\Gamma_2\}.
\end{align*}
Combining the sums over $\vE_1, \vE_2, \vE_3$, for $i\in \bX_2$ and $j\in \bX_1\cup \bX_2$, \eqref{perturb} leads to
\begin{align*}
|\tGT_{ij}-\tilde P_{ij}|\leq (20\n q+1)(2K+2^{2\n+3}q)|m_{sc}| q^{r}+\frac{8(\n+1)}{\sqrt{d-1}}\max\{\Gamma_1,\Gamma_2\},
\end{align*}
given $\sqrt{N\eta}\geq M(d-1)^{\ell+1}$, $\sqrt{d-1}\geq 20\n$ and $\n'q^r\ll1$.
Moreover, taking the maximum over all $i\in \bX_2$ and $j\in \bX_1\cup \bX_2$, we get
\begin{align}\label{Gamma2bound8}
 \Gamma_2\leq (20\n q+1)(2K+2^{2\n+3}q)|m_{sc}| q^{r}+\frac{8(\n+1)}{\sqrt{d-1}}\max\{\Gamma_1,\Gamma_2\}.
\end{align}

Next, we estimate $\Gamma_1$.
To this end, we decompose the set $\vE$ (as in \eqref{e:defB2}) according to the cases in \eqref{tGPbound}
as $\vE=\vE_1'\cup \vE_2'\cup \vE_3'\cup \vE_4'$, where
\begin{align*}
\vE_1'&=\{(x,y): x\in \bX_1, \dist_{\cGT}(i,x)\leq 2r\},\\
\vE_2'&=\{(x,y): x,y\in \bX_1, \dist_{\cGT}(i,x)> 2r\},\\
\vE_3'&=\{(x,y): x\in \bX_1,y\in \bX_2, \dist_{\cGT}(i,x)> 2r\},\\
\vE_4'&=\{(x,y): x\in \bX_2\}.
\end{align*}
For $(x,y)\in \vE_1'$, $|\vE_1'|\leq 2\n$ from \eqref{distfinitedega} and  \eqref{distfinitedegb}. $|\tilde P_{yj}-\tGT_{yj}|\leq \max\{\Gamma_1,\Gamma_2\}$ by the definition of $\Gamma_1, \Gamma_2$. Therefore, by \eqref{tGPbound}, 
\begin{equation*}
\sum_{\vE_1'}\sumdots
\leq \frac{1}{\sqrt{d-1}}\sum_{(x,y)\in \vE_1'}|\hGT_{ix}|\max\{\Gamma_1,\Gamma_2\}\leq \frac{4\n|\msc|}{\sqrt{d-1}}\max\{\Gamma_1,\Gamma_2\}.
\end{equation*}
For $(x,y)\in \vE_2'$, $|\vE_2'|\leq 2\n'$ from \eqref{deg}, and  $|\tilde P_{yj}-\tGT_{yj}|\leq \Gamma_1$ from the definition of $\Gamma_1$. Thus, by \eqref{tGPbound},
\begin{equation*}
\sum_{\vE_2'}\sumdots
\leq \frac{1}{\sqrt{d-1}}\sum_{(x,y)\in \vE_2'}|\hGT_{ix}|\Gamma_1 \leq 4K\n'q^{r+1}\Gamma_1.
\end{equation*}
For $(x,y)\in \vE_3'$, we have $|\vE_3'|\leq 2\nu\leq 4(d-1)^{\ell+1}$,
and $|\tilde P_{yj}-\tGT_{yj}|\leq \Gamma_2$ from the definition of $\Gamma_2$. Thus, by \eqref{tGPbound},
\begin{equation*}
\sum_{\vE_3'} \sumdots
\leq \frac{1}{\sqrt{d-1}}\sum_{(x,y)\in \vE_3'}|\hGT_{ix}|\Gamma_2\leq 8K(d-1)^{\ell+1}q^{r+1}\Gamma_2.
\end{equation*}
For $(x,y)\in \vE_4'$,
we have $|\vE_4'|\leq \nu\leq 2(d-1)^{\ell+1}$, and $|\tilde P_{yj}-\tGT_{yj}|\leq \max\{\Gamma_1,\Gamma_2\}$ from the definition of $\Gamma_1,\Gamma_2$. Thus, by \eqref{tGPbound},
\begin{equation*}
\sum_{\vE_4'}\sumdots
\leq \frac{1}{\sqrt{d-1}}\sum_{(x,y)\in \vE_4'}|\hGT_{ix}|(\Gamma_1+\Gamma_2)\leq \frac{4(d-1)^{\ell+1}M}{\sqrt{d-1}\sqrt{N\eta}}\max\{\Gamma_1,\Gamma_2\}.
\end{equation*}
Using $r=2\ell+1$, and combining the above estimates in \eqref{perturb},
we obtain that, for all $i,j\in \bX_1$,
\begin{equation*}
|\tGT_{ij}-\tilde P_{ij}|
\leq (20\n q+1)(2K+2^{2\n+3}q)|m_{sc}| q^{r}+\frac{8(\n+1)}{\sqrt{d-1}}\Gamma_1+\left(8K+\frac{4\n+4}{\sqrt{d-1}}\right)\Gamma_2,
\end{equation*}
given $\sqrt{N\eta}\geq M(d-1)^{\ell+1}$ and $\n'q^r\ll1$.
Taking the maximum over the left-hand side, we have
\begin{equation}\label{Gamma1bound8}
 \Gamma_1\leq (20\n q+1)(2K+2^{2\n+3}q)|m_{sc}| q^{r}+\frac{8(\n+1)}{\sqrt{d-1}}\Gamma_1+\left(8K+\frac{4\n+4}{\sqrt{d-1}}\right)\Gamma_2.
\end{equation}

Finally,
the claim \eqref{Gammabound} follows by combining \eqref{Gamma2bound8} and \eqref{Gamma1bound8}, provided
that  $\sqrt{d-1}\geq \max\{(\n+1)^2 2^{2\n+10}, 2^8(\n+1)K\}$. Therefore for any $i,j\in \bX_1\cup \bX_2$, we have
\begin{align}\begin{split}\label{boundtGTij}
 |\tGT_{ij}-P_{ij}(\cE_{r}(i,j,\tGT))|&\leq   |(\tGT-\tilde P)_{ij}|+|\tilde P_{ij}-P_{ij}(\cE_{r}(i,j,\tGT))|\\
 &\leq 24K^2|\msc|q^r+2^{2\n+3}|\msc|q^{r+1}\leq (24K^2+1)|\msc|q^r,
\end{split}\end{align}
which implies the bound stated in \eqref{e:stabilitytGT}.
\end{proof}

\begin{proof}[Proof of \eqref{e:stabilitytGT} for the remaining case]
For $i\notin \bX_1\cup \bX_2$ and $j\in \bX_1\cup \bX_2$,
first note that $\cE_r(i,j,\hGT)=\cE_{r}(i,j,\tGT)$ and that both graphs have the same deficit function.
To prove \eqref{e:stabilitytGT}, we will show that $|\tGT_{ij}-\hGT_{ij}|$ is small.
To this end, we start from the resolvent identity \eqref{tGT}, which states that
\begin{equation}  \label{tGT-bis}
|\tGT_{ij}-\hGT_{ij}|\leq \frac{1}{\sqrt{d-1}}\sum_{(x,y)\in \vE}|\hGT_{ix}||\tGT_{yj}|.
\end{equation}
By the definition of the sets $\bX_1,\bX_2$, for any $(x,y)\in \vE$,
we have $\dist_{\cGT}(i,x)>2r$ and $|\hGT_{ix}|\leq 2K|\msc|q^r$ by \eqref{tGPbound}.
We simply decompose the set $\vE$ (as in \eqref{e:defB2}) according to their distance to the vertex $j$ as $\vE=\vE_1\cup \vE_2$, where 
\begin{align*}
\vE_1&=\{(x,y): \dist_{\tcGT}(y,j)<2r\},\\
\vE_2&=\{(x,y): \dist_{\tcGT}(y,j)\geq 2r\}.
\end{align*}
For $(x,y)\in \vE_1$, we have $|\vE_1|\leq 10\n$ by \eqref{e:componentest} in Claim~\ref{c:componentestimate}.
Moreover, $|\tGT_{yj}|\leq |P_{yj}(\cE_{r}(y,j,\tGT))|+(24K^2+1)|\msc|q^r\leq 2|m_{sc}|$ by \eqref{boundtGTij}.
Thus, combining with \eqref{tGPbound}, we have 
\begin{equation*}
\sum_{\vE_1}\sumdots\leq 40K\n |m_{sc}|q^{r+1},
\end{equation*}
where here $\sumdots$ denotes the terms in the sum in \eqref{tGT-bis}.
For $(x,y)\in \vE_2$, we have $|\vE_2|\leq 2\nu\leq 2(d-1)^{\ell+1}$,
and $|\tGT_{yj}|\leq (24K^2+1)|\msc|q^r$ by \eqref{boundtGTij}, since $P_{yj}(\cE_{r}(y,j,\tGT))=0$.
Thus, combining with \eqref{tGPbound},
\begin{equation*}
\sum_{\vE_2}\sumdots\leq 2K(24K^2+1)|\msc|^2q^{2r}\frac{2(d-1)^{\ell+1}}{\sqrt{d-1}}.
\end{equation*}
Combining the sums over $\vE_1,\vE_2$, we get 
\begin{equation*}
 |\tGT_{ij}-\hGT_{ij}|
 \leq 40K\n q^{r+1}+2K(24K^2+1)|\msc|^2q^{2r}\frac{2(d-1)^{\ell+1}}{\sqrt{d-1}}\leq 100K^3|m_{sc}|q^r,
\end{equation*}
provided that $\sqrt{d-1}\geq 20\n$.
Similarly, in the case $i,j\notin \bX_1\cup \bX_2$, we have
\begin{equation*}
  |\tGT_{ij}-\hGT_{ij}|
  \leq \frac{1}{\sqrt{d-1}}\sum_{(x,y)\in \vE}|\hGT_{ix}||\tGT_{yj}|
  \leq 2K(24K^2+1)|\msc|^2q^{2r}\frac{2(d-1)^{\ell+1}}{\sqrt{d-1}}\leq 100K^3 |m_{sc}|q^r.
\end{equation*}
Therefore, for $i\notin \bX_1\cup \bX_2$ and $j\in \bX_1\cup \bX_2$ or $i,j\notin \bX_1\cup \bX_2$, we obtain
\begin{align*}
|\tGT_{ij}-P_{ij}(\cE_{r}(i,j,\tGT))|
&\leq   |\tGT_{ij}-\hGT_{ij}|+|\hGT_{ij}-P_{ij}(\cE_{r}(i,j,\hGT))|\\
&\leq100K^3 |m_{sc}|q^r+2K|\msc|q^r\leq 2^7K^3|\msc|q^r.
\end{align*}
This completes the proof of \eqref{e:stabilitytGT}.
\end{proof}

\section{Improved decay in the switched graph}
\label{sec:boundarydecay}

In the graph $\tcG=T_{\bf S}(\cG)$, the edge boundary $\del_E \cT$ and the vertex boundary $\del \cT$ of $\cT$ are given by
\begin{equation}\label{defI}
  \del_E\cT=\{(l_1, \ta_1), (l_2,\ta_2),\dots, (l_\mu, \ta_\mu)\},\quad
  \bI\deq\del\cT=\{\ta_1,\ta_2, \dots, \ta_\mu\},
\end{equation}
where the vertices $\ta_i=c_i$ with $i\in\qq{1,\nu}$ are those that get switched,
and the vertices $\ta_i=a_i$ with $i\in \qq{\nu+1, \mu}$ are those for which the switching does not take place. 
Here recall from Remark~\ref{defrmk} that we assume without loss of generality that the index set of admissible switchings is $W_{\bf S}=\qq{1,\nu} \subset \qq{1,\mu}$.

The result of this section is the following proposition,
showing that (\rn{1}) between most vertices in $\bI$ the Green's function is small;
(\rn{2}) for any vertex not in $\bI$, the Green's function between it and most vertices in $\bI$ is also small.
This decay asserted by the proposition is better than that between the boundary vertices of $\cT$ which we assumed in the unswitched graph.
This improvement is crucial for the subsequent sections, in particular for the derivation of
the self-consistent equation.

\begin{proposition} \label{greendist-new}
Under the same assumptions as in Proposition \ref{prop:stabilitytGT}, let ${\bf S}\in F(\cG)$ (as in Section \ref{cdFG}) and assume that $\tcG=T_{\bf S}(\cG)\in \bar \Omega$ (as in Section \ref{sec:outline-structure}).
Then there exists $J \subset \qq{1,\nu}$ with $|J|\geq \nu-\n'-6\n$ such that,
for any $k\in J$, 
\begin{alignat}{2}
\label{e:greendistIJ}
|\tGT_{ic_k}| &\leq 2^{9}K^4|\msc|q^{2r+1} &&\quad\text{if $i=\ta_j$ for some $j\in \qq{1,\mu}\setminus J$},
\\
\label{e:greendistJJ}
|\tGT_{ic_k}| &\leq 2^{12}K^5|\msc|q^{3r+2} &&\quad \text{if $i=\ta_j$ for some $j\in J\setminus \{k\}$},
\\
\label{e:greendistNJ}
|\tGT_{ic_k}| &\leq 2^{12}K^5|\msc|q^{2r+1} &&\quad  \text{if $i\not\sim b_k$ and $\dist_{\tcGT}(i,a_k)\geq 2r$},
\end{alignat}
provided that $\sqrt{d-1}\geq \max\{(\n+1)^2 2^{2\n+10}, 2^8(\n+1)K\}$, $\n'q^r\ll1$ and $\sqrt{N\eta}q^{3r+2}\geq M$.
\end{proposition}

The proposition uses the randomness of the resampling
via the properties of the Green's function that are encoded by the $\rm S'$-cells.
Indeed, recall that if $c_k$ was a random index, independent of $\tcGT$ and $i$, then the
size of the right-hand sides would be of order $1/\sqrt{N\eta} \ll |\msc| q^{3r+2}$
by the Ward identity \eqref{e:Ward}.
The remainder of this section is devoted to the proof of the proposition.

\subsection{Preliminaries}\label{sec:defJ}

To prove Proposition~\ref{greendist-new},
we use the same setup as in the proof of \eqref{e:stabilitytGT}.
Thus, from \eqref{e:XYdef} and the paragraph below, recall
the sets $\bX_1,\bX_2$ and the graphs $\cG_0,\hcG_0,\tcG_0$,
and that the set $\bX_1\cup \bX_2$ is contained in $\bG_0$ (the vertex set of $\cG_0$).
We also recall the $\rm S'$-cells defined in Section~\ref{sec:defcells}.

We will prove Proposition~\ref{greendist-new}
with the set $J\subset \qq{1,\nu}$ given by the set 
of indices $k \in  \qq{1,\nu}$ such that the following conditions hold:
\begin{enumerate}
\item
$b_k, c_k\in \bX_2$ (i.e.\ the $S$-cell containing $b_k$ and $c_k$ is not $S_1$); 
\item
$\cB_R(c_k, \cGT)$ is a tree;
\item
the $\rm S'$-cell $\cC'$ containing $b_k$  and $c_k$ is not $\cC'_1$ {(as implied by (\rn 1))} and satisfies
\begin{equation}\label{e:defJ}
\dist_{\tcGT}(\cC',\{a_m: m\in\qq{1,\mu}\setminus \{k\}\}\cup\{b_m, c_m: m\in\qq{1,\nu}\setminus\{k\}\})>R/4.
\end{equation}
\end{enumerate}
By the assumption ${\bf S}\in F(\cG)$, and using the definition of $F(\cG)$ given in Section~\ref{cdFG},
note that \eqref{treeneighbor} and \eqref{e:Spcellbd} hold.
\eqref{treeneighbor} implies that condition (\rn 2) in the definition of $J$ is true for all $k\in \qq{1,\nu}$ with at most $\n$ exceptions.
\eqref{e:Spcellbd} implies condition (\rn 1),
and further that condition (\rn 3) is true for all $k\in \qq{1,\nu}$ with at most $\n'+5\n$ exceptions.
It follows that
\begin{equation*}
|J|\geq \nu-\n'-6\n,
\end{equation*}
as asserted in the statement of Proposition~\ref{greendist-new}.
With this definition of $J$,
to prove Proposition~\ref{greendist-new},
we now follow the structure described below \eqref{e:stabilityGS}
(without the localization step, which is not required here).

\paragraph{Starting point}
For the remainder of this section,
we fix $k \in J$ and denote the $\rm S'$-cell containing $c_k$ by $\cC'$.
Notice that, by the definition of $J$, the $\rm S'$-cell $\cC'$ is not $\cC_1'$, and that it is equal to the $\rm S$-cell containing $c_k$.
For any $i$ arising in the statement of Proposition~\ref{greendist-new},
we either have $i\in \bI$, in which case $i$ and $c_k$ are in different $\rm S$-cells
 (by definition of $J$, the $\rm S$-cell of $c_k$ does not contain any $\ta_j$), or 
otherwise $i\not \sim  b_k$. Noticing that $b_k$ and $c_k$ are in the same $\rm S$-cell,
in both cases, the estimate \eqref{e:diffcellest} with $j=c_k$ holds.
Therefore, since the graph $\tcGT$ is given by adding the edges $\{a_i,b_i\}_{i\leq \nu}$ to $\hcGT$,
by the resolvent formula \eqref{e:resolv},  we have
\begin{align}\label{tGijbound}
|\tGT_{ic_k}|
=
\absa{\hGT_{ic_k}+
\frac{1}{\sqrt{d-1}}\sum_{(x,y) \in \vE}\hGT_{ix}\tGT_{yc_k} }
\leq \frac{2M}{\sqrt{N\eta}}+
\frac{1}{\sqrt{d-1}}\sum_{(x,y) \in \vE}|\hGT_{ix}\tGT_{yc_k}|,
\end{align}
where the summation is over the ordered pairs
\begin{equation}
(x,y) \in \vE = \{(a_1,b_1),\dots, (a_\nu,b_\nu), (b_1,a_1),\dots, (b_\nu, a_\nu)\}.
\end{equation}
By our assumption on $\eta$, the first term on the right-hand side of \eqref{tGijbound} is
smaller than the right-hand sides of \eqref{e:greendistIJ}--\eqref{e:greendistNJ},
so we only need to estimate the sum on the right-hand side of \eqref{tGijbound}.

\paragraph{Green's function estimates}

To estimate the sum on the right-hand side of \eqref{tGijbound},
we use the following estimates on Green's functions, which hold for $(x,y) \in \vE$:
\begin{align}
  \label{e:hGTixbd}
  |\hGT_{ix}|
  &\leq
  \begin{cases}
    2|\msc| & (\text{all $x$}),\\
    2K|\msc|q^r & (\dist_{\hcGT}(i,x) \geq 2r),\\
    2M/\sqrt{N\eta} & (\text{$i$ and $x$ are in different $\rm S$-cells, or $i\not\sim$ the $\rm S$-cell containing $x$}),
  \end{cases}
  \\
  \label{e:tGTyckbd}
  |\tGT_{yc_k}|
  &\leq
  \begin{cases}
    2|\msc| & (\text{all $y$}),\\
    2^{7}K^3|\msc|q^r & (\dist_{\tcGT}(y, c_k) \geq 2r).
  \end{cases}
\end{align}
The last bound in \eqref{e:hGTixbd} holds by \eqref{e:diffcellest}.
The remaining estimates follow from Propositions~\ref{prop:stabilitytGT},
together with
\eqref{e:boundPiimsc} for the bound for all $x,y$;
with $P_{ix}(\cE_r(i,x,\hcGT))=0$ for for the bound for $\dist_{\hcGT}(i,x) \geq 2r$;
and with $P_{yc_k}(\cE_r(y,c_k,\tcGT))=0$ for the bound for $\dist_{\tcGT}(y,c_k) \geq 2r$.

\paragraph{Distance estimates}

Since the estimates \eqref{e:hGTixbd}--\eqref{e:tGTyckbd} depend on distances,
we need some estimates on distances in the graphs $\hcGT$ and $\tcGT$.
These are summarized in the following lemma.

\begin{lemma} \label{lem:tcGTcbdist}
Let $k \in J$ and $\cC'$ be the $\rm S'$-cell that contains $c_k$. Then the following estimates hold.
\begin{enumerate}
\item In the graph $\tcGT$, the vertex $c_k$ is far away from $\{a_1,\dots, a_\mu, b_1, \dots, b_\nu\}$:
\begin{equation} \label{e:yckdist}
\dist_{\tcGT}(c_k,\{a_1,\dots, a_\mu, b_1,\dots, b_\nu\}) > 2r.
\end{equation}
\item If  $\dist_{\tcGT}(i,\cC')> 2r$, then
\begin{equation} \label{e:iakdist}
\dist_{\hcGT}(i,a_k) \geq  \dist_{\tcGT}(i,a_k) \geq 2r.
\end{equation}
\item If $i \in \bX_1$ and $\dist_{\tcGT}(i,a_k) \geq 2r$, then
\begin{equation} \label{e:icCdist}
\dist_{\tcGT}(i,\cC')> 2r.
\end{equation}
\end{enumerate}
\end{lemma}

Notice also that, by the definition of $J$, we have $\{m\in \qq{1,\nu}: b_m\in \cC'\}=\{k\}$.

\begin{proof}
To prove (i), it follows from \eqref{e:defJ} from the definition of $J$ that 
\begin{align*}
\dist_{\tcGT}(c_k,\{a_m: m\in\qq{1,\mu}\setminus\{k\}\cup \{b_m: m\in \qq{1,\nu}\setminus\{k\}\}\}>R/4> 2r.
\end{align*}
It remains to prove $\dist_{\tcGT}(c_k,\{a_k,b_k\}\}> 2r$. Given any geodesic in $\tcGT$ from $c_k$ to $\{a_k,b_k\}$, we distinguish two cases.
In the first case that the geodesic contains any of the edges $\{a_m,b_m\}_{m\leq \nu}$,
the condition \eqref{e:defJ}
which holds by the definition of $J$, implies that its length is larger than $2r$.
In the second case that the geodesic contains none of the edges $\{a_m,b_m\}_{m\leq \nu}$,
it a path on the graph $\hcGT$.

Therefore, to prove (i), it suffices to show that \eqref{e:yckdist} holds with the graph $\tcGT$ replaced by $\hcGT$.
By the condition $b_k, c_k\in \bX_2$, and since $b_k, c_k$ are adjacent in $\cGT$, it follows from Lemma~\ref{l:distdiffcC}
that $\dist_{\cGT}(b_k,a_k)>8r$, and therefore that
\begin{equation*} 
  \dist_{\hcGT}(c_k,a_k)\geq \dist_{\cGT}(c_k,a_k)> 8r >2r.
\end{equation*}
Moreover, since  $c_k$ has radius-$R$ tree neighborhood in $\cGT$, and since in $\hcGT$ the edge $\{b_k,c_k\}$ is removed compared to $\cGT$, we have
\begin{equation*} 
  \dist_{\hcGT}(b_k,c_k)>R>2r.
\end{equation*}
This completes the proof of \eqref{e:yckdist} with $\tcGT$ replaced by $\hcGT$, and thus the proof of (i).

For (ii), since $a_k$ and $b_k\in \cC'$ are adjacent in the graph $\tcGT$, we have
\begin{align*}
\dist_{\tcGT}(i,a_k) \geq \dist_{\tcGT}(i, \cC')-1\geq 2r.
\end{align*}
The first inequality in \eqref{e:iakdist} is trivial since $\hcGT \subset \tcGT$.

To prove (iii), 
note that any geodesic from $i$ to $\cC'$ in $\tcGT$ either
contains $a_k$, or does not contain the edge $\{a_k,b_k\}$.
In the first case that the geodesic contains $a_k$, its length is at least $1+\dist_{\tcGT}(i,a_k)>2r$, as desired.
In the second case,
\begin{equation*}
\dist_{\tcGT}(i,\cC')
\geq
\dist_{\tcGT\setminus\{a_k,b_k\}}( \bX_1\cup \bX_2\setminus \cC',\cC')=\dist_{\cGT}(\bX_1\cup \bX_2\setminus \cC',\cC') > 4r
,
\end{equation*}
where the first inequality holds since $i \in \bX_1 \cup \bX_2 \setminus \cC'$,
and the last inequality follows from the definition of the $\rm S$-cells and Lemma~\ref{l:distdiffcC}.  Recall that the graph $\tcGT$ is obtained from $\cGT$ by adding the edges $\{a_m,b_m\}_{m\leq \nu}$ and removing the edges $\{b_m,c_m\}_{m\leq \nu}$. And by the definition of the set $J$, we know $\{b_k,c_k\}\subset \cC'$ and $\{b_m, c_m: m\in \qq{1,\nu}\setminus \{k\}\}\subset \bX_1\cup \bX_2\setminus\cC'$. Therefore, the graph $\tcGT\setminus\{a_k,b_k\}$ and $\cGT$ are different only on the subgraphs induced on $\cC'$ and $\bX_1\cup \bX_2\setminus \cC'$, and the equality in the above equation holds.
\end{proof}

\begin{remark}\label{r:RWrep}
Recall the random walk representation of the Green's function from Section~\ref{sec:Gexpan}.
In terms of the random walk heuristic,
together with the a priori estimates \eqref{e:hGTixbd}--\eqref{e:tGTyckbd} on the Green's function,
one can understand the bounds of Proposition~\ref{greendist-new} as follows.
For the left-hand side of \eqref{e:greendistIJ}, the walk with most weight is $i\rightarrow a_k\rightarrow b_k\rightarrow c_k$.
Since $\dist_{\tcGT}(i,a_k) \geq 2r$ and $\dist_{\tcGT}(b_k,c_k)\geq 2r$,
the walks $i\rightarrow a_k$ and $b_k\rightarrow c_k$ each contribute at least a small factor $q^r$;
the walk $a_k\rightarrow b_k$ has at least one step and thus contributes at least a factor $q$.
Therefore $|\tGT_{ic_k}|\lesssim q^r \times q\times q^r=q^{2r+1}$.
For the left-hand side of \eqref{e:greendistJJ}, the walk with most weight is $i=c_j\rightarrow b_j \rightarrow a_j\rightarrow a_k\rightarrow b_k\rightarrow c_k$,
and it therefore follows that $|\tGT_{ic_k}|\lesssim q^r\times q \times q^r\times q \times q^{r}=q^{3r+2}$.
For the left-hand side of \eqref{e:greendistNJ}, the walk with most weight is $i\rightarrow a_k\rightarrow b_k\rightarrow c_k$, and thus $|\tGT_{ic_k}|\lesssim q^{2r+1}$.  
\end{remark}

The proof of Proposition~\ref{greendist-new} essentially follows from
the heuristic described in Remark~\ref{r:RWrep}, which can be made rigorous by 
combining the estimates
on the Green's function of \eqref{e:hGTixbd}--\eqref{e:tGTyckbd}
with those on the distances stated in
Lemma~\ref{lem:tcGTcbdist}. This requires a division into a number of cases and is done carefully below.

\subsection{Proof of \eqref{e:greendistIJ}}

Let
\begin{align}
 \label{Xcell} 
 \Gamma_1&:=\max \ha{ |\tGT_{ic_k}| : i\in \bX_1 \text { such that $\dist_{\tcGT}(i,  \cC')>2r$}},
 \\
 \label{Ycell}
 \Gamma_2 &:=\max \ha{ |\tGT_{ic_k}| : i \in  \bX_2 \text{ and } i \not\in \cC'}.
\end{align}
Thus $\Gamma_1$ is the maximal size of the Green's function between $c_k$ and vertices in $\bX_1$ which is away from $\cC'$,
and $\Gamma_2$ is the maximal size of the Green's function between $c_k$ and vertices in $\bX_2$ which is in different $\rm S'$-cells from $c_k$.

\begin{proposition} \label{lem:Gamma}
\begin{align}\label{Gamma}
\max\{\Gamma_1,\Gamma_2\}
\leq 2^9K^4|\msc|q^{2r+1},
\end{align}
provided that $\sqrt{d-1}\geq \max\{(\n+1)^2 2^{2\n+10}, 2^8(\n+1)K\}$, $\n'q^r\ll1$ and $\sqrt{N\eta}q^{2r+2}\geq M$.
\end{proposition}

Given Proposition~\ref{lem:Gamma}, the claim \eqref{e:greendistIJ} is an immediate consequence.

\begin{proof}[Proof of \eqref{e:greendistIJ}]
It suffices to show that the left-hand side of \eqref{e:greendistIJ} is bounded by $\max\{\Gamma_1,\Gamma_2\}$.
First, if $i \in\bX_2$, then $i=c_l$ for some $l \neq k$, and by the definition of $J$, then $c_l \not\in \cC'$. 
Thus the left-hand sides of \eqref{e:greendistIJ} is bounded by $\Gamma_2$.
Second, if $i \in \bX_1$, then either $i=a_l$ or $i=c_l$ for some $l \neq k$.
In either case, by the definition of $J$, $\dist_{\tcGT}(i,\cC')>R/4-2r>2r$. Thus the left-hand side of \eqref{e:greendistIJ} is bounded by $\Gamma_1$.  
\end{proof}

\begin{proof}[Proof of Proposition~\ref{lem:Gamma}]
We first derive a bound for $\Gamma_1$. 
Let $i$ obey the conditions in the definition of $\Gamma_1$ in \eqref{Xcell}.
We divide the sum over the set $\vE$ in \eqref{tGijbound} according to the cases in \eqref{e:hGTixbd}--\eqref{e:tGTyckbd} as $\vE=\vE_1 \cup \cdots \cup \vE_5$,
where
\begin{align*}
  \vE_1 &= \{(a_k,b_k) \},
  \\
  \vE_2 &= \{(b_k,a_k) \},
  \\
  \vE_3 &= \{(b_l,a_l): l \neq k, b_l \in \bX_2 \},
  \\
  \vE_4 &= \{(a_l,b_l): l \neq k, b_l \in \bX_2\},
  \\
  \vE_5 &= \{(a_l,b_l), (b_l,a_l): l \neq k,  b_l \in \bX_1 \}.
\end{align*}
For $(a_k,b_k) \in \vE_1$,
we have $\dist_{\hcGT}(i,a_k)\geq 2r$ by \eqref{e:iakdist} 
and $\dist_{\tcGT}(b_k,c_k)\geq 2r$ by \eqref{e:yckdist}.
Thus, by \eqref{e:hGTixbd}--\eqref{e:tGTyckbd},
\begin{align*}
\sum_{\vE_1} \sumdots \leq
\frac{(2K|\msc|q^r) (2^{7}K^3|\msc|q^r)}{\sqrt{d-1}}.
\end{align*}
For $(b_k,a_k) \in \vE_2$,
we have $b_k \in \bX_2$ and $i \in \bX_1$, which implies $i$ and $b_k$ are in different $\rm S$-cells.
Thus, by \eqref{e:hGTixbd}--\eqref{e:tGTyckbd},
\begin{align*}
\sum_{\vE_2} \sumdots \leq
\frac{2M}{\sqrt{N\eta}}
\frac{2|\msc|}{\sqrt{d-1}}
\leq
\frac{4qM}{\sqrt{N\eta}}.
\end{align*} 
For $(b_l,a_l) \in \vE_3$,
we again have that $i$ and $b_l$ are in different $\rm S$-cells (since $b_l \in \bX_2$)
and $\dist_{\tcGT}(a_l,c_k)\geq 2r$ by \eqref{e:yckdist}.
Thus, by \eqref{e:hGTixbd}--\eqref{e:tGTyckbd} and $|\vE_3| \leq \mu\leq 2(d-1)^{\ell+1}$,
\begin{align*}
\sum_{\vE_3} \sumdots
\leq
\frac{M}{\sqrt{N\eta}} 2^{9}K^3 (d-1)^{\ell+1} q^{r+1}.
\end{align*} 
For $(a_l,b_l) \in \vE_4$,
there are at most $\n+1$ indices $l$ such that $\dist_{\cGT}(i,a_l)\leq \dist_{\hcGT}(i,a_l)<2r$
by \eqref{distfinitedega},
and at most $|\vE_4|\leq\mu\leq 2(d-1)^{\ell+1}$ indices such that $\dist_{\hcGT}(i,a_l)\geq 2r$.
Moreover, we have $b_l \in \bX_2$ and also $b_l \not\in \cC'$ by the definition of $J$. 
Thus, by \eqref{e:hGTixbd} and the definition of $\Gamma_2$,
\begin{equation*}
\sum_{\vE_4} \sumdots
\leq
\left((\n+1)\frac{2|\msc|}{\sqrt{d-1}}+2(d-1)^{\ell+1}\frac{2K|\msc|q^r}{\sqrt{d-1}}\right)\Gamma_2
.
\end{equation*}
For $(x,y) \in \vE_5$,
there are at most $10\n$ pairs $(x,y) \in \vE_5$ such that $\dist_{\tcGT}(i,x)<2r$ by \eqref{e:lessshortdist} in Proposition~\ref{structuretGT},
at most $2\n'$ pairs such that $\dist_{\tcGT}(i,x)\geq 2r$ since $|\bX_1 \cap \{b_1,\dots, b_\nu\}|\leq \n'$ by \eqref{deg}.
Thus, by \eqref{e:hGTixbd} and $|\tGT_{yc_k}|\leq\Gamma_1$
\begin{equation*}
\sum_{\vE_5} \sumdots \leq
\left(10\n \frac{2|\msc|}{\sqrt{d-1}}+2\n'\frac{2K|\msc|q^r}{\sqrt{d-1}}\right)\Gamma_1
.
\end{equation*}
Combining the sums over $\vE_1, \dots \vE_5$, and taking the maximum over $i$ obeying the conditions in the definition of $\Gamma_1$ in \eqref{Xcell}, we get
\begin{equation}\label{Gamma1bound}
\Gamma_1\leq \left(2+4q+2^9K^3\right)\frac{M}{\sqrt{N\eta}}+ 2^8K^4|\msc|q^{2r+1}+(20\n q+4K\n'q^{r+1})\Gamma_1+(2(\n+1) q+4K)\Gamma_2.
\end{equation}

To bound $\Gamma_2$, let $i\in \bX_2$ be as in the definition of $\Gamma_2$.
Let $\cC''$ be the $\rm S'$-cell containing $i$,
and notice that $\cC'\neq \cC'', \cC_1'$ from the definition of $\Gamma_2$. We now divide
$\vE = \vE_1' \cup \cdots \cup \vE_4'$ where
\begin{align}
  \vE_1' &= \{(x,y): x \in X_1\},\\
  \vE_2' &= \{(b_l,a_l): b_l \in  \cC'\} = \{(b_k,a_k)\},\\
  \vE_3' &= \{(b_l,a_l): b_l \in \cC''\},\\
  \vE_4' &= \{(b_l,a_l): b_l \in \bX_2 \setminus (\cC'\cup\cC'')\}.
\end{align}
For $(x,y) \in \vE_1'$, $i$ and $x$ are in different $\rm S$-cells (since $x \in \bX_1$ and $i \in \bX_2$)
and $\dist_{\tcGT}(y,c_k) > 2r$ by \eqref{e:yckdist}.
Since $|\vE_1'|\leq 2\mu\leq 4(d-1)^{\ell+1}$
\begin{equation*}
\sum_{\vE_1'} \sumdots \leq
4(d-1)^{\ell+1}\frac{2M}{\sqrt{N\eta}}\frac{2^{7}K^3|\msc|q^r}{\sqrt{d-1}}.
\end{equation*}
For $(b_k,a_k) \in \vE_2'$,  
$i$ and $b_k$ are in different $\rm S$-cells since $i\in \cC''$ and $b_k\in\cC'$ by assumption.
Moreover, we have $\dist_{\tcGT}(a_k,c_k)> 2r$ by \eqref{e:yckdist}.
Thus
\begin{equation*}
\sum_{\vE_2'} \sumdots
\leq
\frac{2M}{\sqrt{N\eta}}\frac{2^{7}K^3|\msc|q^r}{\sqrt{d-1}}.
\end{equation*}
For $(b_l,a_l) \in \vE_3'$, 
there are at most $5\n$ indices $l$ such that $\dist_{\tcGT}(i,b_l)<2r$ by \eqref{e:lessshortdist} in Proposition~\ref{structuretGT},
and at most $|\vE_3'|\leq |\{l\in\qq{1,\nu}: b_l\in \cC'' \}|\leq  \n'+5\n$ indices such that $\dist_{\tcGT}(i,b_l)\geq 2r$ by \eqref{e:Spcellbd}.
Moreover,
$|\tGT_{a_lc_k}|\leq \Gamma_1$ (since $\dist_{\tcGT}(a_l, \cC')>R/4-2r>2r$ by the definition of $J$).
Thus
\begin{equation*}
\sum_{\vE_3'} \sumdots
\leq
\left(5\n\frac{2|\msc|}{\sqrt{d-1}}+(\n'+5\n)\frac{2K|\msc|q^{r}}{\sqrt{d-1}}\right)\Gamma_1.
\end{equation*}
For $(b_l,a_l) \in \vE_4'$,
$i$ and $b_l$ are in different $\rm S$-cells; $a_l$ and $c_k$ are in different $\rm S$-cells (since $a_l\in \cC_1'$ and $c_k\in \cC'$); there are at most $|\vE'_4|\leq \mu\leq 2(d-1)^{\ell+1}$ terms.
Thus
\begin{equation*}
\sum_{\vE_4'} \sumdots
\leq
2(d-1)^{\ell+1}\frac{2M}{\sqrt{N\eta}}\frac{\Gamma_1}{\sqrt{d-1}}.
\end{equation*}
Combining the sums over $\vE_1', \dots, \vE_4'$, and taking the maximum over $i$ obeying the conditions in the definition of $\Gamma_2$, we get
\begin{equation}\label{Gamma2bound}
\Gamma_2\leq \left(2+2^{10}K^3+2^8K^3q^{r+1}\right)\frac{M}{\sqrt{N\eta}}+ \left(10\n q+2K(\n'+5\n)q^{r+1}+\frac{4(d-1)^{\ell+1/2}M}{\sqrt{N\eta}}\right)\Gamma_1.
\end{equation}

In summary, in \eqref{Gamma1bound} and \eqref{Gamma2bound}, we have shown that
\begin{equation*}
  \Gamma_1 \leq \fa  +\fb\Gamma_1 + \fc \Gamma_2,
  \qquad
  \Gamma_2 \leq \fd  + \fe\Gamma_1,
\end{equation*}
where $\fa,\fb,\fc,\fd,\fe$ are explicit constants given in \eqref{Gamma1bound} and \eqref{Gamma2bound}.
By plugging the second estimate into the first one, noticing $\fb+\fc\fe<1$, and using the explicit values of $\fa,\fb,\fc,\fd,\fe$, it follows that
\begin{equation*}
\Gamma_1 \leq (\fa+\fc\fd)/(1-(\fb+\fc\fe))\leq 2^9K^4|\msc|q^{2r+1},
\qquad
\Gamma_2 \leq \fd+\fe\Gamma_1\leq 2^9K^4|\msc|q^{2r+1},
\end{equation*}
provided that $\sqrt{d-1}\geq \max\{(\n+1)^2 2^{2\n+10}, 2^8(\n+1)K\}$, $\n'q^r\ll1$ and $\sqrt{N\eta}q^{2r+2}\geq M$.
\end{proof}

\subsection{Proofs of \eqref{e:greendistNJ} and \eqref{e:greendistJJ}}

\begin{proof}[Proof of \eqref{e:greendistNJ}]
To bound the left-hand side of \eqref{e:greendistNJ},
consider first the case that 
$i \in \bX_1 \cup \bX_2$: (\rn 1) if $i\in \bX_1$ and $\dist_{\tcGT}(i, a_k)\geq 2r$, it follows by \eqref{e:icCdist} that $\dist_{\tcGT}(i, \cC')> 2r$. Thus the left-hand side of \eqref{e:greendistNJ} is bounded by $\Gamma_1$; (\rn 2) if $i\in \bX_2$ and $i\not\sim b_k$, then $i\not \in \cC'$, and the left-hand side of \eqref{e:greendistNJ} is bounded by $\Gamma_2$.
Therefore \eqref{e:greendistNJ} follows from Proposition~\ref{lem:Gamma}.

For the remaining case $i \not\in \bX_1 \cup \bX_2$ and $i \not\sim b_k$,
we bound the sum over $\vE$ in \eqref{tGijbound}. 
By the definition of $\bX_1$ and $\bX_2$, $\dist_{\cGT}(i,\{a_1,\dots,a_\mu,b_1,\dots, b_\nu\})>2r$,
and therefore also in $\hcGT \subset \cGT$.
Thus \eqref{e:hGTixbd} implies $|\hGT_{ix}|\leq 2K|\msc|q^r$ for all $x \in \{a_1,\dots,a_\nu ,b_1, \dots, b_\nu\}$.

For $(x,y)=(a_k,b_k),(b_k,a_k)$, we have $\dist_{\tcGT}(y,c_k) > 2r$ by \eqref{e:yckdist},
and thus $|\tGT_{yc_k}|\leq2^{7}K^3|\msc|q^r$ by \eqref{e:tGTyckbd}.
The remaining $y \neq a_k,b_k$ satisfy either the condition in \eqref{Xcell} or in \eqref{Ycell}.
Therefore $|\tGT_{yc_k}|\leq \max\{\Gamma_1,\Gamma_2\}\leq \Gamma_1$,
and there are at most $2\mu \leq 2d(d-1)^{\ell}$ such terms.

In summary, we have shown
\begin{align*}
|\tGT_{ic_k}|\leq \frac{2M}{\sqrt{N\eta}}+2^9K^4|\msc|q^{2r+1}+(2Kq^{r+1})(2d(d-1)^{\ell})\Gamma_1\leq 2^{12}K^5|\msc|q^{2r+1},
\end{align*}
provided that $\sqrt{N\eta}\geq Mq^{-2r-2}$, where we used $r=2\ell+1$.\end{proof}

\begin{proof}[Proof of \eqref{e:greendistJJ}]
It remains to estimate $\tGT_{c_jc_k}$ for $j\in J \setminus \{k\}$.
As previously, we denote by $\cC'$ the $\rm S'$-cell containing $c_k$,
and now denote by $\cC''$ the $\rm S'$-cell containing $c_j$.
The estimates in Lemma~\ref{lem:tcGTcbdist} on distances from $c_k$ also apply with $c_k$ replaced by $c_j$.
Similarly to the bound of $\Gamma_2$, we use formula \eqref{tGijbound} and devide $\vE$ as $\vE=\vE_1\cup \cdots \cup \vE_5$, where 
\begin{align*}
  \vE_1 &= \{(x,y): x \in \bX_1, x \neq a_k\},\\
  \vE_2 &= \{(a_k,b_k)\},\\
  \vE_3 &= \{(b_l,a_l): b_l \in \cC'\} = \{(b_k,a_k)\},\\
  \vE_4 &= \{(b_l,a_l): b_l \in  \cC''\} = \{(b_j,a_j)\},\\
  \vE_5 &= \{(b_l,a_l): b_l \in \bX_2 \setminus (\cC'\cup\cC'')\}.
\end{align*}
Notice that for any $x\in\{a_1,\dots, a_{\nu}, b_1,\dots, b_\nu\}\setminus\{b_j\}$, by the definition of $J$, $x,c_j$ are in different $\rm S$-cells, and thus $|\hGT_{c_j x}|\leq 2M/\sqrt{N\eta}$ by \eqref{e:hGTixbd}.
Moreover, by the definition of $J$, for any $y\in\{a_m, b_m: m\in\qq{1,\nu}\setminus\{k\}\}$, we have $\dist_{\tcGT}(y, \cC')>R/4-2r>2r$ and thus $y$ satisfies the condition either in \eqref{Xcell} or in \eqref{Ycell}. 
It follows that $|\tGT_{yc_k}|\leq \max\{\Gamma_1,\Gamma_2\}\leq \Gamma_1$.
For $(x,y) \in \vE_1$.
Since $|\vE_1| \leq 2\mu\leq 4(d-1)^{\ell+1}$, it follows that
\begin{equation*}
\sum_{\vE_1}\sumdots
\leq
4(d-1)^{\ell+1}\frac{2M}{\sqrt{N\eta}}\frac{\Gamma_1}{\sqrt{d-1}}.
\end{equation*}
For $(x,y) = (a_k,b_k) \in \vE_2$. By \eqref{e:yckdist}, $\dist_{\tcGT}(b_k,c_k)>2r$, and thus $|\tGT_{b_kc_k}|\leq 2^7K^3|\msc|q^r$.
\begin{equation*}
\sum_{\vE_2}\sumdots
\leq \frac{2M}{\sqrt{N\eta}}\frac{2^7K^3|\msc|q^r}{\sqrt{d-1}}.
\end{equation*}
For $(x,y) = (b_k,a_k) \in \vE_3$. By \eqref{e:yckdist}, $\dist_{\tcGT}(a_k,c_k)>2r$, and thus $|\tGT_{a_kc_k}|\leq 2^7K^3|\msc|q^r$.
\begin{equation*}
\sum_{\vE_3}\sumdots
\leq
\frac{2M}{\sqrt{N\eta}}\frac{2^7K^3|\msc|q^r}{\sqrt{d-1}}.
\end{equation*}
For $(x,y)=(b_j,a_j) \in \vE_4$,
by \eqref{e:yckdist}  with $c_k$ replaced by $c_j$,
we have the distance estimates $$\dist_{\tcGT}(c_j, \{a_1, \dots, a_\mu,b_1, \dots, b_\nu\} > 2r.$$
In particular, $\dist_{\hcGT}(c_j,b_j)\geq \dist_{\tcGT}(c_j, b_j)>2r$, and $|\hGT_{c_j b_j}|\leq 2K|\msc|q^r$ by \eqref{e:hGTixbd}. Thus
\begin{equation*}
\sum_{\vE_4}\sumdots
\leq \frac{2K|\msc|q^{r}}{\sqrt{d-1}}\Gamma_1.
\end{equation*}
For $(x,y) \in \vE_5$,
since $|\vE_5|\leq \mu\leq 2(d-1)^{\ell+1}$, it follows that
\begin{equation*}
\sum_{\vE_5}\sumdots\leq 2(d-1)^{\ell+1}\frac{2M}{\sqrt{N\eta}}\frac{\Gamma_1}{\sqrt{d-1}}.
\end{equation*}
The above discussion combined with \eqref{Gamma} leads to the estimate
\begin{equation*}
|\tGT_{c_jc_k}|
\leq \frac{M}{\sqrt{N\eta}}\left(2+2^9K^3q^{r+1} +12(d-1)^{\ell+1/2}\Gamma_1\right)+2Kq^{r+1}\Gamma_1
\leq 2^{12}K^5|\msc|q^{3r+2},
\end{equation*}
provided that $\sqrt{N\eta}q^{3r+2}\geq M$.
\end{proof}

\section{Stability estimate for the switched graph}
\label{sec:weakstab}

\begin{proposition}\label{prop:tGweakstab}
Under the assumptions of Propositions~\ref{prop:stabilitytGT},
for ${\bf S} \in F(\cG)$ (as in Section \ref{cdFG})
such that $\tcG=T_{\bf S}(\cG)\in \bar \Omega$ (as in Section \ref{sec:outline-structure}),
the Green's function of the switched graph satisfies the weak stability estimate
that for all $i,j \in \qq{N}$,
\begin{align}\label{tGweakstab}
  |\tG_{ij}(z)|\leq |\tG_{jj}(z)|\leq 2.
\end{align}
Moreover, the off-diagonal entries of the Green's function satisfy the following improved estimates around vertex $1$.
For all vertices $x\in \qq{2, N}$,
\begin{align}\label{G1xbound}
  \left|\tG_{1x}-P_{1x}(\cE_r(1,x,\tcG))\right|\leq  (\n+1)2^{2\n+14}K^3|\msc|q^{r+1}. 
\end{align} 
For all estimates, we assume that $\sqrt{d-1}\geq \max\{(\n+1)^2 2^{2\n+10}, 2^8(\n+1)K\}$, $\n'^2q^\ell\ll1$ and $\sqrt{N\eta}q^{3r+2}\geq M$.
\end{proposition}

\subsection{Preparation of the proof}
As in \eqref{defI}, we denote by  $\del_E\cT$
the boundary edges of $\cT$ in the switched graph $\tcG$,
and the corresponding boundary vertex set by $\bI=\{\ta_1,\ta_2,\dots,\ta_\mu\}$.
Let $J$ be the index set of Proposition~\ref{greendist-new}.
Throughout the following proof,
$\Cw$ represents constants that may differ from line to line,
but depends only on the constant $K$ of \eqref{asumpGT} and the excess $\n$.
As in previous proofs, we follow the structure described below \eqref{e:stabilityGS}.

\paragraph{Localization}

To prove Proposition~\ref{improvetG}, we replace $P_{ij}(\cE_r(i,j,\tcG))$ by 
a vertex independent Green's function $P_{ij}$ according to Remark~\ref{rk:princG0}, applied with $\tcG_0=\cB_{3r}(1,\tcG)$ and $\bX=\bB_{2r}(1,\tcG)$. We abbreviate
\begin{align*}
\tcG_1 = \TE(\tcG_0),\quad \td P=G(\tcG_1),\quad \tcGT_1=\TE(\tcGT_0),\quad \td P^{(\T)}=G(\tcGT_1),
\end{align*}
Notice that $\tcGT_1$ is the same as removing the vertices $\bT$ from $\tcG_1$, and thus  $\td P^{(\T)}=G^{(\T)}(\tcG_1)$. 

\begin{claim}
Let $k\in J$ (as in Section \ref{sec:defJ}), and let $\bK$ be the connected component of $\tcG_0$ containing $\ta_k=c_k$.
Then
\begin{align}\label{e:unique}
\{m\in\qq{1,\mu}: \ta_m\in \bK\}=\{k\},
\end{align}
and
\begin{align}\label{e:diam}
\max_{i\in \bK}\dist_{\tcGT}(i,\ta_k)\leq 3r.
\end{align}
\end{claim}
\begin{proof}
\eqref{e:unique} follows from the condition \eqref{e:defJ} in the definition of $J$,
i.e.\ from $\dist_{\tcGT}(\ta_k, \{\ta_m: m\in\qq{1,\mu}\setminus \{k\}\})>R/4>6r$.
\eqref{e:diam} follows from \eqref{e:unique} and the construction of the subgraph $\tcG_0$.
\end{proof}

\paragraph{Verification of assumptions in Proposition \ref{boundPij}}
As subgraphs of $\tcG$, both $\tcG_0$ and $\tcGT_0$ have excess at most $\n$.
The deficit function $g$ of $\tcG_0$ vanishes.
By Proposition~\ref{GTstructure}, on each connected components of $\tcGT_0$,
the deficit function obeys $\sum g(v)\leq \n+1\leq 8\n$.
Thus the assumptions for \eqref{e:compatibility} are verified for both graphs $\tcG_0$ and $\tcGT_0$, and we have \eqref{replaceEr0}:
\begin{align}\label{tGTreplaceEr}
\left|P_{ij}(\cE_{r}(i,j,\tcG))-\td P_{ij}\right|, \left|P_{ij}(\cE_{r}(i,j,\tcGT))-\td P_{ij}^{(\T)}\right|
\leq 2^{2\n+3}|\msc|q^{r+1}
\end{align}
for $i,j\in \bX$, provided that $\sqrt{d-1}\geq 2^{\n+2}$.

\paragraph{Starting point}

The normalized adjacency matrices of $\tcG$ and $\tcG_1$ respectively have the block form
\begin{align*}
\left[
\begin{array}{cc}
H & \tilde{B'}\\
\tilde{B} & D
\end{array}
\right],\quad
\left[
\begin{array}{cc}
H & \tilde{B}_1'\\
\tilde{B}_1 & D_1
\end{array}
\right],
\end{align*}
where $H$ is the normalized adjacency matrix for $\cT$,
and $\tilde B$ (respectively $\tilde B_1$) corresponds to the edges from $\bI$ to $\T_\ell$,
where $\bI$ is the set of boundary vertices of $\cT$ in the switched graph $\tcG$ as defined in \eqref{defI}, and $\bT_\ell$ is the inner vertex boundary of $\cT$ as in \eqref{def:Ti}.
To be precise, the nonzero entries of $\tilde B$ and $\tilde B_1$ occur for the indices
$(i,j) \in \bI\times \T_\ell$ and take values $1/\sqrt{d-1}$.
Notice that $\tilde B_{ij}=(\tilde B_1)_{ij}$; in the rest of this section we will therefore not distinguish
$B$ and $\tilde B$.

By the Schur Complement formula \eqref{e:Schur}, we have
\begin{align}
  \label{G1xSchur1}
 \tG|_\T&=(H-z-\tilde{B}'\tGT\tilde{B})^{-1},\\
 \label{G1xSchur2}
 \td P|_\T&=(H-z-\tilde{B}'\tP^{(\T)}\tilde{B})^{-1},
\end{align}
and, by the resolvent identity \eqref{e:resolv}, the difference of \eqref{G1xSchur1} and \eqref{G1xSchur2} is
\begin{equation}\label{Tterm}
 \tG|_\T-\td P|_{\T}=(\tG-\tP)\tilde{B}'(\tGT-\tP^{(\T)})\tilde{B}\tP+\tP\tilde{B}'(\tGT-\tP^{(\T)})\tilde{B}\tP.
\end{equation}

In terms of the random walk heuristic described in Section~\ref{sec:Gexpan},
\eqref{Tterm} has the interpretation that only walks that exit $\T$ contribute
(see Figure~\ref{fig:Tterm}).
We will adopt suggestive terminology corresponding to the random walk picture below.
By Proposition~\ref{greendist-new}, the Green's function $\tGT$ is small between most vertices in $\bI$.
This is the main reason that the right-hand side of \eqref{Tterm} is small.
In the following, we analyze the various contributions precisely.

\begin{figure}[t]
\centering
\input{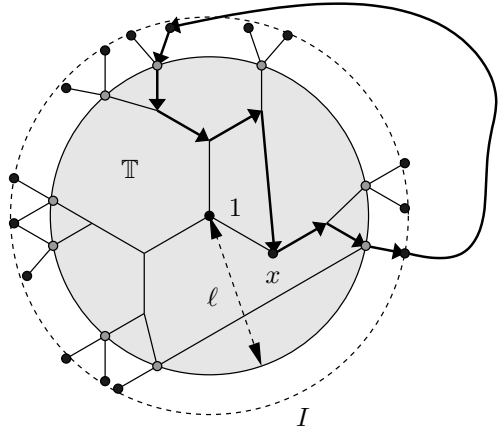}
\caption{Only walks that exit $\T$ contribute to \eqref{Tterm}.
\label{fig:Tterm}}
\end{figure}

\subsection{Boundary}
The following lemma estimates the weight of ``walks'' from $x \in \T$ to $\T_\ell$, the inner vertex boundary of $\cT$.
It depends on the distance of $x$ to the boundary, or equivalently that from $x$ to  $1$.

\begin{lemma}\label{sumP}
Assume that $\tcG_0$ has excess at most $\n$.
For vertices $x\in \T_{\ell_1}$, i.e.\ $x$ is at distance $\ell_1$ from vertex $1$, we have
\begin{align}\label{sumtleaf1}
\sum_{k\in\qq{1,\mu}}|\tP_{l_k x}|\leq (\n+1)2^{\n+3}(\ell_1+1)|\msc|^{\ell-\ell_1+1}(d-1)^{(\ell-\ell_1)/2+1}.
\end{align}
For vertices $x\in \T_{\ell_1}$ and $y\in \T_{\ell_2}$, with $\ell_1\geq \ell_2$, we have
\begin{align}\label{sumtleaf2}
\sum_{k\in\qq{1,\mu}}|\tP_{l_k x}||\tP_{l_k y}|\leq \frac{(\n+1)^2 2^{2\n+6}(\ell_1-\ell_2+2)}{(d-1)^{(\ell_1-\ell_2)/2-1}}|\msc|^{2\ell-\ell_1-\ell_2+2}.
\end{align}
\end{lemma}

The proof of the lemma uses the following combinatorial estimate on the distances of a vertex $x$ to $\bT_\ell$
(which is the inner vertex boundary of $\cT$).

\begin{lemma}\label{l:distancextoboundary}
Assume that the graph $\tcG_0 = \cB_{3r}(1,\cG)$ has excess at most $\n$. Given $x \in \T_{\ell_1}$,
let $L_x$  be the multiset consisting of $2(\n+1)(d-1)^{\ell-\ell_3}$ copies of the number $q^{\ell+\ell_1-2\ell_3}$ for $\ell_3\in\qq{0,\ell_1}$,
and let $K_x$ be the multiset $K_x = \{ q^{\dist_{\tcG_0}(x,i)}:i\in \T_\ell\}$.
Then the $k$-th largest number of $K_x$ 
is smaller than or equal to the $k$-th largest number of $L_x$.
\end{lemma}

We postpone the proof of the lemma to Appendix~\ref{app:distxtobd}.
Given the lemma, the proof of Lemma~\ref{sumP} is completed as follows.

\begin{proof}[Proof of Lemma~\ref{sumP}]
To prove \eqref{sumtleaf1}, we use
\begin{align*}
\sum_{k\in\qq{1,\mu}}|\tP_{l_k x}|
\leq (d-1)\sum_{i\in \T_\ell}|\tP_{ix}|
\leq 2^{\n+2} (d-1)|\msc| \sum_{i\in \T_\ell} q^{\dist_{\tcG_0}(x,i)},
\end{align*}
by Proposition \ref{boundPij}.
Defining the multiset $L_x$ as in Lemma~\ref{l:distancextoboundary}, the
inequality continuous with
\begin{align*}
2^{\n+2}(d-1) |\msc| \sum_{i\in \T_\ell} q^{\dist_{\tcG_0}(x,i)}
&\leq 2^{\n+2} (d-1) |\msc| (2\n+2)\sum_{\ell_3=0}^{\ell_1}(d-1)^{\ell-\ell_3}q^{\ell+\ell_1-2\ell_3}\\
&\leq (\n+1)2^{\n+3}(\ell_1+1)|\msc|^{\ell-\ell_1+1}(d-1)^{(\ell-\ell_1)/2+1}.
\end{align*}
This finishes the proof of \eqref{sumtleaf1}. 

For the proof of \eqref{sumtleaf2}, we use
\begin{align*}
\sum_{k\in\qq{1,\mu}}|\tP_{l_k x}||\tP_{l_k y}|\leq (d-1)\sum_{i\in \T_\ell}|\tP_{ix}||\tP_{iy}|\leq 2^{2\n+4}|\msc|^2(d-1)\sum_{i\in \T_\ell}q^{\dist_{\tcG_0}(x,i)}q^{\dist_{\tcG_0}(y,i)}.
\end{align*}
We define the multisets $L_x$ and $L_y$ as in Lemma~\ref{l:distancextoboundary}. More precisely, $L_x$ consists of $2(\n+1)(d-1)^{\ell-\ell_3}$ copies of $q^{\ell+\ell_1-2\ell_3}$ for $\ell_3\in\qq{0,\ell_1}$,
and $L_y$ consists of $2(\n+1)(d-1)^{\ell-\ell_3}$ copies of $q^{\ell+\ell_2-2\ell_3}$ for $\ell_3\in\qq{0,\ell_2}$.
By the rearrangement inequality, we have
\begin{align*}
&\sum_{i\in \T_\ell}q^{\dist_{\tcG_0}(x,i)}q^{\dist_{\tcG_0}(y,i)}\\
&\leq 4(\n+1)^2 \left(\sum_{\ell_3=0}^{\ell_2} (d-1)^{\ell-\ell_3}q^{\ell+\ell_1-2\ell_3}q^{\ell+\ell_2-2\ell_3}
+\sum_{\ell_3=\ell_2+1}^{\ell_1}(d-1)^{\ell-\ell_3}q^{\ell+\ell_1-2\ell_3}q^{\ell-\ell_2}\right)\\
&\leq \frac{4(\n+1)^2 (\ell_1-\ell_2+2)}{(d-1)^{(\ell_1-\ell_2)/2}}|\msc|^{2\ell-\ell_1-\ell_2}.
\end{align*}
This finishes the proof of \eqref{sumtleaf2}.
\end{proof}

\begin{remark}
In the worst case, when $x=1$, we have
\begin{align}\label{worstsumtleaf1}
\sum_{k\in \qq{1,\mu}}|\tP_{l_k x}|\leq (\n+1)2^{\n+3}|\msc|^{\ell+1}(d-1)^{\ell/2+1}.
\end{align}
Moreover, when $x,y\in \T_{\ell_1}$, we have
\begin{align}\label{worstsumtleaf2}
\sum_{k\in\qq{1,\mu}}|\tP_{l_k x}||\tP_{l_ky}|\leq (\n+1)^2 2^{2\n+7}|\msc|^{2\ell-2\ell_1+2}(d-1).
\end{align}
These special cases will be used below.
\end{remark}

\subsection{Outside $\T$}
The following proposition shows that the weight of ``walks'' outside $\T$ is small.
It essentially follows from Proposition~\ref{greendist-new}.

\begin{proposition}\label{l:smallsum}
  Under the assumptions of Proposition~\ref{prop:tGweakstab}, 
  for any vertex $j\in \qq{N} \setminus \T$ such that $\dist_{\tcG}(1,j)\leq 2r$, we have
  \begin{align}\label{sumtGtPi}
    \sum_{k\in \qq{1,\mu}}|\tGT_{\ta_k j}-\tP^{(\T)}_{\ta_kj}|\leq \Cw \n'|\msc|q^{r}.
  \end{align}
  Moreover, for any vertex $j\in \qq{N}\setminus\T$ such that $\dist_{\tcG}(1,j)> 2r$, we have
  \begin{align}\label{sumtGti}
    \sum_{k\in \qq{1,\mu}}|\tGT_{\ta_k j}|\leq \Cw \n'|\msc|q^{r},
  \end{align}
  and
   \begin{align}\label{sumtGtPij}
    \sum_{k\neq m\in \qq{1,\mu}}|\tGT_{\ta_k\ta_m}-\tP^{(\T)}_{\ta_k\ta_m}|\leq \Cw \n'^2|\msc|q^{r},
  \end{align}
  where the constants $\Cw$ depend only on  the excess $\n$ and $K$ (from Proposition~\ref{prop:stabilitytGT}).
  For all estimates, we assume $\sqrt{d-1}\geq \max\{(\n+1)^2 2^{2\n+10}, 2^8(\n+1)K\}$, $\n'q^r\ll1$ and $\sqrt{N\eta}q^{3r+2}\geq M$.
\end{proposition}

\begin{claim}
Let $j \in \qq{N} \setminus \T$ be as in the statement of Proposition~\ref{l:smallsum}. Then
\begin{align}\label{e:lessshortdist-tilde}
  &|\{k\in\qq{1,\nu}: \dist_{\tcGT}(j,a_k)\leq R/4\}|\leq 5\n,\\
 \label{e:fewclose-tilde}
  &|\{k\in\qq{1,\nu}: \dist_{\tcGT}(j,\ta_k)\leq R/2\}|\leq \n+1.
\end{align}
\end{claim}

\begin{proof}
The first claim follows from \eqref{e:lessshortdist}.
The second one follows from \eqref{distfinitedega} by considering the graph $\tcG$, since by our assumption $\tcG\in \bar \Omega$,
the $R$-neighborhood $\cB_{R}(1, \tcG)$ has excess at most $\n$.
\end{proof}

\begin{proof}[Proof of Proposition \ref{l:smallsum}]
Recall the index set $J\subset \qq{1,\nu}$ defined previously in Section \ref{sec:defJ}.
To prove \eqref{sumtGtPi}, we decompose $\qq{1,\mu}$ according to the relations between $\{a_k, b_k, c_k\}$ and  vertex $j$ as
$\qq{1,\mu}=J_1\cup J_2$, where 
\begin{align*}
J_1=&\{k\in J: j\not\sim b_k,\text{ }\dist_{\tcGT}(j,a_k)\geq 2r, \text{ and } \dist_{\tcGT}(j,\ta_k)\geq R/2\},\\
J_2=&\qq{1,\mu}\setminus J_1.
\end{align*}
By the defining relation \eqref{switchnum} of $F(\cG)$ and Proposition \ref{greendist-new}, we have $|J|\geq \nu-\n'-6\n\geq \mu-\n'-9\n$.
Combining with \eqref{e:lessshortdist-tilde},  \eqref{e:fewclose-tilde} and \eqref{finitedeg},
which states that $|\{k\in \qq{1,\nu}: j\sim b_k\}|< \n'$, we get 
\begin{align*}
|J_1|\geq \mu-2\n'-15\n,\quad |J_2|\leq 2\n'+15\n.
\end{align*}
Bounding by the total number of terms, we also have $|J_1|\leq \mu\leq 2(d-1)^{\ell+1}$.
Now, for $k\in J_1$, we have $\ta_k=c_k$ and the conditions for \eqref{e:greendistNJ} are satisfied.
Moreover, for $k\in J_1$, we have $\dist_{\tcGT}(\ta_k,j)\geq R/2$, and therefore by \eqref{e:diam}
the vertices $\ta_k$ and $j$ are in different connected components of $\tcGT_0$;
it follows that $|\tP^{(\T)}_{\ta_kj}|=0$.
Therefore, by \eqref{e:greendistNJ},
\begin{align}\label{sumtGtPiexp1}
 \sum_{k\in J_1}|\tGT_{\ta_k j}-\tP^{(\T)}_{\ta_k j}|
 =\sum_{k\in J_1}|\tGT_{\ta_kj}|
 \leq 2(d-1)^{\ell+1}(2^{12}K^5|\msc|q^{2r+1})\leq 2^{13}K^5|m_{sc}|q^r.
\end{align}
For $k\in J_2$, by \eqref{e:stabilitytGT} and \eqref{tGTreplaceEr},
\begin{align}\label{sumtGtPiexp2}
\sum_{k\in J_2}|\tGT_{\ta_kj}-\tP^{(\T)}_{\ta_kj}|\leq (2\n'+16\n)(2^{7}K^3+2^{2\n+3}q)|\msc|q^r.
\end{align}
Then \eqref{sumtGtPi} follows by combining \eqref{sumtGtPiexp1} and \eqref{sumtGtPiexp2}.

For \eqref{sumtGti}, again, we split the sum over $J_1$ and over $J_2$ as above.
For $k\in J_1$, similarly to \eqref{sumtGtPi}, we have
\begin{align}\label{sumtGtPiexp3}
 \sum_{k\in J_1}|\tGT_{\ta_kj}|
 \leq 2(d-1)^{\ell+1}2^{12}K^5|\msc|q^{2r+1}\leq 2^{13}K^5|m_{sc}|q^r.
\end{align}
For $k\in J_2$, we note that $\dist_{\tcG}(\ta_k,j)\geq \dist_{\tcG}(1,j)-\dist_{\tcG}(1,\ta_k)\geq 2r-\ell>r$, so that $P_{\ta_kj}(\cE_r(\ta_k,j,\tcGT))=0$.  Therefore, by \eqref{e:stabilitytGT}, we have
\begin{align}\label{sumtGtPiexp4}
\sum_{k\in J_2}|\tGT_{\ta_kj}|\leq (2\n'+15\n)2^{7}K^3|\msc|q^r.
\end{align}
Again \eqref{sumtGti} follows by combining \eqref{sumtGtPiexp3} and \eqref{sumtGtPiexp4}.

For \eqref{sumtGtPij}, we split the sum over 
\begin{align*}
\{k\neq m\in \qq{1,\mu}\}=&\{k\neq m \in \qq{1,\mu}\setminus J\}\cup \{k\in \qq{1,\mu}\setminus J,  m \in J\}\\
\cup& \{k\in J,  m \in \qq{1,\mu}\setminus J\}\cup \{k\neq m\in J\}.
\end{align*}
Since $|\qq{1,\mu}\setminus J|\leq \n'+9\n$, for $k\neq m\in \qq{1,\mu}\setminus J$, by Proposition \ref{prop:stabilitytGT} and \eqref{tGTreplaceEr},
\begin{align*}
\sum_{k\neq m\in \qq{1,\mu}\setminus J}|\tGT_{\ta_k\ta_m}-\tP^{(\T)}_{\ta_k\ta_m}|\leq (\n'+9\n)^2(2^{7}K^3+2^{2\n+3}q)|\msc|q^r.
\end{align*}
For $k\in \qq{1,\mu}\setminus J,  m \in J$,  by \eqref{e:unique} $\ta_k$ and $\ta_m$ are in different connected components of $\tcGT_0$, and thus 
$|\tP_{\ta_k\ta_m}^{(\T)}|=0$. By \eqref{e:greendistIJ}, 
\begin{align*}
\sum_{k\in \qq{1,\mu}\setminus J,  m \in J}|\tGT_{\ta_k\ta_m}-\tP^{(\T)}_{\ta_k\ta_m}|
=\sum_{k\in \qq{1,\mu}\setminus J,  m \in J}|\tGT_{ij}|
\leq& (\n'+9\n)2(d-1)^{\ell+1}(2^{9}K^4|\msc|q^{2r+1})\\
\leq& (\n'+9\n)2^{10}K^4|\msc|q^{r}.
\end{align*}
The same estimate holds for $k\in J,  m \in \qq{1,\mu}\setminus J$.
For $k\neq m\in J$, the same reasoning as above gives $\tP_{\ta_k\ta_m}^{(\T)}=0$.
By \eqref{e:greendistJJ} and noticing that $|J|\leq \mu\leq 2(d-1)^{\ell+1}$, we have
\begin{align*}
\sum_{k\neq m\in J}|\tGT_{\ta_k\ta_m}-\tP^{(\T)}_{\ta_k\ta_m}|=\sum_{k\neq m\in J}|\tGT_{\ta_k\ta_m}|\leq 4(d-1)^{2\ell+2}2^{12}K^5|\msc|q^{3r+2}\leq 2^{14}K^5|\msc|q^{r}.
\end{align*}
Now \eqref{sumtGtPij} follows by combining the above four cases.
\end{proof}

\subsection{Proof of \eqref{G1xbound}}

The proof of \eqref{G1xbound} follows essentially from \eqref{Tterm}
and the fact the difference of $\tGT$ and $\tP^{(\T)}$ is small (Proposition \ref{prop:stabilitytGT}).
\begin{claim}\label{newrigid}
For all $x\in \T$,
\begin{align*}
|\tilde{G}_{1x}-\tP_{1x}|\leq (\n+1)2^{2\n+13}K^3|\msc|q^r.
\end{align*}
Moreover, for $x\in \T\setminus \{1\}$, we have the stronger estimate
\begin{align}\label{betterest}
 |\tilde{G}_{1x}-\tP_{1x}|\leq  (\n+1)2^{2\n+14}K^3|\msc|q^{r+1}.
\end{align}
\end{claim}

\begin{proof}
Let $\Gamma_1=\max_{x\in \T}|\tilde{G}_{1x}-\tP_{1x}|$.
Then the first term on the right-hand side of \eqref{Tterm} is bounded by
\begin{align}\begin{split}\label{firstterm}
& |((\tG-\tP)\tilde{B}'(\tGT-\tP^{(\T)})\tilde{B}\tP)_{1x}|
  \leq \frac{\Gamma_1}{d-1}\sum_{k\in \qq{1,\mu}}\sum_{m\in \qq{1,\mu}}|\tGT_{\ta_k\ta_m}-\tP^{(\T)}_{\ta_k\ta_m}||\tP_{l_m x}|\\
 &\leq \frac{\Gamma_1(\Cw \n'q^{r+1})}{\sqrt{d-1}}\sum_{m\in\qq{1,\mu}}|\tP_{l_mx}|
 \leq  \Gamma_1(\Cw \n'q^{r+1})((\n+1)2^{\n+3}|\msc|^{\ell+1}(d-1)^{(\ell+1)/2})\\
 &\leq\Cw \n'|\msc|q^{\ell+1}\Gamma_1,
\end{split}
\end{align}
where we used \eqref{sumtGtPi} and \eqref{worstsumtleaf1}.
Next we bound the second term on the right-hand side of \eqref{Tterm}.
For, say $x\in \T_{\ell_1}$, we have
\begin{align}\begin{split}\label{sumtwo}
|(\tP\tilde{B}'(\tGT-\tP^{(\T)})\tilde{B}\tP)_{1x}|
 &\leq \frac{1}{d-1}\sum_{k\in\qq{1,\mu}}|\tP_{1l_k}||\tGT_{\ta_k\ta_k}-\tP^{(\T)}_{\ta_k\ta_k}||\tP_{l_kx}|\\
 &\qquad+\frac{1}{d-1}\sum_{k\neq m\in \qq{1,\mu}}|\tP_{1l_k}||\tGT_{\ta_k\ta_m}-\tP^{(\T)}_{\ta_k\ta_m}||\tP_{l_mx}|.
\end{split}\end{align}
By \eqref{e:boundPij} and \eqref{e:boundPiimsc}, for any $k,m\in \qq{1,\mu}$, 
\begin{align}\label{PBbound}
|\tP_{1l_k}|\leq 2^{\n+2}|\msc|q^{\ell},\quad |\tP_{l_m x}|\leq 2|\msc|.
\end{align}
We can estimate the first term in \eqref{sumtwo} in the following way:
\begin{align*}
\frac{1}{d-1}\sum_{k\in\qq{1,\mu}}|\tP_{1l_k}||\tGT_{\ta_k\ta_k}-\tP^{(\T)}_{\ta_k\ta_k}||\tP_{l_kx}|
&\leq 2^{\n+2}q^{\ell+1} (2^{7}K^3+2^{2\n+3}q)q^{r+1}\sum_{k\in\qq{1,\mu}}|\tP_{l_kx}|\\
&\leq (\n+1) 2^{2\n+5}(\ell_1+1)(2^{7}K^3+2^{2\n+3}q)q^{r+\ell_1},
\end{align*}
where in the first inequality we used \eqref{PBbound}, \eqref{e:stabilitytGT} and  \eqref{tGTreplaceEr}, in the second inequality we used estimate \eqref{sumtleaf1}. 
For the second term in \eqref{sumtwo}, we have
\begin{align*}
&\frac{1}{d-1}\sum_{k\neq m\in \qq{1,\mu}}|\tP_{1l_k}||\tGT_{\ta_k\ta_m}-\tP^{(\T)}_{\ta_k\ta_m}||\tP_{l_mx}|\\
&\leq 2^{\n+2}q^{\ell+1} \sum_{k\neq m\in\qq{1,\mu}}|\tGT_{\ta_k\ta_m}-\tP^{(\T)}_{\ta_k\ta_m}|(2q)
\leq \Cw \n'^2 q^{r+\ell+2}, 
\end{align*}
where we used \eqref{sumtGtPij} and \eqref{PBbound}.
It follows that 
\begin{align}\label{tG1xminusP1x}
|\tG_{1x}-\tP_{1x}|\leq \Cw \n'|\msc|q^{\ell+1}\Gamma_1+(\n+1) 2^{2\n+5}(\ell_1+1)(2^{7}K^3+2^{2\n+3}q)q^{r+\ell_1}+\Cw \n'^2 q^{r+\ell+2} .
\end{align}
By taking the maximum over $x\in \T$ and rearranging it,
we have $\Gamma_1\leq (\n+1)2^{2\n+13}K^3|\msc|q^r$, provided that $\n'^2q^\ell\ll 1$ and $\sqrt{d-1}\geq (\n+1)^2 2^{2\n+10}$.

For \eqref{betterest}, it follows from \eqref{tG1xminusP1x}, the estimate of $\Gamma_1$ and $\ell_1\geq 1$,
\begin{align}
\begin{split}
\label{eqn:betteresttG1xminusP1x}
 |\tG_{1x}-\tP_{1x}|
 &\leq C\n'q^{r+\ell+1}+(\n+1) 2^{2\n+6}(2^{7}K^3+2^{2\n+3}q)q^{r+1}+\Cw \n'^2 q^{r+\ell+2} \\
 &\leq (\n+1) 2^{2\n+6}(2^{7}K^3+1)q^{r+1},
\end{split}
\end{align}
provided that $\n'^2q^\ell\ll 1$ and $\sqrt{d-1}\geq (\n+1)^2 2^{2\n+10}$. 
\end{proof}

\begin{proof}[Proof of \eqref{G1xbound}]
For $x \in \T\setminus\{1\}$, the estimate \eqref{G1xbound} follows from \eqref{eqn:betteresttG1xminusP1x} and \eqref{tGTreplaceEr}:
\begin{align*}
  &\left|\tG_{1x}-P_{1x}(\cE_r(1,x,\tcG))\right|
  \leq   \left|\tG_{1x}-\tP_{1x}\right|+\left|P_{1x}(\cE_{r}(1,x,\tcG))-\tP_{1x}\right|\\
  &\leq (\n+1) 2^{2\n+6}(2^{7}K^3+1)q^{r+1}+2^{2\n+3}|\msc|q^{r+1}
  \leq (\n+1)2^{2\n+14}K^3|\msc|q^{r+1}.
\end{align*}
Thus it only remains to prove \eqref{G1xbound} for $x\not\in \T$.

For $x\in \bB_{2r}(1,\tcG)\setminus \T$, we have by the Schur complement formula \eqref{e:Schur1}:
\begin{equation*}
 \tilde{G}=-\tilde{G}\tilde{B}'\tGT,\quad
  \tP=-\tP\tilde{B}'\tP^{(\T)}.
\end{equation*}
Therefore, by taking the difference of these two equations,
\begin{equation}\label{sumtwo2}
|\tilde{G}_{1x}-\tP_{1x}|\leq \frac{1}{\sqrt{d-1}}\sum_{k\in\qq{1,\mu}}|\tG_{1l_k}||\tP^{(\T)}_{\ta_kx}-\tGT_{\ta_kx}| + \frac{1}{\sqrt{d-1}}\sum_{k\in\qq{1,\mu}}|\tG_{1l_k}-\tP_{1l_k}||\tP^{(\T)}_{\ta_kx}|.
\end{equation}
For the first term in \eqref{sumtwo2}, notice that by combining \eqref{PBbound} and \eqref{betterest}, we have
\begin{align}\label{GBbound}
|\tG_{1l_k}|
\leq |\tG_{1l_k}-\tP_{1l_k}|+|\tP_{1l_k}|
\leq 2^{\n+3}|\msc|q^{\ell}. 
\end{align}
The first term in \eqref{sumtwo2} is bounded by
\begin{align*}
\frac{1}{\sqrt{d-1}}\sum_{k\in\qq{1,\mu}}|\tG_{1l_k}||\tP^{(\T)}_{\ta_kx}-\tGT_{\ta_kx}| 
\leq \Cw q^{\ell+1} \sum_{k\in\qq{1,\mu}}|\tP^{(\T)}_{\ta_kx}-\tGT_{\ta_kx}| 
\leq \Cw \n'q^{r+\ell+1},
\end{align*}
where we used \eqref{sumtGtPi}.
For the second term in \eqref{sumtwo2}, since $\tcG\in \bar\Omega$, its radius-$R$ neighborhood of vertex $1$ has excess at most $\n$.
By \eqref{distfinitedega} there are at most $\n+1$ indices $k\in\qq{1,\mu}$, such that $\ta_k$ is in the same connected component as $x$ in the graph $\tcG_0$.
Thus, $\tP^{(\T)}_{\ta_kx}$ are zero for all $k\in \qq{1,\mu}$ except for at most $\n+1$ of them, and they are bounded  $|\tP^{(\T)}_{\ta_kx}|\leq 2|m_{sc}|$ by \eqref{e:boundPiimsc}. 
\begin{align*}
\frac{1}{\sqrt{d-1}}\sum_{k\in\qq{1,\mu}}|\tG_{1l_k}-\tP_{1l_k}||\tP^{(\T)}_{\ta_kx}|
\leq (\n+1)^2 2^{2\n+15}K^3|\msc|q^{r+2} ,
\end{align*}
where we used \eqref{betterest}. Combining the arguments above, they lead to
\begin{align*}
|\tilde{G}_{1x}-\tP_{1x}|\leq& 2^5K^3|\msc|q^{r+1}+\Cw \n'q^{r+\ell+1},
\end{align*}
given that  $\sqrt{d-1}\geq (\n+1)^2 2^{2\n+10}$. The estimate \eqref{G1xbound} for $x\in \bB_{2r}(1,\tcG)\setminus \T$ follows, provided $\n'^2q^\ell\ll1$.

For  $x\notin \bB_{2r}(1,\tcG)$, we have
\begin{align}
\label{NTterms2} |\tilde{G}_{1x}|\leq  \frac{1}{\sqrt{d-1}}\sum_{k\in\qq{1,\mu}}|\tG_{1l_k}||\tGT_{\ta_kx}| \leq 2^{\n+3}q^{\ell+1}\sum_{k\in\qq{1,\mu}}|\tGT_{\ta_kx}|\leq \Cw \n'q^{r+\ell+1}\ll q^{r+1},
\end{align}
where we used \eqref{GBbound} and \eqref{sumtGti}. This finishes the proof of \eqref{G1xbound}.
\end{proof}

\subsection{Proof of \eqref{tGweakstab}}

\begin{proof}[Proof of \eqref{tGweakstab}]
For $x,y\in \T$, we denote $\Gamma=\max_{x,y \in \T}\{|\tilde{G}_{xy}-\tP_{xy}|\}$.
Then, by the Schur complement formula \eqref{Tterm},
\begin{align}\label{GPxy}
|\tG_{xy}-\tP_{xy}|\leq |((\tG-\tP)\tilde{B}'(\tGT-\tP^{(\T)})\tilde{B}\tP)_{xy}|+|(\tP\tilde{B}'(\tGT-\tP^{(\T)})\tilde{B}\tP)_{xy}|.
 \end{align}
The estimate of the first term follows the same argument as that for \eqref{firstterm}:
\begin{align*}
|((\tG-\tP)\tilde{B}'(\tGT-\tP^{(\T)})\tilde{B}\tP)_{xy}|
\leq C\n'|\msc|q^{\ell+1}\Gamma
\leq \Gamma/2.
\end{align*}
For the second term, similarly, we have
\begin{align*}
 &\left|(\tP\tilde{B}'(\tGT-\tP^{(\T)})\tilde{B}\tP)_{xy}\right|\\
 &\leq\frac{1}{d-1} \sum_{k\in \qq{1,\mu}}|\tP_{xl_k}||\tGT_{\ta_k\ta_k}-\tP^{(\T)}_{\ta_k\ta_k}||\tP_{l_ky}|+\frac{1}{d-1}\sum_{k\neq m\in  \qq{1,\mu}}|\tP_{xl_k}||\tGT_{\ta_k\ta_m}-\tP^{(\T)}_{\ta_k\ta_m}||\tP_{l_my}|\\
 &\leq \frac{\Cw |\msc|q^{r}}{d-1} \sum_{k\in \qq{1,\mu}}|\tP_{xl_k}||\tP_{l_ky}|+(\Cw q)^2\sum_{k\neq m\in  \qq{1,\mu}}|\tGT_{\ta_k\ta_m}-\tP^{(\T)}_{\ta_k\ta_m}|\leq \Cw \n'^2 |\msc|q^r,
\end{align*}
where we bounded $|\tP_{xl_k}|, |\tP_{l_my}|\leq \Cw |\msc|$ and used the estimates \eqref{e:stabilitytGT}, \eqref{worstsumtleaf2} and \eqref{sumtGtPij}. Therefore, by taking supremum of both sides of \eqref{GPxy} and rearranging, we have $\Gamma\leq \Cw \n'^2|\msc|q^{r}$.

For $x\in \T$ and $y\in \bB_{2r}(1,\tcG)\setminus \T$, the same argument as for \eqref{sumtwo2} implies:
\begin{align*}
 |\tG_{xy}-\tP_{xy}|
 &\leq \frac{1}{\sqrt{d-1}}\sum_{k\in\qq{1,\mu}}|\tG_{xl_k}||\tP^{(\T)}_{\ta_kx}-\tGT_{\ta_kx}| + \frac{1}{\sqrt{d-1}}\sum_{k\in\qq{1,\mu}}|\tG_{xl_k}-\tP_{xl_k}||\tP^{(\T)}_{\ta_kx}|.
\\
 &\leq \Cw q \sum_{k\in \qq{1,\mu}}|\tP^{(\T)}_{\ta_ky}-\tGT_{\ta_ky}| + C\n'^2 q^{r+1}\sum_{k\in \qq{1,\mu}}|\tP^{(\T)}_{\ta_ky}| 
 \leq \Cw \n'^2|\msc|q^{r+1}.
\end{align*}
where we bounded $|\tG_{xl_k}|\leq \Cw |\msc|$, and used the estimate \eqref{sumtGtPi}, the bound for $\Gamma$ and the fact that for all $k\in\qq{1,\mu}$ with at most $\n+1$ exceptions, $\tP^{(\T)}_{\ta_ky}$ are zero.

For $x\in \T$ and $y\notin \bB_{2r}(1,\tcG)$, similarly, we have:
\begin{align*}
 |\tilde{G}_{xy}|\leq  \frac{1}{\sqrt{d-1}}\sum_{k\in\qq{1,\mu}}|\tG_{xl_k}||\tGT_{\ta_ky}| \leq  \Cw q\sum_{k\in\qq{1,\mu}}|\tGT_{\ta_ky}|\leq \Cw \n'|\msc|q^{r+1},
\end{align*}
where we bounded $|\tG_{xl_k}|\leq \Cw |\msc|$ and used the estimate \eqref{sumtGti}.

For $x,y\in \bB_{2r}(1,\tcG)/\T$, we have the Schur complement formula \eqref{e:Schur}:
\begin{align*}
 \tilde{G}=&\tGT+\tGT\tilde B\tilde G\tilde B' \tGT,\\
 \tilde{P}=&\tP^{(\T)}+\tP^{(\T)}\tilde B\tilde P\tilde B'\tP^{(\T)}.
\end{align*}
By taking the difference, 
\begin{align*}
 \tilde{G}-\tP =&\tGT-\tP^{(\T)}+(\tGT-\tP^{(\T)})\tilde B\tilde G \tilde B' \tGT\\
 +&\tilde{P}^{(\T)}\tilde B(\tilde G-\tilde{P}) \tilde B' \tGT+\tP^{(\T)}\tilde B\tP\tilde B' (\tGT-\tP^{(\T)}).
\end{align*}
Notice that $|\tGT_{xy}-\tP^{(\T)}_{xy}|\leq C|\msc|q^r$, $|\tG_{l_kl_m}|\leq \Cw |\msc|$, $|\tP_{l_kl_m}|\leq \Cw|\msc|$ and $|\tilde G_{l_kl_m}-\tilde{P}_{l_kl_m}|\leq \Cw \n'^2|\msc|q^{r}$, we have
\begin{align}\begin{split}\label{GPxyout}
& |\tilde{G}_{xy}-\tP_{xy}|\leq \Cw |\msc|q^r
 +\sum_{k,m\in\qq{1,\mu}}|\tGT_{x\ta_k}-\tP^{(\T)}_{x\ta_k}|\frac{\Cw |\msc|}{d-1}|\tGT_{\ta_my}|\\
+&\sum_{k,m\in\qq{1,\mu}}|\tP^{(\T)}_{x\ta_k}|\frac{\Cw \n'^2 |\msc|q^r}{d-1}|\tGT_{\ta_my}|
+\sum_{k,m\in\qq{1,\mu}}|\tP^{(\T)}_{x\ta_k}|\frac{\Cw |\msc|}{d-1}|\tGT_{\ta_my}-\tP^{(\T)}_{\ta_my}|.
\end{split}
\end{align}
The following estimates follow from Proposition \ref{l:smallsum}:
\begin{align*}
&\sum_{k\in\qq{1,\mu}}|P^{(\T)}_{x\ta_k}|, \sum_{m\in\qq{1,\mu}}|\tGT_{\ta_my}| \leq \Cw |\msc|,\\
&\sum_{k\in\qq{1,\mu}}|\tGT_{x\ta_k}-P^{(\T)}_{x\ta_k}|, \sum_{m\in\qq{1,\mu}}|\tGT_{\ta_my}-P^{(\T)}_{\ta_my}| \leq \Cw \n'|\msc|q^r.
\end{align*}
Therefore \eqref{GPxyout} simplifies to
\begin{align*}
 |\tilde{G}_{xy}-\tP_{xy}|\leq \Cw \n'^2|\msc|q^r.
\end{align*}
Finally, for $x\not \in \bB_{2r}(1,\tcG)$ or $y\not \in \bB_{2r}(1,\tcG)$, by symmetry we assume $x\not \in \bB_{2r}(1,\tcG)$, we have 
\begin{align*}
|(\tGT\tilde B\tilde G\tilde B' \tGT)_{xy}|\leq& \frac{1}{d-1}\sum_{k,m \in \qq{1,\mu}}|\tGT_{x\ta_k}||\tG_{l_kl_m} |\tGT_{\ta_my}|\\
\leq&  \sum_{k,m \in \qq{1,\mu}}|\tGT_{x\ta_k}|\frac{\Cw |\msc|}{d-1}|\tGT_{\ta_my}| \leq \Cw \n'|\msc|q^{r+2},
\end{align*}
where we used \eqref{sumtGti}, thus,
\begin{align*}
| \tilde{G}_{xy}-\tGT_{xy}|\leq \Cw \n'|\msc|q^{r+2}.
\end{align*}
Altogether we proved that
\begin{align*}
  \left|\tG_{ij}-\tP_{ij}(\cE_r(i,j,\tcG))\right|\leq \Cw |\msc|\n'^2q^r\ll 	q^{\ell+1}, 
\end{align*}
provided that $\n'^2q^\ell\ll1$. The weak stability estimates \eqref{weakstab} follows by combining with \eqref{e:boundPiimsc}.
This finishes the proof of \eqref{tGweakstab}.
\end{proof}

\section{Concentration in the switched graph}
\label{sec:concentration}

The result of this section is the following proposition, which shows that the average
of the Green's function of $\tcGT$ over the vertex boundary  of $\T$
concentrates under resampling of the edge boundary of $\T$.
This part is where the condition that the edge boundary contains $\gg \log N$ edges is
important.

More precisely,
recall the vertex boundary $\bI=\{\ta_1,\ta_2,\dots, \ta_\mu\}$ of $\cT$ in $\tcG$ from \eqref{defI}.
For any finite graph $\cH$ (not necessarily regular and not necessary on $N$ vertices), we define
\begin{equation} \label{e:IG}
\I(\cH, z)=\frac{1}{Nd}\sum_{(i,j)\in \vec{E}} G_{jj}^{(i)}(\cH,z),
\end{equation}
where $\vec E$ denotes the set of oriented edges of $\cH$,
and $G^{(i)}(\cH,z)$ the Green's function of the graph obtained from $\cH$ by removing the vertex $i$.
Notice that we always normalize \eqref{e:IG} by $Nd$, irregardless of the actual number of oriented edges
in $\cH$ (which can be smaller than $Nd$).

\begin{proposition}\label{tGconcentration}
Sets $\Omega_1^+(z,\ell)$ and $\bar\Omega$ are as defined in Section \ref{sec:outline-structure}. Let $z\in\C_+$ and $\cG \in \Omega_1^+(z,\ell)$.
Then there exists an event $F'(\cG) \subset F(\cG)$ (as in Section \ref{cdFG}) with probability $\P_{\cG}(F'(\cG))= 1-o(N^{-\n+\delta})$
such that for any ${\bf S}\in F'(\cG)$ with $T_{\bf S}(\cG)\in \bar\Omega$,
\begin{align}\label{eqn:tGconcentration}
\left| \frac{1}{\mu}\sum_{k=1}^\mu\left(\tG^{(\T)}_{\ta_k\ta_k}-P_{\ta_k\ta_k}(\cE_r(\ta_k,\ta_k,\tcG^{(\T )}))\right)-\left(\I(\tcG)-\msc\right)\right|\leq \frac{2(\log N)^{1/2+\delta} |\msc|q^r}{\sqrt{\mu}},
 \end{align}
provided that $\sqrt{d-1}\geq \max\{(\n+1)^2 2^{2\n+10}, 2^8(\n+1)K\}$, $\n'^2q^\ell\ll1$ and $\sqrt{N\eta}q^{3r+2}\geq M$.
\end{proposition}

To prove Proposition~\ref{tGconcentration}, in Lemma~\ref{lem:econcentration},
we first show a similar statement for the unswitched graph $\cGT$
in which the problem becomes a concentration problem of \emph{independent} random variables.
Then we prove Proposition~\ref{tGconcentration} by comparision,
using the estimates of Proposition~\ref{prop:stabilitytGT},
and the fact that the change from $\I(\tcG,z)$ to $\I(\cGT,z)$ is small (Lemma~\ref{l:IGchange}).
Proposition~\ref{prop:stabilitytGT}
is applicable since, by the definition of set $\Omega_1^+(z,\ell)$ in Section~\ref{sec:outline},
any graph $\cG\in \Omega_1^+(z,\ell)$ satisfies the assumptions in Proposition~\ref{prop:stabilitytGT} with $K=2^{10}$.

The following proposition is used repeatedly in this section.
It follows from exactly the same argument as Proposition~\ref{stabilityGS}, and we therefore omit the proof. 

\begin{lemma}
\label{stablermone}
Given $z\in\C_+$, a constant $K'\geq 2$, and $\cG \in \bar\Omega$.
Let $\cal H$ be one of the graphs $\cGT$, $\hcGT$, $\tcGT$ or $\tcG$, and suppose that 
 \begin{align}
  \left|G_{ij}(\cal H, z)-P_{ij}(\cE_r(i,j,\cal H),z)\right|\leq K'|\msc|q^r.
 \end{align}
Then, for any vertices $i,j$ in $\cal H^{(x)}$, we have
\begin{align}
 |G_{ij}(\cal H^{(x)},z)-P_{ij}(\cE_{r}(i,j,\cal H^{(x)}),z)|\leq 2K' |\msc| q^r, 
\end{align}
provided that $\sqrt{d-1}\geq (\n+1)^2 2^{2\n+10}$.
Here all graphs have deficit function $g=d-\deg$,
and we recall that $\cal H^{(x)}$ is the graph obtained from $\cH$ by removing the vertex $x$.
\end{lemma}

\subsection{Estimate for the unswitched graph}

The next lemma shows concentration of a certain average of the Green's function
in the unswitched graph.

\begin{lemma} \label{lem:econcentration}
For any $z\in \C_+$ and $\cal G\in \Omega_1^+(z,\ell)$, we define the set 
$F'(\cG) \subset F(\cG)$ (as in Section \ref{cdFG}) such that
\begin{align}
\label{econcentration}
  \left| \frac{1}{\mu}\sum_{k=1}^\mu\left(G^{(\T b_k)}_{c_kc_k}-P_{c_kc_k}(\cE_r(c_k,c_k,\cal G^{(\T b_k)}))\right)-\left(\I(\cal G^{(\T)})-\msc\right)\right|\leq \frac{(\log N)^{1/2+\delta} |\msc|q^r}{\sqrt{\mu}}.
 \end{align}
Then $\P_{\cG}(F'(\cG))= 1-o(N^{-\n+\delta})$.
\end{lemma}

\begin{proof}
Let
\begin{align*}
X_k=G^{(\T b_k)}_{c_kc_k}-P_{c_kc_k}(\cE_r(c_k,c_k,\cal G^{(\T b_k)})),\quad k\in\qq{1,\mu}.
\end{align*}
Conditioned on the graph $\cGT$, the random sets $\vec S_1, \vec S_2,\dots, \vec S_\mu$ are independent
and identically distributed, and thus $X_1,X_2,\dots, X_\mu$ are i.i.d random variables.
By Lemma~\ref{stablermone} and the assumption that $\cG\in \Omega_1^+(z,\ell)$,
for any $k\in \qq{1,\mu}$, we have
\begin{align*}
|X_k|\leq 2K|\msc|q^r,
\end{align*}
where  $K=2^{10}$.
By Azuma's inequality for independent random variables, it therefore follows that
\begin{align}\label{azuma}
 \P_{\cG}\left(\left| \frac{1}{\mu}\sum_{k=1}^\mu X_k-\E[X_k]\right|\geq \frac{2Kt|\msc|q^r}{\sqrt{\mu}}\right)\leq e^{-t^2/2}.
\end{align}
In the following, we still need to estimate $\E[X_k]$. Let $\vec E$ be the set of oriented edges of $\cGT$.
By definition, $\T$ is the $\ell$ neighborhood of the vertex $1$, and by the trivial bound it intersects at most
$d+d(d-1)+\cdots+d(d-1)^{\ell}  \leq 2(d-1)^{\ell+1}$ edges.
Thus $Nd-4(d-1)^{\ell+1}\leq |\vec E|\leq Nd$.
Using that, by Lemma~\ref{stablermone},
we also have $|G_{ii}^{(\T j)}-P_{ii}(\cE_{r}(i,i,\cG^{(\T j)}))|\leq 2K|\msc|q^r$,
it follows that
\begin{align}
\notag\E[X_k]
&=\frac{1}{|\vec E|} \sum_{(i,j)\in\vec E}G_{ii}^{(\T j)}-P_{ii}(\cE_{r}(i,i,\cG^{(\T j)}))\\
\notag
&=\frac{1}{Nd} \sum_{(i,j)\in\vec E}G_{ii}^{(\T j)}-P_{ii}(\cE_{r}(i,i,\cG^{(\T j)}))+O_{\leq}\left(\frac{8K|m_{sc}|q^r(d-1)^{\ell+1}}{Nd}\right)\\
\label{avIG}
&=\I(\cGT)-\frac{1}{Nd} \sum_{(i,j)\in\vec E}P_{ii}(\cE_{r}(i,i,\cG^{(\T j)}))+O_{\leq}\left(\frac{8K}{N}\right).
\end{align}
Moreover, since by assumption $\cal G\in \bar\Omega$, all except for at most $N^{\delta}$ vertices have radius-$R$ tree neighborhoods in $\cG$, and therefore
\begin{align*}
&|\{i\in \qq{1, N}\setminus \T: \cB_{r}(i, \cGT) \text{ is not a $d$-regular tree}\}|\\
&\leq |\{i\in \qq{1, N}: \cB_{r}(i, \cG) \text{ is not a tree}\}|+|\{i\in \qq{1, N}: \dist_{\cG}(i, \T)\leq r\}|
\leq N^{\delta}+2(d-1)^{r+\ell}\leq 2N^{\delta} 
.
\end{align*}
For the vertices $i$ contained in the set on the left-hand side,
we have the bound $|P_{ii}(\cE_{r}(i,i,\cG^{(\T j)}))|\leq 2|\msc|$ from \eqref{e:boundPii}.
For the other vertices $i$, whose $r$-neighborhood in $\cGT$ is a $d$-regular tree,
we have the equality $P_{ii}(\cE_{r}(i,i,\cG^{(\T j)}))=\msc$. Therefore
\begin{equation} \label{avemsc}
  \frac{1}{Nd} \sum_{(i,j)\in\vec E}P_{ii}(\cE_{r}(i,i,\cG^{(\T j)}))=\msc+O_{\leq }(8N^{-1+\delta})
  .
\end{equation}
Combining \eqref{avIG}, \eqref{avemsc}, and taking $t=(\log N)^{1/2+\delta}/(4K)$ in \eqref{azuma}, we get
\begin{align*}
 \P_{\cG}\left(\left| \frac{1}{\mu}\sum_{i=1}^\mu X_i-\left(\I(\cGT)-\msc\right)\right|\geq \frac{2Kt|\msc|q^r}{\sqrt{\mu}}+\frac{10}{N^{1-\delta}}\right)\leq e^{-(\log N)^{1+2\delta}/(32K^2)}.
\end{align*}
Since $N^{-1+\delta}\ll (\log N)^{1/2+\delta}|\msc|q^r/\sqrt{\mu}$,
it follows that \eqref{econcentration} holds with overwhelming probability,
and we can define $F'(\cG)\subset F(\cG)$ as claimed with probability
\begin{align*}
\P_{\cG}(F'(\cG))\geq \P_{\cG}(F(\cG))-e^{-(\log N)^{1+2\delta}/(32K^2)}
=1-o(N^{-\n+\delta}),
\end{align*}
where we used \eqref{FGmeasure}.
This completes the proof.
\end{proof}

\subsection{Changing $Q(\cGT)$ to $Q(\tcG)$}

The next lemma shows that we can replace $Q(\cGT)$ by $Q(\cG)$ up to a small error.
It follows from the general insensitivity of the quantity $Q$ to small changes of the graph.

\begin{lemma}\label{l:IGchange}
For $z\in \C_+$, $\cG \in \Omega_1^+(z,\ell)$ and ${\bf S} \in F(\cG)$ with $T_{\bf S}(\cG)\in \bar\Omega$,
we have
\begin{align}\label{e:IGchange}
|\I(\cGT,z)-\I(\tcG,z)|
\leq \frac{36d^{2\ell+2}}{N\eta}
,
\end{align}
provided that $\sqrt{d-1}\geq \max\{(\n+1)^2 2^{2\n+10}, 2^8(\n+1)K\}$, $\n'^2q^\ell\ll1$ and $\sqrt{N\eta}q^{3r+2}\geq M$.
\end{lemma}

The proof of Lemma~\ref{l:IGchange} uses Lemma~\ref{lem:sumward} below,
which is a direct consequence of the Ward identity \eqref{e:Ward}.

\begin{lemma}\label{lem:sumward}
Given a graph $\cG$ with degree bounded by $d$.
We denote by $\vec{E}$ the set of oriented edges of $\cG$,
by $H$ its normalized adjacency matrix, and by $G=(H-z)^{-1}$ its Green's function.
Then, if for some $z\in \C_+$ and any $(i,j)\in \vec{E}$, it holds that
\begin{align}\label{weakstab}
  |G_{ij}(z)|\leq |G_{jj}(z)|\leq 2,
\end{align}
then for any vertex $x\in \cG$,
\begin{equation} \label{e:sumward}
 \sum_{(i,j)\in \vec{E}}|G_{ix}^{(j)}(z)|^2\leq \frac{8d}{\eta}.
\end{equation}
\end{lemma}

\begin{proof}
By the Schur complement formula \eqref{e:Schurixj} and the Ward identity \eqref{e:Ward}, we obtain
\begin{align*}
  \sum_{(i,j)\in \vec{E}}|G_{ix}^{(j)}|^2
  &=\sum_{(i,j)\in \vec{E}}\left|G_{ix}-\frac{G_{ij}G_{jx}}{G_{jj}}\right|^2
     \leq \sum_{(i,j)\in \vec{E}}2\left|G_{ix}\right|^2+2\left|\frac{G_{ij}G_{jx}}{G_{jj}}\right|^2\\
  &\leq 4\sum_{(i,j)\in \vec{E}}|G_{ix}|^2\leq 4\sum_{i}\deg_\cG(i)|G_{ix}|^2
        \leq \frac{4d \Im[G_{xx}]}{\eta}\leq \frac{8d}{\eta},
\end{align*}
as claimed.
\end{proof}

We will prove Lemma~\ref{l:IGchange} in two steps, by proving
\begin{equation}\label{GTtotGTconcentration}
|\I(\cGT)-\I(\hcGT)|
\leq \frac{d\mu}{2N\eta}, \qquad
|\I(\hcGT)-\I(\tcGT)|
\leq \frac{d\mu}{2N\eta}
,
\end{equation}
and
\begin{equation}\label{tGTtotGconcentration}
|\I(\tcGT)-\I(\tcG)|\leq  \frac{34d^{2\ell+2}}{N\eta}.
\end{equation}
Then \eqref{e:IGchange} follows by combining \eqref{GTtotGTconcentration} and \eqref{tGTtotGconcentration},
and using that $\mu\leq 2(d-1)^{\ell+1}$.
In preparation, we recall from Proposition~\ref{prop:tGweakstab} that,
for all vertices $i,j\in \qq{N}$,
\begin{equation}\label{weakstabinconcen}
  |\tG_{ij}(z)|\leq |\tG_{jj}(z)|\leq 2.
\end{equation}

\begin{proof}[Proof of \eqref{GTtotGTconcentration}]
The proofs of both estimates in \eqref{GTtotGTconcentration} are analogous, and we only prove the first one.
Denote by $\vec{E}$ the set of oriented edges of $\hcGT$,
and by $\Delta=\sum_{k=1}^\nu (e_{b_kc_k}+e_{c_kb_k})/\sqrt{d-1}$ the difference of the normalized adjacency matrices of the graphs $\hcGT$ and $\cGT$.
Then by the resolvent formula \eqref{e:resolv},
\begin{align*}
\sum_{(i,j)\in \vec{E}}|\hG^{(\T j)}_{ii}-G^{(\T j)}_{ii}|
& \leq  \sum_{x,y}\sum_{(i,j)\in \vec{E}}|\hG^{(\T j)}_{ix}\Delta_{xy}G^{(\T j)}_{yi}|\\
& \leq \sum_{x,y}\Delta_{xy}\left(\sum_{(i,j)\in \vec{E}}|\hG^{(\T j)}_{ix}|^2\sum_{(i,j)\in \vec{E}}|G^{(\T j)}_{yi}|^2\right)^{1/2}
\leq \frac{16d\mu}{\eta\sqrt{d-1}}
,
\end{align*}
where we used \eqref{e:sumward}
(and that both graphs $\cGT$ and $\hcGT$ satisfy condition \eqref{weakstab} by the definition of $\Omega_1^+(z,\ell)$ and \eqref{boundhGT}).
Therefore,
\begin{align}\begin{split}\label{hIGchange-pf}
\left|\I(\cGT)-\I(\hcGT)\right|
&\leq \frac{1}{Nd}\sum_{k\in\qq{1,\nu}}|G_{b_kb_k}^{(\T c_k)}+G_{c_kc_k}^{(\T b_k)}|
+\frac{1}{Nd}\sum_{(i,j)\in \vec{E}}|\hG^{(\T j)}_{ii}-G^{(\T j)}_{ii}|\\
&\leq \frac{4\nu|\msc|}{Nd}+\frac{16d\mu}{N\eta\sqrt{d-1}}\leq \frac{d\mu}{2N\eta},
\end{split}\end{align}
where in the estimate of the first term, we used $|G_{b_kb_k}^{(\T c_k)}|,|G_{c_kc_k}^{(\T b_k)}|\leq 2|\msc|$
which follows from combining \eqref{e:defOmega+}, Lemma~\ref{stablermone}, and \eqref{e:boundPii}.
\end{proof}

\begin{proof}[Proof of \eqref{tGTtotGconcentration}]
The normalized adjacency matrices of $\tcG$ takes the block form
\begin{align*}
\left[
\begin{array}{cc}
H & \tilde{B'}\\
\tilde{B} & D
\end{array}
\right],
\end{align*}
where $H$ is the normalized adjacency matrix for $\cT$,
and $\tilde B$ corresponds to the edges from $\bI$ to $\T_\ell$,
where $\bI$ is the set of boundary vertices of $\cT$ in the switched graph $\tcG$ as defined in \eqref{defI}.
We denote by $\vec{E}$ the set of oriented edges of $\tcGT$.
By the Schur complement formula \eqref{e:Schur}, we have
\begin{equation*}
 \sum_{(i,j)\in \vec{E}}|\tG^{(\T j)}_{ii}-\tG^{(j)}_{ii}|
 \leq  \frac{1}{d-1}\sum_{k,m\in\qq{1,\mu}}\sum_{(i,j)\in \vec{E}}|\tG^{(\T j)}_{i\ta_k}\tG^{(j)}_{l_kl_m}\tG^{(\T j)}_{\ta_mi}|.
 \end{equation*}
It follows from \eqref{weakstabinconcen} and \eqref{e:Schurixj} that $|\tG^{(j)}_{l_kl_m}|\leq 4$.
Therefore the above expression is bounded by
\begin{align*}
  \sum_{(i,j)\in \vec{E}}|\tG^{(\T j)}_{ii}-\tG^{(j)}_{ii}|
  &\leq \frac{4}{d-1}\sum_{k,m\in\qq{1,\mu}}\sum_{(i,j)\in \vec{E}}|\tG^{(\T j)}_{i\ta_k}\tG^{(\T j)}_{\ta_mi}|\\
  &\leq \frac{4}{d-1}\sum_{k,m\in\qq{1,\mu}}\left(\sum_{(i,j)\in \vec{E}}|\tG^{(\T j)}_{i\ta_k}|^2\sum_{(i,j)\in \vec{E}}|\tG^{(\T j)}_{\ta_mi}|^2\right)^{1/2}
  \leq \frac{32d\mu^2}{\eta(d-1)}\leq \frac{32d^{2\ell+2}}{\eta},
\end{align*}
where we used \eqref{e:sumward} (since $\tcGT$ satisfies condition \eqref{weakstab} thanks to the definition of $\Omega_1^+(z,\ell)$ and \eqref{e:stabilitytGT}).  Therefore, we have
\begin{align}\begin{split}\label{tGIGchange}
\left|\I(\tcGT)-\I(\tcG)\right|
&\leq \frac{1}{Nd}\sum_{\{i,j\}\text{ incident to }\T}|\tG_{ii}^{(j)}+\tG_{jj}^{(i)}|
+\frac{1}{Nd} \sum_{(i,j)\in \vec{E}}|\tG^{(\T j)}_{ii}-G^{(j)}_{ii}|\\
&\leq \frac{16(d-1)^{\ell+1}}{Nd}+\frac{32d^{2\ell+2}}{N\eta}\leq 
\frac{34d^{2\ell+2}}{N\eta},
\end{split}\end{align}
where for the first term we used $|\tG_{ii}^{(j)}|,|\tG_{jj}^{(i)}|\leq 4$ from \eqref{weakstabinconcen} and \eqref{e:Schurixj} .
\end{proof}

\subsection{Adding of switched vertices}

Recall the index set $J\subset\qq{1,\nu}$ from Proposition~\ref{greendist-new}.
In this subsection, we show that the following lemma.

\begin{lemma} \label{lem:Gsw}
For $z\in \C_+$, $\cG\in \Omega_1^+(z,\ell)$ and ${\bf S} \in F(\cG)$ with $T_{\bf S}(\cG)\in\bar\Omega$, for any $k\in J$, we have
\begin{equation}\label{diff1361}
|\hG^{(\T b_k)}_{c_k c_k}-G^{(\T b_k)}_{c_kc_k}|
\leq
16q^{2r} 
\end{equation}
and
\begin{equation}\label{diff1362}
|\tGT_{c_kc_k}-\hG^{(\T b_k)}_{c_kc_k}|
\leq
2^{10}K^4|\msc|q^{2r},
\end{equation}
where  $K=2^{10}$. For both estimates, we assume $\sqrt{d-1}\geq \max\{(\n+1)^2 2^{2\n+10}, 2^8(\n+1)K\}$, $\n'q^r\ll1$ and $\sqrt{N\eta}q^{2r}\geq M$.
\end{lemma}

To prove Lemma~\ref{lem:Gsw} we need the estimates summarized in the following lemma.

\begin{lemma}
Let $z\in \C_+$, $\cG \in \Omega_1^+(z,\ell)$, and ${\bf S} \in F(\cG)$ with $T_{\bf S}(\cG)\in\bar \Omega$.
Then for any index $k\in J$,
the vertex $c_k$ is far away from $\{a_1,\dots, a_\mu, b_1,\dots, b_\nu\}$:
\begin{align}\label{dist1}
\dist_{\hcGT}(c_k,\{a_1,\dots, a_\mu, b_1,\dots, b_\nu\})\geq\dist_{\tcGT}(c_k,\{a_1,\dots, a_\mu, b_1,\dots, b_\nu\})>2r.
\end{align}
Moreover, for any $x\in \{a_1,\dots, a_\mu, b_1,\dots, b_\nu\}$,
\begin{align}\label{greenest1}
|\hGT_{c_kx}|\leq 2K|\msc|q^r,\quad |\tGT_{c_kx}|\leq 2^7K^3|\msc|q^r, \quad |\hGT_{b_kb_k}|\geq |m_{sc}|/2,
\end{align}
where $K=2^{10}$ and we assume that $\sqrt{d-1}\geq \max\{(\n+1)^2 2^{2\n+10}, 2^8(\n+1)K\}$, $\n'q^r\ll1$ and $\sqrt{N\eta}\geq M(d-1)^{\ell+1}$.
\end{lemma}
\begin{proof}
\eqref{dist1} is \eqref{e:yckdist}.
The first two estimates in \eqref{greenest1} follow from \eqref{dist1} and $K=2^{10}$ in Proposition~\ref{prop:stabilitytGT}.
The last estimate in \eqref{greenest1} follows by taking $K=2^{10}$ in Proposition~\ref{prop:stabilitytGT} and \eqref{e:boundPii}.
\end{proof}

\begin{proof}[Proof of Lemma~\ref{lem:Gsw}]
Notice that for $\cG\in \Omega_1^+(z,\ell)$, the assumptions in Proposition \ref{prop:stabilitytGT} hold for $K=2^{10}$. By the resolvent identity \eqref{e:resolv},
\begin{align*}
|\hG^{(\T b_k)}_{c_k c_k}-G^{(\T b_k)}_{c_k c_k}|\leq  \sum_{x,y}|\hG^{(\T b_k)}_{c_k x}|| \Delta_{xy}||\hG^{(\T b_k)}_{y c_k}|,
  \end{align*}
 where $\Delta = \sum_{m \in \qq{1,\nu}\setminus \{k\}} (e_{c_mb_m}+e_{b_mc_m})/\sqrt{d-1}$. By our choice of the index set $J$, $\{b_k,c_k\}$ and $\{b_m,c_m:m\in\qq{1,\nu}\setminus\{k\}\}$ are in different $\rm S$-cells. Thus $|\hGT_{c_kx}|,|\hGT_{b_kx}|\leq 2M/\sqrt{N\eta}$ by \eqref{e:diffcellest}. Therefore, using \eqref{e:Schurixj}, and notice $|\hGT_{b_kb_k}|\geq |\hGT_{c_kb_k}|$, we get
\begin{align*}
|\hG_{c_kx}^{(\T b_k)}|\leq |\hG_{c_kx}^{(\T )}|+\left|\frac{\hG_{c_kb_k}^{(\T )}\hG_{b_kx}^{(\T )}}{\hG_{b_kb_k}^{(\T )}}\right|\leq |\hG_{c_kx}^{(\T )}|+|\hG_{b_kx}^{(\T )}|\leq \frac{4M}{\sqrt{N\eta}}.
\end{align*}
The same estimate holds for $|\hGT_{yc_k}|$. Thus the second term is bounded by
\begin{align*}
\sum_{x,y}|\hG^{(\T b_k)}_{c_k x}||\Delta_{xy}||\hG^{(\T b_k)}_{y c_k}|\leq 4(d-1)^{\ell+1/2}\left(\frac{2M}{\sqrt{N\eta}}\right)^2
\leq 16q^{2r},
\end{align*}
provided that $\sqrt{N\eta}\geq Mq^{-3r/2}$.
Similarly, now setting $\Delta=\sum_{m=1}^{\nu}(e_{a_mb_m}+e_{b_ma_m})/\sqrt{d-1}$, by the resolvent identity \eqref{e:resolv} and \eqref{e:Schurixj},
we have
\begin{align}
\label{eqn:termtG-hG}|\tGT_{c_kc_k}-\hG^{(\T b_k)}_{c_kc_k}|
\leq  |\tGT_{c_kc_k}-\hG^{(\T)}_{c_kc_k}|+ |\hG^{(\T b_k)}_{c_kc_k}-\hG^{(\T)}_{c_kc_k}|
\leq \sum_{x,y}|\hGT_{c_kx}||\Delta_{xy}||\tGT_{yc_k}|+ \left|\frac{\hGT_{c_kb_k}\hGT_{b_k c_k}}{\hGT_{b_kb_k}}\right|
\end{align}
For the last term in \eqref{eqn:termtG-hG}, by \eqref{greenest1}, $|\hGT_{c_kb_k}|\leq 2K|\msc|q^r$ and $|\hGT_{b_kb_k}|\geq |\msc|/2$, and therefore
\begin{align*}
\left|\frac{\hGT_{c_kb_k}\hGT_{b_k c_k}}{\hGT_{b_kb_k}}\right|\leq\frac{(2K|\msc|q^r)^2}{|\msc|/2}
=
8K^2|\msc|q^{2r}.
 \end{align*} 
For the sum  on the right-hand side of \eqref{eqn:termtG-hG}, we can split it into two, 
\begin{align*}
\sum_{x,y}|\hGT_{c_kx}||\Delta_{xy}||\tGT_{yc_k}|=|\hGT_{c_kb_k}||\Delta_{b_ka_k}||\tGT_{a_kc_k}|+\sum_{(x,y)\neq (b_k,a_k)}|\hGT_{c_kx}||\Delta_{xy}||\tGT_{yc_k}|.
\end{align*}
Again, we have $|\hGT_{c_kx}|\leq 2M/\sqrt{N\eta}$, for $x\in \{b_m:m\in\qq{1,\nu}\setminus\{k\}\}\cup \{a_m: m\in \qq{1,\nu}\}$.
Combining with \eqref{greenest1}, it follows that 
\begin{align*}
\eqref{diff1362} &\leq \frac{(2K|\msc|q^r)(2^{7}K^3|\msc|q^r)}{\sqrt{d-1}}+\frac{2M}{\sqrt{N\eta}}\frac{4(d-1)^{\ell+1}}{\sqrt{d-1}}(2^{7}K^3|\msc|q^{r})+8K^2|\msc|q^{2r}
\\
&\leq 2^{10}K^4|\msc|q^{2r},
\end{align*}
provided that $\sqrt{N\eta}q^{2r}\geq M$.
\end{proof}

\subsection{Proof of Proposition~\ref{tGconcentration}}

Finally, using the previous lemmas, we can proof Proposition~\ref{tGconcentration}.

\begin{proof}[Proof of Proposition~\ref{tGconcentration}]
For $k\in J$, the $r$-neighborhood of $c_k$ is a $d$-regular tree with root degree $d-1$ in
any of the graphs $\cG^{(\T b_k)}$, $\hcG^{(\T b_k)}$ and $\tcG^{(\T)}$; therefore, by \eqref{e:Gtreemsc},
\begin{align*}
P_{c_kc_k}\left(\cE_r\left(c_k,c_k,\cG^{(\T b_k)}\right)\right)=P_{c_kc_k}\left(\cE_r\left(c_k,c_k,\hcG^{(\T b_k)}\right)\right)
=P_{c_kc_k}\left(\cE_r\left(c_k,c_k,\tcG^{(\T )}\right)\right)=\msc
.
\end{align*}
On the other hand, for the indices $k\in \qq{1,\mu}\setminus J$, by \eqref{e:defOmega+}, Proposition~\ref{prop:stabilitytGT}, and Lemma~\ref{stablermone},
using that for $\cG\in \Omega_1^+(z,\ell)$, the assumption of Proposition~\ref{prop:stabilitytGT} holds with $K=2^{10}$,
we have
\begin{align}\begin{split}\label{eqn:somebounds}
\left|G^{(\T b_k)}_{c_kc_k}-P_{c_kc_k}\left(\cE_r\left(c_k,c_k,\cG^{(\T b_k)}\right)\right)\right|&\leq 2K|\msc|q^r,\\
\left|\hG^{(\T b_k)}_{c_kc_k}-P_{c_kc_k}\left(\cE_r\left(c_k,c_k,\hcG^{(\T b_k)}\right)\right)\right|&\leq 4K|\msc|q^r,\\
\left|\tG^{(\T )}_{\ta_k\ta_k}-P_{\ta_k\ta_k}\left(\cE_r\left(\ta_k,\ta_k,\tcG^{(\T)}\right)\right)\right|&\leq 2^{7}K^3|\msc|q^r.
\end{split}\end{align}
The above estimates \eqref{eqn:somebounds} and \eqref{diff1361} give
\begin{align}\label{hchange}
\begin{split}
&\left| \frac{1}{\mu}\sum_{k=1}^\mu\left(G^{(\T b_k)}_{c_kc_k}-P_{c_kc_k}\left(\cE_r\left(c_k,c_k,\cG^{(\T b_k)}\right)\right)\right)- \left(\hG^{(\T b_k)}_{c_kc_k}-P_{c_kc_k}\left(\cE_r\left(c_k,c_k,\hcG^{(\T b_k)}\right)\right)\right) \right|\\
&\leq \frac{6K(\mu-|J|)|\msc|q^{r}}{\mu}+\frac{1}{\mu}\sum_{k\in J}|G_{c_kc_k}^{(\T b_k)}-\hG_{c_kc_k}^{(\T b_k)}|\\
&\leq \frac{6K(\n'+9\n)|\msc|q^{r}}{\mu}+(8K^2|\msc|q^{2r} + 16q^{2r})
\leq \frac{(\log N)^{1/2+\delta}|\msc|q^r}{4\sqrt{\mu}}.
\end{split}
\end{align}
Moreover, by the above estimates \eqref{eqn:somebounds}, \eqref{diff1362} and using $\ta_k=c_k$ for  $k\in J$, we have
\begin{align}\label{tGTchange}
\begin{split}
&\left| \frac{1}{\mu}\sum_{k=1}^\mu\left(\tG^{(\T )}_{\ta_k\ta_k}-P_{\ta_k\ta_k}\left(\cE_r\left(\ta_k,\ta_k,\tcG^{(\T )}\right)\right)\right)- \left(\hG^{(\T b_k)}_{c_kc_k}-P_{c_kc_k}\left(\cE_r\left(c_k,c_k,\hcG^{(\T b_k)}\right)\right)\right) \right|\\
&\leq \frac{(4K+2^7K^3)(\mu-|J|)|\msc|q^{r}}{\mu}+\frac{1}{\mu}\sum_{k\in J}|\tG_{c_kc_k}^{(\T )}-\hG_{c_kc_k}^{(\T b_k)}|\\
&\leq \frac{(4K+2^7K^3)(\n'+9\n)|\msc|q^{r}}{\mu}+2^{10}K^4|\msc|q^{2r}\leq \frac{(\log N)^{1/2+\delta}|\msc|q^r}{4\sqrt{\mu}}.
\end{split}
\end{align}
In the above estimates we used $\ell\geq 4\log_{d-1}\log N$ by \eqref{e:constchoice} so that $\sqrt{\mu}\gg \log N=\n'$.
The left-hand side of \eqref{eqn:tGconcentration} is bounded by
\begin{align*}
|\eqref{econcentration}|+|\eqref{e:IGchange}|+|\eqref{tGTchange}|+|\eqref{econcentration}|\leq \frac{2(\log N)^{1/2+\delta} |\msc|q^r}{\sqrt{\mu}},
\end{align*}
provided that $\sqrt{N\eta}q^{3r+2}\geq M$.
\end{proof}

\section{Improved approximation in the switched graph}
\label{sec:improved}

The results of this section are the following proposition, stating
that the Green's function obeys better estimates than the original one near vertex $1$.
As in the previous sections, we write $\tcG=T_{\bf S}(\cG)$
and assume that ${\bf S}\in F'(\cG)$ (as in Lemma \ref{lem:econcentration}) is such that $\tcG=T_{\bf S}(\cG)\in \bar\Omega$ (as in Section \ref{sec:outline-structure}).
Throughout the proof, $\Cw$ represents constants depending only on the constant $K$ from \eqref{asumpGT} and the excess $\n$,
which may be different from line to line.

\begin{proposition}\label{improvetG}
Under the assumptions of Propositions~\ref{prop:stabilitytGT},
for ${\bf S} \in F'(\cG)$ such that $\tcG=T_{\bf S}\cG\in \bar\Omega$,
the Green's function of the switched graph 
satisfies the following improved estimates near vertex $1$.
\begin{enumerate}
\item[(i)]
For the vertex $x=1$,
\begin{align}\label{G11bound}
  \tG_{11}=P_{11}(\cE_r(1,1,\tcG))+ \frac{\md^2\msc^{2\ell}\mu}{(d-1)^{(\ell+1)}}(\I(\tcG)-\msc)
  +O_{\leq}\left(2^{2\n+10}K^3|\msc|q^{r+1}\right).
\end{align}
 \item[(ii)]
For all vertices $x\in \qq{2,N}$,
\begin{align}\label{G1xbound_copy}
  \left|\tG_{1x}-P_{1x}(\cE_r(1,x,\tcG))\right|\leq  (\n+1)2^{2\n+14}K^3|\msc|q^{r+1}.
\end{align} 
\end{enumerate}
Moreover, if the vertex $1$ has radius-$R$ tree neighborhood in the graph $\tcG$,
then the following stronger estimates hold.
\begin{enumerate}
\item[(i')]
For the vertex $x=1$,
\begin{align}\label{G11treebound}
  \tG_{11} = \md + \md^2\msc^{2\ell}\frac{d}{d-1}(\I(\tcG)-\msc)+O_{\leq}\left( \frac{4(\log N)^{1/2+\delta} |\msc|q^{r}}{\sqrt{d(d-1)^{\ell}}}\right).
\end{align}
\item[(ii')]
For the the average of $\tG_{1x}$ over the vertices $x$ adjacent to $1$,
\begin{align}\label{G1itreebound}
 \frac{1}{d}\sum_{1\sim x}\tG_{1x}+\frac{\md\msc}{\sqrt{d-1}}
=\frac{-\md^2\msc^{2\ell-1}(1+\msc^2)}{\sqrt{d-1}}(\I(\tcG)-\msc)+O_{\leq}\left( \frac{16(\log N)^{1/2+\delta} |\msc|q^{r+1}}{\sqrt{d(d-1)^{\ell}}}\right).
\end{align}
\end{enumerate}
For all estimates we assume that  $\sqrt{d-1}\geq \max\{(\n+1)^2 2^{2\n+10}, 2^8(\n+1)K\}$, $\n'^2q^\ell\ll1$ and $\sqrt{N\eta}q^{3r+2}\geq M$, and the global quantity $\I(\tcG)$ is as defined in \eqref{e:IG}.
\end{proposition}

We use the same set-up as in Section~\ref{sec:weakstab},
and notice that \eqref{G1xbound_copy} is \eqref{G1xbound}.

\subsection{Proof of \eqref{G11bound} and \eqref{G11treebound}}

By \eqref{Tterm}, we have
\begin{align}\begin{split}\label{G11-P11}
 \tilde{G}_{11}-\tP_{11}
 &= \frac{1}{d-1}\sum_{k\in\qq{1,\mu}}\tP_{1 l_k}^2(\tGT_{\ta_k\ta_k}-\tP^{(\T)}_{\ta_k\ta_k})+ \frac{1}{d-1}\sum_{k\neq m\in \qq{1,\mu}} \tG_{1 l_k}\tP_{1 l_m}(\tGT_{\ta_k\ta_m}-\tP^{(\T)}_{\ta_k\ta_m}) \\
 & +\frac{1}{d-1}\sum_{k\in \qq{1,\mu}} (\tG_{1l_k}-\tP_{1 l_k})\tP_{1 l_k}(\tGT_{\ta_k\ta_k}-\tP^{(\T)}_{\ta_k\ta_k}).
 \end{split}
\end{align}
For the last term on the right-hand side of \eqref{G11-P11}, we have
\begin{align*}
&\left|\frac{1}{d-1}\sum_{k\in \qq{1,\mu}} (\tG_{1l_k}-\tP_{1 l_k})\tP_{1 l_k}(\tGT_{\ta_k\ta_k}-\tP^{(\T)}_{\ta_k\ta_k})\right|
\leq\sum_{k\in\qq{1,\mu}}\Cw q^{r+2} q^{\ell+1}(|\msc|q^{r})
\leq \Cw |\msc|q^{r+\ell+2},
\end{align*}
where we used \eqref{betterest} for the first factor, \eqref{PBbound} for the second factor, and \eqref{e:stabilitytGT} for the last factor.
For the second term on the right-hand side of \eqref{G11-P11}, we have
\begin{align*}
  \left|\frac{1}{d-1}\sum_{k\neq m\in \qq{1,\mu}} \tG_{1 l_k}\tP_{1 l_m}(\tGT_{\ta_k\ta_m}-\tP^{(\T)}_{\ta_k\ta_m})\right|
  &\leq \Cw (q^{\ell+1})(q^{\ell+1})\sum_{k\neq m\in \qq{1,\mu}} |(\tGT_{\ta_k\ta_m}-\tP^{(\T)}_{\ta_k\ta_m})|\\
  &\leq \Cw q^{2\ell+2}(\n'^2|\msc|q^r)\leq \Cw |\msc|q^{r+\ell+2},
\end{align*}
provided that $\n'^2q^{\ell}\ll1$, where we used \eqref{GBbound} for the first factor, \eqref{PBbound} for the second factor, and \eqref{sumtGtPij} for the last factor.
Therefore \eqref{G11-P11} is bounded by
\begin{equation}\label{newG11-P11}
 \tilde{G}_{11}-\tP_{11}= \frac{1}{d-1}\sum_{k\in\qq{1,\mu}}\tP_{1 l_k}^2(\tGT_{\ta_k\ta_k}-\tP^{(\T)}_{\ta_k\ta_k})+O\left( |\msc|q^{r+\ell+2}\right),
\end{equation}
where the implicit constant depends only on the excess $\n$ and $K$ from \eqref{asumpGT}.

\begin{proof}[Proof of \eqref{G11treebound}]
If the radius-$R$ neighborhood of the vertex $1$ is a tree, then by Proposition~\ref{greentree},
\begin{align*}
 \tP_{1 l_k}^2=\frac{\md^2\msc^{2\ell}}{(d-1)^{\ell}}, \quad \tP_{11}=m_d,
\end{align*}
and 
\begin{equation}\label{tree11term}
 \tilde{G}_{11}-\md=\frac{\md^2\msc^{2\ell}}{(d-1)^{\ell+1}} \sum_{k\in\qq{1,\mu}}(\tGT_{\ta_k\ta_k}-\tP^{(\T)}_{\ta_k\ta_k})+O\left( |\msc|q^{r+\ell+2}\right).
\end{equation}

Notice that $\mu=d(d-1)^{\ell}$ under the assumption that the $R$-neighborhood is a tree.
Moreover, for all $k\in\qq{1,\mu}$,
\begin{align*}
\tP^{(\T)}_{\ta_k\ta_k}=\tP_{\ta_k\ta_k}\left(\cE_{r}\left(\ta_k,\ta_k,\tcGT\right)\right)=m_{sc},
\end{align*}
and by Proposition \ref{tGconcentration}, we can simplify \eqref{tree11term} to get
\begin{align}\label{mktree}
  \tilde{G}_{11}=\md+\frac{d}{d-1}\md^2\msc^{2\ell}(\I(\tcG)-\msc)+
  O_{\leq }\left( \frac{4(\log N)^{1/2+\delta} |\msc|q^r}{\sqrt{d(d-1)^\ell}}\right).
\end{align}
This finishes the proof of \eqref{G11treebound}.
\end{proof}

\begin{proof}[Proof of \eqref{G11bound}]
Since by assumption $\tcG\in \bar\Omega$,
the radius-$R$ neighborhood of the vertex $1$ has excess at most $\n$.
Therefore, there are at most $2\n(d-1)^{\ell}$ indices $k\in \qq{1,\mu}$ such that the non-backtracking path from $1$ to $l_k$ of length $\ell$ is not unique. Let
\begin{align*}
J'=\{k\in \qq{1,\mu}:  \text{non-backtracking path from $1$ to $l_k$ of length $\ell$ is unique} \}.
\end{align*}
For  $k\in  J'$, by \eqref{eqn:formulaPij} in the proof of Proposition~\ref{boundPij}, we have 
\begin{align*}
\left|\tP_{1 l_k}-\frac{\md (-m_{sc})^\ell}{(d-1)^{\ell/2}}\right|\leq |\md|\sum_{k\geq 2}2^{\n k}q^{\ell+k-1}\leq 2^{2\n}|\md|\frac{3}{2}q^{\ell+1},
\end{align*}
provided that  $\sqrt{d-1}\geq 2^{\n+2}$.
Therefore, for all $k\in J'$, the following estimate holds
\begin{align*}
\frac{\tP_{1 l_k}^2}{d-1}=\frac{\md^2\msc^{2\ell}}{(d-1)^{\ell+1}}+O_{\leq}\left(2^{2\n+2}q^{2\ell+3}\right),
\end{align*}
For $k\in \qq{1,\mu}\setminus J'$ by \eqref{PBbound}, we have $|\tP_{1l_k}|\leq 2^{\n+2}|\msc|q^{\ell}$. 
Notice that $|J'|\leq \mu\leq d(d-1)^\ell$ and $|\qq{1,\mu}\setminus J'|\leq 2\n (d-1)^{\ell}$, it follows that 
\begin{align}\begin{split}\label{e:sumPcoef}
&\frac{1}{d-1}\sum_{k\in\qq{1,\mu}}\left|\tP_{1 l_k}^2-\frac{\md^2 m_{sc}^{2\ell}}{(d-1)^{\ell}}\right|
=\frac{1}{d-1}\sum_{k\in \qq{1,\mu}\setminus J'}(\cdots)+\frac{1}{d-1}\sum_{k\in J'}(\cdots).
 \\
\leq& 
  2\n (d-1)^{\ell}2^{2\n+5}q^{2\ell+2}  +d(d-1)^{\ell} 2^{2\n+2}q^{2\ell+3}
  \leq 2^{2\n+2}(dq+16\n )q^2
\end{split}\end{align}
Combining \eqref{e:sumPcoef}, \eqref{e:stabilitytGT} and \eqref{tGTreplaceEr}, \eqref{newG11-P11} leads to
\begin{align}
 \tilde{G}_{11}-\tP_{11}
\notag &= \frac{\md^2\msc^{2\ell}}{(d-1)^{(\ell+1)}}\sum_{k\in\qq{1,\mu}}(\tGT_{\ta_k\ta_k}-\tP^{(\T)}_{\ta_k\ta_k}) +O_{\leq}\left(2^{2\n+2}(dq+16\n)q^2\max_{k\in \qq{1,\mu}}
  |\tGT_{\ta_k\ta_k}-\tP^{(\T)}_{\ta_k\ta_k}|
  \right)\\
\label{11term} &=\frac{\md^2\msc^{2\ell}}{(d-1)^{(\ell+1)}}\sum_{k\in\qq{1,\mu}}\left(\tGT_{\ta_k\ta_k}-\tP_{\ta_k\ta_k}(\cE_{r}(\ta_k,\ta_k,\tcGT))\right)+\cal E
 ,
\end{align}
where the error term is bounded
\begin{align*}
|\cal E|
&\leq 2^{2\n+2}(dq+16\n)q^2  (2^7K^3|\msc|q^r+2^{2\n+3}|\msc|q^{r+1})
+|\md|^2 \mu 2^{2\n+3}|\msc|q^{r+1}/(d-1)^{\ell+1}\\
&\leq 3\times2^{2\n+8}K^3|\msc|q^{r+1},
\end{align*}
provided that $\sqrt{d-1}\geq 2^6\n$.
Therefore, by Proposition \ref{tGconcentration}, we can simplify \eqref{11term} to get
\begin{align}\label{mk}
  \tilde{G}_{11}=\tP_{11}(\cE_{r}(1,1,\tcG))+ \frac{\md^2\msc^{2\ell}\mu}{(d-1)^{(\ell+1)}}(\I(\tcG)-\msc)+
  O_{\leq}\left(2^{2\n+10}K^3|\msc|q^{r+1}\right).
\end{align}
This finishes the proof of \eqref{G11bound}. 
\end{proof}

\subsection{Proof of \eqref{G1itreebound}}

\begin{proof}[Proof of \eqref{G1itreebound}]
For any vertex $x$ adjacent to $1$, by \eqref{Tterm} we have, 
\begin{align}
\begin{split}\label{G1x-P1x}
\tilde{G}_{1x}-\tP_{1x}
 &= \frac{1}{d-1}\sum_{k\in\qq{1,\mu}}\tP_{1 l_k}\tP_{xl_k}(\tGT_{\ta_k\ta_k}-\tP^{(\T)}_{\ta_k\ta_k})+ \frac{1}{d-1}\sum_{k\neq m\in \qq{1,\mu}} \tG_{1 l_k}\tP_{x l_m}(\tGT_{\ta_k\ta_m}-\tP^{(\T)}_{\ta_k\ta_m}) \\
 & +\frac{1}{d-1}\sum_{k\in \qq{1,\mu}} (\tG_{1l_k}-\tP_{1 l_k})\tP_{x l_k}(\tGT_{\ta_k\ta_k}-\tP^{(\T)}_{\ta_k\ta_k}).
\end{split}
\end{align}
For the last term on the right-hand side of \eqref{G1x-P1x},
\begin{align*}
\left|\frac{1}{d-1}\sum_{k\in \qq{1,\mu}} (\tG_{1l_k}-\tP_{1 l_k})\tP_{x l_k}(\tGT_{\ta_k\ta_k}-\tP^{(\T)}_{\ta_k\ta_k})\right|
\leq \Cw q^{r+2}q^{r+1}\sum_{k\in\qq{1,\mu}}|\tP_{xl_k}|
\leq \Cw |\msc|q^{r+\ell+3}.
\end{align*}
where in the first inequality, we used \eqref{betterest} for the first factor,
and \eqref{e:stabilitytGT} for the last factor;
in the second inequality, we used \eqref{sumtleaf1} for the case $x\in \T_1$.
For the second term on the right-hand side of \eqref{G1x-P1x}, we have
\begin{align*}
\left|\frac{1}{d-1}\sum_{k\neq m\in \qq{1,\mu}} \tG_{1 l_k}\tP_{x l_m}(\tGT_{\ta_k\ta_m}-\tP^{(\T)}_{\ta_k\ta_m})\right|
  &\leq \Cw q^{2\ell+1} \sum_{k\neq m\in \qq{1,\mu}}|\tGT_{\ta_k\ta_m}-\tP^{(\T)}_{\ta_k\ta_m}|\\
  &\leq \Cw \n'^2|\msc|q^{r+2\ell+1}
  \leq \Cw |\msc|q^{r+\ell+1},
\end{align*}
provided that $\n'^2q^{\ell}\ll 1$, where we used \eqref{sumtGtPij}. 
Therefore, they together lead to
\begin{align}\label{greendist1}
 \tilde{G}_{1x}-\tP_{1x}=\frac{1}{d-1}\sum_{k\in\qq{1,\mu}}\tP_{1 l_k}\tP_{xl_k}(\tGT_{\ta_k\ta_k}-\tP^{(\T)}_{\ta_k\ta_k})+O\left( |\msc|q^{r+\ell+1}\right),
\end{align}
where the implicit constant depends only on the excess $\n$ and $K$. Especially, if vertex $1$ has radius-$R$ neighborhood, then by Proposition \ref{greentree}
\begin{align*}
\tP_{1x}=-\frac{\md\msc}{\sqrt{d-1}},\quad \tP^{(\T)}_{\ta_k\ta_k}=m_{sc},\quad  \tP_{1l_k}=\frac{m_d(-m_{sc})^\ell}{(d-1)^{\ell/2}},\quad \tP_{xl_k}=\md\left(\frac{-m_{sc}}{\sqrt{d-1}}\right)^{\dist_{\tcG}(x,l_k)}
\end{align*}
for any index $k\in\qq{1,\mu}$. Thus averaging \eqref{greendist1} over all the vertices $x$ adjacent to $1$ (in the following, we write $x\sim 1$ when the vertex $x$ is adjacent to $1$), we get
\begin{align*}
\frac{1}{d}\sum_{x\sim 1}\tG_{1x}+\frac{\md\msc}{\sqrt{d-1}}
&=\frac{1}{d(d-1)}\sum_{k\in \qq{1,\mu}}(\tGT_{\ta_k\ta_k}-\tP^{(\T)}_{\ta_k\ta_k})\sum_{x\sim 1}\tP_{1 l_k}\tP_{xl_k}+O\left(|\msc|q^{r+\ell+1}\right)\\
&=\frac{\md(-\msc)^\ell}{d(d-1)^{\ell/2+1}}\sum_{k\in \qq{1,\mu}}(\tGT_{\ta_k\ta_k}-\tP^{(\T)}_{\ta_k\ta_k})\sum_{x\sim 1}\tP_{xl_k}+O\left(|\msc|q^{r+\ell+1}\right)\\
&=\frac{-\md^2\msc^{2\ell-1}(1+\msc^2)}{d(d-1)^{\ell+1/2}}\sum_{k\in \qq{1,\mu}}(\tGT_{\ta_k\ta_k}-\tP^{(\T)}_{\ta_k\ta_k})+O\left( |\msc|q^{r+\ell+1}\right)\\
&=\frac{-\md^2\msc^{2\ell-1}(1+\msc^2)}{\sqrt{d-1}}(\I(\tcG)-\msc)+  O_{\leq}\left( \frac{16(\log N)^{1/2+\delta} |\msc|q^{r+1}}{\sqrt{d(d-1)^{\ell}}}\right).
\end{align*}
In the third line, we used the fact that for any index $k\in\qq{1,\mu}$, among the $d$ children of vertex $1$, one of them is distance $\ell-1$ to the vertex $l_k$, and the others are distance $\ell+1$ to the vertex $l_k$. In the last line, we used Proposition \eqref{tGconcentration}, and $|\md^2\msc^{2\ell-1}(1+\msc^2)|\leq 4$. This finishes the proof of Proposition \ref{improvetG}.
\end{proof}

\section{Proof of main results}
\label{sec:pfmr}

In this section, we use the estimates established in the previous sections
to prove Theorem~\ref{thm:mr}.

\subsection{Summary of estimates}

By combining the propositions of the previous sections, we obtain the following sequence of propositions,
relating the sets
\begin{equation*}
\Omega^-(z,\ell) \subset \Omega(z,\ell) \subset  \Omega_1^+(z,\ell) \subset \bar\Omega \subset \GNd,\quad \Omega_1'(z,\ell) \subset \bar\Omega \subset \GNd,
\end{equation*}
defined in Section~\ref{sec:outline-structure}.
We also recall the parameters from Section~\ref{sec:outline-parameters}, assume that
\begin{align}\label{choosep1}  
\ell\in\qq{\ell_*,2\ell_*}, \quad r=2\ell+1,\quad
\end{align}
and \eqref{e:Mnprime}, namely that
\begin{align}\label{choosep2}  
\n'=\lfloor \log N \rfloor,\quad M=(d-1)^{9\ell}(\log N)^{\delta}.
\end{align}
Since, for $|z|\geq 2d-1$, the claim of Theorem~\ref{thm:mr} follows from Proposition~\ref{prop:betacomp},
it suffices to prove the claim of Theorem~\ref{thm:mr} on the following slightly smaller domain
\begin{equation}
\label{spectralDrecall}
\cal D^*\deq \ha{z\in \C_+: |z|\leq 2d,\quad \Im [z]\geq  \frac{(\log N)^{48\alpha+1}}{N},\quad |z\pm 2|\geq (\log N)^{-\alpha/2+1}},
\end{equation}
which is the intersection of $\cal D$ (as in \eqref{spectralD}) with $\{z\in \bC_+: |z|\leq 2d\}$.

\begin{proposition}[Initial estimates]\label{initialestimates}
Under the assumptions of Theorem \ref{thm:mr}, and the choices of parameters given in \eqref{choosep1} and \eqref{choosep2}, for $N\geq N(\n,d,\delta)$ large enough, we have
\begin{align}\label{barOmegalarge}
\P(\bar\Omega) = 1-o(N^{-\n+\delta}).
\end{align}
Moreover, for any $z\in \C_+$ such that $|z|\geq 2d-1$, we have $\bar \Omega\subset \Omega^-(z,\ell)$. 
\end{proposition}

\begin{proof}
The estimate \eqref{barOmegalarge} follows from Proposition~\ref{prop:structure},
and the inclusion $\bar \Omega\subset \Omega^-(z,\ell)$ from Proposition~\ref{prop:betacomp}.
\end{proof}

Given a graph $\cal G$ and a vertex $i$,
we resample the edge boundary of $\cB_\ell(i, \cal G)$ using switchings;
without loss of generality we assume $i=1$.
Denote the resampled graph by $T_{\bf S}(\cal G)$ (which depends on the choice of $i$);
$\bf S$ is the resampling data (whose distribution depends on $\cal G$).

\begin{proposition}[Stability under resampling] \label{prop:stability}
Under the assumptions of Theorem \ref{thm:mr}, and the choices of parameters given in \eqref{choosep1} and \eqref{choosep2}, for $z\in \cal D^*$, $N\geq N(\alpha,\n,d,\delta)$ large enough, and any $\cal G\in \Omega(z,\ell)$, the following holds. (\rn{1})  $\cal G\in \Omega_1^+(z,\ell)$. (\rn{2}) There exists a set $F(\cal G)\subset \sS(\cal G)$ with $\P_{\cG}(F(\cal G))= 1-o(N^{-\n+\delta})$
such that for any ${\bf S}\in F(\cal G)$ with $T_{\bf S}(\cal G)\in \bar \Omega$, we have  $T_{\bf S}(\cG) \in \Omega_1^+(z,\ell)$.
\end{proposition}

\begin{proof}
The first statement $\cG\in \Omega_1^+(z,\ell)$ follows from Proposition~\ref{prop:stabilityGT},
and the second statement follows from Proposition~\ref{prop:stabilitytGT} with $K=2$.
\end{proof}

\begin{proposition}[Improvement under resampling] \label{prop:stability+}
Under the assumptions of Theorem \ref{thm:mr}, and the choices of parameters given in \eqref{choosep1} and \eqref{choosep2}, for $z\in \cal D^*$, $N\geq N(\alpha,\n,d,\delta)$ large enough, and any $\cal G\in \Omega_1^+(z,\ell)$, there exists a set $F'(\cal G)\subset \sS(\cal G)$ with $\P_{\cG}(F'(\cal G))= 1-o(N^{-\n+\delta})$
such that for any ${\bf S}\in F'(\cal G)$ with $T_{\bf S}(\cal G)\in \bar \Omega$, we have  $T_{\bf S}(\cG) \in \Omega'_1(z,\ell)$.
\end{proposition}

\begin{proof}
The definition of the set $F'(\cG)$ and its properties are given in Proposition~\ref{tGconcentration}.
The final statement $T_{\bf S}(\cG)\in \Omega_1'(z,\ell)$ follows from Propositions~\ref{prop:stabilitytGT} and \ref{improvetG} by taking $K=2^{10}$.
\end{proof}

The improvement under resampling above applies to the switched graphs $T_{\bf S}(\cG)$.
However, by general properties of $T$, it implies an improvement on the original space of graphs.

\begin{proposition}[Improvement on original space] \label{prop:improvement2}
Under the assumptions of Theorem \ref{thm:mr}, and the choices of parameters given in \eqref{choosep1} and \eqref{choosep2}, for $z\in \cal D^*$, we have
\begin{equation} \label{e:Omegaimp1}
\P(\Omega(z,\ell) \setminus (\Omega(z,\ell)\cap \Omega_1'(z,\ell))) =o(N^{-\n+\delta}).
\end{equation}
\end{proposition}

\begin{proof}
By Propositions~\ref{initialestimates}--\ref{prop:stability+},
the conditions of Proposition~\ref{prop:resample} are satisfied with $q_0,q_1,q_2=o(N^{-\n+\delta})$,
and $\bar\Omega$ as in Section~\ref{sec:outline}, $\Omega =\Omega(z,\ell)$, $\Omega^+ = \Omega_1^+(z,\ell)$, and $\Omega' = \Omega_1'(z,\ell)$.
Therefore, Proposition~\ref{prop:resample} implies
\begin{align*}
\P(\Omega(z,\ell) \setminus (\Omega(z,\ell)\cap \Omega_1'(z,\ell)))= o(N^{-\n+\delta}),
\end{align*}
which was the claim.
\end{proof}

Clearly,
by the same argument or by symmetry,
\eqref{e:Omegaimp1} also holds with vertex $1$ replaced by any other vertex $i \in \qq{N}$.
In particular, for any graph in the intersection of the $\Omega_i'(z,\ell)$ over $i \in \qq{N}$,
we have the following improved estimates for the entries of its Green's function.

\begin{proposition}[Self-consistent equation] \label{prop:sce1}
Under the assumptions of Theorem \ref{thm:mr}, and the choices of parameters given in \eqref{choosep1} and \eqref{choosep2}, for any $z\in \cal D^*$ (as in \eqref{spectralDrecall}) and $N\geq N(\alpha,\n,d,\delta)$ large enough, we have
\begin{align}\label{IGequation}
 \I(\cal G)-\msc= \frac{d-2}{d-1}\md\msc^{2\ell+1}(\I(\cal G)-\msc)+ O \left(\frac{(\log N)^{1/2+\delta}|\msc|q^r}{(d-1)^{(\ell+1)/2}}\right).
\end{align}
\end{proposition}

\begin{proof} 
As noted above, the same statement as in Proposition \ref{prop:improvement2} holds with vertex $1$ replaced by any other vertex $i \in \qq{N}$.
On the union of the $\Omega_i'(z,\ell)$, the improvement then holds for all $i$ simultaneously, and 
by a union bound
\begin{align*}
\P(\Omega(z,\ell)\setminus \cap_{i\in \qq{N}} \Omega_i'(z,\ell))\leq \sum_{i=1}^{N}\P(\Omega(z,\ell)\setminus \Omega_i'(z,\ell))=o(N^{-\n+1+\delta}).
\end{align*}

For any graph $\cal G\in \cap_{i\in \qq{N}} \Omega_i'(z,\ell)$, by the definition of $\Omega_i'(z,\ell)$ (as in Section \ref{sec:outline}), we have
\begin{align}\label{Giiequation}
  G_{ii}(\cal G,z)=P_{ii}(\cE_r(i,i,\cal G),z)+ \frac{\md^2\msc^{2\ell}\mu}{(d-1)^{(\ell+1)}}(\I(\cal G)-\msc)
  +O_{\leq}\left(2^{2\n+40}|\msc| q^{r+1}\right),
\end{align}
and, for any $j\neq i$, we have the bound
\begin{align}\label{e:estimateforoffdiag}
  \left|G_{ij}(\cal G,z)-P_{ij}(\cE_r(i,j,\cal G),z)\right|
  \leq (\n+1)2^{2\n+44}|\msc| q^{r+1}.
\end{align} 
In the following, we derive an approximate self-consistent equation for $\I(\cal G)-\msc$. Using the Green's function identity \eqref{e:Schurixj},
notice that 
\begin{align}\label{IGest}
\I(\cal G)-\msc
=\frac{1}{Nd}\sum_{(i,j)\in \vec E}(G_{jj}^{(i)}-\msc)
&=\frac{1}{Nd}\sum_{(i,j)\in \vec E} \pa{G_{jj}-\frac{G_{ij}G_{ij}}{G_{ii}}-\msc}
\nonumber\\
&=\frac{1}{N}\sum_{i}\pa{G_{ii}-\frac{1}{d}\sum_{j: j\sim i}\frac{G_{ij}G_{ij}}{G_{ii}}-\msc},
\end{align}
where $\vec E$ is the set of oriented edges of $\cG$, and where here $j\sim i$ means that the vertices $i$ and $j$ are adjacent to each other.
Since $\cal G\in \bar\Omega$, at least $N-N^\delta$ of the vertices of $\cG$ have radius-$R$ tree neighborhoods.
The contribution to $\I(\cal G)-\msc$ from those vertices which do not have radius-$R$ tree neighborhoods is $O(N^{\delta-1})$.
For any vertex $i$ that has radius-$R$ tree neighborhood, by the definition of $\Omega_i'(z,\ell)$,
\begin{align}
\label{Giiest}G_{ii}-\md&=\md^2\msc^{2\ell}\frac{d}{d-1}(\I(\cal G)-\msc)+  O\left(\frac{(\log N)^{1/2+\delta}|\msc|q^r}{(d-1)^{(\ell+1)/2}}\right),\\
\label{Gijest} \frac{1}{d}\sum_{j\sim i}G_{ij}+\frac{\md\msc}{\sqrt{d-1}}
&=\frac{-\md^2\msc^{2\ell-1}(1+\msc^2)}{\sqrt{d-1}}(\I(\cG)-\msc)+  O\left(\frac{(\log N)^{1/2+\delta}|\msc|q^{r+1}}{(d-1)^{(\ell+1)/2}}\right).
\end{align}
Also, by the stability estimate Claim \ref{newrigid}, for any vertex $i$ with radius-$R$ tree neighborhood, and vertex $j$ adjacent to $i$, 
we have $|G_{ii}-\md|=O(|\msc|q^r)$ and $|G_{ij}-\md\msc/\sqrt{d-1}|=O(|\msc|q^r)$,
where the implicit constant depends only on $\n$.
It follows that
\begin{align}\begin{split}
\frac{G_{ij}G_{ij}}{G_{ii}}-\frac{\md\msc^2}{d-1}
&=\frac{-\frac{2\md^2\msc}{\sqrt{d-1}}(G_{ij}+\frac{\md\msc}{\sqrt{d-1}})-\frac{\md^2\msc^2}{d-1}(G_{ii}-\md)+\md(G_{ij}+\frac{\md\msc}{\sqrt{d-1}})^2}{\md G_{ii}}\\
\label{Gijfirst}
&=-\frac{2\msc}{\sqrt{d-1}}(G_{ij}+\frac{\md\msc}{d-1})-\frac{\msc^2}{d-1}(G_{ii}-\md)+O(q^{2r}).
\end{split}
\end{align}
Combining \eqref{Giiest}--\eqref{Gijfirst}, we get
\begin{align*}
G_{ii}-\frac{1}{d}\sum_{j: j\sim i}\frac{G_{ij}G_{ij}}{G_{ii}}-\msc
=\frac{d-2}{d-1}\msc^{2\ell+1}\md(\I(\cal G)-\msc)+O  \left(\frac{(\log N)^{1/2+\delta}|\msc|q^r}{(d-1)^{(\ell+1)/2}}\right),
\end{align*}
for all vertices $i$ which have radius-$R$ tree neighborhoods.
Then averaging over $i\in \qq{N}$, we obtain \eqref{IGequation}, as claimed.
\end{proof}

The equation \eqref{IGequation} implies
\begin{align}\label{IGequation-bis}
 \I(\cG)-\msc= \pa{1-\frac{d-2}{d-1}\md\msc^{2\ell+1}}^{-1} O \left(\frac{(\log N)^{1/2+\delta}|\msc|q^r}{(d-1)^{(\ell+1)/2}}\right),
\end{align}
provided that the term in the first bracket does not vanish.
To use this equation to show that the left-hand side is small, we require a lower bound on the term in the first bracket
on the right-hand side.
Since $1-(d-2)\md\msc^{2\ell+1}/(d-1)$ may be zero on the spectral domain $\cal D^*$, 
such a bound only holds on an $\ell$-dependent subset of the spectral domain, which we now define.
(In Section~\ref{sec:pfmr-decomp}, we will use the flexibility in the choice of $\ell \in \qq{\ell_*,2\ell_*}$ to recover the entire spectral domain.)

First, we define the Joukowsky transform $\phi$ to be
the holomorphic bijection from the upper half unit disk $\mathbb{D}_+$ to the upper half plane $\C_+$
given by
\begin{equation*}
\phi: w \in \mathbb{D}_+  \mapsto -\left(w+w^{-1}\right)\in \C_+
  .
\end{equation*}
It is the functional inverse of $z \mapsto \msc(z)$, i.e.\ $\msc(\phi(w))=w$.
For any $\ell\in\qq{\ell_*, 2\ell_*}$ as in \eqref{e:constchoice}, we define 
the small error parameter
\begin{align}\label{e:defeps}
\epsilon_\ell:=\frac{(\log N)^{1/2+2\delta}}{(d-1)^{(\ell+1)/2}}\ll (\log N)^{1-\alpha/2},
\end{align}
as well as the sets $\tilde \Lambda_\ell\subset \mathbb D_+$ and $\Lambda_\ell \subset \C_+$ by
\begin{align}
\label{lambdadef}
\tilde \Lambda_{\ell}\deq\left\{\msc(z): z\in\C_+, \left|1-\frac{d-2}{d-1}\msc^{2\ell+1}{ (z)}\md{ (z)}\right|\geq \epsilon_\ell\right\},
\qquad
\Lambda_\ell\deq \phi(\tilde \Lambda_\ell).
\end{align}

\begin{proposition}[Self-consistent equation] \label{prop:sce}
Under the assumptions of Theorem \ref{thm:mr}, and the choices of parameters given in \eqref{choosep1} and \eqref{choosep2}, for any $z\in \cal D^*$ (as in \eqref{spectralDrecall}) and $N\geq N(\alpha,\n,d,\delta)$ large enough, we have
\begin{align*}
\P(\Omega(z,\ell)\setminus \cap_{i\in \qq{N}} \Omega_i'(z,\ell)) =o(N^{-\n+1+\delta}).
\end{align*}
Moreover, for $z\in \cal D^*\cap \Lambda_\ell$ and any $\cG \in  \bigcap_{i\in \qq{N}} \Omega_i'(z,\ell) $,
the normalized Green's function of $\cG$ satisfies, for any $i,j\in \qq{N}$,
\begin{align*}
 \left|G_{ij}(\cal G,z)-P_{ij}(\cE_r(i,j,\cal G),z)\right|
 \leq  (\n+1)2^{2\n+44}|\msc| q^{r+1},
 \end{align*} 
 where $r=2\ell+1$.
\end{proposition}

\begin{proof}
Let $z\in \cal D^*\cap \Lambda_\ell$. Then by the definition of the set $\tilde \Lambda_\ell$ in \eqref{lambdadef}, we have
\begin{align*}
 \left|1-\frac{d-2}{d-1}\msc^{2\ell+1}\md\right|\geq \frac{(\log N)^{1/2+2\delta}}{(d-1)^{(\ell+1)/2}},
\end{align*}
and \eqref{IGequation-bis} implies
\begin{align*}
|\I(\cal G)-\msc|=O\left(|\msc|q^r (\log N)^{-\delta}\right),
\end{align*}
where the implicit constant depends only on $\n$.
Plugging the above expression into \eqref{Giiequation}, we get 
\begin{align*} 
G_{ii}(\cal G,z)=P_{ii}(\cE_r(i,i,\cal G),z)
 +O_{\leq }\left((1+2^{2\n+40})|\msc| q^{r+1}\right),
 \end{align*}
 for $N$ large enough.  This finishes the proof of Proposition \ref{prop:sce} by combining with \eqref{e:estimateforoffdiag}.
\end{proof}

\subsection{Decomposition of the spectral domain}
\label{sec:pfmr-decomp}

The following lemma gives a precise description of the sets $\tilde\Lambda_\ell$,
stating that, except for two small regions near $\pm1$,
the half disk $\mathbb D_+$ is contained in the union of the sets $\tilde\Lambda_\ell$, i.e., in $\cup_{\ell\in\qq{\ell_*, 2\ell_*}}\tilde\Lambda_\ell$.
To be precise, we define the spectral domains
\begin{align}
\label{spectralDell}
\cal D_{\ell}
&=\{z\in \C_+: |z|\leq 2d,\quad \Im [z]\geq (d-1)^{24\ell}\log N/N ,\quad |z\pm 2|\geq 4\epsilon_{\ell}\},
\\
\label{spectraltD}
\tilde {\cal D}_{\ell}
&=\mathbb D_+\setminus\left\{w=e^{\ii\theta}r\in \mathbb D_+: |\theta|\leq \epsilon_\ell, 1-\epsilon_\ell\leq r<1\right\}.
\end{align}

\begin{lemma}{\label{domainp}}
For any $\ell\in\qq{\ell_*,2\ell_*}$, define $\tilde\Lambda_{\ell}$ as in \eqref{lambdadef}.
Then ${\mathbb D}_+\setminus\tilde\Lambda_\ell$ is contained in
\begin{multline*}
\left\{w=e^{\ii\theta}r: 0<\theta\leq \epsilon_\ell, 1-\epsilon_\ell\leq r<1\right\} 
\cup
\left\{w= e^{\pi \ii}e^{\ii \theta}r: -\epsilon_\ell\leq \theta<0, 1-\epsilon_\ell\leq r<1\right\}
\cup \\
\bigcup_{k=1}^{\ell}\left\{w=e^{\frac{k\pi \ii}{\ell+1}}e^{\ii\theta}r: |\theta|\leq \frac{2\pi}{d(\ell+1)}, 0<r< 1\right\}.
\end{multline*}
As a consequence, for any $\theta_0\in (0,\pi)$, there exists some $\ell\in \qq{\ell_*, 2\ell_*}$ such that 
\begin{align}\label{e:tDsubset}
\tilde{\cal D}_{\ell}\cap \{w=e^{\ii\theta_0}r: 0< r< 1\}\subset \tilde\Lambda_\ell.
\end{align}
\end{lemma}
\begin{proof}
By the definition of the set $\tilde\Lambda_\ell$, its complement is  
\begin{equation*}
\left|1-\frac{d-2}{d-1}\msc^{2\ell+1}\frac{\msc}{1-\msc^2/(d-1)}\right|< \epsilon_\ell.
\end{equation*}
This implies that
\begin{equation*}
\left|1-\frac{\msc^2}{d-1}-\frac{d-2}{d-1}\msc^{2\ell+2}\right|< 2\epsilon_\ell,
\end{equation*}
and therefore
\begin{equation}\label{mscdomain}
\left|1-\msc^{2\ell+2}\right|< \left|\frac{\msc^2}{d-1}-\frac{\msc^{2\ell+2}}{d-1}\right|+2\epsilon_\ell<\frac{2}{d-1}+2\epsilon_\ell.
\end{equation}
From direct computation, for any $0\leq r\leq 1$ and $\theta \in [-\pi, \pi]$, we have the following simple estimate:
\begin{align}
\label{eitr}\frac{1-r}{2}+\frac{\sqrt{r}|\theta|}{\pi}\leq |1-e^{\ii\theta}r|\leq (1-r) +|\theta| .
\end{align}
Therefore \eqref{mscdomain} implies
\begin{align*}
 \msc\in \bigcup_{k=0}^{\ell+1}\left\{w=e^{\frac{k\pi \ii}{\ell+1}}e^{\ii\theta}r\in \mathbb{D}_+: |\theta|\leq \frac{2\pi}{d(\ell+1)}, 0< 1-r^{2(\ell+1)}\leq \frac{5}{d}\right\}.
\end{align*}
We have better estimates if $\msc=e^{\ii\theta}r$ or $\msc=e^{\ii\pi}e^{\ii\theta}r$,
for some $|\theta|\leq \frac{2\pi}{d(\ell+1)}$ and  $0<1-r^{2(\ell+1)}\leq \frac{5}{d}$.
In this case, on the complement of $\tilde\Lambda_\ell$, 
\begin{align*}
\left|\frac{d-2}{d-1}(1-\msc^{2\ell+2})\right|
\leq \left|\frac{1-\msc^2}{d-1}\right|+ 2\epsilon_\ell.
\end{align*}
Combining the above expression with \eqref{eitr}, we get
\begin{align*}
\frac{d-2}{d-1} \left(\frac{1-r^{2(\ell+1)}}{2}+\frac{2(\ell+1)r^{(\ell+1)}|\theta|}{\pi}\right)
\leq \frac{1}{d-1}\left(1-r^2+2|\theta|\right)+2\epsilon_\ell.
\end{align*}
It follows that $|\theta|\leq \epsilon_\ell$ and $1-\epsilon_\ell\leq r<1$. This finishes the proof of the first statement.

For the second statement, if $\theta_0\in (0, (1-\frac{2}{d})\frac{\pi}{\ell_*+1})\cup ((1+\frac{2}{d})\frac{\ell_*\pi}{\ell_*+1},\pi)$, then $\{z=e^{\ii \theta_0}r: 0\leq r\leq 1\}\cap \tilde{\cal D}_{\ell_*}\subset \tilde \Lambda_{\ell_*}$. In the following we consider the case, $\theta_0\in [(1-\frac{2}{d})\frac{\pi}{\ell_*+1}, (1+\frac{2}{d})\frac{\ell_*\pi}{\ell_*+1}]$. We use the convention that $(a\mod \pi )\in [-\pi/2, \pi/2)$, for any $a\in \R$.  If we can find some $\ell\in \qq{\ell_*, 2\ell_*}$, such that $((\ell+1) \theta_0 \mod \pi)\in [-\pi/2,-\pi/8]\cup [\pi/8, \pi/2) $, then there exists some integer $k$ such that
\begin{align*}
-\frac{3\pi}{8(\ell+1)}\leq \left|\theta_0 -\frac{k\pi}{\ell+1}-\frac{\pi}{2(\ell+1)}\right|\leq \frac{3\pi}{8(\ell+1)},
\end{align*}
and thus $\{z=e^{\ii \theta_0}r: 0< r< 1\}\in \tilde\Lambda_{\ell}$. In the following we prove such $\ell$ exists. By symmetry we assume $\theta_0\in [(1-\frac{2}{d})\frac{\pi}{\ell_*+1},\frac{\pi}{2}]$. We consider the following numbers,
\begin{align}\label{sequence}
(\ell_*+1)\theta_0 \mod \pi, (\ell_*+2)\theta_0\mod \pi,\dots, (2\ell_*+1)\theta_0\mod\pi.
\end{align}
If $((\ell_*+1)\theta_0 \mod \pi ) \in [-\pi/2,-\pi/8]\cup [\pi/8, \pi/2)$, then we can take $\ell=\ell_*$. Otherwise, we assume $((\ell_*+1)\theta_0\mod \pi) \in (-\pi/8,\pi/8)$. Since $(2\ell_*+1)\theta_0-(\ell_*+1)\theta_0=\ell_*\theta_0 \geq (1-\frac{2}{d})\frac{\ell_*\pi}{\ell_*+1}\geq\frac{\pi}{2}$, the above sequence \eqref{sequence} can not all stay in the interval $(-\pi/8,\pi/8)$. Say $(\ell+1)\theta_0 \mod \pi$ is the first number  in the above sequence  which is not in $(-\pi/8,\pi/8)$. We can take this $\ell$, then $(\ell\theta_0\mod\pi)\in(-\pi/8,\pi/8)$, and $((\ell+1)\theta_0\mod \pi) \in [-\pi/2, -3\pi/8)\cup [\pi/8, \pi/2)$. This finishes the proof. 
\end{proof}

\begin{lemma}
\begin{enumerate}
\item
For the choice of parameters in \eqref{choosep1}--\eqref{choosep2}, for any $z\in \cal D_\ell$,
all of the conditions in 
Propositions~\ref{prop:betacomp}, \ref{prop:stabilityGT}, \ref{prop:stabilitytGT} and  \ref{improvetG} are satisfied
for $K \in \{2,2^{10}\}$; i.e.,
\begin{align*}
\sqrt{d-1}\geq \max\{(\n+1)^2 2^{2\n+10}, 2^8(\n+1)K\},\quad \n'^2q^\ell\ll1, \quad \sqrt{N\Im[z]}q^{3r+2}\gg M.
\end{align*}
\item For any $\ell\in \qq{\ell_*,2\ell_*}$, we have
\begin{align}\label{e:Dsub1}
\cal D^*\subset \cal D_{\ell}\subset \phi(\tilde{\cal D}_{\ell}).
\end{align}
\end{enumerate}
\end{lemma}

\begin{proof}
It is straightforward to check that (\rn1) holds. In the following, we therefore only prove (\rn2).
For this, notice that $(d-1)^{24\ell}\leq (d-1)^{48\ell_*}\leq (\log N)^{48\alpha}$, and that in combination with \eqref{e:defeps},
it follows that $\cal D^*\subset \cal D_\ell$.
For the second inclusion in \eqref{e:Dsub1}, observe that,
for any $w=e^{\ii \theta}r$ such that $0< \theta \leq \epsilon_{\ell}$ and $1-\epsilon_{\ell} \leq r\leq 1$, we have
\begin{align*}
|\phi(w)+2|=\left|e^{\ii \theta}r+\frac{1}{e^{\ii\theta}r}-2\right|< 4\epsilon_{\ell},
\end{align*} 
and that we have similar estimates for $w=e^{\ii\pi}e^{\ii \theta}r$ with $-\epsilon_{\ell}\leq \theta\leq 0$ and $1-\epsilon_{\ell} \leq r\leq 1$. Therefore,
\begin{align*}
\{z\in \C_+: |z\pm2|\geq 4\epsilon_{\ell}\}\subset \phi (\tilde {\cal D}_{\ell}),
\end{align*}
and $\cal D_\ell\subset \phi(\tilde{\cal D}_\ell)$ follows. This finishes the proof of \eqref{e:Dsub1}.
\end{proof}

\subsection{Proof of Theorem~\ref{thm:mr}}

We define a lattice on $\mathbb{D}_+$ by
\begin{align*}
\tilde{L}:=\left\{e^{\ii \theta} r\in \mathbb{D}_+: \theta\in\frac{\pi\Z}{N^3}, r\in\frac{\Z}{N^3}\right\}.
\end{align*}
The image of $\tilde L$ under the Joukowsky transform defines a discrete approximation of $\cal D^*$ by
 \begin{align}\label{defLattice}
L:=\phi(\tilde L)\cap \cal D^*.
\end{align}
Notice that $\cal D^*$ can indeed be well approximated by $L$, in the sense that for any $z\in \cal D^*$ there is some $z'\in L$ such that $|z-z'|=(\log N)^{O(1)}/N^3$.
Therefore, by the following claim, we only need to prove Theorem~\ref{thm:mr} for $z\in L$.
The claim is a consequence of the Lipschitz property of Green's function.

\begin{claim}\label{c:LipschitzG}
For any $\ell\in\qq{\ell_*,2\ell_*}$, and $z,z'\in \cal D^*$ with $|z-z'|=(\log N)^{O(1)}/N^3$, we have
\begin{align*}
\Omega^-(z,\ell)\subset \Omega(z',\ell).
\end{align*}
\end{claim}
\begin{proof}
For any graph $\cG\in \Omega^-(z,\ell)$, the Green's function of its normalized adjacency matrix satisfies
\begin{align*}
|G_{ij}(z)-G_{ij}(z')|\leq |z-z'|\sum_{m=1}^N|G_{im}(z)G_{mj}(z')|\leq (\log N)^{O(1)}/N,
\end{align*}
where we used $|z-z'|=(\log N)^{O(1)}/N^3$, $|G_{im}(z)|=O(1)$ from the definition \eqref{defOmega-} of $\Omega^-(z,\ell)$ and \eqref{e:boundPiimsc},
as well as the trivial bound $|G_{mj}(z')|\leq 1/\eta\leq N$.
Moreover, the same estimate holds for $|P_{ij}(\cE_r(i,j,\cal G),z)-P_{ij}(\cE_r(i,j,\cal G),z')|$. As a result, 
\begin{align*}
  \left|G_{ij}(z')-P_{ij}(\cE_r(i,j,\cal G),z')\right|
  =\left|G_{ij}(z)-P_{ij}(\cE_r(i,j,\cal G),z)\right| +(\log N)^{O(1)}/N\leq |\msc|q^r,
\end{align*}
and the claim follows.
\end{proof}

\begin{claim}\label{c:Omegainclusion}
For any $\ell\in\qq{\ell_*,2\ell_*}$, we have
\begin{align*}
\Omega^-(z,\ell)\subset \Omega^-(z,\ell_*),
\end{align*}
provided that $\sqrt{d-1}\geq 2^{2\n+3}$.
\end{claim}
\begin{proof}
Let $\cG\in \Omega^-(z,\ell)$. Then, by \eqref{e:compatibility} in Proposition~\ref{boundPij}, 
\begin{align*}
 &\left|G_{ij}(z)-P_{ij}(\cE_{r_*}(i,j,\cal G),z)\right|\\
 &\leq \left|G_{ij}(z)-P_{ij}(\cE_{r}(i,j,\cal G),z)\right|
 +\left|P_{ij}(\cE_{r_*}(i,j,\cal G),z)-P_{ij}(\cE_{r}(i,j,\cal G),z)\right|\\
 &\leq \frac{1}{2}|\msc|q^{r}+\delta_{r\neq r_*}2^{2\n+3}|\msc|q^{2\ell_*+2}\leq \frac{1}{2}|\msc|q^{r_*},
 \end{align*}
 provided that $\sqrt{d-1}\geq 2^{2\n+3}$. Therefore, $\cG\in \Omega^{-}(z,\ell_*)$, as claimed.
\end{proof}

\begin{proof}[Proof of Theorem \ref{thm:mr}]
For any $\theta_0\in \frac{\pi \Z}{N^3}\cap (0,\pi)$, the Joukowsky transform $\phi$ sends the ray $\{w=e^{\ii \theta_0}r: 0<r<1\}$ to a branch of some hyperbola.
With $\theta_0$ fixed, we consider the set
\begin{align*}
\left\{r\in \frac{\Z}{N^3}: \phi(e^{\ii \theta_0}r)\in \cal D^*\right\}=\left\{\frac{k_0}{N^3}, \frac{k_0+1}{N^3},\frac{k_0+2}{N^3},\dots, \frac{k_1}{N^3}\right\},
\end{align*}
for some $0<k_0\leq k_1<N^3$, and denote
\begin{align*}
z_k=\phi\left(\frac{e^{\ii \theta_0}k}{N^3}\right), \quad 1\leq k\leq N^3.
\end{align*}
One can check that
$k_0/N^3\geq 1/(3d)$, $|z_{k_0}|\geq 2d-1$, and 
$|z_{k+1}-z_k|\leq 10d^2/N^3$ for $k_0\leq k\leq k_1.$

By Proposition \ref{domainp}, there exists some $\ell\in\qq{\ell_*, 2\ell_*}$ such that ${\tilde {\cal D}_{\ell}}\cap \{e^{\ii \theta_0} r: 0< r< 1\}\subset\tilde \Lambda_\ell$. Therefore, combining with \eqref{e:Dsub1}, we know that $z_{k_0}, z_{k_0+1},\dots, z_{k_1}\in \Lambda_\ell$.
By Proposition \ref{initialestimates}, $\bar\Omega\subset \Omega^-(z_{k_0},\ell)$, and 
\begin{align}\label{initialloss}
\P(\Omega^-(z_{k_0},\ell))= 1-o(N^{-\n+\delta}).
\end{align}

For any $k_0\leq k\leq k_1-1$, it follows from Claim \ref{c:LipschitzG} that
\begin{align}
\label{O-O}\Omega^-(z_k,\ell)\subset \Omega(z_{k+1},\ell).
\end{align}
By Proposition \ref{prop:sce}, we have
\begin{align}
\label{O/O}\P(\Omega(z_{k+1},\ell)\setminus \cap_{i\in \qq{N}} \Omega_i'(z_{k+1},\ell)) =o(N^{-\n+1+\delta}),
\end{align}
and 
\begin{align}\label{OinO}
\cap_{i\in \qq{N}} \Omega_i'(z_{k+1},\ell) \subset \Omega^-(z_{k+1},\ell),
\end{align}
provided that $\sqrt{d-1}\geq (\n+1)2^{2\omega+45}$. It follows from combining \eqref{O-O}--\eqref{OinO} that
\begin{align}\label{losslittle}
\P(\Omega^-(z_{k},\ell)\setminus \Omega^-(z_{k+1},\ell))=o(N^{-\n+1+\delta}).
\end{align}

By definition, on the set $\cap_{k=k_0}^{k_1}\Omega^-(z_k,\ell)$, we have
\begin{align*}
  \left|G_{ij}(z)-P_{ij}(\cE_{r}(i,j,\cal G),z)\right|\leq \frac{1}{2}|\msc|q^{r}, 
\end{align*}
for any $z=z_{k_0}, z_{k_0+1},\dots, z_{k_1}$.
Moreover, combining \eqref{initialloss} and \eqref{losslittle}, the above holds with high probability,
\begin{align}\label{e:hp1}
\P(\cap_{k=k_0}^{k_1}\Omega^-(z_k,\ell))= 1-(k_1-k_0+1)o(N^{-\n+1+\delta})=1-o(N^{-\n+4+\delta}).
\end{align}
Combining with Claim~\ref{c:Omegainclusion}, the estimate \eqref{e:hp1} implies
\begin{align*}
\P(\cap_{k=k_0}^{k_1}\Omega^-(z_k,\ell_*))= 1-o(N^{-\n+4+\delta}).
\end{align*}

The above argument is independent of $\theta_0\in \frac{\Z}{N^3}\cap (0,\pi)$.
Thus, by a union bound, with probability at least $1-o(N^{-\n+7+\delta})$, uniformly in $z\in L$, we have
\begin{align}\label{finalestimate}
  \left|G_{ij}(z)-P_{ij}(\cE_{r_*}(i,j,\cal G),z)\right|\leq \frac{1}{2}|\msc|q^{r_*}.
\end{align}
Since for any $z\in \cal D^*$, there is some $z'\in L$ such that $|z-z'|=(\log N)^{O(1)}/N^3$,
the Lipschitz property of Green's function, Claim \ref{c:LipschitzG},
implies that the above estimate \eqref{finalestimate} holds uniformly for $z\in \cal D^*$,
with possibly a slightly larger constant:
\begin{align}\label{finalestimates}
  \left|G_{ij}(z)-P_{ij}(\cE_{r_*}(i,j,\cal G),z)\right|\leq |\msc|q^{r_*}.
\end{align}
This is  \eqref{e:locallaw}, and thus the proof of Theorem~\ref{thm:mr} is complete.
\end{proof}

\appendix
\section{Combinatorial estimates for random regular graphs}

\subsection{Proof of Proposition~\ref{prop:structure}}
\label{app:structure}

\begin{proof}[Proof of \eqref{e:structure1}]
For $\omega=1$, a proof of the statement is given in \cite[Lemma~2.1]{MR2667423} or \cite[Lemma~7]{Bord15},
for example.
The more general statement follows from the same proof. More precisely, in \cite[(2.4)]{MR2667423}, it is shown that
for any $i \in \qq{N}$, 
the excess $X_i$ in $\cB_R(i,\cG)$ is stochastically dominated by a binomial random variable with
$n=d(d-1)^R$ trials and success probability $p=d(d-1)^{R-1}/N$.
It follows that
\begin{equation*}
  \P(X_i > \n) = O\pbb{\binom{d(d-1)^R}{\n+1} \pa{\frac{d(d-1)^{R-1}}{N}}^{\omega+1}}
  = O(N^{-\omega-1} (d-1)^{2R (\omega+1)})
  = O(N^{-\omega-1+2\kappa (\omega+1)}).
\end{equation*}
By a union bound, and using $\kappa < \delta/(2\omega+2)$, therefore
\begin{align*}
  \P(X_i > \n \text{ for some $i \in \qq{N}$}) = O(N^{-\omega+2\kappa(\omega+1)}) =
  o(N^{-\omega+\delta}),
\end{align*}
as claimed.
\end{proof}

\begin{proof}[Proof of \eqref{e:structure2}]
The claim follows from \cite[Theorem 4]{MR2097332}, for example.
Indeed,  if $\cB_R(i,\cG)$ is not a tree, then some edge in $\cB_R(i,\cG)$ must lie on a cycle of length at most $k=2R$,
and any edge that lies on such a cycle is in $\cB_R(j,\cG)$ for at most $2(d-1)^R$ vertices $j \in \qq{N}$. Thus
\begin{equation}
  |\{i \in \qq{N}: \text{$\cB_R(i,\cG)$ is not a tree}\}| \leq 2(d-1)^R X = 2N^\kappa X
\end{equation}
where $X$ is the number of edges in $\cG$ which lie on cycles of length at most $k$.
With $k=2R$ and $A\geq 2$ in \cite[Theorem 4]{MR2097332}, we obtain
\begin{equation}
  \P(X=M)
  \leq (e^{5(A-1)}A^{-5A})^{(d-1)^k} \leq 
  e^{-c(d-1)^k}
  = e^{-cN^{2\kappa}}
\end{equation}
if $M=20Ak(d-1)^k$, where $c$ is some universal constant.
Let $M_0=40 k (d-1)^{k} \leq 80 R (d-1)^{2R} \leq 80RN^{2\kappa}$.
By a union bound, then
\begin{equation}
  \P(X \geq M_0)
  \leq N e^{-cN^{2\kappa}}
  \leq e^{-cN^{2\kappa}/2}.
\end{equation}
Thus, with probability $1-e^{-cN^{2\kappa}/2}$, and using $\kappa < \delta/(2\n+2)\leq  \delta/4$, $R=\lfloor \kappa \log_{d-1} N\rfloor\ll N^\kappa$, we have
\begin{equation}
  |\{i \in \qq{N}: \text{$\cB_R(i,\cG)$ is not a tree}\}| \leq 2 N^\kappa X\leq 2N^\kappa M_0 \leq 160RN^{3\kappa} \leq N^\delta,
\end{equation}
which is better than claimed.
\end{proof}

\subsection{Proof of Proposition~\ref{prop:numberpath}}
\label{app:numberpath}

\begin{proof}[Proof of \eqref{pathnum}]
We fix vertices $i,j$ and an integer $k$.
Given a graph $\cG$,
we denote by $t_k(\cG)$ the total number of non-backtracking paths from $i$ to $j$ of length less than $\dist_{\cG}(i,j)+k$.
We modify the graph $\cG$ in the three steps
such that, in each step, $t_k$ does not decrease,
and the excess remains the same.
Then it suffices to prove \eqref{pathnum} for the final graph.
 
\smallskip\noindent
\emph{Step 1.}
Given an edge $e = \{x,y\} \in \cG$ that is not a self-loop and not on a geodesic from $i$ to $j$,
we shrink the edge $e$ to a point (remove $e$ and identify its incident vertices), and so obtain a new graph $\cG'$. 
There is a bijection between the oriented edges of the graph $\cG\setminus\{e\}$ and those of the graph $\cG'$.

Now we show that the total number of non-backtracking paths from $i$ to $j$ of length less than
$\dist_{\cG}(i,j)+k=\dist_{\cG'}(i,j)+k$ in $\cG'$ is at least $t_k$.
Let $(\vec e_1, \vec e_2,\vec e_3, \dots)$ be any non-backtracking path from $i$ to $j$ in the graph $\cG$
that is not a geodesic.
If some $\vec e_{\beta}$ is $(x,y)$ or $(y,x)$, we remove it from the path and view the remaining part 
as a path from $i$ to $j$ in the graph $\cG'$. In this way we get a shorter path from $i$ to $j$ in  $\cG'$.
The new path is still non-backtracking, and we can recover the original path in $\cG$ from the new path in $\cG'$ since $x\neq y$.
Therefore the total number of non-backtracking paths from $i$ to $j$ of length less than
$\dist_{\cG}(i,j)+k=\dist_{\cG'}(i,j)+k$ in $\cG'$ is at least $t_k$.

We repeat this procedure with edges $e$ (not on a geodesic) chosen arbitrarily as long as possible.
This creates a new graph $\cal G_1$ (which may depend on the choice of edges in the steps) with vertex set $\bG_1$.
By construction, the edges in $\cG_1$ are either self-loops or on geodesics from $i$ to $j$.
Thus the vertex set of $\cG_1$ decomposes into
\begin{align}\label{decomp}
  \bG_1=\bV_0\cup \bV_1\cup \cdots \cup \bV_{\dist_{\cG_1}(i,j)},
  \quad \text{where $\bV_m:=\{v\in \bG_1: \dist_{\cG_1}(i,v)=m\}$,}
\end{align}
or equivalently, $\bV_{\dist_{\cG_1}(i,j)-m}:=\{v\in \bG_1: \dist_{\cG_1}(v,j)=m\}$.
In particular, $\bV_0=\{i\}$ and $\bV_{\dist_{\cG_1}(i,j)}=\{j\}$.
Any edge in $\cG_1$ is either a self-loop or has
one vertex in $\bV_m$ and the other vertex in $\bV_{m+1}$, for some $m\in\qq{0,\dist_{\cG_1}(i,j)-1}$. 
The excess of $\cG_1$ is $\omega$.
 
\smallskip\noindent
\emph{Step 2.}
Given two edges $e=\{v_m,v_{m+1}\}$ and $e'=\{v_m,v_{m+1}'\}$ with $v_m\in \bV_m$ and $v_{m+1}\neq v_{m+1}'\in \bV_{m+1}$,
we remove the edge $e'$ and identify $v_{m+1}'$ with $v_{m+1}$, thus creating a new graph $\cG'_1$.
Again there is a bijection between the oriented edges of the graph $\cG_1\setminus\{e'\}$ and those of the graph $\cG'_1$.

Now we show that the total number of non-backtracking paths from $i$ to $j$ of length less than
$\dist_{\cG}(i,j)+k=\dist_{\cG'_1}(i,j)+k$ in $\cG_1'$ is at least $t_k$.
Let $(\vec e_1, \vec e_2,\vec e_3,\dots)$ be any non-backtracking path from $i$ to $j$ in the graph $\cG_1$.
If $\vec e_{\beta}=(v_{m+1}',v_{m})$ and $\vec e_{\beta+1}\neq (v_m,v_{m+1})$, we replace $\vec e_{\beta}$ by $(v_{m+1},v_m)$; if $\vec e_{\beta}=(v_{m+1}',v_{m})$ and $\vec e_{\beta+1}= (v_m,v_{m+1})$, we remove both $\vec e_{\beta}$ and $\vec e_{\beta+1}$; if $\vec e_{\beta}=(v_{m},v_{m+1}')$ and $\vec e_{\beta-1}\neq (v_{m+1},v_{m})$, we replace $\vec e_{\beta}$ by $(v_m,v_{m+1})$; if $\vec e_{\beta}=(v_{m},v_{m+1}')$ and $\vec e_{\beta-1}= (v_{m+1},v_{m})$, we remove both $\vec e_{\beta}$ and $\vec e_{\beta-1}$.
Then we view the remaining part  as a path from $i$ to $j$ in the graph $\cG_1'$,
whose length is at most as long as that of the original path.
The new path is still non-backtracking, we can recover the original path in $\cG_1$ from the new path in $\cG'_1$ since $v_{m+1}\neq v_{m+1}'$.
Therefore the total number of non-backtracking paths from $i$ to $j$ of length less than $\dist_{\cG}(i,j)+k=\dist_{\cG'_1}(i,j)+k$ in $\cG'_1$ is at least $t_k$.

For any $m\in \qq{0,\dist_{\cG_1}(i,j)-2}$, if in the new graph $|\{v:\dist_{\cG_1}(i,v)=m+1\}|\geq 2$,
we can repeat the above process to reduce it by one.
We repeat this procedure as long as possible, choosing at every step edges $e$ and $e'$ arbitrarily
such that the conditions are satisfied.
Finally, we obtain a graph $\cG_2$ (which again is not unique)
that has exactly $\dist_{\cG_2}(i,j)+1$ vertices, $\{v_0=i,v_1,v_2,\dots,v_{\dist_{\cG_2}(i,j)}=j\}$,
such that $\dist_{\cal G_2}(i,v_m)=m$ for $m\in\qq{0,\dist_{\cG_2(i,j)}}$.
The excess of $\cG_2$ is $\n$.

\smallskip\noindent
\emph{Step 3.}
In the final step, given any edge $e$ from $v_m$ to $v_{m+1}$, if it is the
only edge from $v_m$ to $v_{m+1}$, we shrink it to a point.
This preserves non-backtracking paths, and it reduces the distance between $i$ and $j$ by one.
By shrinking all edges of multiplicity one, we obtain a graph $\cal G_3$.
The number of non-backtracking paths from $i$ to $j$ of length less than $\dist_{\cal G_3}(i,j)+k$ is at least $t_k$, and the excess of $\cG_3$ is $\n$.

\smallskip\noindent
\emph{Final step.}
To bound the number of non-backtracking paths from $i$ to $j$ in $\cG$, it suffices to
estimate the number of non-backtracking paths from $i$ to $j$ in the graph $\cG_3$.
Let $\ell=\dist_{\cal G_3}(i,j)$,
$s$ be the total number of self-loops in $\cal G_3$,
$w_m+1$ the multiplicity of the edge $\{v_{m-1},v_m\}$, for $m\in\qq{1, \ell}$,
and set $w=\max_{1\leq m\leq \ell} w_m$.
Since $\cal G_3$ has excess $\n$,
$s+\sum_{m=1}^{\ell} w_m=\n$. The maximum degree of the graph $\cal G_3$ is bounded by $2s+2w+2$. 
Now any non-backtracking path from $i$ to $j$ of length $\ell+k$
necessarily contains the edges $(v_0,v_1),(v_1,v_2),\dots, (v_{\ell-1},v_{\ell})$,
and for each of them there are $w_1+1,w_2+2,\dots, w_\ell+1$ choices respectively.
For other steps, there are at most $2s+1+2w$ choices. The total number of such paths is bounded by 
\begin{align}\label{edgebound}
  {\ell+k\choose \ell}(2s+1+2w)^k\prod_{m=1}^{\ell}(w_m+1),
\end{align}
under the condition $s+\sum_{m=1}^{\ell} w_m=\n$.
Note that \eqref{edgebound} increases if we decrease $s$ by $1$ and increase some $w_m$ in such a way that $w$ increases by $1$.
Therefore \eqref{edgebound} achieves its maximum at $s=0$. We denote  
\begin{align*}
  a_k:= {\ell+k\choose \ell}(1+2w)^k\prod_{m=1}^{\ell}(w_m+1),
\end{align*}
Since $1+n \leq 2^n$ for any $n\in\N_0$ and $\sum_{m=1}^{\ell} w_m= \n$, we then have $a_0\leq \prod_{m=1}^{\ell}(w_m+1)\leq 2^\n$.
For $a_k$ with $k\geq 1$, notice that $\n = \sum_{m=1}^\ell w_m \geq w + (\ell-1)$ so that $w\leq \n-(\ell-1)$, and thus
  \begin{align*}
  a_k\leq \frac{\ell+k}{k}(1+2w)a_{k-1}
  \leq (\ell+1)(2\n-2\ell+3)a_{k-1}
  \leq \frac{(2\n+5)^2}{8}a_{k-1}\leq (2^\n-1) a_{k-1},
  \end{align*}
  given that $\n\geq 6$.
  Therefore 
  \begin{align*}
   t_k\leq a_0+a_1+\cdots +a_{k-1}
   \leq 2^{\n}\left(1+(2^\n-1)+\cdots (2^\n-1)^{k-1}\right) 
 \leq 2^{\n k}.
  \end{align*}
  This finishes the proof. 
\end{proof}

\begin{proof}[Proof of \eqref{pathN}]
Let $\bH$ be the vertex set of $\cH$, $\n_0$ be the excess of the subgraph $\cal H$,  and $\bar\cH$ the subgraph induced by $\cG$ on $\bH$.
If $\dist_{\cG}(i,j)\geq \ell+1$, then \eqref{pathnum} implies
\begin{align*}
&\#\{\text{non-backtracking paths from $i$ to $j$ of length $\ell+k$, not completely in $\cal H$}\}\\
&\leq \#\{\text{non-backtracking paths from $i$ to $j$ of length $\ell+k$}\}
\leq 2^{\n k},
\end{align*}
and the claim \eqref{pathN} follows.
Therefore,
in the following, assume that $\dist_{\cG}(i,j)\leq \ell$, and also that $\cH,\cG$ are connected
(otherwise, we can replace $\cH$ by its connected component containing $i$ and $j$,
and $\cG$ by its connected component containing $\cH$).
For any non-backtracking path from $i$ to $j$ which is not completely contained in $\cH$,
let $\vec e$ be the first edge in the path which does not belong to $\cH$.
There are three possibilities for such edge $e$:
(\rn{1}) $e\in \bar\cH$. We denote the set of such edges by $E_1$.
(\rn{2}) If we remove $e$ from $\cG$, then $\cG\setminus\{e\}$ breaks into two connected components.
It is necessary that the component not containing $i,j$ contains cycles. We denote the set of such edges by $E_2$.
(\rn{3}) $e\not\in \bar\cH$, and if we remove $e$ from $\cG$, $\cG\setminus\{e\}$ is still connected. We denote the set of such edges by $E_3$. 

We consider the graph $\cG\setminus\{E_1\cup E_2\cup E_3\}$, from $\cG$ by removing edges $E_1\cup E_2\cup E_3$. It consists some many connected components, one corresponds to the graph $\cH$, others are in one-to-one correspondence with the connected components of $\cG\setminus \bar\cH$, the graph from removing $\cH$ from $\cG$. Notice from the definition of these edge sets $E_1, E_2, E_3$, each connected component of $\cG\setminus \bar \cH$ contains exactly one edges in $E_2$ or at least two edges in $E_3$.  Therefore, $\cG\setminus\{E_1\cup E_2\cup E_3\}$ has at most $1+|E_2|+|E_3|/2$ connected components, where $1$ represents the component $\cH$. For the excess of $\cG\setminus\{E_1\cup E_2\cup E_3\}$, since its subgraph $\cH$ has excess $\n_0$, and each new components, due to removing of edges in $E_2$, has excess at least $1$, $\cG\setminus\{E_1\cup E_2\cup E_3\}$ has excess at least $\n_0+|E_2|$. 

\begin{claim}
\begin{align}\label{sumedge}
2|E_1|+|E_2|+|E_3|\leq 2(\n-\n_0)
\end{align}
\end{claim}
\begin{proof}
To prove \eqref{sumedge}, for any finite graph $\cal X$, set
\begin{align}\label{apc-c}
  \chi(\cal X) = \#\text{connected components}(\cX)-\text{excess}(\cX).
\end{align}
By the definition of excess, $\chi(\cal X)= \#\text{vertices}(\cX)- \#\text{edges}(\cX)$, 
we have $\chi(\cal X \setminus e) = \chi(\cal X) + 1$ for any graph $\cal X$ and any edge $e$ in $\cal X$.
Since the graph $\cG$ is connected and has excess at most $\n$, it follows that $\chi(\cG)\geq 1-\n$.
Thus if we remove $E_1\cup E_2\cup E_3$ from $\cG$, the remaining graph has excess at least $\n_0+|E_2|$ and at most $1+|E_2|+|E_3|/2$ connected components. Therefore
\begin{align*}
1+|E_2|+|E_3|/2-|E_2|-\n_0\geq \chi(\cG\setminus(E_1\cup E_2\cup E_3))\geq 1-\n+|E_1|+|E_2|+|E_3|,
\end{align*}
and thus $ |E_1|+|E_2|+|E_3|/2\leq \n-\n_0$. \eqref{sumedge} follows.
\end{proof}

In the following we count the number of length $\ell+k$ non-backtracking paths from $i$ to $j$, containing $\vec e=(i_1,j_1)$ as the first edge not in $\cal H$, i.e., $\{i_1,j_1\}\in E_1\cup E_2\cup E_3$. Let $\dist_{\cG}(i,i_1)=\ell_1$ and $\dist_{\cG}(j_1,j)=\ell_2$. Since $\{i_1,j_1\}$ is not in $\cal H$, it is necessary that $\ell_1+\ell_2\geq \ell$. Thus, $\vec e$ must be the $\ell_1+1, \ell_1+2,\dots$, or $(\ell_1+k)$-th step in the path. The total number of such non-backtracking paths is bounded by
\begin{align*}
& \sum_{k_1=1}^{k}\#\{\text{non-backtracking paths from $i$ to $i_1$ of length $\ell_1+k_1-1$, in $\cal H$}\}\\
& \qquad \times \#\{\text{non-backtracking paths from $j_1$ to $j$ of length $\ell+k-\ell_1-k_1$, in $\cal G$}\}\\
&\leq\sum_{k_1=1}^{k} 2^{\n_0k_1}2^{\n(k-k_1+1)}\leq 2^{\n(k+1)}\sum_{k_1=1}^k 2^{(\n_0-\n)k_1}.
\end{align*}
Since by \eqref{sumedge}, $2|E_1|+|E_2|+|E_3|\leq 2(\n-\n_0)$, there are at most $2(\n-\n_0)$ choices for the oriented edge $\vec e$, the total number of such non-backtracking paths is bounded by
\begin{align*}
&\#\{\text{non-backtracking paths from $i$ to $j$ of length $\ell+k$, not completely in $\cal H$}\}\\
& \leq 2(\n-\n_0)2^{\n(k+1)}\sum_{k_1=1}^k 2^{(\n_0-\n)k_1}\leq 2^{\n(k+1)+1}.
\end{align*}
This completes the proof.
\end{proof}

\subsection{Proof of Lemma~\ref{l:distancextoboundary}}
\label{app:distxtobd}

 To understand the distances $\dist_{\tcG}(x,i)$ for all $i\in \T_\ell$, we need some more notations.
A {\it simple pruning}\cite[Definition 4.4]{MR2437174} is the operation of removing one leaf and its incident edge from a graph.
By repeating pruning on the graph $\tcG_0$, we get a graph $\tcG_2$ with vertex set $\tilde{\bG}_2$, such that it contains at most two leave vertices: $1$ and $x$.
\begin{claim}
\begin{align}\label{sumedge2}
|\tilde{\bG}_2\cap \T_k|\leq 2\n+1,\quad 0\leq k\leq \ell.
\end{align}
\end{claim}
\begin{proof}
For $k=0$, \eqref{sumedge2} holds trivially, $|\tilde{\bG}_2\cap \T_0|=1\leq 2\n+1$. For $k\geq 1$, say $\tilde{\bG}_2\cap \T_k=\{v_1,v_2,\dots, v_m\}$.
By our construction of $\tcG_2$, there are vertices $v_1',v_2',\dots,v_m'\in \T_{k-1}$
such that the edges $\{v_1',v_1\}, \{v_2',v_2\},\dots, \{v_m',v_m\}\in \tcG_2$.
For any $i\in\qq{1,m}$, if we remove the edge $\{v_i',v_i\}$ from $\tcG_2$, the graph $\tcG_2$ will either still be connected;
or it will break into two connected components, one contains vertex $1$,
and the other contains vertex $x$ or some cycles.
Let $m_1, m_2$  the number of edges in the first case and second case respectively. If we remove all edges $\{v_1',v_1\}, \{v_2',v_2\},\dots, \{v_m',v_m\}$, there will be at most $1+m_1/2+m_2$ connected components, and at least excess $m_2-{\bf 1}_{m_2>0}$ left. Notice that the graph $\tcG_2$ is connected and has excess at most $\n$. Recall the function $\chi$ as in \eqref{apc-c}, we have
\begin{align*}
1+m_2+m_1/2-(m_2-{\bf 1}_{m_2>0})\geq \chi(\tcG_2\setminus\{\{v_1',v_1\},\dots, \{v_m',v_m\}\})\geq 1-\n+m_1+m_2.
\end{align*}
Therefore $m_1+m_2\leq 2\n+1$, and the claim follows.
\end{proof}

With the above preparations, we can prove Lemma~\ref{l:distancextoboundary} as follows.

\begin{proof}[Proof of Lemma~\ref{l:distancextoboundary}]
Fix a geodesic $\cal P$ (viewed as a sequence of oriented edges) in $\tcG_0$ from vertex $x$ to vertex $i\in \T_\ell$, there are three possibilities for its step $(v',v)$: (\rn{1}) the edge is downward, i.e. $\dist_{\tcG_0}(1,v)=\dist_{\tcG_0}(1,v')+1$;  (\rn{2}) the edge is horizontal, i.e. $\dist_{\tcG_0}(1,v)=\dist_{\tcG_0}(1,v')$, in this case $v\in \tilde{\bG}_2$; (\rn{3}) the edge is upward, i.e. $\dist_{\tcG_0}(1,v)=\dist_{\tcG_0}(1,v')-1$, in this case $v\in \tilde{\bG}_2$. We denote $(v',v)$ the last step in $\cal P$, which is horizontal or upward. Then $v\in \tilde{\bG}_2$ and we say the vertex $i$ is associated with the vertex $v$ (which may not be unique). By our choice of $(v',v)$, the steps from $v$ to $i$ in $\cal P$ are all downward,  thus $v \in \T$. Moreover we have the estimate: for any vertex $i\in \T_\ell$ associated with $v$
\begin{align*}
\dist_{\tcG_0}(x,i)
\geq& \left|\dist_{\tcG_0}(1,x)-\dist_{\tcG_0}(1,v)\right|+\dist_{\tcG_0}(v,i)\\
=&\left|\dist_{\tcG_0}(1,x)-\dist_{\tcG_0}(1,v)\right|+\left|\ell-\dist_{\tcG_0}(1,v)\right|.
\end{align*} 
Especially, if $v\in \T_{\ell_3}$, i.e. $v$ is distance $\ell_3$ from vertex $1$, the above relation simplifies to
\begin{align*}
\dist_{\tcG_0}(x,i)\geq |\ell_1-\ell_3|+(\ell-\ell_3),
\end{align*}
and by noticing $q<1$, we have
\begin{align}\label{qdistancebound}
q^{\dist_{\tcG_0}(x,i)}\leq \left\{
\begin{array}{cc}
q^{\ell+\ell_1-2\ell_3},& \text{ if }\ell_3\leq \ell_1,\\
q^{\ell-\ell_1}, & \text { if }\ell_3\geq \ell_1.
\end{array}
\right.
\end{align}
In this way, each vertex $i\in \T_\ell$ is associated with some vertex $v\in \tilde{\bG}_2$. If $v\in \tilde{\bG}_2\cap \T_{\ell_3}$,
the total number of vertices in $\T_\ell$ associated with $v$ is at most $(d-1)^{\ell-\ell_3}$, since they are all distance $\ell-\ell_3$ away from $v$. The total number of vertices $i\in \T_\ell$ associated with some $v\in \tilde{\bG}_2\cap \{\T_{\ell_1}\cup \T_{\ell_1+1}\cdots\cup \T_{\ell}\}$ is bounded by $(2\n+1)(1+(d-1)+\cdots+(d-1)^{\ell_1})\leq 2(\n+1)(d-1)^{\ell_1}$, provided that $d\geq 2\n+3$.  Notice that we have the decomposition
\begin{align*}
\{q^{\dist_{\tcG_0}(x,i)}:i\in \T_\ell\}=\cup_{\ell_3\in \qq{0,\ell_1-1}}\{q^{\dist_{\tcG_0}(x,i)}:i\in \T_\ell, \text{$i$ is associated with some $v\in \tilde{\bG}_2\cap \T_{\ell_3}$}\}\\
\cup\{q^{\dist_{\tcG_0}(x,i)}:i\in \T_\ell, \text{$i$ is associated with some $v\in \tilde{\bG}_2\cap \{\T_{\ell_1}\cup \T_{\ell_1+1}\cdots\cup \T_{\ell}\}$}\}.
\end{align*}
Lemma \ref{l:distancextoboundary} follows by combining with \eqref{qdistancebound} and \eqref{sumedge2}.
\end{proof}

\section{Properties of the Green's functions}
\label{app:Green}

Throughout this paper, we repeatedly use some (well-known) identities for Green's functions,
which we collect in this appendix.

\subsection{Resolvent identity}

The following well-known identity is referred as resolvent identity:
for two invertible matrices $A$ and $B$ of the same size, we have
\begin{equation} \label{e:resolv}
  A^{-1} - B^{-1} = A^{-1}(B-A)B^{-1}=B^{-1}(B-A)A^{-1}.
\end{equation}

\subsection{Schur complement formula}

Given an $N\times N$ matrix $M$ and an index set $\T \subset \qq{N}$, recall that we denote by
$M|_\T$ the $\T \times \T$-matrix obtained by restricting $M$ to $\T$,
and that by $M^{(\T)} = M|_{\qq{N}\setminus \T}$ we denote the matrix obtained by removing
the rows and columns with indices in $\T$.
Thus, for any $\T \subset \qq{N}$,
any symmetric matrix $H$ can be written (up to rearrangement of indices) in the block form
\begin{equation}
  H = \begin{bmatrix} A& B'\\ B &D  \end{bmatrix},
\end{equation}
with $A=H|_{\T}$ and $D=H^{(\T)}$.
The Schur complement formula asserts that, for any $z\in \C_+$,
\begin{equation} \label{e:Schur}
 G=(H-z)^{-1}= \begin{bmatrix}
   (A-B'\GT B)^{-1} & -(A-B'\GT B)^{-1}B'\GT \\
   -\GT B(A-B'\GT B)^{-1} & \GT+\GT B(A-B'\GT B)^{-1}B'\GT 
 \end{bmatrix},
\end{equation}
where $\GT=(D-z)^{-1}$.
Throughout the paper, we often use the following special cases of \eqref{e:Schur}:
\begin{align} \begin{split}\label{e:Schur1}
  G|_{\T} &= (A-B'\GT B)^{-1},\\
  G|_{\T^c}-G^{(\T)}&=G|_{\T^c\T}(G|_{\T})^{-1}G|_{\T\T^c},\\
  G|_{\T\T^c}&=-G|_{\T}B'\GT,
  \end{split}
\end{align}
as well as the special case
\begin{equation} \label{e:Schurixj}
G_{ij}^{(k)} = G_{ij}-\frac{G_{ik}G_{kj}}{G_{kk}}.
\end{equation}

\subsection{Ward identity}

For any symmetric $N\times N$ matrix $H$, its Green's function $G(z)=(H-z)^{-1}$ satisfies
the \emph{Ward identity}
\begin{equation} \label{e:Ward}
  \sum_{j=1}^{N} |G_{ij}(z)|^2= \frac{\im G_{jj}(z)}{\eta},
\end{equation}
where $\eta=\Im [z]$. This identity follows from \eqref{e:resolv} with $A = H-z$ and $B=(H- z)^*$.
In particular, \eqref{e:Ward} provides a bound for the sum $\sum_{j=1}^{N} |G_{ij}(z)|^2$
in terms of the diagonal of the Green's function.
For an explanation why this algebraic identity has the interpretation of a Ward, see e.g.\ \cite[p.147]{MR3204347}.

\subsection{Covering map}

For any vertex $i$, the vector $(G_{i1}, G_{i2}, G_{i3}, \dots)\in \ell^2(\qq{N})$ is uniquely determined by the following relations:
\begin{align}\label{relation}
\begin{split}
1+zG_{ii}=\frac{1}{\sqrt{d-1}}\sum_{k:i\sim k}G_{ik}\\
zG_{ij}=\frac{1}{\sqrt{d-1}}\sum_{k:j\sim k}G_{ik}
\end{split}
\end{align}
where $l \sim k$ denotes that $l$ and $k$ are adjacent in $\cG$, i.e., that $A_{kl}=1$.

\begin{lemma}
Given a covering $\pi:\tcG \to \cG$ of graphs,
denote the Green's function of $\tcG$ by $\tG$ and that of $\cG$ by $G$. Then
for all vertices $i,j$ in $\cG$, the Green's functions obey
\begin{align}\label{defg-bis}
G_{ij}=\sum_{y:\pi(y)=j}\tilde{G}_{xy}.
\end{align}
\end{lemma}

\begin{proof}[Proof of \eqref{defg-bis}]
We give the proof for simple graphs $\cG, \tcG$.
(The statement also holds for graphs with self-loops and multiple edges
if $\sum_{k:i\sim k}$ is interpreted as the sum of all the oriented edges $(i,k)$;
especially, a self-loop should be counted twice.)
Clearly, $\tG$ satisfies the relations \eqref{relation} with $G$ replaced by $\tG$.
For any fixed $x\in \tcG$ such that $\pi(x)=i$, we can define:
\begin{align}\label{defgcopy}
G_{ij}=\sum_{y:\pi(y)=j}\tilde{G}_{xy},
\end{align}
if the right-hand side is summable.
Assuming that for any $j$ the right-hand side of \eqref{defgcopy} is well defined,
we verify that $(G_{ij})_j$ satisfies the relation \eqref{relation}, and thus that it gives the Green's function of $H$.
Indeed,
\begin{align*}
1+zG_{ii}&=1+z\sum_{y:\pi(y)=i}\tilde G_{xy}=1+z\tilde G_{xx}+z\sum_{y:\pi(y)=i,y\neq x}\tilde G_{xy}\\
&=\frac{1}{\sqrt{d-1}}\sum_{w:w\sim x}\tilde G_{xw}+\frac{1}{\sqrt{d-1}}\sum_{y:\pi(y)=i,y\neq x}\sum_{w:w\sim y}\tilde G_{xw}
=\frac{1}{\sqrt{d-1}}\sum_{y:\pi(y)=i}\sum_{w:w\sim y}\tilde G_{xw}.
\end{align*}
Since there is no self-loop and multi-edge in our graph $\cal G$, for any $y_1\neq y_2$ with $\pi(y_1)=\pi(y_2)=i$ and $w_1\sim y_1$ and $w_2\sim y_2$, it is necessary that $w_1\neq w_2$. Therefore:
\begin{equation*}
\frac{1}{\sqrt{d-1}}\sum_{y:\pi(y)=i}\sum_{w:w\sim y}\tilde G_{xw}
=\frac{1}{\sqrt{d-1}}\sum_{k:i\sim k}\sum_{w:\pi(w)=k}\tilde G_{xw}
=\frac{1}{\sqrt{d-1}}\sum_{k:i\sim k}G_{ik}.
\end{equation*}
Similarly, for the second relation \eqref{relation},
\begin{equation*}
zG_{ij}=z\sum_{y:\pi(y)=j}\tilde G_{xy}
=\frac{1}{\sqrt{d-1}}\sum_{y:\pi(y)=j}\sum_{w:w\sim y}\tilde G_{xw}
=\frac{1}{\sqrt{d-1}}\sum_{k:j\sim k}\sum_{w:\pi(w)=k}\tilde G_{xw}=\frac{1}{\sqrt{d-1}}\sum_{k:j\sim k}G_{ik},
\end{equation*}
as needed.
\end{proof}

\bibliography{all}

\def\polhk#1{\setbox0=\hbox{#1}{\ooalign{\hidewidth
  \lower1.5ex\hbox{`}\hidewidth\crcr\unhbox0}}}
\begin{thebibliography}{10}

\bibitem{1509.03368}
B.~Adlam and Z.~Che.
\newblock Spectral statistics of sparse random graphs with a general degree
  distribution, 2015.
\newblock Preprint, arxiv:1509.03368.

\bibitem{MR2215610}
M.~Aizenman, R.~Sims, and S.~Warzel.
\newblock Absolutely continuous spectra of quantum tree graphs with weak
  disorder.
\newblock {\em Comm. Math. Phys.}, 264(2):371--389, 2006.

\bibitem{MR2259096}
M.~Aizenman, R.~Sims, and S.~Warzel.
\newblock Fluctuation based proof of the stability of ac spectra of random
  operators on tree graphs.
\newblock In {\em Recent advances in differential equations and mathematical
  physics}, volume 412 of {\em Contemp. Math.}, pages 1--14. Amer. Math. Soc.,
  Providence, RI, 2006.

\bibitem{MR2257129}
M.~Aizenman, R.~Sims, and S.~Warzel.
\newblock Stability of the absolutely continuous spectrum of random
  {S}chr\"odinger operators on tree graphs.
\newblock {\em Probab. Theory Related Fields}, 136(3):363--394, 2006.

\bibitem{MR2241809}
M.~Aizenman and S.~Warzel.
\newblock Persistence under weak disorder of {AC} spectra of quasi-periodic
  {S}chr\"odinger operators on trees graphs.
\newblock {\em Mosc. Math. J.}, 5(3):499--506, 742, 2005.

\bibitem{MR2329431}
M.~Aizenman and S.~Warzel.
\newblock The canopy graph and level statistics for random operators on trees.
\newblock {\em Math. Phys. Anal. Geom.}, 9(4):291--333 (2007), 2006.

\bibitem{MR2885163}
M.~Aizenman and S.~Warzel.
\newblock Disorder-induced delocalization on tree graphs.
\newblock In {\em Mathematical results in quantum physics}, pages 107--109.
  World Sci. Publ., Hackensack, NJ, 2011.

\bibitem{MR2905787}
M.~Aizenman and S.~Warzel.
\newblock Absolutely continuous spectrum implies ballistic transport for
  quantum particles in a random potential on tree graphs.
\newblock {\em J. Math. Phys.}, 53(9):095205, 15, 2012.

\bibitem{MR3055759}
M.~Aizenman and S.~Warzel.
\newblock Resonant delocalization for random {S}chr\"odinger operators on tree
  graphs.
\newblock {\em J. Eur. Math. Soc. (JEMS)}, 15(4):1167--1222, 2013.

\bibitem{MR3405613}
M.~Aizenman and S.~Warzel.
\newblock On the ubiquity of the {C}auchy distribution in spectral problems.
\newblock {\em Probab. Theory Related Fields}, 163(1-2):61--87, 2015.

\bibitem{MR3364516}
M.~Aizenman and S.~Warzel.
\newblock {\em Random operators}, volume 168 of {\em Graduate Studies in
  Mathematics}.
\newblock American Mathematical Society, Providence, RI, 2015.
\newblock Disorder effects on quantum spectra and dynamics.

\bibitem{AEK16}
O.H. Ajanki, L.~Erd{\H{o}}s, and T.~Kr{\"u}ger.
\newblock Universality for general wigner-type matrices.
\newblock {\em Probability Theory and Related Fields}, pages 1--61, 2016.

\bibitem{MR875835}
N.~Alon.
\newblock Eigenvalues and expanders.
\newblock {\em Combinatorica}, 6(2):83--96, 1986.
\newblock Theory of computing (Singer Island, Fla., 1984).

\bibitem{1512.06624}
N.~Anantharaman.
\newblock Quantum ergodicity on large graphs, 2015.
\newblock Preprint, arXiv:1512.06624.

\bibitem{MR3322309}
N.~Anantharaman and E.~Le~Masson.
\newblock Quantum ergodicity on large regular graphs.
\newblock {\em Duke Math. J.}, 164(4):723--765, 2015.

\bibitem{1607.04785}
A.~Backhausz and B.~Szegedy.
\newblock On the almost eigenvectors of random regular graphs, 2016.
\newblock Preprint. arXiv:1607.04785.

\bibitem{MR1194071}
H.~Bass.
\newblock The {I}hara-{S}elberg zeta function of a tree lattice.
\newblock {\em Internat. J. Math.}, 3(6):717--797, 1992.

\bibitem{1505.06700-aop}
R.~Bauerschmidt, J.~Huang, A.~Knowles, and H.-T. Yau.
\newblock Bulk eigenvalue statistics for random regular graphs.
\newblock {\em Ann. Probab.}, 2016+.
\newblock To appear.

\bibitem{1503.08702-cpam}
R.~Bauerschmidt, A.~Knowles, and H.-T. Yau.
\newblock Local semicircle law for random regular graphs.
\newblock {\em Comm. Pure Appl. Math.}, 2016+.
\newblock To appear.

\bibitem{Bord15}
C.~Bordenave.
\newblock A new proof of {F}riedman's second eigenvalue {T}heorem and its
  extension to random lifts.
\newblock Preprint, arXiv:1502.04482, 2015.

\bibitem{MR3541852}
P.~Bourgade, L.~Erd{\H{o}}s, H.-T. Yau, and J.~Yin.
\newblock Fixed energy universality for generalized {W}igner matrices.
\newblock {\em Comm. Pure Appl. Math.}, 69(10):1815--1881, 2016.

\bibitem{1609.09022}
P.~Bourgade, J.~Huang, and H.-T. Yau.
\newblock Eigenvector statistics of sparse random matrices, 2016.
\newblock Preprint, arXiv:1609.09022.

\bibitem{BY2016}
P.~Bourgade and H.-T. Yau.
\newblock The eigenvector moment flow and local quantum unique ergodicity.
\newblock {\em Comm. Math. Phys.}, pages 1--48, 2016.

\bibitem{MR3038543}
S.~Brooks and E.~Lindenstrauss.
\newblock Non-localization of eigenfunctions on large regular graphs.
\newblock {\em Israel J. Math.}, 193(1):1--14, 2013.

\bibitem{1505.03887}
S.~Brooks, E.L. Masson, and E.~Lindenstrauss.
\newblock Quantum ergodicity and averaging operators on the sphere, 2015.

\bibitem{MR0391792}
J.M. Combes and L.~Thomas.
\newblock Asymptotic behaviour of eigenfunctions for multiparticle
  {S}chr\"odinger operators.
\newblock {\em Comm. Math. Phys.}, 34:251--270, 1973.

\bibitem{Cook2015}
N.~Cook.
\newblock On the singularity of adjacency matrices for random regular digraphs.
\newblock {\em Probab. Theory Related Fields}, pages 1--58, 2015.

\bibitem{1508.00208}
N.A. Cook.
\newblock The circular law for signed random regular digraphs, 2015.
\newblock Preprint, arXiv:1508.00208.

\bibitem{1510.06013}
N.A. Cook, L.~Goldstein, and T.~Johnson.
\newblock Size biased couplings and the spectral gap for random regular graphs,
  2015.
\newblock Preprint. arXiv:1510.06013.

\bibitem{PhysRevLett.113.046806}
A.~De~Luca, B.L. Altshuler, V.E. Kravtsov, and A.~Scardicchio.
\newblock Anderson localization on the bethe lattice: Nonergodicity of extended
  states.
\newblock {\em Phys. Rev. Lett.}, 113:046806, Jul 2014.

\bibitem{MR3078290}
I.~Dumitriu, T.~Johnson, S.~Pal, and E.~Paquette.
\newblock Functional limit theorems for random regular graphs.
\newblock {\em Probab. Theory Related Fields}, 156(3-4):921--975, 2013.

\bibitem{MR3025715}
I.~Dumitriu and S.~Pal.
\newblock Sparse regular random graphs: spectral density and eigenvectors.
\newblock {\em Ann. Probab.}, 40(5):2197--2235, 2012.

\bibitem{0907.5065}
Y.~Elon.
\newblock Gaussian waves on the regular tree, 2009.
\newblock Preprint, arXiv:0907.5065.

\bibitem{MR2964770}
L.~Erd{\H{o}}s, A.~Knowles, H.-T. Yau, and J.~Yin.
\newblock Spectral statistics of {E}rd{\H o}s-{R}\'enyi {G}raphs {II}:
  {E}igenvalue spacing and the extreme eigenvalues.
\newblock {\em Comm. Math. Phys.}, 314(3):587--640, 2012.

\bibitem{MR3098073}
L.~Erd{\H{o}}s, A.~Knowles, H.-T. Yau, and J.~Yin.
\newblock Spectral statistics of {E}rd{\H o}s-{R}\'enyi graphs {I}: {L}ocal
  semicircle law.
\newblock {\em Ann. Probab.}, 41(3B):2279--2375, 2013.

\bibitem{MR2662426}
L.~Erd{\H{o}}s, S.~P{\'e}ch{\'e}, J.A. Ram{\'{\i}}rez, B.~Schlein, and H.-T.
  Yau.
\newblock Bulk universality for {W}igner matrices.
\newblock {\em Comm. Pure Appl. Math.}, 63(7):895--925, 2010.

\bibitem{MR2639734}
L.~Erd{\H{o}}s, J.A. Ram{\'{\i}}rez, B.~Schlein, and H.-T. Yau.
\newblock Universality of sine-kernel for {W}igner matrices with a small
  {G}aussian perturbation.
\newblock {\em Electron. J. Probab.}, 15:no. 18, 526--603, 2010.

\bibitem{MR2810797}
L.~Erd{\H{o}}s, B.~Schlein, and H.-T. Yau.
\newblock Universality of random matrices and local relaxation flow.
\newblock {\em Invent. Math.}, 185(1):75--119, 2011.

\bibitem{ErdYauBook}
L.~Erd{\H{o}}s and H.-T. Yau.
\newblock {\em {D}ynamical {A}pproach {T}o {R}andom {M}atrix {T}heory}.
\newblock Preliminary version, available at
  \url{http://www.math.harvard.edu/~htyau/}.

\bibitem{MR3372074}
L.~Erd{\H{o}}s and H.-T. Yau.
\newblock Gap universality of generalized {W}igner and {$\beta$}-ensembles.
\newblock {\em J. Eur. Math. Soc. (JEMS)}, 17(8):1927--2036, 2015.

\bibitem{MR2981427}
L.~Erd{\H{o}}s, H.-T. Yau, and J.~Yin.
\newblock Bulk universality for generalized {W}igner matrices.
\newblock {\em Probab. Theory Related Fields}, 154(1-2):341--407, 2012.

\bibitem{MR2437174}
J.~Friedman.
\newblock A proof of {A}lon's second eigenvalue conjecture and related
  problems.
\newblock {\em Mem. Amer. Math. Soc.}, 195(910):viii+100, 2008.

\bibitem{FriedmanKahnSzemeredi}
J.~Friedman, J.~Kahn, and E.~Szemer{\'e}di.
\newblock On the second eigenvalue of random regular graphs.
\newblock In {\em Proceedings of the Twenty-first Annual ACM Symposium on
  Theory of Computing}, STOC '89, pages 587--598, New York, NY, USA, 1989. ACM.

\bibitem{MR3433288}
L.~Geisinger.
\newblock Convergence of the density of states and delocalization of
  eigenvectors on random regular graphs.
\newblock {\em J. Spectr. Theory}, 5(4):783--827, 2015.

\bibitem{MR3369316}
L.~Geisinger.
\newblock Poisson eigenvalue statistics for random {S}chr\"odinger operators on
  regular graphs.
\newblock {\em Ann. Henri Poincar\'e}, 16(8):1779--1806, 2015.

\bibitem{MR2218022}
M.D. Horton, D.B. Newland, and A.A. Terras.
\newblock The contest between the kernels in the {S}elberg trace formula for
  the {$(q+1)$}-regular tree.
\newblock In {\em The ubiquitous heat kernel}, volume 398 of {\em Contemp.
  Math.}, pages 265--293. Amer. Math. Soc., Providence, RI, 2006.

\bibitem{1510.06390}
J.~Huang and B.~Landon.
\newblock Spectral statistics of sparse {E}rd{\H{o}}s-{R}{\'e}nyi graph
  {L}aplacians, 2015.
\newblock Preprint, arXiv:1510.06390.

\bibitem{1504.05170}
J.~Huang, B.~Landon, and H.-T. Yau.
\newblock Bulk universality of sparse random matrices, 2015.
\newblock Preprint, arXiv: 1504.05170.

\bibitem{MR0223463}
Y.~Ihara.
\newblock On discrete subgroups of the two by two projective linear group over
  {${\germ p}$}-adic fields.
\newblock {\em J. Math. Soc. Japan}, 18:219--235, 1966.

\bibitem{MR1691538}
D.~Jakobson, S.D. Miller, I.~Rivin, and Z.~Rudnick.
\newblock Eigenvalue spacings for regular graphs.
\newblock In {\em Emerging applications of number theory ({M}inneapolis, {MN},
  1996)}, volume 109 of {\em IMA Vol. Math. Appl.}, pages 317--327. Springer,
  New York, 1999.

\bibitem{MR1810949}
K.~Johansson.
\newblock Universality of the local spacing distribution in certain ensembles
  of {H}ermitian {W}igner matrices.
\newblock {\em Comm. Math. Phys.}, 215(3):683--705, 2001.

\bibitem{MR3315475}
T.~Johnson.
\newblock Exchangeable pairs, switchings, and random regular graphs.
\newblock {\em Electron. J. Combin.}, 22(1):Paper 1.33, 28, 2015.

\bibitem{MR0109367}
H.~Kesten.
\newblock Symmetric random walks on groups.
\newblock {\em Trans. Amer. Math. Soc.}, 92:336--354, 1959.

\bibitem{MR1492789}
A.~Klein.
\newblock Extended states in the {A}nderson model on the {B}ethe lattice.
\newblock {\em Adv. Math.}, 133(1):163--184, 1998.

\bibitem{MR3034787}
A.~Knowles and J.~Yin.
\newblock Eigenvector distribution of {W}igner matrices.
\newblock {\em Probab. Theory Related Fields}, 155(3-4):543--582, 2013.

\bibitem{1504.03605}
B.~Landon and H.-T. Yau.
\newblock Convergence of local statistics of {D}yson {B}rownian motion, 2015.
\newblock Preprint, arXiv:1504.03605.

\bibitem{MR3245884}
E.~Le~Masson.
\newblock Pseudo-differential calculus on homogeneous trees.
\newblock {\em Ann. Henri Poincar\'e}, 15(9):1697--1732, 2014.

\bibitem{MR2667423}
E.~Lubetzky and A.~Sly.
\newblock Cutoff phenomena for random walks on random regular graphs.
\newblock {\em Duke Math. J.}, 153(3):475--510, 2010.

\bibitem{MR963118}
A.~Lubotzky, R.~Phillips, and P.~Sarnak.
\newblock Ramanujan graphs.
\newblock {\em Combinatorica}, 8(3):261--277, 1988.

\bibitem{1401.0019}
A.D. Luca, A.~Scardicchio, V.E. Kravtsov, and B.L. Altshuler.
\newblock Support set of random wave-functions on the {B}ethe lattice, 2013.
\newblock Preprint, arXiv:1401.0019.

\bibitem{MR3374962}
A.W. Marcus, D.A. Spielman, and N.~Srivastava.
\newblock Interlacing families {I}: {B}ipartite {R}amanujan graphs of all
  degrees.
\newblock {\em Ann. of Math. (2)}, 182(1):307--325, 2015.

\bibitem{MR939574}
G.A. Margulis.
\newblock Explicit group-theoretic constructions of combinatorial schemes and
  their applications in the construction of expanders and concentrators.
\newblock {\em Problemy Peredachi Informatsii}, 24(1):51--60, 1988.

\bibitem{MR629617}
B.D. McKay.
\newblock The expected eigenvalue distribution of a large regular graph.
\newblock {\em Linear Algebra Appl.}, 40:203--216, 1981.

\bibitem{MR790916}
B.D. McKay.
\newblock Asymptotics for symmetric {$0$}-{$1$} matrices with prescribed row
  sums.
\newblock {\em Ars Combin.}, 19(A):15--25, 1985.

\bibitem{MR2097332}
B.D. McKay, N.C. Wormald, and B.~Wysocka.
\newblock Short cycles in random regular graphs.
\newblock {\em Electron. J. Combin.}, 11(1):Research Paper 66, 12 pp.
  (electronic), 2004.

\bibitem{PhysRevE.90.052109}
F.L. Metz, G.~Parisi, and L.~Leuzzi.
\newblock Finite-size corrections to the spectrum of regular random graphs: An
  analytical solution.
\newblock {\em Phys. Rev. E}, 90:052109, Nov 2014.

\bibitem{MR2433888}
S.J. Miller and T.~Novikoff.
\newblock The distribution of the largest nontrivial eigenvalues in families of
  random regular graphs.
\newblock {\em Experiment. Math.}, 17(2):231--244, 2008.

\bibitem{MR2647344}
I.~Oren and U.~Smilansky.
\newblock Trace formulas and spectral statistics for discrete {L}aplacians on
  regular graphs ({II}).
\newblock {\em J. Phys. A}, 43(22):225205, 13, 2010.

\bibitem{1601.03678}
S.~O'Rourke, V.~Vu, and K.~Wang.
\newblock Eigenvectors of random matrices: A survey, 2016.
\newblock Preprint. arXiv:1601.03678.

\bibitem{MR3385636}
D.~Puder.
\newblock Expansion of random graphs: new proofs, new results.
\newblock {\em Invent. Math.}, 201(3):845--908, 2015.

\bibitem{MR2072849}
P.~Sarnak.
\newblock What is an expander?
\newblock {\em Notices Amer. Math. Soc.}, 51(7):762--763, 2004.

\bibitem{MR3204183}
U.~Smilansky.
\newblock Discrete graphs---a paradigm model for quantum chaos.
\newblock In {\em Chaos}, volume~66 of {\em Prog. Math. Phys.}, pages 97--124.
  Birkh\"auser/Springer, Basel, 2013.

\bibitem{MR3204347}
T.~Spencer.
\newblock Duality, statistical mechanics, and random matrices.
\newblock In {\em Current developments in mathematics 2012}, pages 229--260.
  Int. Press, Somerville, MA, 2013.

\bibitem{MR2784665}
T.~Tao and V.~Vu.
\newblock Random matrices: universality of local eigenvalue statistics.
\newblock {\em Acta Math.}, 206(1):127--204, 2011.

\bibitem{MR2930379}
T.~Tao and V.~Vu.
\newblock Random matrices: universal properties of eigenvectors.
\newblock {\em Random Matrices Theory Appl.}, 1(1):1150001, 27, 2012.

\bibitem{MR2768284}
A.~Terras.
\newblock {\em Zeta functions of graphs}, volume 128 of {\em Cambridge Studies
  in Advanced Mathematics}.
\newblock Cambridge University Press, Cambridge, 2011.
\newblock A stroll through the garden.

\bibitem{MR1967891}
A.~Terras and D.~Wallace.
\newblock Selberg's trace formula on the {$k$}-regular tree and applications.
\newblock {\em Int. J. Math. Math. Sci.}, (8):501--526, 2003.

\bibitem{MR2999215}
L.V. Tran, V.H. Vu, and K.~Wang.
\newblock Sparse random graphs: eigenvalues and eigenvectors.
\newblock {\em Random Structures Algorithms}, 42(1):110--134, 2013.

\bibitem{MR2432537}
V.~Vu.
\newblock Random discrete matrices.
\newblock In {\em Horizons of combinatorics}, volume~17 of {\em Bolyai Soc.
  Math. Stud.}, pages 257--280. Springer, Berlin, 2008.

\bibitem{MR1725006}
N.C. Wormald.
\newblock Models of random regular graphs.
\newblock In {\em Surveys in combinatorics, 1999 ({C}anterbury)}, volume 267 of
  {\em London Math. Soc. Lecture Note Ser.}, pages 239--298. Cambridge Univ.
  Press, Cambridge, 1999.

\end{thebibliography}
\bibliographystyle{plain}

\end{document}